\documentclass[12pt]{amsart}
\title{Cónicas y Superficies Cuádricas}
\author{Jaime Chica\\ Jonathan Taborda}
\email{taborda50@gmail.com}
\usepackage{amsmath,amssymb,amsthm,graphicx,graphics,cancel,pifont,pb-diagram,epstopdf,yfonts,lettrine,pdfcprot,eso-pic,wallpaper,mathrsfs,mathpazo}
\usepackage[activeacute,english,spanish]{babel}
\usepackage[T1]{fontenc}
\usepackage{appendix}
\usepackage{hyperref}
\hypersetup{bookmarksopen,bookmarksnumbered,colorlinks,linkcolor=blue,urlcolor=blue}
\newtheorem{defin}{Definici\'{o}n}[section]

\newtheorem{obser}{Observaci\'{o}n}[section]
\newtheorem{prop}{Proposici\'{o}n}[section]
\newtheorem{ejer}{Ejercicio}[section]
\newtheorem{ejem}{Ejemplo}[section]

\newtheorem{sol}{Solución}[section]

\newtheorem{pro}{Problema}[section]

\newtheorem{cas}{Caso}[section]

\begin{document}\maketitle
\selectlanguage{english}
\begin{abstract}
There are two problems Analytical Geometry with facing anyone who studies this discipline:
define the nature of the locus represented by the general equation $2^{do}$ degree in two or three variables:\\
That curve represents the plane?\\
What surface is in space?\\
These two problems are posed and solved by applying the study of matrices and spectral theory.
\end{abstract}
\selectlanguage{spanish}
\tableofcontents
\section{Introducción}
Hay en la Geometría Analítica dos problemas con los que se enfrenta todo el que estudie ésta disciplina:\\
definir la naturaleza del Lugar Geométrico representado por la ecuación general de $2^{do}$ grado en dos o tres variables. Por ejemplo,
\begin{enumerate}
\item[i)] donde la ecuación $3x^2-2xy+y^2-x+y+5=0,$ ?`Que curva representa en el plano? y\\
\item[ii)] donde la ecuación $x^2+2y^2-5z^2+2xy+3xz-4yz+2x+3y-5z+8=0,$ ?`Qué superficie representa en el espacio?
\end{enumerate}
La respuesta en el primer caso es que la curva es una cónica (elipse, parábola, hipérbola) ó una cónica degenerada y en el segundo caso es una superficie cuádrica (cono, cilindro, elipsoide, paraboloide,...) ó un caso degenerado de ellas.\\
En el $1^{er}$ caso es posible estudiar el problema con elementos que proporciona la Trigonometría empleando funciones del ángulo doble para definir el $\measuredangle$ que deben girarse los ejes para conseguir anular el término mixto.\\
Pero para estudiar el segundo problema es imprescindible el empleo de las matrices y los valores propios junto con el Teorema Espectral para matrices simétricas.\\
Estos dos problemas serán planteados y resueltos aplicando, como ya se dijo, las matrices y la teoría espectral que acaba de estudiarse en los dos capítulos anteriores.
\section[Lugares geométricos representados por la ecuación general de segundo grado en dos variables...]{Lugares geométricos representados por la ecuación general de segundo grado en dos variables: $Ax^2+2Bxy+Cy^2+2Dx+2Ey+F=0$}
\lettrine{E}l problema que vamos a abordar es el siguiente: dada la ecuación
\begin{equation}\label{1}
Ax^2+2Bxy+Cy^2+2Dx+2Ey+F=0
\end{equation}
donde $A,B,C,D,E,F$ son constantes reales dadas, ?`Qué conjunto de puntos del plano $xy$ la satisfacen?\\
La ecuación [\ref{1}]:
\begin{itemize}
\item Una componente cuadrática: $$Ax^2+2Bxy+Cy^2=\left(\begin{array}{cc}x&y\end{array}\right)\left(\begin{array}{cc}A&B\\B&C\end{array}\right)\left(\begin{array}{c}x\\y\end{array}\right)$$
\item Una componente lineal: $$2Dx+2Ey=\left(\begin{array}{cc}2D&2E\end{array}\right)\left(\begin{array}{c}x\\y\end{array}\right)$$
\item Un término independiente:$$F$$
\end{itemize}
Asumiremos que No todos los coeficientes $A,B,C$ de la componente cuadrática se anulan, ya que si así fuese, [\ref{1}] representaría la recta del plano $$2Dx+2Ey+F=0$$ y no hay nada que analizar.\\
Como se desprenderá del estudio que vamos a hacer, la ecuación [\ref{1}] puede representar:\\
\[\text{Una cónica}
\begin{cases}
\text{Circunferencia}\\
\text{Elipse}\\
\text{Hipérbola}\\
\text{Parábola}
\end{cases}
\]
O una
$$\text{cónica degenerada}
\begin{cases}
\text{dos rectas paralelas}\\
\text{dos rectas concurrentes}\\
\text{una recta}\\
\text{un punto}\\
\emptyset.\,\,\text{O sea que\,\, $\nexists(x,y)\in\mathbb{R}^2$\,\, que satisfagan [\ref{1}]}
\end{cases}
$$
todo dependerá, en el fondo, de los \underline{invariantes} de [\ref{1}] y de los \underline{valores propios} de la matriz $$M=\left(\begin{array}{cc}A&B\\B&C\end{array}\right)$$
Llamaremos en adelante cónica al lugar geométrico representado por [\ref{1}], o sea, el conjunto de puntos del plano que satisfacen [\ref{1}].\\
Consideremos, pues, la cónica [\ref{1}]. Reducirla es definir el lugar geométrico representado por la ecuación.\\
Asosiado a [\ref{1}] hay dos funciones:
\begin{figure}[ht!]
\begin{center}
  \includegraphics[scale=0.6]{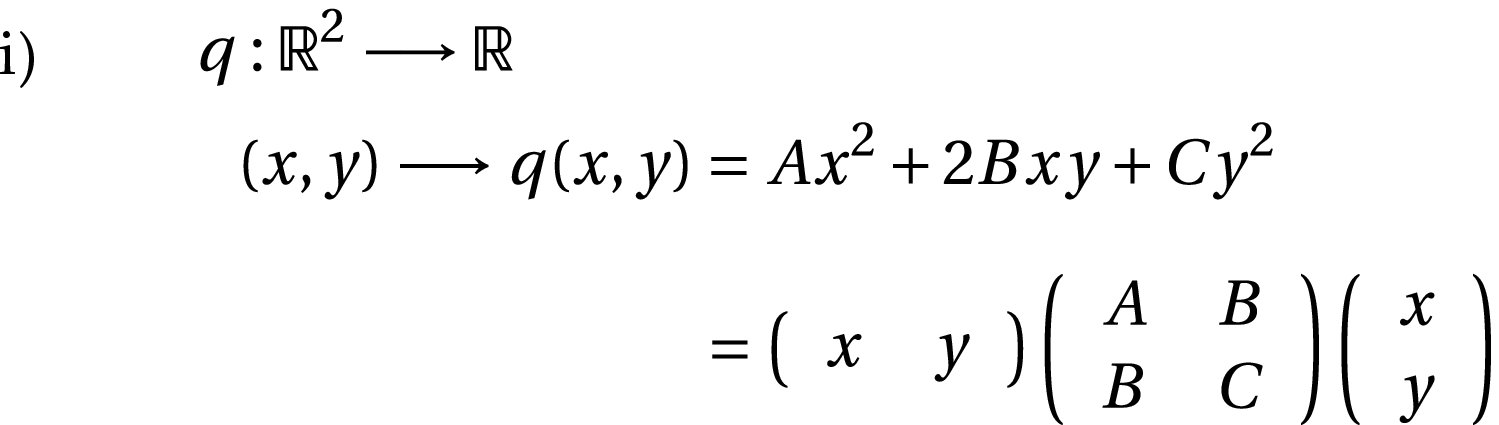}\\
\end{center}
\end{figure}
\newline
$q$ es la f.c (forma cuadrática) asociada a [\ref{1}]
\begin{enumerate}
\item[ii)]
\begin{align*}
f&:\mathbb{R}^2\longrightarrow\mathbb{R}\\
&(x,y)\longrightarrow f(x,y)=q(x,y)+\left(2D\,\,2E\right)\dbinom{x}{y}+F
\end{align*}
\end{enumerate}
El Kernel de $f$ (o el núcleo de $f$) es el conjunto $$\ker f=f^{-1}({0})=\left\{(x,y)\diagup f(x,y)=0\right\}.$$
La cónica [\ref{1}] no es más que el $\ker f$. Además es claro que un punto $P(x,y)$ está en la cónica\,\, $\Longleftrightarrow f(x,y)=0\Longleftrightarrow q(x,y)+\left(2D\,\,2E\right)\dbinom{x}{y}+F=0.$\\
Notese además que
\begin{align*}
\dfrac{\partial f}{\partial x}&=2Ax+2By+2D\\
\dfrac{\partial f}{\partial y}&=2Bx+2Cy+2E
\end{align*}
La naturaleza del lugar representado por [\ref{1}] está íntimamente ligada a ciertos escalares que se construyen con los coeficientes de la matriz simétrica $$\left(\begin{array}{ccc}A&B&D\\B&C&E\\D&E&F\end{array}\right)$$
Ello son:
$$\Delta=\left|\begin{array}{ccc}A&B&D\\B&C&E\\D&E&F\end{array}\right|=ACF-AE^2-B^2F+2BDE-CD^2.$$
$\Delta$ es llamado el invariante cúbico de [\ref{1}] ó el discriminante de la cónica.\\
$$\delta=\left|\begin{array}{cc}A&B\\B&C\end{array}\right|=AC-B^2=M_{33}$$
$M_{33}:$ menor principal $3-3$ de la matriz
\begin{equation}\label{2}
\left(\begin{array}{ccc}A&B&D\\B&C&E\\D&E&F\end{array}\right)
\end{equation}
$\delta$ es llamado el invariante cuadrático lineal de [\ref{1}].\\
$\omega=A+C;$ $\omega$ es llamado el invariante linaeal de [\ref{1}].\\
Estos escalares se llaman los \underline{invariantes de la cónica} porque son cantidades que como se verá no cambian de valor cuando sometemos [\ref{1}] bien sea a una trasposición ó una rotación de ejes. Los menores principales de orden dos de [\ref{2}] son:
$$M_{11}=\left|\begin{array}{cc}C&E\\E&F\end{array}\right|;\hspace{0.5cm}M_{22}=\left|\begin{array}{cc}A&D\\D&F\end{array}\right|;\hspace{0.5cm}M_{33}=\delta=\left|\begin{array}{cc}A&B\\B&C\end{array}\right|$$
De los tres, solo $M_{33}$ es invariante. Estos tres menores, sobre todo $\delta=M_{33}$, van a ser importantes en el problema de la reducción de la cónica.
\begin{ejem}
\begin{enumerate}
\item[1)] Un caso en que el lugar es $\emptyset.$\\
$$f(x,y)=x^2+y^2-4x-6y+24=0$$
Para identificar el lugar tratemos de completar trinomios cuadrados perfectos.\\
\begin{align*}
f(x,y)&=\left(x^2-4x+4\right)+\left(y^2-6y+9\right)+24-13\\
&={\left(x-2\right)}^2+{\left(y-3\right)}^2+11=0
\end{align*}
\end{enumerate}
$$\nexists(x,y)\in\mathbb{R}^2\hspace{0.5cm}\text{t.q}\hspace{0.5cm}f(x,y)=0.$$
El \underline{lugar es $\emptyset$}.\\
\item[2)] Sea $f(x,y)=x^2+y^2-4x-6y+13=0.$ Entonces
\begin{align*}
f(x,y)&=\left(x^2-4x+4\right)+\left(y^2-6y+9\right)+\cancel{13}-\cancel{13}\\
&={\left(x-2\right)}^2+{\left(y-3\right)}^2=0.
\end{align*}
El único punto que satisface $f(x,y)=0$ es (2,3).\\
Así que el \underline{lugar es un punto}.\\
\item[3)] Dos rectas paralelas.\\
Sean
\begin{align}
l(x,y)=3x-2y+2=0\label{3}\\
m(x,y)=3x-2y+1=0\label{4}
\end{align}
Definamos
\begin{align}
f(x,y)&=l(x,y)\cdot m(x,y)\notag\\
&=\left(3x-2y+2\right)\left(3x-2y+1\right)\notag\\
&=9x^2-12xy+4y^2+9x-6y+2=0\label{5}
\end{align}
Es claro que
\begin{align*}
f(x,y)&=0\Longleftrightarrow l(x,y)\cdot m(x,y)=0\\
&\Longleftrightarrow l(x,y)=0\hspace{0.5cm}\text{ó}\hspace{0.5cm}m(x,y)=0\\
&\Longleftrightarrow (x,y)\hspace{0.5cm}\text{está en la recta [\ref{3}] ó en la recta [\ref{4}]}
\end{align*}
Luego el lugar representado por [\ref{5}] consta de dos rectas paralelas.\\
\item[4)] Dos rectas que se cortan.\\
Sean
\begin{align}
l(x,y)=3x-2y+1=0\label{6}\\
m(x,y)=x+y-2=0\label{7}
\end{align}
[\ref{6}] y [\ref{7}] representan dos rectas que se cortan.\\
Definamos
\begin{align}
f(x,y)&=l(x,y)\cdot m(x,y)\notag\\
&=\left(3x-2y+1\right)\left(x+y-2\right)\notag\\
&=3x^2+xy-2y^2-5x+5y-2=0\label{8}
\end{align}
Es claro que
\begin{align*}
f(x,y)&=0\Longleftrightarrow l(x,y)\cdot m(x,y)=0\\
&\Longleftrightarrow l(x,y)=0\hspace{0.5cm}\text{ó}\hspace{0.5cm}m(x,y)=0\\
&\Longleftrightarrow (x,y)\hspace{0.5cm}\text{está en la recta [\ref{6}] ó en la recta [\ref{7}]}
\end{align*}
El lugar répresentado por [\ref{8}] consta de dos rectas que se cortan.\\
\item[5)] Una recta.\\
Sea $$l(x,y)=x+y-2=0$$
Definamos
\begin{align}
f(x,y)&=l(x,y)\cdot l(x,y)=\left(x+y-2\right)\left(x+y-2\right)\notag\\
&=x^2+2xy+y^2-4x-4y+4=0\label{9}
\end{align}
Es claro que $$f(x,y)=0\Longleftrightarrow l(x,y)=x+y-2=0$$
El lugar representado por [\ref{9}] es una recta.
\end{ejem}
\section{Invariantes de una cónica}
\begin{prop}
Consideremos la cónica de la ecuación [\ref{1}]. Los números reales $$\omega=A+C,\hspace{0.5cm}\delta=\left|\begin{array}{cc}A&B\\B&C\end{array}\right|=AC-B^2\hspace{0.5cm}\text{y}\hspace{0.5cm}\Delta=\left|\begin{array}{ccc}A&B&C\\B&C&E\\D&E&F\end{array}\right|$$
no cambian al realizar una traslación, una rotación de los ejes $xy$, o una combinación de ambas transformaciones.
\end{prop}
\begin{proof}
\begin{enumerate}
\item[(1)] Supongamos que realizamos una traslación de los ejes xy al punto $O'(h,k)$. (Fig.1.1). Entonces se definen
ejes X-Y con origen en O' y paralelos a x-y teniendose que
\begin{align*}
x&=X+h\\
y&=Y+k
\end{align*}
que llevamos a [\ref{1}] obteniendose
\begin{figure}[ht!]
\begin{center}
  \includegraphics[scale=0.5]{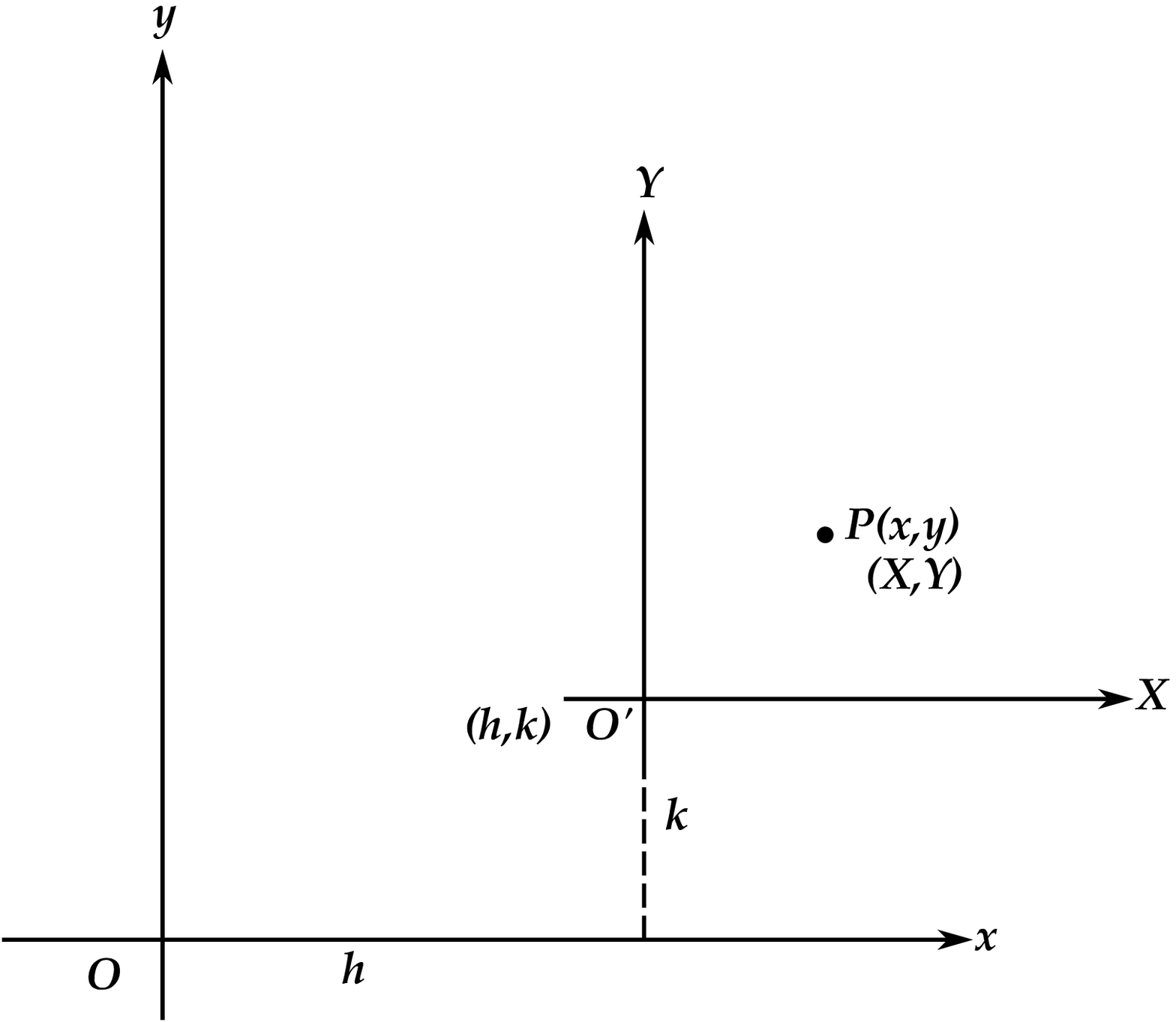}\\
  \caption{}\label{}
\end{center}
\end{figure}
$$A{\left(X+h\right)}^2+2B\left(X+h\right)\left(Y+k\right)+C{\left(Y+k\right)}^2+2D\left(X+h\right)+2E\left(Y+k\right)+F=0$$
O sea que\\
$AX^2+2BXY+CY^2+2\left(Ah+Bk+D\right)X+2\left(Bh+Ck+E\right)Y+\left(Ah^2+2Bhk+Ck^2+2Dh+2Ek+F\right)=0$\\
ecuación que podemos escribir en la forma
\begin{align*}
&A'X^2+2B'XY+C'Y^2+2D'X+2E'Y+F'=0\hspace{0.5cm}\text{siendo}\\
&A'=A\\
&B'=B\\
&C'=C\\
&D'=Ah+Bk+D=\dfrac{1}{2}\left.\dfrac{\partial f}{\partial x}\right)_{h,k}\\
&E'=Bh+Ck+E=\dfrac{1}{2}\left.\dfrac{\partial f}{\partial y}\right)_{h,k}\\
&F'=Ah^2+2Bhk+Ck^2+2Dh+2Ek+F=f(h,k)
\end{align*}
Nótese que cuando se hace una traslación de ejes al punto $(h,k)$ los coeficientes de la parte cuadrática de la ecuación no se tranasforman.\\
Si lo hacen los coeficientes de la parte lineal y el término independiente que ahora es $f(h,k)$.\\
En imagenes.
\begin{figure}[ht!]
\begin{center}
  \includegraphics[scale=0.5]{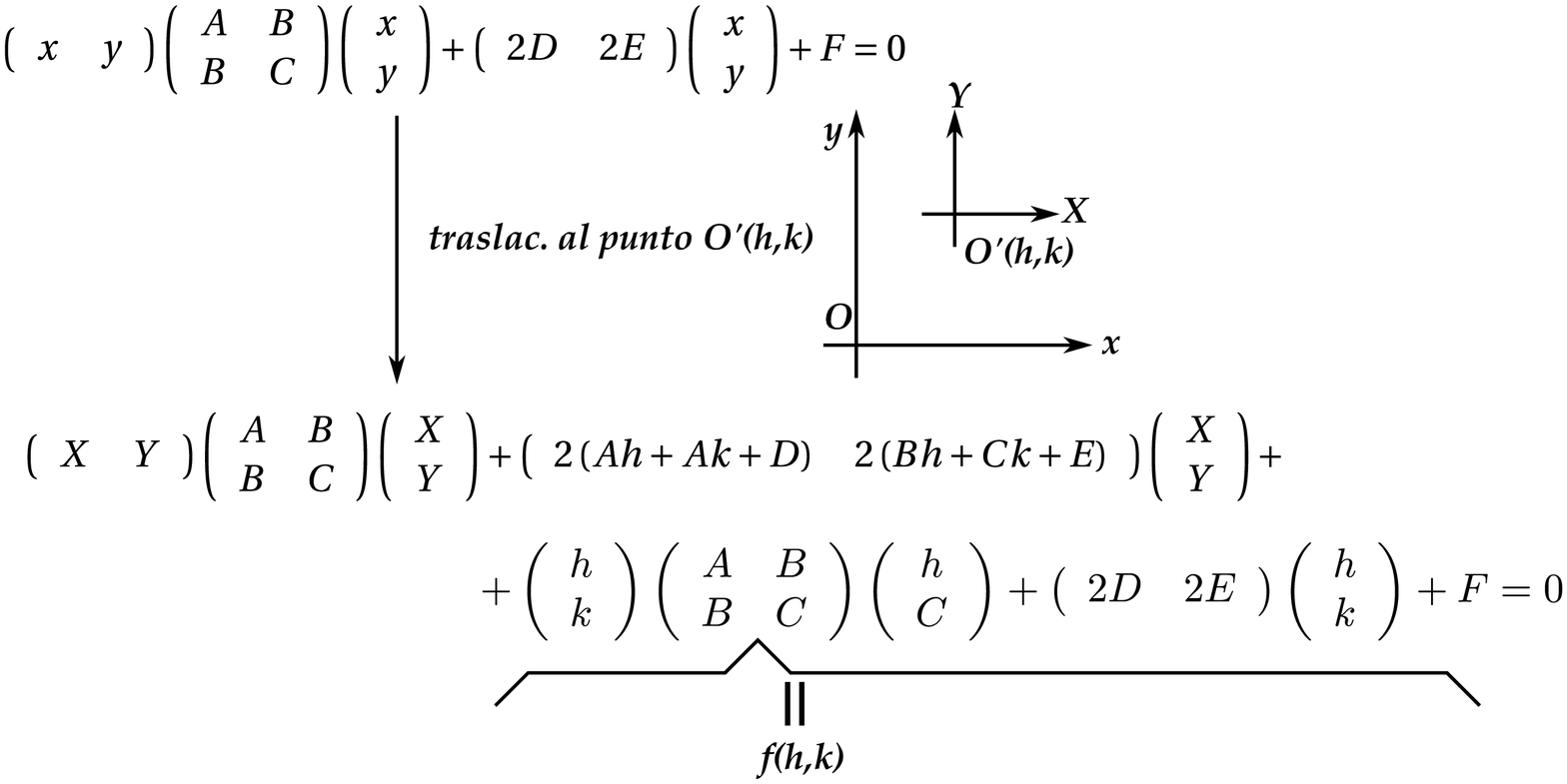}\\
\end{center}
\end{figure}
\newpage
nueva componente lineal:$\left(\begin{array}{cc}2\left(Ah+Bk+D\right)&2\left(Bh+Ck+E\right)\end{array}\right)\left(\begin{array}{cc}X\\Y\end{array}\right)$\\
nuevo término independiente: $f(h,k)=\left(\begin{array}{cc}h&k\end{array}\right)\left(\begin{array}{cc}A&B\\B&C\end{array}\right)\left(\begin{array}{cc}h\\k\end{array}\right)+\left(\begin{array}{cc}2D&2E\end{array}\right)\left(\begin{array}{cc}h\\k\end{array}\right)+F$
Calculemos los nuevos valores de $\omega, \delta,$ y $\Delta$.\\
\begin{align*}
\omega'&=A'+C'=A+C=\omega\\
\delta'&=\left|\begin{array}{cc}A'&B'\\B'&C'\end{array}\right|=\left|\begin{array}{cc}A&B\\B&C\end{array}\right|=AC-B^2=\delta.
\end{align*}
Esto demuestra que $\omega$ y $\delta$ son invariantes por traslación.\\
Veamos ahora que $\Delta$ también se conserva.
\begin{align*}
\Delta'&=\left|\begin{array}{ccc}A'&B'&D'\\B'&C'&E'\\D'&E'&F'\end{array}\right|\\
&=\left|\begin{array}{cccc}A&B&Ah+Bk+D\\B&C&Bh+Ck+E\\Ah+Bk+D&Bh+Ck+E&Ah^2+2Bhk+Ck^2+2Dh+2Dh+2Ek+F\end{array}\right|\\
&\underset{\tiny f_3\rightarrow f_3-hf_1-kf_2}{\underset{\uparrow}{=}}\left|\begin{array}{cccc}A&B&Ah+Bk+D\\B&C&Bh+Ck+E\\D&E&Dh+Ek+F\end{array}\right|\underset{\tiny c_3\rightarrow c_3-hc_1-kc_2}{\underset{\uparrow}{=}}\left|\begin{array}{cccc}A&B&C\\B&C&E\\D&E&F\end{array}\right|=\Delta
\end{align*}
lo que nos demuestra que $\Delta$ también es invariante por traslación.
\item[(2)] Supongamos ahora que rotamos los ejes $xy$ un ángulo
$\overset{\curvearrowleft}{\theta}(0<\theta<\phi/2)$ respecto a $O$.
\begin{figure}[ht!]
\begin{center}
  \includegraphics[scale=0.5]{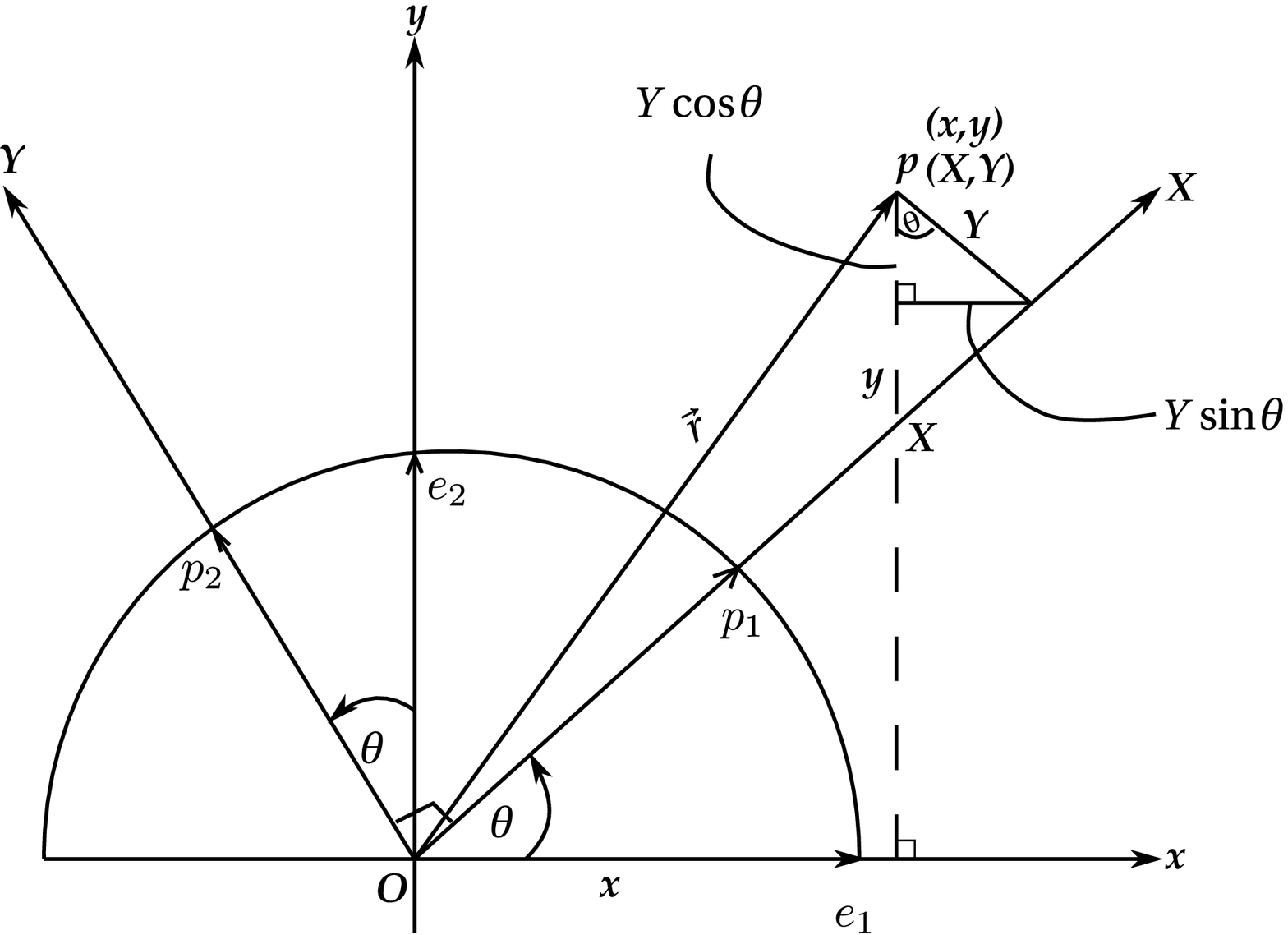}\\
  \caption{}
\end{center}
\end{figure}
\newpage
Se obtiene así un nuevo sistema $XY$ con origen en $O$.\\
Llamemos $\vec{P_1}$ y $\vec{P_2}$ a los vectores unitarios que señalan las dimensiones de los nuevos ejes. Sea $Q$ un punto cualquiera del plano de vector de posición $\vec{r}$ respecto a $O$.\\
$$\left[\vec{r}\right]_{e_\alpha}=\left[I\right]_{e_\alpha}^{p_\alpha}\left[\vec{r}\right]_{p_\alpha}$$
O sea
$$\left(\begin{array}{cc}x\\y\end{array}\right)=\left(\begin{array}{cc}\cos\theta&-\sin\theta\\\sin\theta&\cos\theta\end{array}\right)\left(\begin{array}{cc}X\\Y\end{array}\right)=P\left(\begin{array}{cc}X\\Y\end{array}\right)$$
con $$P=\left(\begin{array}{cc}\cos\theta&-\sin\theta\\\sin\theta&\cos\theta\end{array}\right):\text{ortogonal}$$
La cónica referida a los ejes $xy$ tiene por ecuación:
$$\left(\begin{array}{cc}x&y\end{array}\right)\left(\begin{array}{cc}A&B\\B&C\end{array}\right)\left(\begin{array}{cc}x\\y\end{array}\right)+\left(\begin{array}{cc}2D&2E\end{array}\right)\left(\begin{array}{cc}x\\y\end{array}\right)+F=0$$
Pero $$\left(\begin{array}{cc}x\\y\end{array}\right)=P\left(\begin{array}{cc}X\\Y\end{array}\right)$$
y al transponer, $$\left(\begin{array}{cc}x&y\end{array}\right)=\left(\begin{array}{cc}X&Y\end{array}\right)P^t.$$
Luego $$\left(\begin{array}{cc}X&Y\end{array}\right)\left(P^t\left(\begin{array}{cc}A&B\\B&C\end{array}\right)P\right)\left(\begin{array}{cc}X\\Y\end{array}\right)+\left(\begin{array}{cc}2D&2E\end{array}\right)P\left(\begin{array}{cc}X\\Y\end{array}\right)+F=0$$
es la ecuación de cónica $\diagup XY.$\\
Nótese que $$P^t\left(\begin{array}{cc}A&B\\B&C\end{array}\right)P$$ es simétrica ya que $${\left(P^t\left(\begin{array}{cc}A&B\\B&C\end{array}\right)P\right)}^t=P^t\left(\begin{array}{cc}A&B\\B&C\end{array}\right)P$$
Si llamamos
\begin{equation*}
\left.
\begin{split}
\left(\begin{array}{cc}A'&B'\\B'&C'\end{array}\right)&=P^t\left(\begin{array}{cc}A&B\\B&C\end{array}\right)P\\
\text{y}\hspace{0.5cm}\left(\begin{array}{cc}2D'&2E'\end{array}\right)&=\left(\begin{array}{cc}2D&2E\end{array}\right)P
\end{split}
\right\}\hspace{0.5cm}\star
\end{equation*}
se tiene que $$\left(\begin{array}{cc}A'&B'\\B'&C'\end{array}\right)$$ es simétrica y la ecuación de la cónica $\diagup XY$ es:
$$\left(\begin{array}{cc}X&Y\end{array}\right)\left(\begin{array}{cc}A'&B'\\B'&C'\end{array}\right)\left(\begin{array}{cc}X\\Y\end{array}\right)+\left(\begin{array}{cc}2D&2E\end{array}\right)\left(\begin{array}{cc}X\\Y\end{array}\right)+F=0,$$
o también, $$A'X^2+2B'XY+C'Y^2+2D'X+2E'Y+F=0$$ donde $A',B',C',D',E'$ se calculan utilizando $\star$.\\
Observese que cuando realizamos una rotación de ejes, tanto la parte cuadrática como la parte lineal se transforman. El tétmino independiente no se afecta.\\
En imágenes:
\begin{figure}[ht!]
\begin{center}
  \includegraphics[scale=0.5]{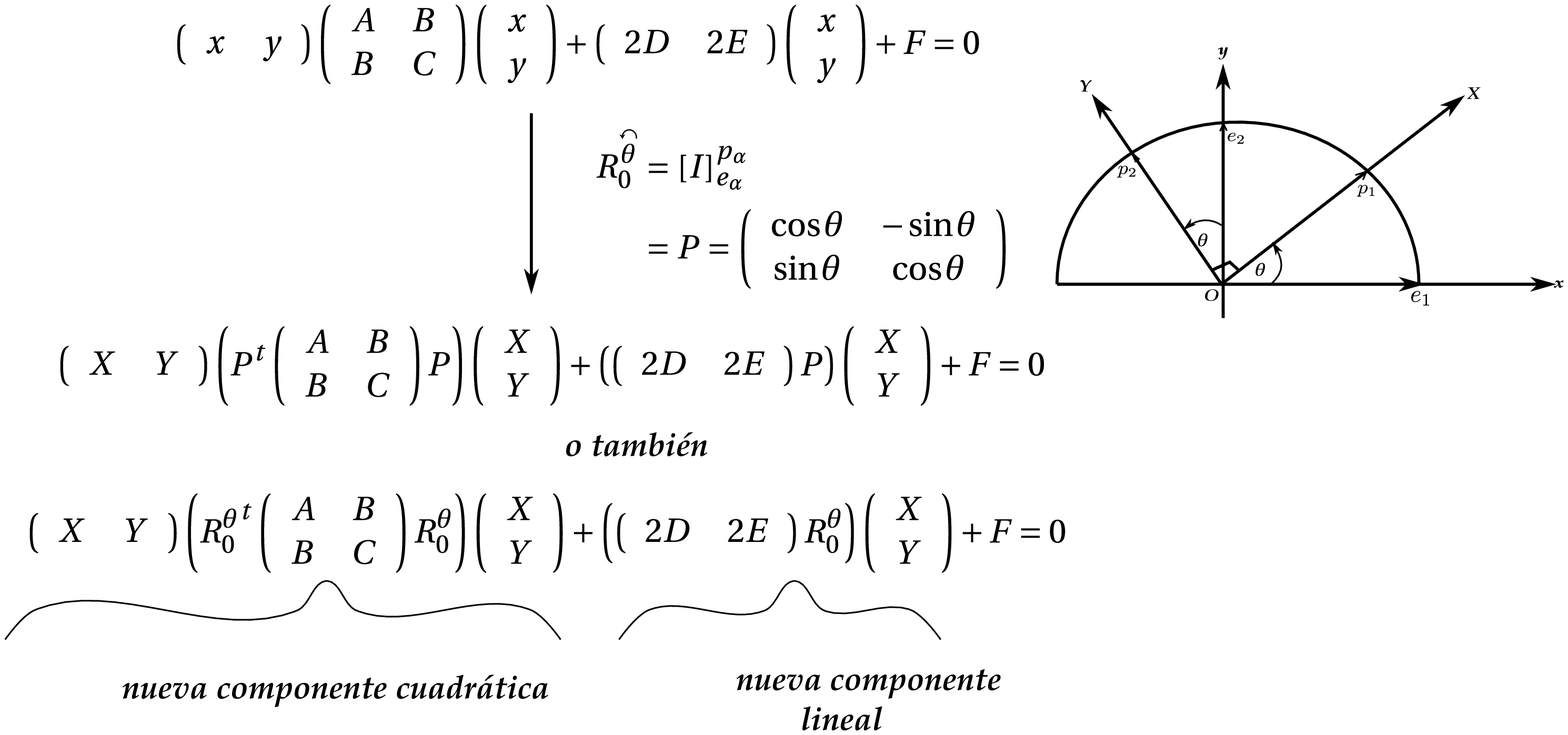}\\
\end{center}
\end{figure}
\newline
Los nuevos valores de $\omega, \delta$ y $\Delta$ son ahora:
\begin{align*}
\omega'&=A'+C'\\
\delta'&=\left|\begin{array}{cc}A'&B'\\B'&C'\end{array}\right|=A'C'-B'^2\\
\Delta'&=\left|\begin{array}{ccc}A'&B'&C'\\B'&C'&E'\\D'&E'&F'\end{array}\right|
\end{align*}
Como $P$ es ortogonal, $$\left(\begin{array}{cc}A'&B'\\B'&C'\end{array}\right)=P^t\left(\begin{array}{cc}A&B\\B&C\end{array}\right)P\sim\left(\begin{array}{cc}A&B\\B&C\end{array}\right)$$
Luego el polinomio característico de $$\left(\begin{array}{cc}A'&B'\\B'&C'\end{array}\right)$$ es el mismo polinomio característico de $$\left(\begin{array}{cc}A&B\\B&C\end{array}\right).$$
O sea que $$\lambda^2-\left(A'+C'\right)\lambda+\left(A'C'-B'^2\right)=\lambda^2-\left(A+C\right)+\left(AC-B^2\right)$$
\begin{align*}
\therefore\hspace{0.5cm}\omega'&=A'+C'=A+C=\omega\\
\delta'&=A'C'-B'^2=AC-B^2=\delta
\end{align*}
Esto demuestra que $\omega$ y $\delta$ son invariantes por una rotación de ejes.\\
Finalmente veremos que $\Delta$ también lo es.\\
Bastará con demostrar que $$\left(\begin{array}{ccc}A'&B'&D'\\B'&C'&E'\\D'&E'&F'\end{array}\right)\sim\left(\begin{array}{ccc}A&B&D\\B&C&E\\D&E&F\end{array}\right)$$
Una vez establecido esto se tendrá que $$\Delta'=\left|\begin{array}{ccc}A'&B'&D'\\B'&C'&E'\\D'&E'&F'\end{array}\right|=\left|\begin{array}{ccc}A&B&D\\B&C&E\\D&E&F\end{array}\right|=\Delta$$
ya que si dos matrices son semejantes tienen el mismo determinante.\\
Definamos
$$\overset{\sim}{P}=\left(\begin{array}{cc|c}P&&0\\ &&0\\ \hline 0&0&1\end{array}\right).$$
Como $P$ es ortogonal, $\overset{\sim}{P}$ también lo es y $${\overset{\sim}{P}}^{-1}={\overset{\sim}{P}}^t=\left(\begin{array}{cc|c}P^t&&0\\ &&0\\ \hline 0&0&1\end{array}\right)$$
Ya teníamos que $$\left(\begin{array}{cc}A'&B'\\B'&C'\end{array}\right)=P^t\left(\begin{array}{cc}A&B\\B&C\end{array}\right)P$$
y que $$\left(\begin{array}{cc}D'&E'\end{array}\right)=\left(\begin{array}{cc}D&E\end{array}\right)P\hspace{0.5cm}\therefore\hspace{0.5cm}\dbinom{D'}{E'}=P^t\dbinom{D}{E}.$$
Así que podemos escribir:
\begin{align*}
\left(\begin{array}{cc|cc}A'&B'&D'\\ B'&C'&E'\\ \hline D'&E'&F'\end{array}\right)&=\left(\begin{array}{c|cc}P^t\left(\begin{array}{cc}A&B\\B&C\end{array}\right)P&P^t\dbinom{D}{E}\\ \hline \left(\begin{array}{cc}D&E\end{array}\right)P&F\end{array}\right)\\
&=\left(\begin{array}{c|cc}P^t&0\\ &0\\ \hline 0&0&1\end{array}\right)\left(\begin{array}{cc|cc}A&B&D\\ B&C&E\\ \hline D&E&F\end{array}\right)\left(\begin{array}{c|cc}P&0\\ &0\\ \hline 0&0&1\end{array}\right)\\
&={\overset{\sim}{P}}^t\left(\begin{array}{cccc}A&B&D\\B&C&E\\D&E&F\end{array}\right)\overset{\sim}{P},
\end{align*}
lo que nos demuestra que
$$\left(\begin{array}{cccc}A'&B'&D'\\ B'&C'&E'\\D'&E'&F'\end{array}\right)\sim\left(\begin{array}{cccc}A&B&D\\ B&C&E\\D&E&F\end{array}\right)$$
Existen otros dos invariantes:
$$D^2+E^2\hspace{0.5cm}\text{y}M_{11}+M_{22}+M_{33}$$ son invariantes por una rotación.\\
Como $$\left(\begin{array}{cc}D'&E'\end{array}\right)=\left(\begin{array}{cc}D&E\end{array}\right)P, \hspace{0.5cm}\dbinom{D'}{E'}=P^t\dbinom{D}{E}$$ Luego $$D'^2+E'^2=\left(\begin{array}{cc}D'&E'\end{array}\right)\dbinom{D'}{E'}=\left(\begin{array}{cc}D&E\end{array}\right)\left(\begin{array}{cc}P&P^t\end{array}\right)\dbinom{D}{E}=D^2+E^2$$
lo que nos demuestra que $D^2+E^2$ es invariante por rotación.\\
$$M_{11}+M_{22}+M_{33}=\left|\begin{array}{cc}E&F\\E&F\end{array}\right|+\left|\begin{array}{cc}A&D\\D&F\end{array}\right|+\left|\begin{array}{cc}A&B\\B&C\end{array}\right|$$
\begin{align}
\Delta&=\left|\begin{array}{ccc}A&B&D\\B&C&E\\D&E&F\end{array}\right|\notag\\
&=CF-E^2+AF-D^2+\delta\notag\\
&=F\left(A+C\right)-\left(D^2+E^2\right)+\delta\notag\\
&=FW-\left(D^2+E^2\right)+\delta\label{10}
\end{align}
Ahora,
\begin{align}
M'_{11}+M'_{22}+M'_{33}&=\left|\begin{array}{ccc}C'&E'\\E'&F'\end{array}\right|+\left|\begin{array}{ccc}A'&D'\\D'&F'\end{array}\right|+\left|\begin{array}{ccc}A'&B'\\B'&C'\end{array}\right|\notag\\
&=C'F'-E'^2+A'F'-D'^2+\delta'\notag\\
&=F'\left(A'+C'\right)-\left(D'^2+E^2\right)+\delta'\notag\\
&\underset{\nearrow}{=}FW-\left(D^2+E^2\right)-\delta\label{11}\\
&\begin{cases}\notag\\
F'=F\\
D'^2+E'^2=D^2+E^2\\
W'=A'+C'=A+C=\omega\\
\delta'=\delta
\end{cases}
\end{align}
En virtud de [\ref{10}] y [\ref{11}], $M'_{11}+M'_{22}+M'_{33}=M_{11}+M_{22}+M_{33}.$ Lo que nos demuestra que $M_{11}+M_{22}+M_{33}$ es un invariante por una rotación.
\end{enumerate}
\end{proof}
\section{Ecuación de incrementos}
Sea $\left(X,Y\right)$ un punto cualquiera del plano $x,y$. Fig.1.3, $\left(X,Y\right)$ no necesariamente un punto de la cónica. Entonces $$F(X,Y)=q(X,Y)+\left(\begin{array}{cc}2D&2E\end{array}\right)\dbinom{X}{Y}+F.$$
\begin{figure}[ht!]
\begin{center}
  \includegraphics[scale=0.5]{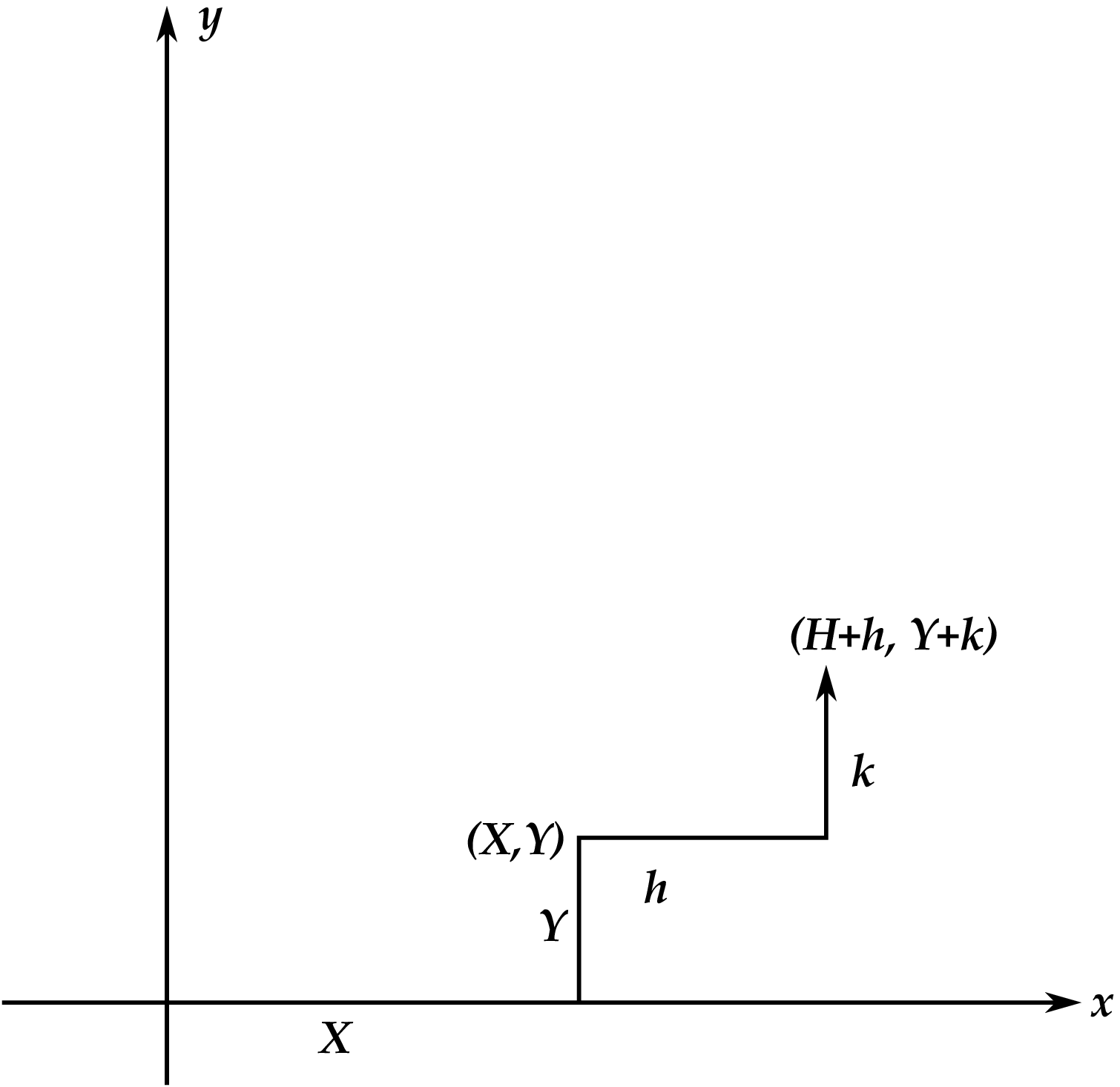}\\
  \caption{}\label{}
\end{center}
\end{figure}
\newline
Se trata de demostrar que $$\forall (h,k)\in\mathbb{R}^2: f\left(X+h,Y+k\right)=q(X,Y)+\left(\left.\dfrac{\partial f}{\partial x}\right)_{h,k}\left.\dfrac{\partial f}{\partial y}\right)_{h,k}\right)\dbinom{X}{Y}+ f(h,k)$$
En efecto,
\begin{align*}
f\left(X+h,Y+k\right)&=A{\left(X+h\right)}^2+2B\left(X+h\right)\left(Y+k\right)+C{\left(Y+k\right)}^2+\\
&+2D\left(X+h\right)+2E\left(Y+k\right)+F\\
&=AX^2+CY^2+2AhX+2CkY+Ah^2+Ck^2+\\
&+2BXY+2BXk+2BYk+2Bhk+2DX+\\
&+2Dh+2EY+2Ek+F\\
&=\left(AX^2+2BXY+CY^2\right)+\left(2Ah+2Bk+2D\right)X+\\
&+\left(2Bh+2Ck+2E\right)Y+\left(Ah^2+2Bhk+Ck^2+2Dh+2Ek+F\right)
\end{align*}
Como $$\dfrac{\partial f}{\partial x}=Ax+2By+2D$$ y $$\dfrac{\partial f}{\partial y}=2Bx+2Cy+2E,$$
\begin{align*}
\left.\dfrac{\partial f}{\partial x}\right)_{h,k}&=2Ah+2Bk+2D\\
\left.\dfrac{\partial f}{\partial y}\right)_{h,k}&=2Bh+2Ck+2E
\end{align*}
y regresando a la ecuación anterior,
\begin{equation}\label{12}
f\left(X+h,Y+k\right)=q(X,Y)+\left(\dfrac{\partial f}{\partial x}\right)_{h,k}\left(\dfrac{\partial f}{\partial y}\right)_{h,k}\dbinom{X}{Y}+f(h,k)
\end{equation}
Hemos demostrado así que $\forall (X,Y)$ y $\forall (h,k)\in\mathbb{R}^2$ se cumple [\ref{12}]. La ecuación [\ref{12}] se llama la \textit{ecuación de incrementos de la cónica} y será utilizada en la sección 1.3 que sigue y más adelante (sección 1.8) para hallar la ecuación de rectas tangentes y normales a la cónica en uno de sus puntos.
\section{Reducción de una cónica}
Reducir la cónica [\ref{1}] es definir que lugar geométrico representa.\\
Nuestro primer problema en la reducción de [\ref{1}] es definir, (si se puede) una transformación que elimine los términos lineales en la ecuación. Es obvio que debemos entonces considerar una cónica [\ref{1}] en la que no todos los coefientes $A,B,C$ de la forma cuadrática se anulan a la vez, ya que si eso sucede, [\ref{1}] tiene la forma $$2Dx+2Ey+F=0$$
que representa una recta y no nada más que decir.\\
Supongamos que bajo esas hipótesis realizamos una traslación de ejes al punto $O'$ de coordenadas $(h,k)\diagup xy,$ Fig. 1.4.\\
El punto $(h,k)$ no tiene que estar en la cónica. Quedan definidos dos ejes $XY$ con origen en $O',\,\,X-Y\parallel x-y$
\begin{figure}[ht!]
\begin{center}
  \includegraphics[scale=0.5]{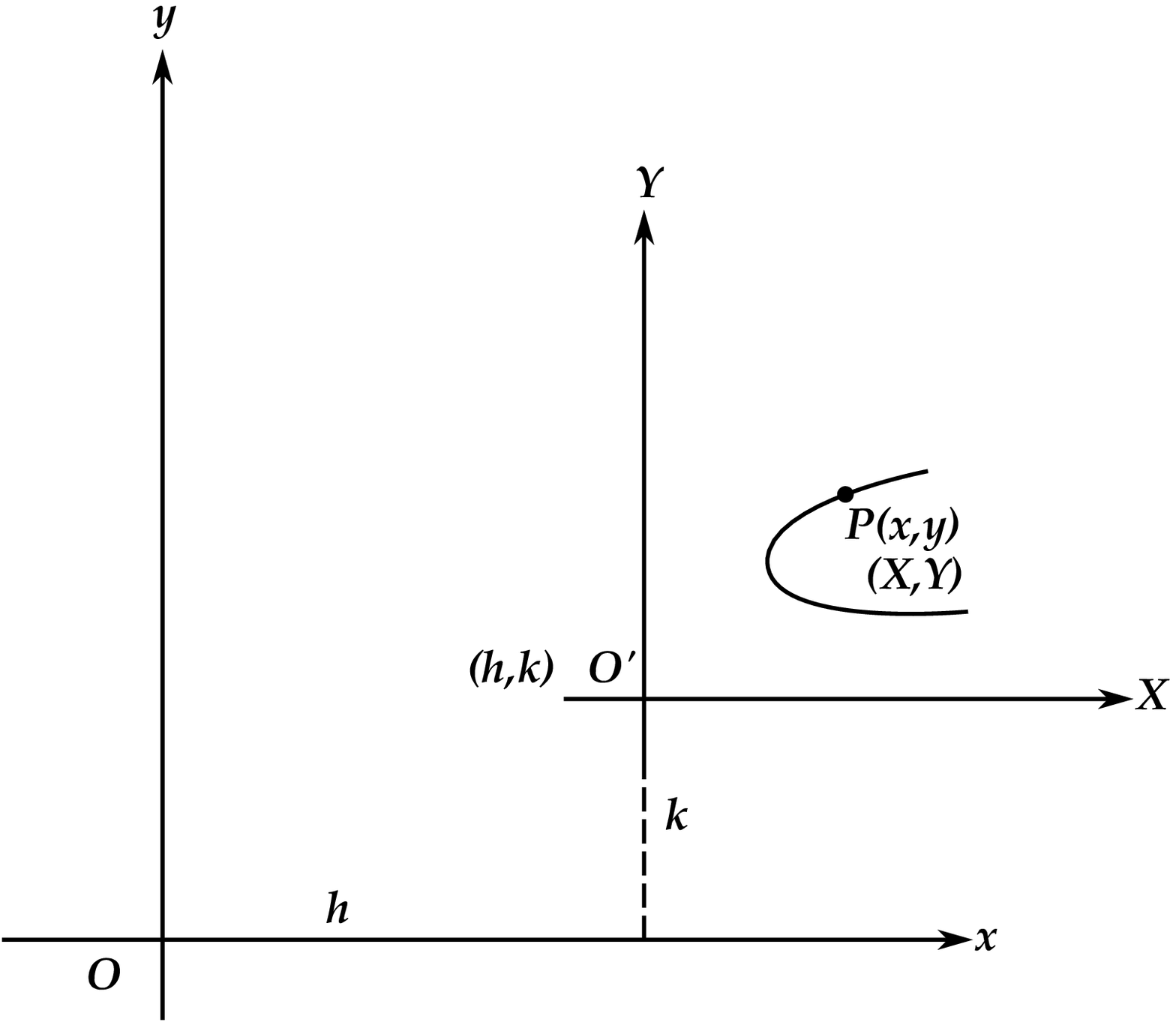}\\
  \caption{}\label{}
\end{center}
\end{figure}
Sea $P(x,y),\,\, P(X,Y)$ un punto de la cónica.\\
Como $P(x,y)$ está en la curva, $Ax^2+2Bxy+Cy^2+2Dx+2Ey+F=0.$\\
El mismo punto $P$ toma coordenadas $(X,Y)\diagup\text{sistema}X-Y.$\\
Las ecuaciones de transformación son:
\begin{align*}
x&=X+h\\
y&=Y+k
\end{align*}
y la ecuación de la cónica $\diagup XY$ es $$A{\left(X+h\right)}^2+2B\left(X+h\right)\left(Y+k\right)+C{\left(Y+k\right)}^2+2D\left(X+h\right)+2E\left(Y+k\right)+F=0$$
que según se acaba de demostrar en la ecuación de incrementos puede escribirse así: $$q(X,Y)+\left(\left.\dfrac{\partial f}{\partial x}\right)_{h,k}\left.\dfrac{\partial f}{\partial y}\right)_{h,k}\right)\dbinom{X}{Y}+f(h,k)=0.$$
Esta sería la ecuación de la cónica $\diagup XY$, o también
\begin{equation}\label{13}
\left(\begin{array}{cc}X&Y\end{array}\right)\left(\begin{array}{cc}A&B\\B&C\end{array}\right)\dbinom{X}{Y}\left(\left.\dfrac{\partial f}{\partial x}\right)_{h,k}\left.\dfrac{\partial f}{\partial y}\right)_{h,k}\right)\dbinom{X}{Y}+F(h,k)=0
\end{equation}
Nótese que
\begin{align}\label{14}
f(h,k)&=Ah^2+2Bhk+Ck^2+2Dh+2Ek+F\notag\\
&=\left(Ah^2+Bhk\right)+2Dh+\left(Bhk+Ck^2\right)+2Ek+F\notag\\
&=\left(Ah+Bk\right)h+\left(Ck+Bh\right)+2Dh+2Ek+F
\end{align}
Si queremos que se anulen los términos lineales en [\ref{13}] debemos escoger $(h,k)$ de modo que
\begin{equation}
\left.\dfrac{\partial f}{\partial x}\right)_{h,k}=0\hspace{0.5cm}\text{y}\hspace{0.5cm}\left.\dfrac{\partial f}{\partial y}\right)_{h,k}=0
\end{equation}
O sea debemos resolver para $(h,k)$ el sistema
\begin{align*}
Ah+Bk+D&=0\\
Bh+Ck+E&=0
\end{align*}
o
\begin{equation}\label{16}
\begin{cases}
Ah+Bk=-D\\
Bh+Ck=-E
\end{cases}
\end{equation}
lo que equivale a decir que debemos hallar los $\dbinom{h}{k}$ tal que $$h\dbinom{A}{B}+k\dbinom{B}{C}=\dbinom{-D}{-E}.$$
Es claro que $\forall(\overset{\sim}{h}, \overset{\sim}{k})$ que sea solución a [\ref{16}], la ecuación del lugar [\ref{13}] tiene la forma $$\left(\begin{array}{cc}X&Y\end{array}\right)\left(\begin{array}{cc}A&B\\B&C\end{array}\right)\dbinom{X}{Y}+f\left(\overset{\sim}{h}, \overset{\sim}{k}\right)=0$$ y se ha conseguido eliminar los términos lineales en [\ref{1}].
\section{Centro de una cónica. Propiedades}
\begin{defin}
Se llama \underline{centro de la cónica [\ref{1}]} a todo $(h,k)\in\Pi$ que sean solución de [\ref{16}].
\end{defin}
\begin{ejer}
Vamos a estudiar los centros de las cónicas en las que algunos de los coeficientes $A,B,C$ son ceros.\\
Hecho esto, estudiaremos el caso de los centros de las cónicas en que ninguno de $A,B,C$ es cero.
\begin{enumerate}
\item[(1)]
$\begin{tabular}{|c|c|c|}&&\\ \hline 0&B&0\end{tabular}$ La cónica es $2Bxy+2Dx+2Ey+F=0$\\
\begin{figure}[ht!]
\begin{center}
  \includegraphics[scale=0.3]{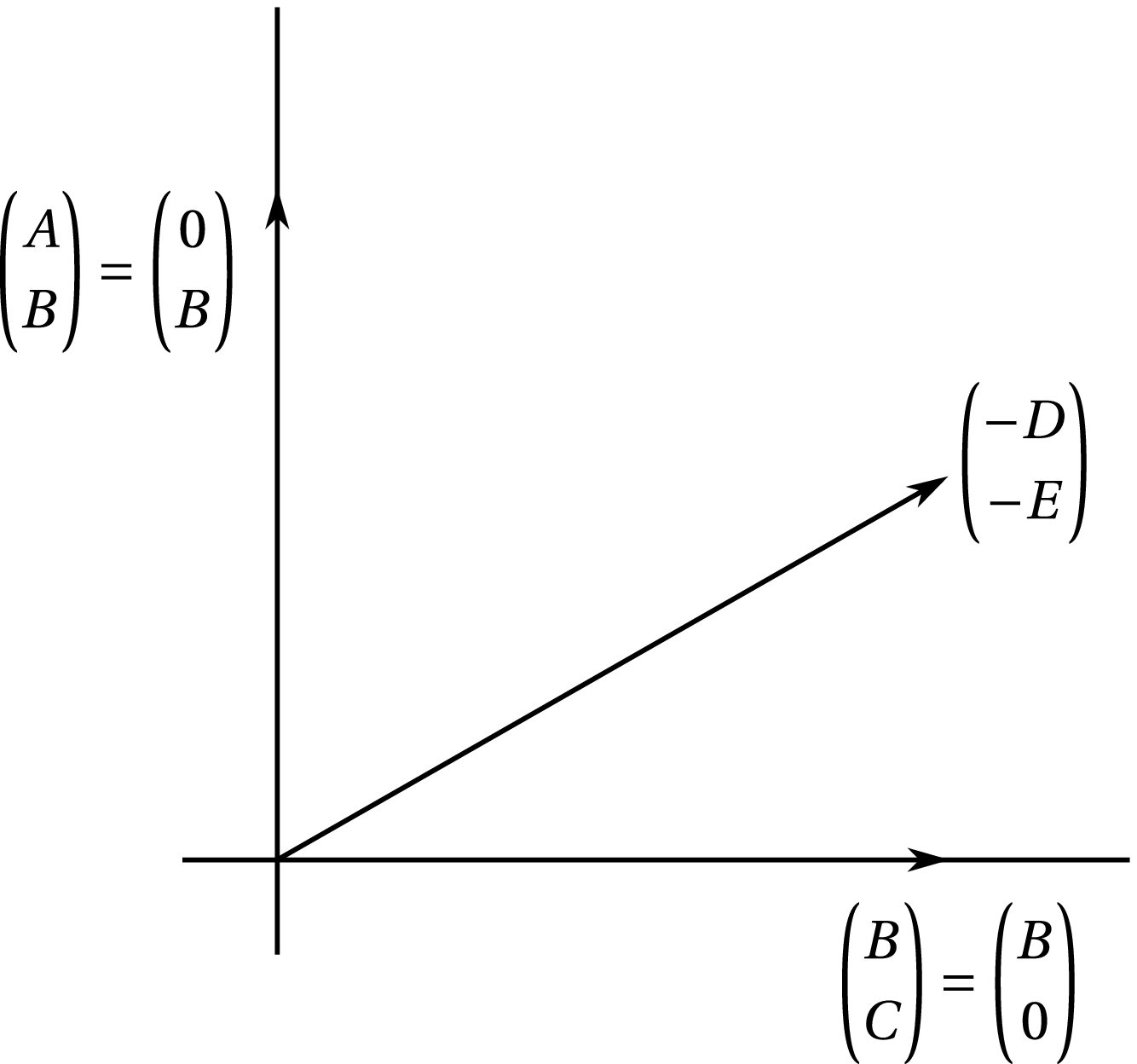}\\
\end{center}
\end{figure}
\newline
\begin{align*}
\left\{\dbinom{A}{B}, \dbinom{B}{C}\right\}&=\left\{\dbinom{0}{B}, \dbinom{B}{0}\right\}\hspace{0.5cm}\text{es Base de $\mathbb{R}^2$}\\
Ah+Bk&=-D\\
Bh+Ck&=-E\hspace{0.5cm}\text{es ahora}:\\
&\begin{cases}
0h+Bk=-D\\
Bh+0k=-E\hspace{0.5cm}\therefore\hspace{0.5cm}\begin{array}{cc}\overset{\sim}{h}=-\dfrac{E}{B}\\\\ \overset{\sim}{k}=-\dfrac{D}{B}\end{array}
\end{cases}
\end{align*}
\underline{la cónica tiene centro único}: $\left(-\dfrac{E}{B}, -\dfrac{D}{B}\right)$\\
Al trasladar los ejes al centro $0'\left(-\underset{\overset{\parallel}{\overset{\sim}{h}}}{\dfrac{E}{B}}, -\underset{\overset{\parallel}{\overset{\sim}{k}}}{\dfrac{D}{B}}\right),$ la ecuación de la cónica $\diagup XY$ con origen en $0'$ es:
$$\Delta=\left|\begin{array}{ccc}0&B&0\\ B&0&0\\ 0&0&f(\overset{\sim}{h},\overset{\sim}{k})\end{array}\right|=-B^2f(\overset{\sim}{h},\overset{\sim}{k})\hspace{0.5cm}\therefore\hspace{0.5cm}f(\overset{\sim}{h},\overset{\sim}{k})=-\dfrac{\Delta}{B^2}$$
y la ecuación de la cónica $\diagup XY$ es: $$2BXY-\dfrac{\Delta}{B^2}=0$$
O sea $$XY=\dfrac{\Delta}{2B^3}$$
\begin{itemize}
\item Si $\Delta\neq 0,$ el lugar es una \underline{hipérbola} de centro $0'$ y que se abre así:
\begin{figure}[ht!]
\begin{center}
  \includegraphics[scale=0.5]{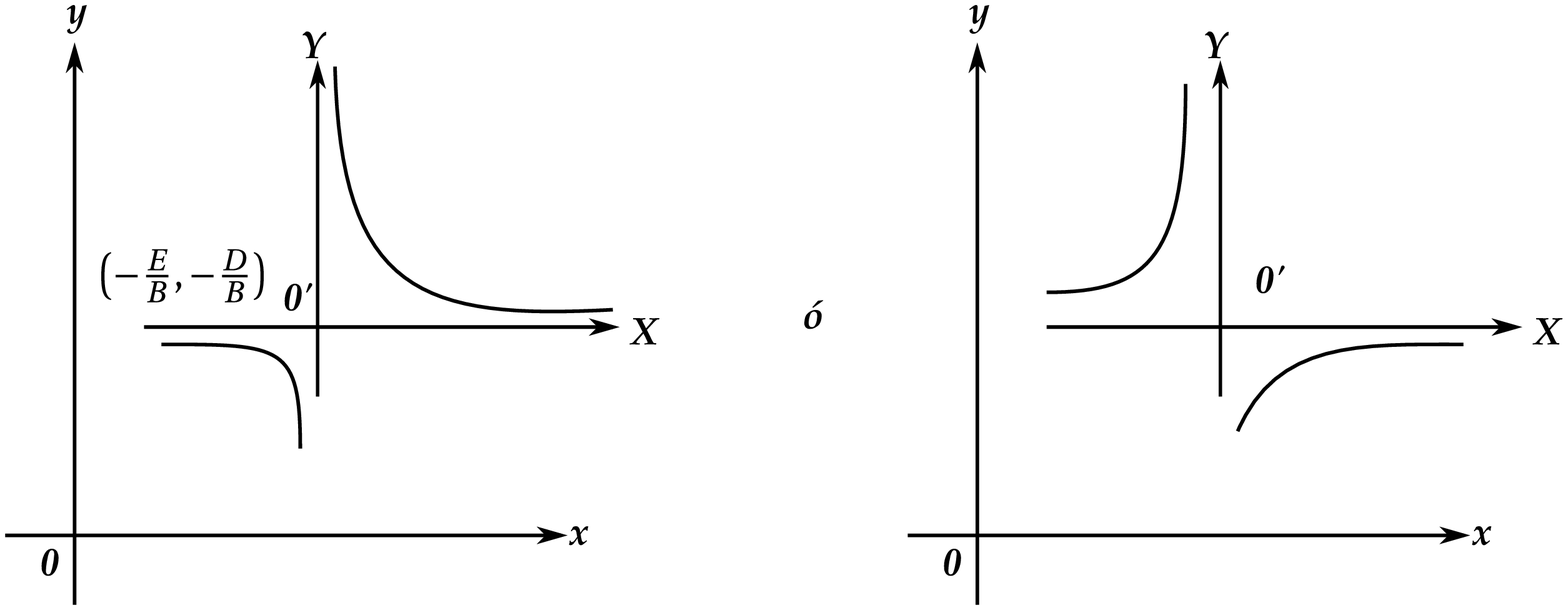}\\
\end{center}
\end{figure}
\newline
\item Si $\Delta=0,\,\,XY=0$. La \underline{cónica consta de dos rectas $\perp$s y concurrentes} en $0'$: el eje $Y$ de ecuación $x=-\dfrac{E}{B}$ y el eje $X$ de ecuación $y=-\dfrac{D}{B}$\\
Las dos rectas son: $\mathscr{L}_1: y+\frac{D}{B}=0$ y $\mathscr{L}_2: x+\frac{E}{B}=0$
\begin{figure}[ht!]
\begin{center}
  \includegraphics[scale=0.5]{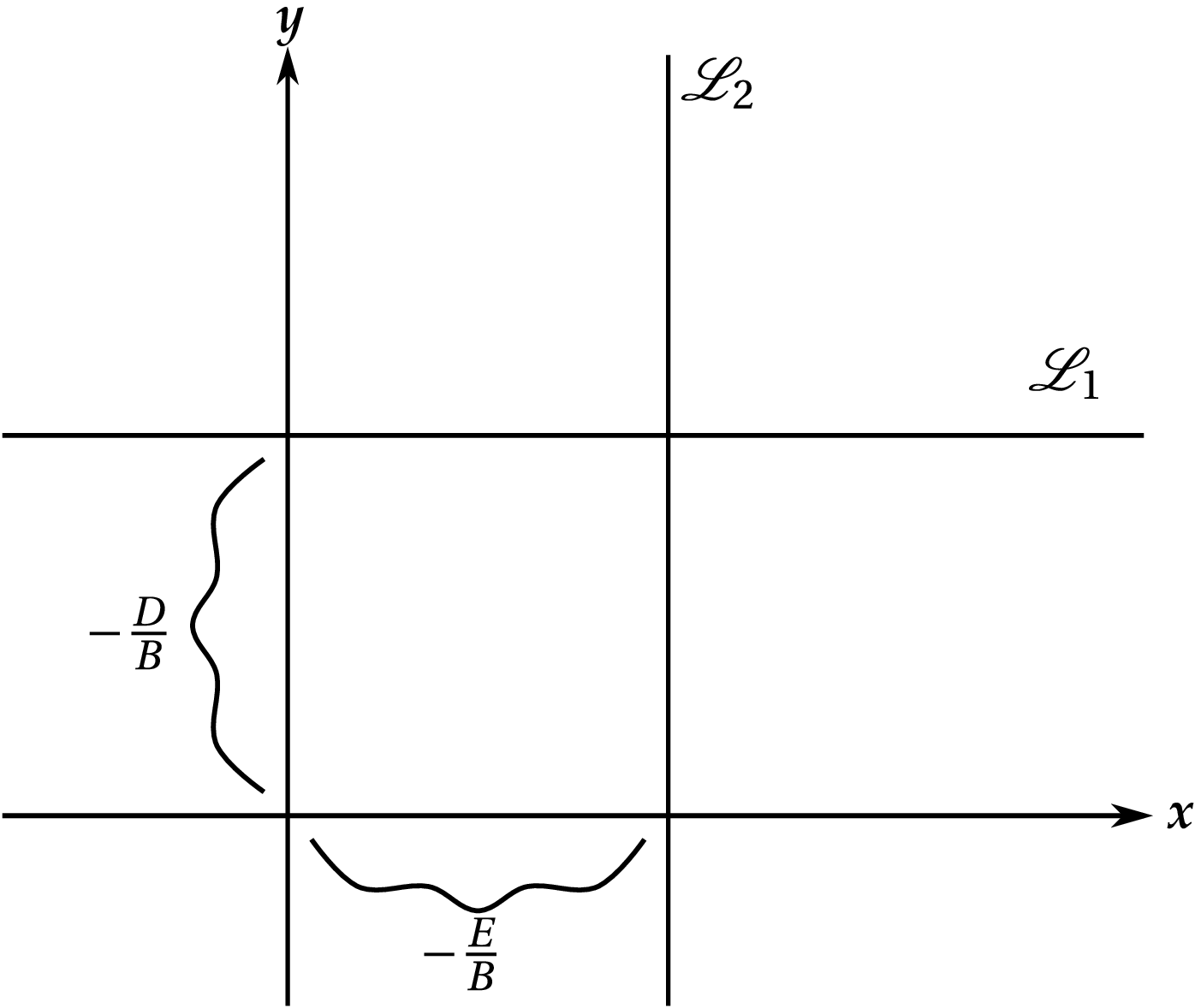}\\
\end{center}
\end{figure}
\newpage
\end{itemize}
Vamos a demostrar que en este caso $f(x,y)$ factoriza como <<el producto>> de las de ambas rectas.\\
O sea, veamos que $$f(x,y)=2B\left(y+\dfrac{D}{B}\right)\left(x+\dfrac{E}{B}\right)$$
$$\left(y+\dfrac{D}{B}\right)\left(x+\dfrac{E}{B}\right)=xy+\dfrac{E}{B}y+\dfrac{D}{B}x+\dfrac{DE}{B^2}\hspace{0.5cm}\star$$
Pero $\Delta=0.$ Luego $f\left(-\dfrac{E}{B}, -\dfrac{D}{B}\right)=0.$\\
O sea que $$2B\left(-\dfrac{E}{B}\right)\left(-\dfrac{D}{B}\right)+2D\left(-\dfrac{E}{B}\right)+2E\left(-\dfrac{D}{B}\right)+F=0$$
i.e., $$\dfrac{-2DE}{B}+F=0\hspace{0.5cm}\therefore\hspace{0.5cm}\dfrac{DE}{B^2}=\dfrac{F}{2B}$$ que llevamos a $\star$ $$\left(y+\dfrac{D}{B}\right)\left(x+\dfrac{E}{B}\right)=xy+\dfrac{D}{B}x+\dfrac{E}{B}y+\dfrac{F}{2B}=\dfrac{1}{2B}\left(2Bxy+2Dx+2Ey+F\right)$$
O sea que $$2B\left(y+\dfrac{D}{B}\right)\left(x+\dfrac{E}{B}\right)=2Bxy+2Dx+2Ey+F=f(x,y)$$
\item[(2)]
$\begin{tabular}{|c|c|c|}
&&\\ \hline A&0&0
\end{tabular}$
La cónica es
\begin{equation}
Ax^2+2Dx+2Ey+F=0
\end{equation}
Para hallar el centro debemos resolver el sistema:
\begin{equation*}
\left.
\begin{split}
Ax+0y&=-D\\
0x+0y&=-E
\end{split}
\right\}\hspace{0.5cm}\star
\end{equation*}
Se presentan dos casos.\\
\item[i)] $E=0$. La cónica es
\begin{equation}
\left\{(x,y)\diagup Ax^2+2Dx+F=0, y\in\mathbb{R}\right\}
\end{equation}
$$\delta=\left|\begin{array}{cc}A&0\\0&0\end{array}\right|=0;\hspace{0.5cm}\Delta=\left|\begin{array}{ccc}A&0&D\\0&0&0\\D&0&F\end{array}\right|=0;\hspace{0.5cm}M_{22}=AF-D^2$$
El sistema $\star$ es
\begin{align*}
Ax+0y&=-D\\
0x+0y&=0
\end{align*}
teniendose que $$\forall y\in\mathbb{R}, \left(-\dfrac{D}{A}, y\right)$$ es solución.\\
Así que hay infinitas soluciones y \underline{la cónica tiene infinitos centros}.\\
Como la segunda ecuación es redundante, los centros se encuentrasn sobre la $x=-\dfrac{D}{A}$ que llamaremos el \underline{eje de centros}.
\begin{figure}[ht!]
\begin{center}
  \includegraphics[scale=0.5]{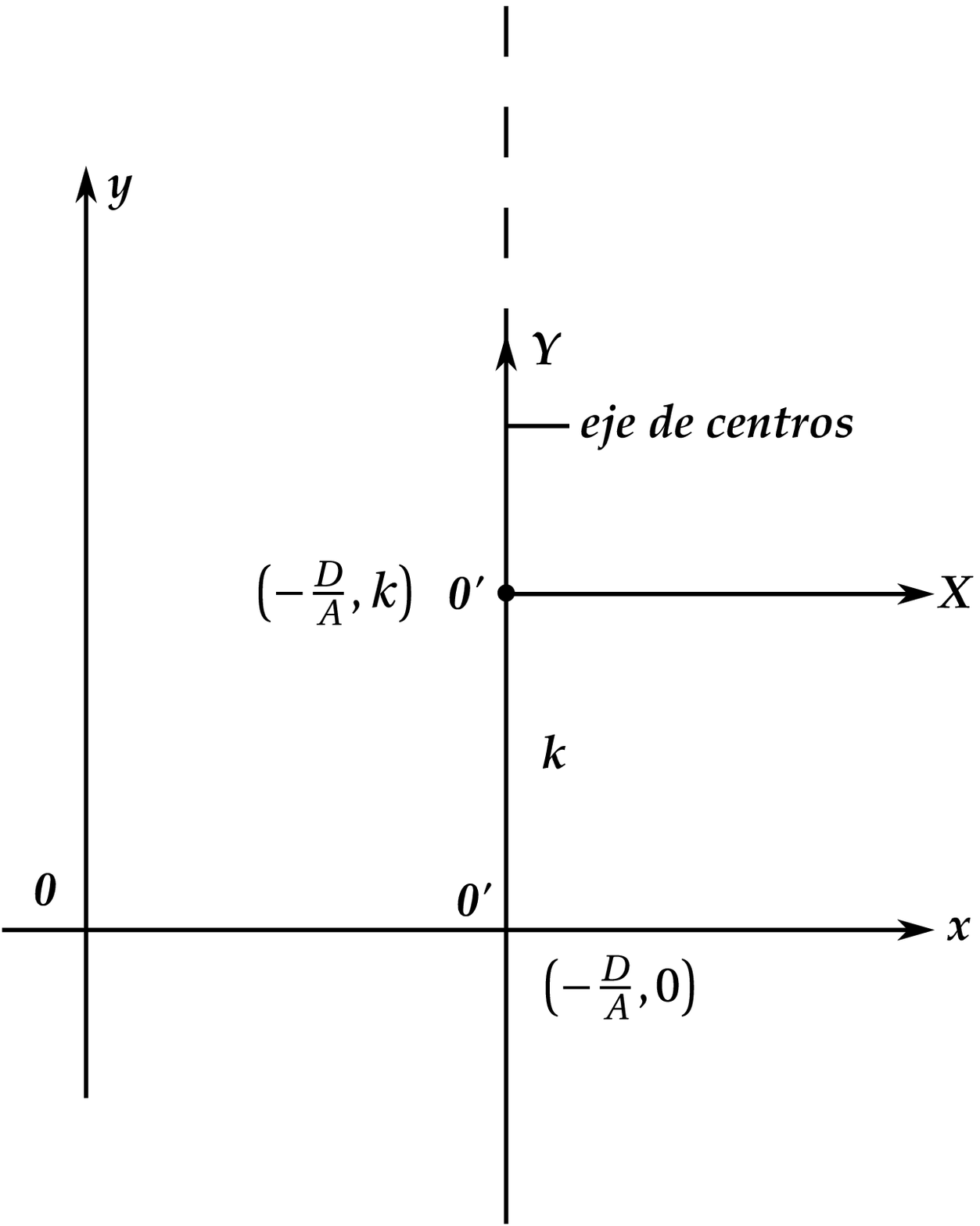}\\
\end{center}
\end{figure}
\newline
Al trasladar los ejes a un punto $0'$ de coordenadas $\left(-\dfrac{D}{A}, k\right)$ del eje de centros, las ecuaciones de la transformación son:\\
$\begin{cases}
x=X-\dfrac{D}{A}\\
y=Y+k
\end{cases}$\\
se eliminan los términos lineales y la ecuación de la cónica $\diagup XY$ es: $$AX^2+f\left(-\dfrac{D}{A}, k\right)=0$$
Pero
\begin{align*}
f\left(-\dfrac{D}{A}, k\right)&=A{\left(-\dfrac{D}{A}\right)}^2+2D\left(-\dfrac{D}{A}, k\right)+F\\
&=\dfrac{D^2}{A}-\dfrac{2D^2}{A}+F=F-\dfrac{D^2}{A}=\text{constante}
\end{align*}
independientemente del punto $0'$ tomado en el eje de centros.\\
Luego la ecuación de la cónica $\diagup XY$ con origen en el punto $0'$ de coordenadas $\left(-\dfrac{D}{A}, 0\right)\diagup xy$ es $$AX^2=\dfrac{D^2}{A}-F=\dfrac{D^2-AF}{A},$$
i.e., $$X^2=\dfrac{D^2-AF}{A^2}=-\dfrac{M_{22}}{A^2}.$$
Si $M_{22}<0$, el \underline{lugar son dos rectas $\parallel$ al eje Y}. El eje de centros es la paralela media de las dos rectas.\\
Si $M_{22}=0$, el \underline{lugar es el eje Y}.\\
Si $M_{22}>0$, el \underline{lugar es $\emptyset$}.\\
\item[$\bullet$] Si $M_{22}<0$, el lugar son dos rectas $\parallel$s al eje Y:
$$\mathscr{L}_1: X=+\sqrt{-\dfrac{M_{22}}{A^2}};\hspace{0.5cm}\mathscr{L}_2: X=-\sqrt{-\dfrac{M_{22}}{A^2}}$$
Vamos a dm. que la cónica <<factoriza>> como el producto de las dos rectas.\\
Como
$\begin{cases}
x&=X-\dfrac{D}{A},\hspace{0.5cm}X=x+\dfrac{D}{A}\\
y&=Y
\end{cases}$
\newline
La ecuación de $\mathscr{L}_1$ es $$x+\dfrac{D}{A}=\sqrt{-\dfrac{M_{22}}{A^2}}$$ y la de $\mathscr{L}_2$ $$x+\dfrac{D}{A}=-\sqrt{-\dfrac{M_{22}}{A^2}}$$
Definamos:
\begin{align*}
\mathscr{L}_1&=x+\dfrac{D}{A}-\sqrt{-\dfrac{M_{22}}{A^2}}\hspace{0.5cm}\text{y}\\
\mathscr{L}_2&=x+\dfrac{D}{A}+\sqrt{-\dfrac{M_{22}}{A^2}}\\
\mathscr{L}_1\cdot\mathscr{L}_2&=\left(x+\dfrac{D}{A}-\sqrt{-\dfrac{M_{22}}{A^2}}\right)\left(x+\dfrac{D}{A}+\sqrt{-\dfrac{M_{22}}{A^2}}\right)\\
&={\left(x+\dfrac{D}{A}\right)}^2-\left(-\dfrac{M_{22}}{A^2}\right)=x^2+\dfrac{2D}{A}x+\dfrac{D^2}{A^2}+\dfrac{M_{22}}{A^2}\\
&\underset{\nearrow}{=}x^2+\dfrac{2D}{A}x+\cancel{\dfrac{D^2}{A^2}}-\dfrac{AF-D^2}{\cancel{A^2}}=x^2+\dfrac{2D}{A}x+\dfrac{F}{A}\\
&M_{22}=AF-D^2\\
&=\dfrac{1}{A}\left(Ax^2+2Dx+F\right)=\dfrac{1}{A}f(x,y)
\end{align*}
O sea que $$f(x,y)=A\cdot\mathscr{L}_1\cdot\mathscr{L}_2$$
A este resultado podríamos haber llegado directamente factorizando el trinomio $Ax^2+2Dx+F$\\
$$Ax^2+2Dx+F=A\left(x^2+\dfrac{2D}{A}x+\dfrac{F}{A}\right)\hspace{0.5cm}\star$$
Si hacemos $$x^2+\dfrac{2D}{A}x+\dfrac{F}{A}=0,\hspace{0.5cm}x=\dfrac{-\dfrac{2D}{A}\pm\sqrt{\dfrac{4D^2}{A^2}-\dfrac{4F}{A}}}{2}$$
$$x=\dfrac{-\dfrac{2D}{A}\pm 2\sqrt{\dfrac{-M_{22}}{A^2}}}{2}$$
Luego $$x^2+\dfrac{2D}{A}x+\dfrac{F}{A}=\left(x+\dfrac{D}{A}-\sqrt{-\dfrac{M_{22}}{A^2}}\right)\left(x+\dfrac{D}{A}+\sqrt{-\dfrac{M_{22}}{A^2}}\right)$$
O sea que
\begin{align*}
f(x,y)&=Ax^2+2Dx+F=A\left(x^2+\dfrac{2D}{A}x+\dfrac{F}{A}\right)\\
&=A\left(x+\dfrac{D}{A}+\sqrt{-\dfrac{M_{22}}{A^2}}\right)\left(x+\dfrac{D}{A}-\sqrt{-\dfrac{M_{22}}{A^2}}\right)\\
&=A\cdot\mathscr{L}_1\cdot\mathscr{L}_2
\end{align*}
\item[$\bullet$] Si $M_{22}=0,$ el lugar es el eje $Y$ ó eje de centros de ecuación$\diagup xy:
x=-\dfrac{D}{A}.$\\
Definimos $\mathscr{L}=x+\dfrac{D}{A}.$ Vamos a dm. que la cónica puede escribirse en la forma $$f(x,y)=A{\left(x+\dfrac{D}{A}\right)}^2=A\mathscr{L}^2.$$
\begin{align*}
f(x,y)&=Ax^2+2Dx+F\\
&=A\left(x^2+\dfrac{2D}{A}x+\dfrac{F}{A}\right)\\
&\underset{\nearrow}{=}A\left(x^2+\dfrac{2D}{A}x+\dfrac{D^2}{A^2}\right)=A{\left(x+\dfrac{D}{A}\right)}^2=A\mathscr{L}^2\\
&\begin{cases}
\text{Como}\quad M_{22}&=0\\
AF-D^2&=0\\
\therefore\quad F=\dfrac{D^2}{A}
\end{cases}
\end{align*}
\item[$\bullet$] Si $M_{22}>0,$ el lugar es $\emptyset.$ Un argumento adicional para probarlo podría se éste.
\begin{align*}
f(x,y)&=Ax^2+2Dx+F\\
&=A\left(x^2+\dfrac{2D}{A}x+\dfrac{D^2}{A^2}\right)+F-\dfrac{D^2}{A}\\
&=A{\left(x+\dfrac{D}{A}\right)}^2+F-\dfrac{D^2}{A}\\
&=A{\left(x+\dfrac{D}{A}\right)}^2+\dfrac{AF-D^2}{A}\underset{\uparrow}{=}A\underset{\overset{\vee}{0} \text{cualq. sea $(x,y)\in\mathbb{R}^2$}}{\underbrace{\left({\left(x+\dfrac{D}{A}\right)}^2+\dfrac{AF-D^2}{A^2}\right)}}\neq0\\
&\hspace{4.5cm}\begin{cases}
\text{Como}\quad M_{22}&=AF-D^2>0\\
&\dfrac{AF-D^2}{A^2}>0
\end{cases}
\end{align*}
lo que dm. que el lugar es $\emptyset.$
\item[ii)] $E\neq 0.$ La cónica es
\begin{equation}\label{19}
Ax^2+2Dx+2Ey+F=0
\end{equation}
$$\delta=\left|\begin{array}{cc}A&0\\0&0\end{array}\right|;\hspace{0.5cm}\Delta=\left|\begin{array}{ccc}A&0&D\\0&0&E\\D&E&F\end{array}\right|=-AE^2\neq 0.$$
El sistema $\star$ es:
\begin{align*}
Ax+0y&=-D\\
0x+0y&=-E
\end{align*}
Como la segunda ecuación no tiene solución, el sistema no tiene solución y en consecuencia la cónica \textit{no tiene centro}. O sea que no es posible definir una traslación que elimine los términos lineales en [\ref{19}].\\
Puede ocurrir\\
\item[ii-1)] que $D\neq 0$.\\
Vamos a ver que es posible definir una traslación de los ejes $xy$ a un punto $0'$ de coordenadas $(a,b)\diagup xy,$ \underline{a y b a determinar}, de modo que se eliminen en [\ref{19}] el término lineal en $x$ y el término independiente.\\
Lo que no es posible, se acaba de demostrar, es definir una traslación a un punto que elimine los dos términos lineales a la vez.\\
Supongamos, pues, que trasladamos los ejes $x-y$ a un punto de coordenadas $(a,b)\diagup xy:$
\begin{figure}[ht!]
\begin{center}
  \includegraphics[scale=0.5]{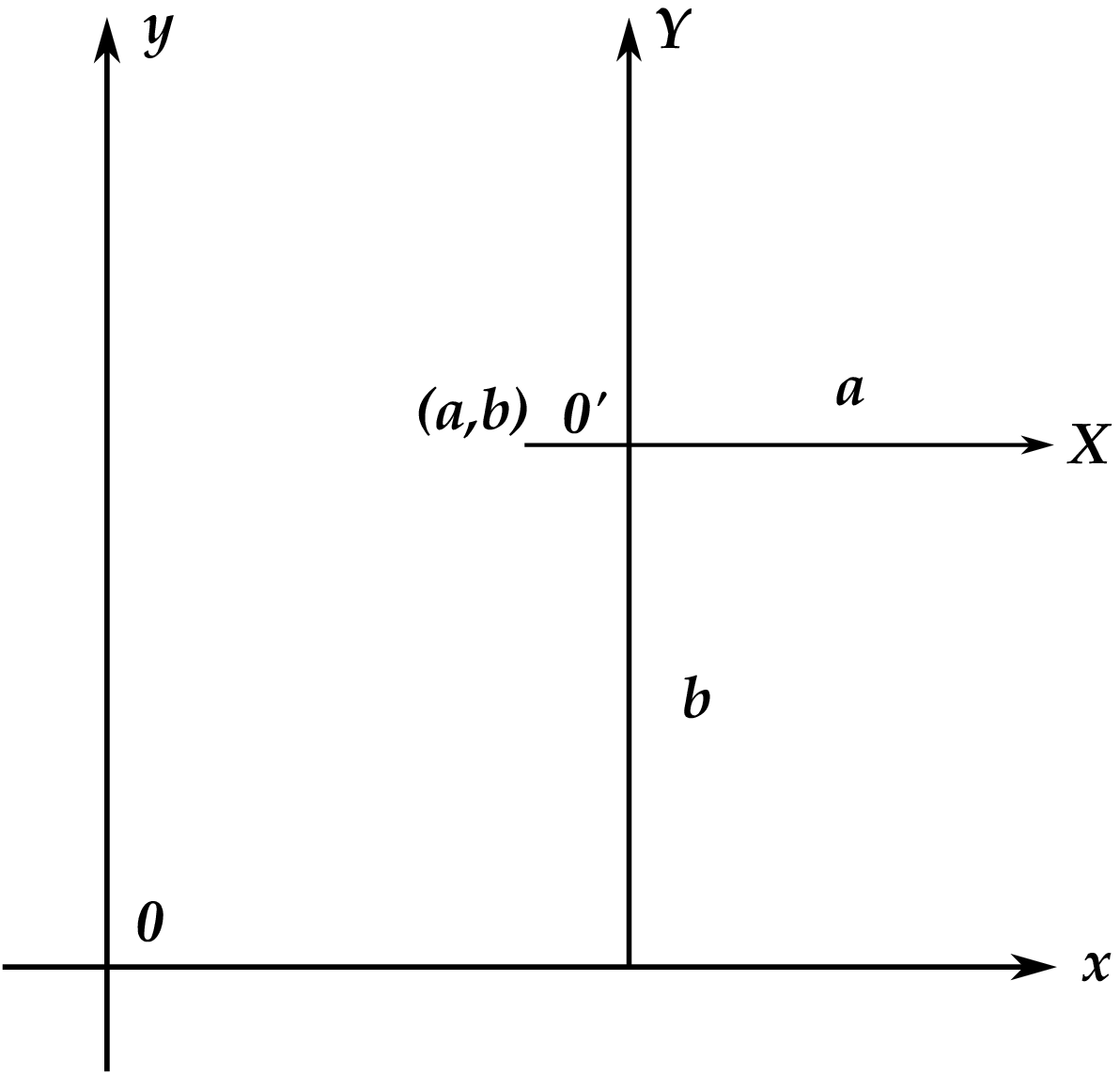}\\
\end{center}
\end{figure}
\newline
Ecuaciones de la trasformación:
\begin{align*}
x&=X+a\\
y&=Y+b\hspace{0.5cm}\text{que llevamos a [\ref{19}]}
\end{align*}
\begin{align*}
&A{\left(X+a\right)}^2+2D\left(X+a\right)+2E\left(Y+b\right)+F=0\\
&Ax^2+2AaX+Aa^2+2DX+2Da+2EY+2Eb+F=0\\
&AX^2+2\left(Aa+D\right)X+2EY+Aa^2+2Da+2Eb+F=0
\end{align*}
Para lo que se busca,\\
$\begin{cases}
Aa+D=0\\
Aa^2+2Da+2Eb+F=0
\end{cases}$\\
De la $1^a$ ecuación, $a=-\dfrac{D}{A}$ que llevamos a la $2^a$
\begin{align*}
2Eb&=-Aa^2-2Da-F\\
&=-A\left(\dfrac{D^2}{A}\right)+2D\left(\dfrac{D}{A}\right)-F\hspace{0.5cm}\therefore\hspace{0.5cm} b=-\dfrac{AF+D^2}{2AE}
\end{align*}
Así que si trasladamos los ejes $xy$ al punto $0'$ de coordenadas $\left(-\dfrac{D}{A},-\dfrac{AF+D^2}{2AE}\right)\diagup xy$ se definen otros ejes $X-Y\parallel$s a $x-y$ y la ecuación de la cónica$\diagup XY$ es: $$AX^2+2EY=0\hspace{0.5cm}\therefore\hspace{0.5cm}Y=-\dfrac{A}{2E}X^2$$
\newpage
\begin{figure}[ht!]
\begin{center}
  \includegraphics[scale=0.5]{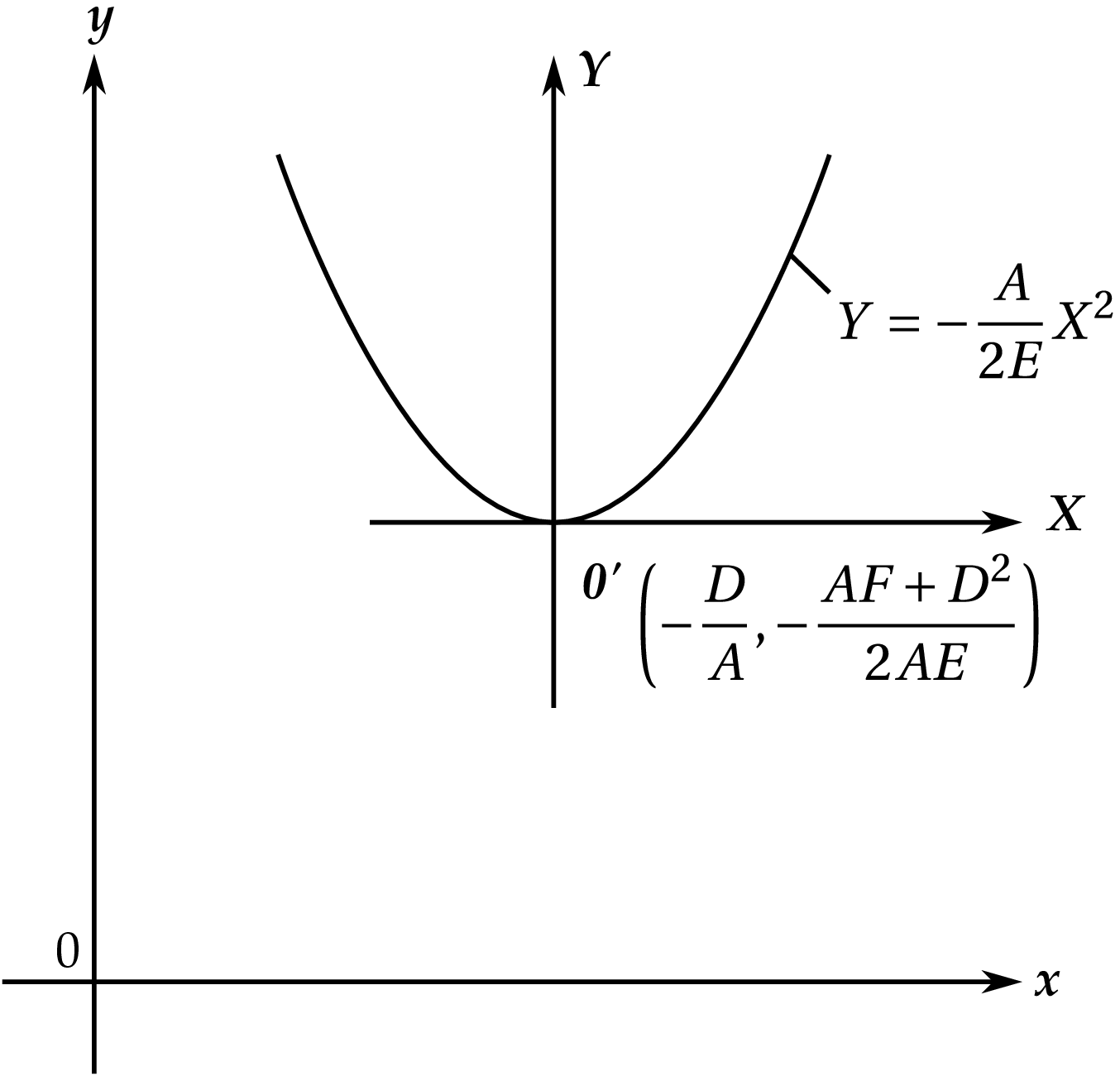}\\
\end{center}
\end{figure}
El \underline{lugar es una parábola} que se abre en el sentido del eje $Y(0-Y)$ de acuerdo al signo de $-\dfrac{A}{2E}.$ El vértice de la parábola es el punto $O'$ de coord.$\left(-\dfrac{D}{A},-\dfrac{AF+D^2}{2AE}\right)\diagup xy$
\item[ii-2)] $D=0$.\\ La ecuación de la cónica es:
\begin{equation}
Ax^2+2Ey+F=0
\end{equation}
$$\therefore\hspace{0.5cm}y=-\dfrac{A}{2E}x^2-\dfrac{F}{2E}.$$
\begin{figure}[ht!]
\begin{center}
  \includegraphics[scale=0.4]{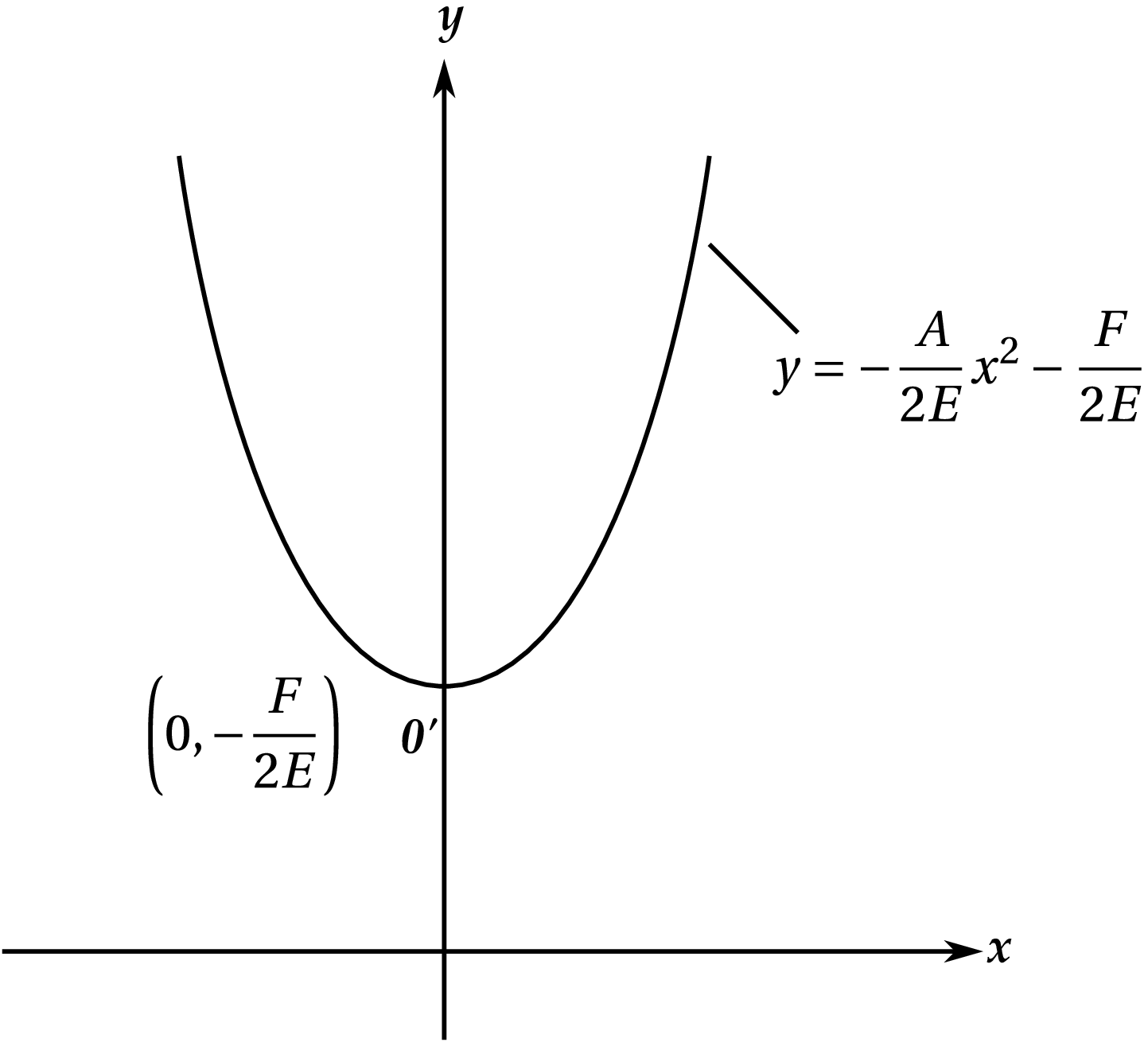}\\
\end{center}
\end{figure}
\newline
La cónica es \underline{una parábola}. El vértice de la parábola es el punto $O'$ de coord.$\left(0,-\dfrac{F}{2E}\right)\diagup xy$
\item[(3)]
$\begin{tabular}{|c|c|c|}&&\\\hline
0&0&c
\end{tabular}$ La cónica es
\begin{equation}\label{21}
\left\{(x,y)\diagup Cy^2+2Dx+2Ey+F=0\right\}
\end{equation}
$$\left\{\dbinom{A}{B},\dbinom{B}{C}\right\}=\left\{\dbinom{0}{0},\dbinom{0}{C}\right\}$$ Los centros son las soluciones al sistema
\begin{equation*}
\left.
\begin{split}
0x+0y&=-D\\
0x+Cy&=-E
\end{split}
\right\}\hspace{0.5cm}\star\star\quad\text{La situación es análoga al caso (2).}
\end{equation*}
Se presentan dos casos.\\
\item[(1)] $D=0.$ La ecuación de la cónica es
\begin{equation}
Cy^2+3Ey+F=0
\end{equation}
\begin{align*}
\delta&=\left|\begin{array}{cc}0&0\\0&C\end{array}\right|=0\\
\Delta&=\left|\begin{array}{ccc}0&0&0\\0&C&E\\0&E&F\end{array}\right|=0;\hspace{0.5cm}M_{11}=CF-E^2
\end{align*}
El sistema $\star\star$ es:
\begin{align*}
0x+0y&=0\\
0x+Cy&=-E
\end{align*}
Luego $\forall x\in\mathbb{R}, \left(x,-\dfrac{E}{C}\right)$ es solución.\\
Así que hay infinitas soluciones y \underline{la cónica tiene $\infty$s centros}.\\
Como la primera ecuación es redundante, los centros se encuentran sobre la recta $y=-\dfrac{E}{C}$ que es el
\underline{eje de centros}:
\begin{figure}[ht!]
\begin{center}
  \includegraphics[scale=0.5]{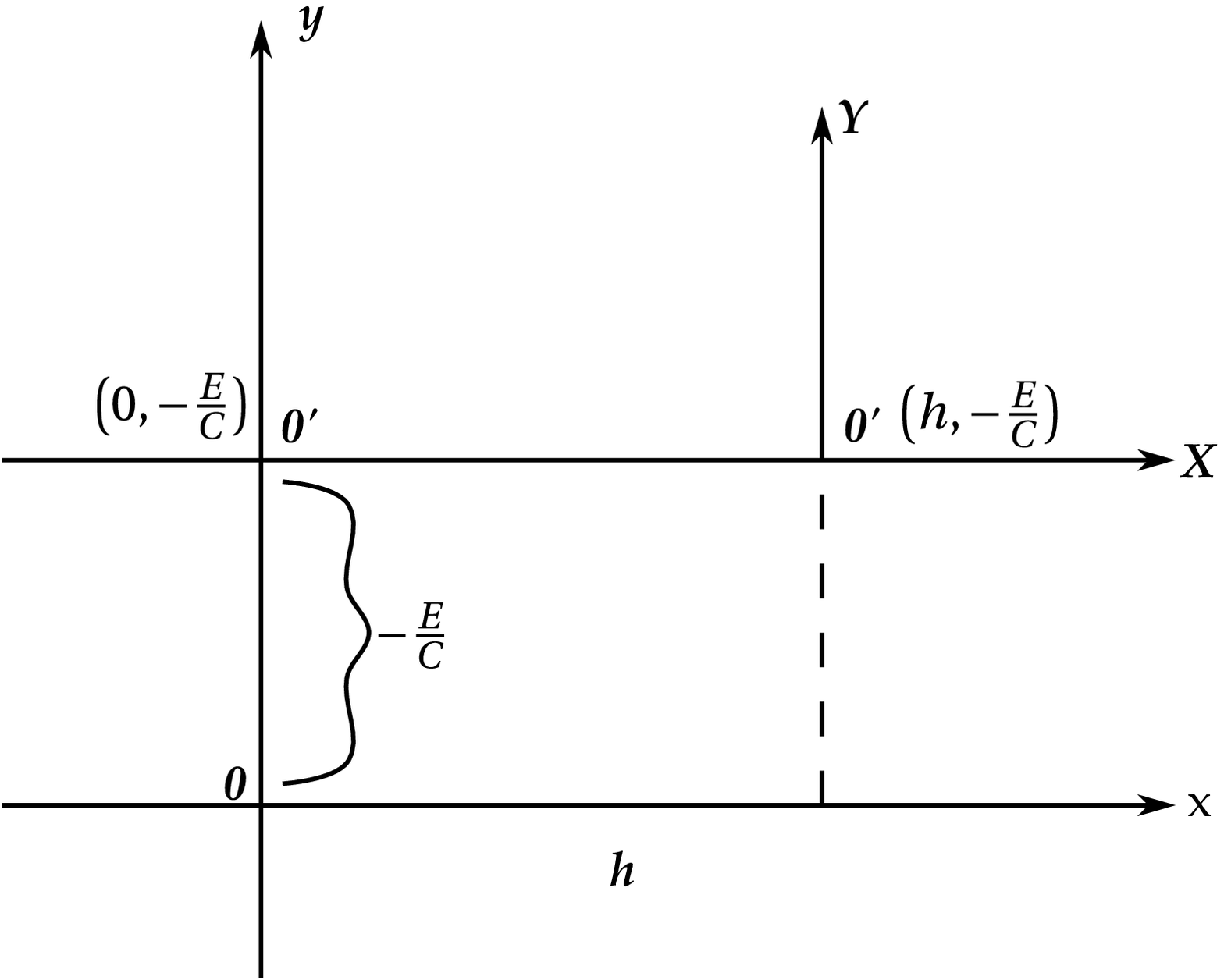}\\
\end{center}
\end{figure}
\newline
Al trasladar los ejes $xy$ a un punto $0'$ del eje de centros y de coordenadas $\left(h,\dfrac{E}{C}\right)\diagup xy$, se eliminan los términos lineales en [\ref{21}] y la ecuación de la cónica$\diagup XY$ es: $$CY^2+f\left(h,-\dfrac{E}{C}\right)=0$$
Ahora,
\begin{align*}
f\left(h,-\dfrac{E}{C}\right)&=C{\left(-\dfrac{E}{C}\right)}^2+2E\left(\dfrac{E}{C}\right)+F\\
&=F-\dfrac{E^2}{C}:\text{cte. independiente del punto $0'$ tomado sobre el eje de centros.}
\end{align*}
Luego la ecuación de la cónica$\diagup XY$ con origen en el punto $0'$ de coordenadas $\left(0, -\dfrac{E}{C}\right)\diagup xy$ es $$CY^2=\dfrac{E^2-CF}{C}\hspace{0.5cm}\text{o}\hspace{0.5cm}Y^2=\dfrac{E^2-CF}{C^2}=-\dfrac{\Delta_{11}}{C^2}$$
Si $M_{11}<0,$ el \underline{lugar son dos rectas $\parallel$s al eje X}. El eje de centros es la paralela media der ambas rectas.\\
Si $M_{11}=0,$ el \underline{lugar es el eje X}, ó eje de centros.\\
Si $M_{11}>0,$ el \underline{lugar es $\emptyset$}.\\
En el primer caso, o sea, cuando $M_{11}<0,$ dm. que la cónica factoriza como el producto de ambas rectas.\\
\item[$\bullet$] Si $M_{11}<0,$ el lugar consta de dos rectas $\parallel$s al eje $X$ y de ecuación
\begin{align*}
\mathscr{L}_1: Y_1=+\sqrt{-\dfrac{M_{11}}{C^2}}\\
\intertext{y}\\
\mathscr{L}_2: Y_2=-\sqrt{-\dfrac{M_{11}}{C^2}}
\end{align*}
Vamos a dm. que la cónica factoriza como el producto de dos rectas.\\
Como
\begin{align*}
x&=X\\
\intertext{Y}\\
y&=Y-\dfrac{E}{C},\quad Y=y+\dfrac{E}{C}
\end{align*}
La ecuación de $\mathscr{L}_1$ es $y+\dfrac{E}{C}=+\sqrt{-\dfrac{M_{11}}{C^2}}$\\
Definimos
\begin{align*}
\mathscr{L}_1&=y+\dfrac{E}{C}-\sqrt{-\dfrac{M_{11}}{C^2}}\\
\intertext{y}\\
\mathscr{L}_2&=y+\dfrac{E}{C}+\sqrt{-\dfrac{M_{11}}{C^2}}\\
\mathscr{L}_1\cdot\mathscr{L}_2&=\left(y+\dfrac{E}{C}-\sqrt{\dfrac{M_{11}}{C^2}}\right)\left(y+\dfrac{E}{C}+\sqrt{\dfrac{M_{11}}{C^2}}\right)\\
&={\left(y+\dfrac{E}{C}\right)}^2-\left(-\dfrac{M_{11}}{C^2}\right)\\
&=y^2+\dfrac{2E}{C}y+\dfrac{E^2}{C^2}+\dfrac{M_{11}}{C^2}\underset{M_{11}\overset{\uparrow}{=}CF-E^2}{=}y^2+\dfrac{2E}{C}y+\dfrac{E^2}{C^2}+\dfrac{CF-E^2}{C^2}\\
&=y^2+\dfrac{2E}{C}y+\dfrac{F}{C}=\dfrac{1}{C}\left(Cy^2+2Ey+F\right)=\dfrac{1}{C}f(x,y)
\end{align*}
O sea que
\begin{align*}
f(x,y)&=C\cdot\mathscr{L}_1\cdot\mathscr{L}_2,\quad\text{i.e.,}\\
Cy^2+2Ey+F&=C\left(y+\dfrac{E}{C}-\sqrt{-\dfrac{M_{11}}{C^2}}\right)\left(y+\dfrac{E}{C}+\sqrt{-\dfrac{M_{11}}{C^2}}\right)
\end{align*}
\item[$\bullet$] Si $M_{11}=0,$ el lugar es el eje $X$ ó eje de centros de ecuación$\diagup
xy:y=-\dfrac{E}{C}.$\\
Definamos $\mathscr{L}=y+\dfrac{E}{C}.$ Vamos a dm. que la cónica puede escribirse en la forma $$f(x,y)=C{\left(y+\dfrac{E}{C}\right)}^2=C\mathscr{L}^2$$
\begin{align*}
f(x,y)&=Cy^2+2Ey+F\\
&=C\left(y^2+\dfrac{2E}{C}y+\dfrac{F}{C}\right)\\
&\underset{\uparrow}{=}C\left(y^2+\dfrac{2E}{C}y+\dfrac{E^2}{C^2}\right)=C{\left(y+\dfrac{E}{C}\right)}^2=C\mathscr{L}^2\\
&\begin{cases}
\text{Como}\quad M_{11}=0\\
CF-E^2=0\\
\therefore\quad F=\dfrac{E^2}{C}
\end{cases}
\end{align*}
\item[$\bullet$] Si $M_{11}>0,$ el lugar es $\emptyset$. Veamos porque.
\begin{align*}
f(x,y)&=Cy^2+2Ey+F\\
&=C\left(y^2+\dfrac{2E}{C}y+\dfrac{E^2}{C^2}\right)+F-\dfrac{E^2}{C}\\
&=C{\left(y+\dfrac{E}{C}\right)}^2+\dfrac{CF-E^2}{C}\\
&\underset{\uparrow}{=}\left({\left(y+\dfrac{E}{C}\right)}^2+\dfrac{CF-E^2}{C^2}\right)\neq0\\
&\begin{cases}
\text{Como}\quad M_{11}&=CF-E^2>0,\\
&\dfrac{CF-E^2}{C^2}>0
\end{cases}
\end{align*}
Lo que dm. que el lugar es $\emptyset$.
\item[(2)] $D\neq 0$.\\
La cónica es
\begin{equation}\label{23}
Cy^2+2Dx+2Ey+F=0
\end{equation}
$$\delta=\left|\begin{array}{cc}0&0\\0&C\end{array}\right|;\hspace{0.5cm}\Delta=\left|\begin{array}{ccc}0&0&D\\0&C&E\\D&E&F\end{array}\right|=CD^2\neq 0$$
El sistema $\star\star$ es:
\begin{align*}
0x+0y&=-D\\
0x+Cy&=-E
\end{align*}
Como la primera ecuación no tiene solución, el sistema no tiene solución y en consecuencia, la \underline{cónica no tiene centro}. O sea que no es posible definir una traslación que elimine en [\ref{23}] los términos lineales.\\
\item[2-i)] que en [\ref{23}], $E\neq 0$.\\
Vamos a ver que en este caso es posible definir una traslación de ejes $xy$ a un punto $0'$ de coordenadas $(a,b)\diagup xy$, \underline{a y b a determinar}, de modo que en [\ref{23}] se elimine el término lineal en $y$ y el término independiente.\\
Supongamos entonces que trasladamos los ejes $xy$ a un punto $0'$ de coordenadas $(a,b)\diagup xy$.
\begin{figure}[ht!]
\begin{center}
  \includegraphics[scale=0.5]{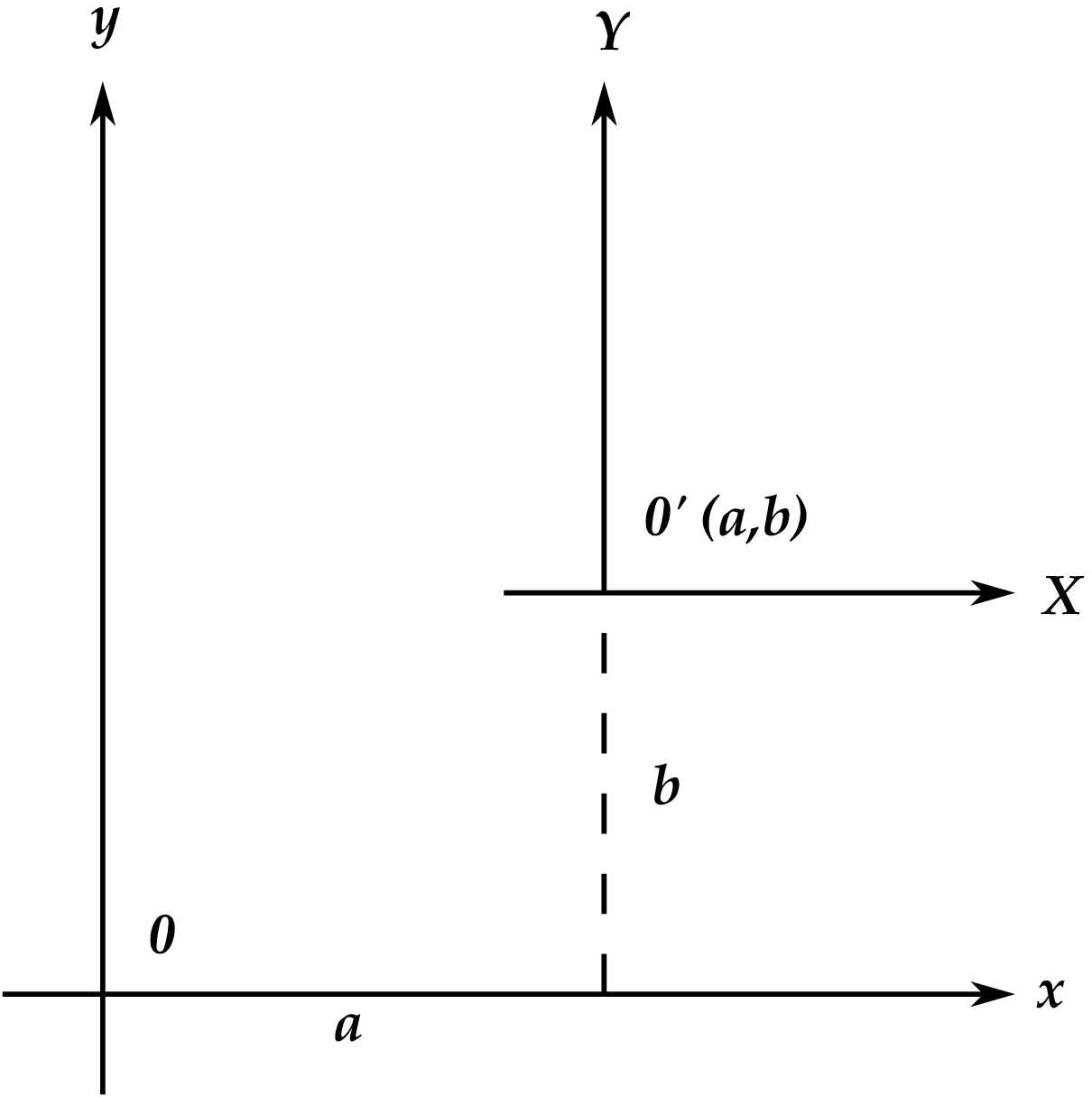}\\
\end{center}
\end{figure}
\newline
Ecuaciones de la transformación:
\begin{align*}
x&=X+a\\
y&=Y+b
\end{align*}
que llevamos a [\ref{23}]
\begin{align*}
&C{\left(Y+b\right)}^2+2D\left(X+a\right)+2E\left(Y+b\right)+F=0\\
&CY^2+2CbY+Cb^2+2DX+2Da+2Da+2EY+2Eb+F=0\\
&CY^2+2DX+2\left(Cb+E\right)Y+\left(Cb^2+2Da+2Eb+F\right)=0
\end{align*}
Para lo que se busca, debemos tener que
\begin{align*}
Cb+E&=0\\
Cb^2+2Da+2Eb+F&=0
\end{align*}
De la primera ecuación, $b=-\dfrac{E}{C}$ que llevamos a la segunda
\begin{align*}
2Da&=-Cb^2-2Eb-F\\
&=-C{\left(-\dfrac{E}{C}\right)}^2-2E\left(-\dfrac{E}{C}\right)-F\\
&=-\dfrac{E^2}{C}-F=-\dfrac{CF+E^2}{C}\hspace{0.5cm}\therefore\hspace{0.5cm}a=-\dfrac{CF+E^2}{2CD}
\end{align*}
Así que si trasladamos los ejes $xy$ al punto $0'$ de coordenadas $\left(-\dfrac{CF+E^2}{2CD}, -\dfrac{E}{C}\right)\diagup xy,$ la ecuación de la cónica$\diagup XY$ es: $$CY^2+2DX=0\hspace{0.5cm}\therefore\hspace{0.5cm}X=-\dfrac{C}{2D}Y^2$$
\newpage
\begin{figure}[ht!]
\begin{center}
  \includegraphics[scale=0.5]{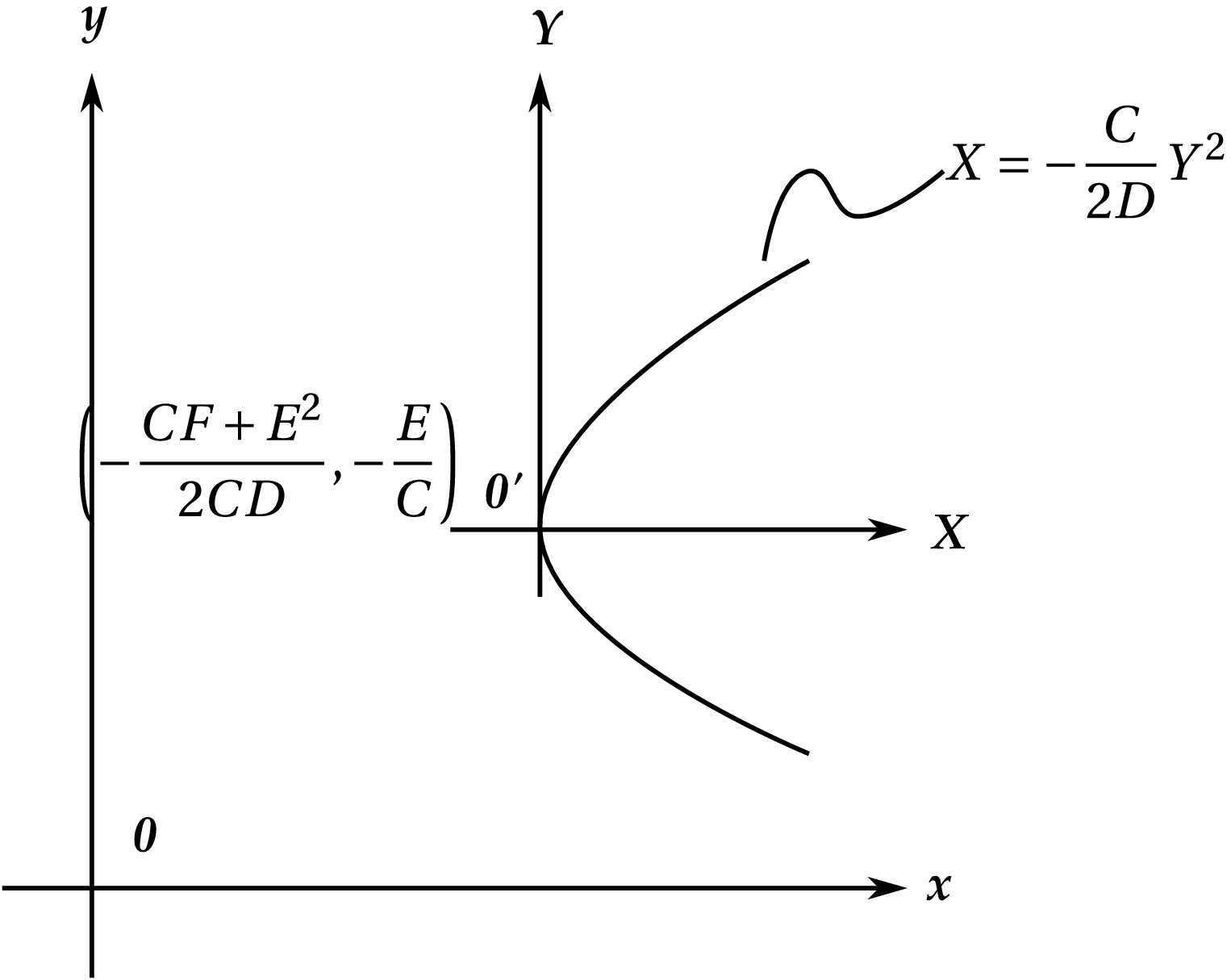}\\
\end{center}
\end{figure}
El lugar es \underline{una parábola} que se abre en el sentido del $X\left(0-X\right)$ dependiendo del signo de $-\dfrac{C}{2D}$. El vértice de la parábola es el punto $O'$ de coord.$\left(-\dfrac{CF+E^2}{2CD},-\dfrac{E}{C}\right)\diagup xy.$\\
\item[2-ii)] que en [\ref{23}], $E=0$. La ecuación de la cónica es
\begin{align*}
&Cy^2+2Dx+F=0\\
&\therefore\hspace{0.5cm}x=-\dfrac{C}{2D}y^2-\dfrac{F}{2D}
\end{align*}
\begin{figure}[ht!]
\begin{center}
  \includegraphics[scale=0.4]{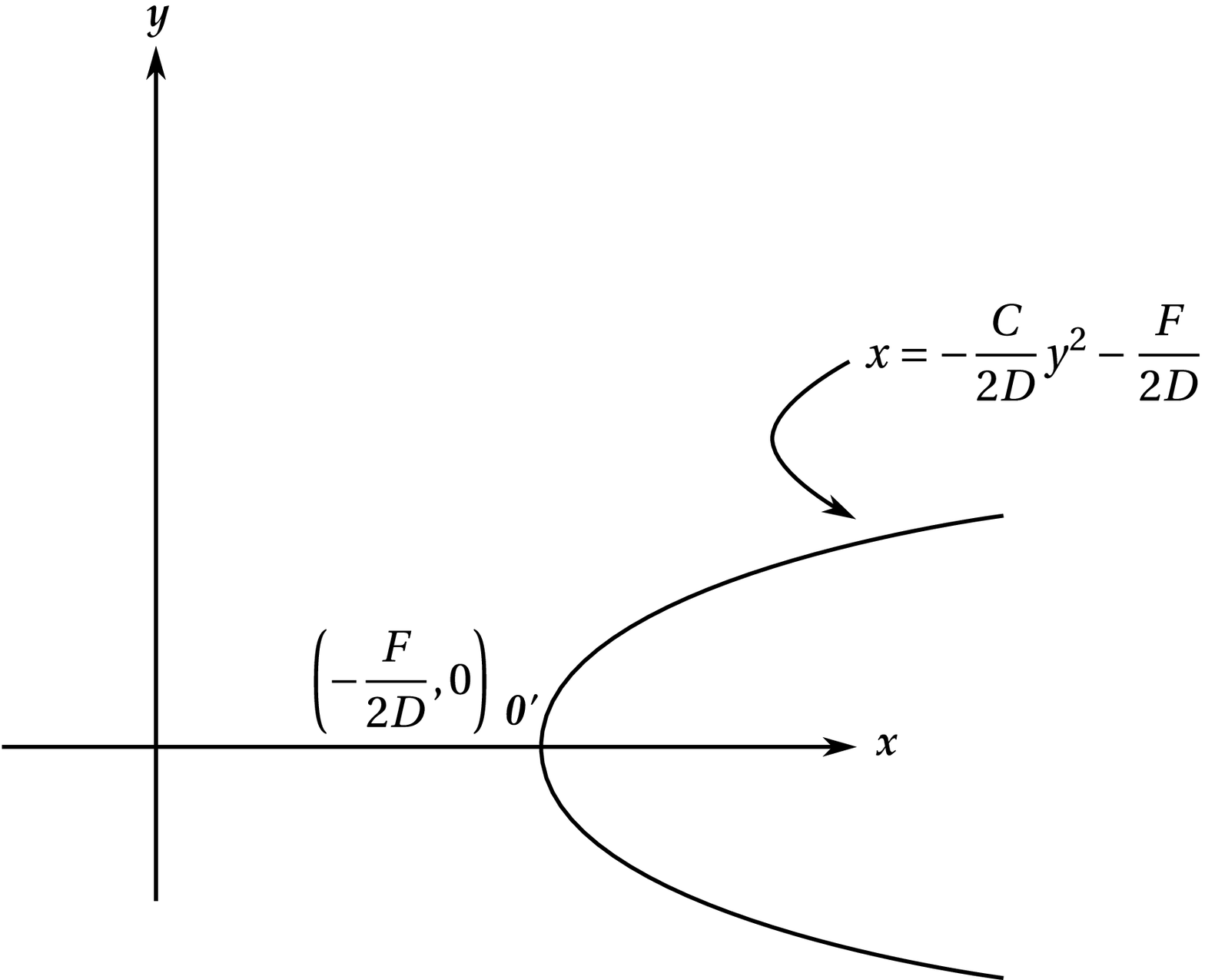}\\
\end{center}
\end{figure}
La \underline{cónica es una parábola} que se abre según el eje $x$. El vértice de la parábola es el punto $O'$ de coord.$\left(-\dfrac{F}{2D},0\right)\diagup xy.$
\newpage
\item[(4)] Hay un solo coeficiente $A,B,C$ que se anula. Se presentan varios casos.
\item[i)]
$\begin{tabular}{|c|c|c|}0&B&C\\ \hline&&\end{tabular}$ La cónica es $2Bxy^2+Cy^2+2Dx+2Ey+F=0.$\\
$\left\{\dbinom{A}{B},\dbinom{B}{C}\right\}=\left\{\dbinom{0}{B},\dbinom{B}{C}\right\}:\text{Base de $\mathbb{R}^2$}$
\begin{figure}[ht!]
\begin{center}
  \includegraphics[scale=0.5]{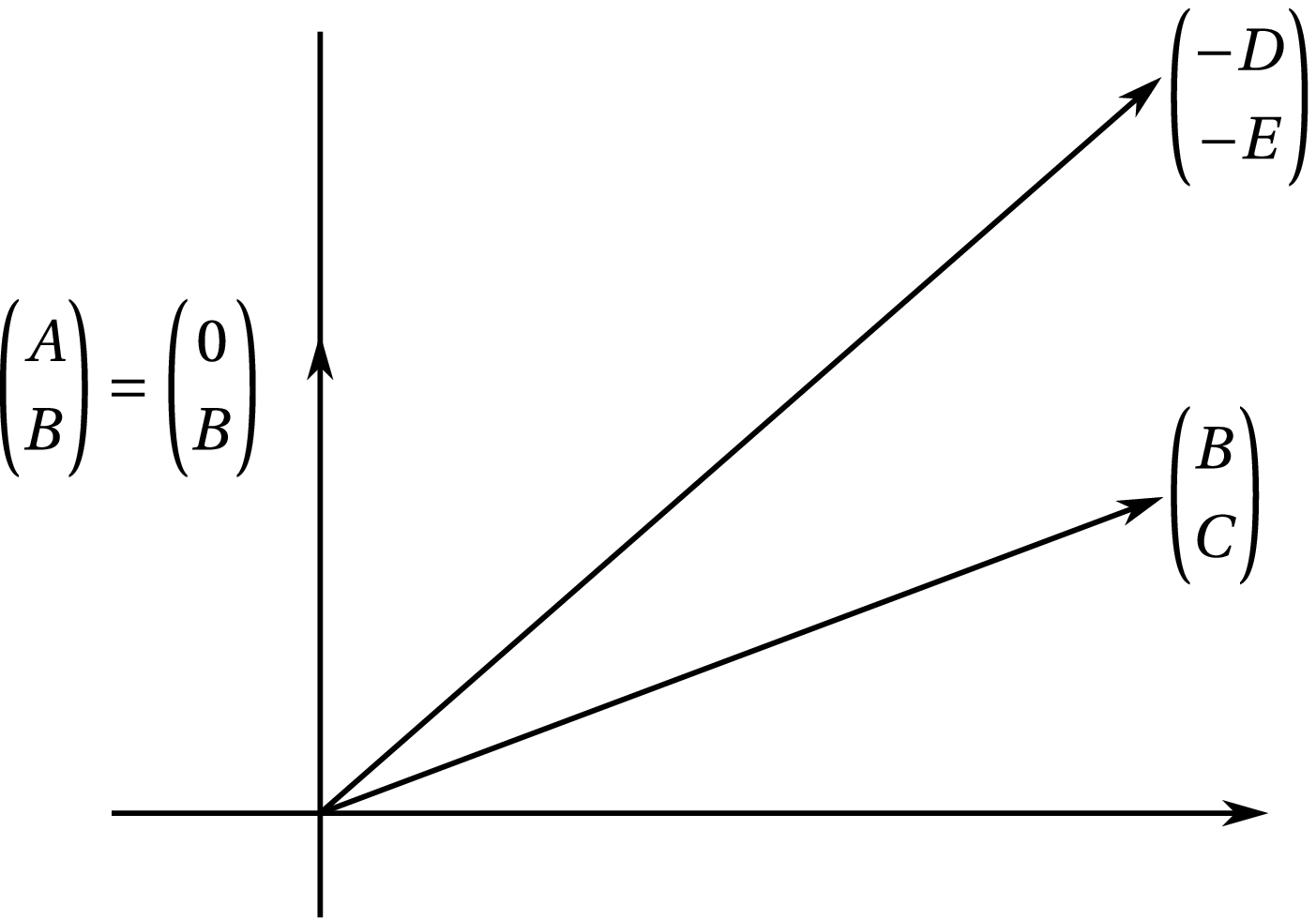}\\
\end{center}
\end{figure}
\newline
El sistema
\begin{align*}
Ah+Bk=-D\\
Bh+Ck=-E
\end{align*}
es ahora
$$\begin{cases}
Bk&=-D\\
Bh+Ck&=-E
\end{cases}$$\\
\underline{Hay centro único} que se determina así:\\
de la primera ecuación, $\overset{\sim}{k}=-\dfrac{D}{B}$ que llevamos a la segunda
\begin{align*}
B\overset{\sim}{k}=-E-C\overset{\sim}{k}&=-E-C\left(-\dfrac{D}{B}\right)\\
&=\dfrac{CD}{B}-E=\dfrac{CD-BE}{B}
\end{align*}
Así que el centro es el punto $$0'\left(\dfrac{CD-BE}{B^2}, -\dfrac{D}{B}\right)$$
Al trasladar los ejes al punto $0'\left(\overset{\sim}{h}, \overset{\sim}{k}\right),$ la ecuación de la cónica$\diagup XY$ es $$2BXY+CY^2+f\left(\overset{\sim}{h}, \overset{\sim}{k}\right)=0$$
Ahora, $$\Delta=\left|\begin{array}{ccc}0&B&0\\B&C&0\\0&0&f\left(\overset{\sim}{h}, \overset{\sim}{k}\right)\end{array}\right|=-B^2f\left(\overset{\sim}{h}, \overset{\sim}{k}\right)\hspace{0.5cm}\text{y como $B\neq 0$},$$
$$f\left(\overset{\sim}{h}, \overset{\sim}{k}\right)=-\dfrac{\Delta}{B^2}.$$
La ecuación de la cónica$\diagup XY$ es entonces
\begin{align*}
2BXY+CY^2=\dfrac{\Delta}{B^2}&\underset{\uparrow}{=}-\dfrac{\Delta}{\delta}\\
&\delta=AC-B^2=-B^2\neq 0
\end{align*}
\begin{figure}[ht!]
\begin{center}
  \includegraphics[scale=0.5]{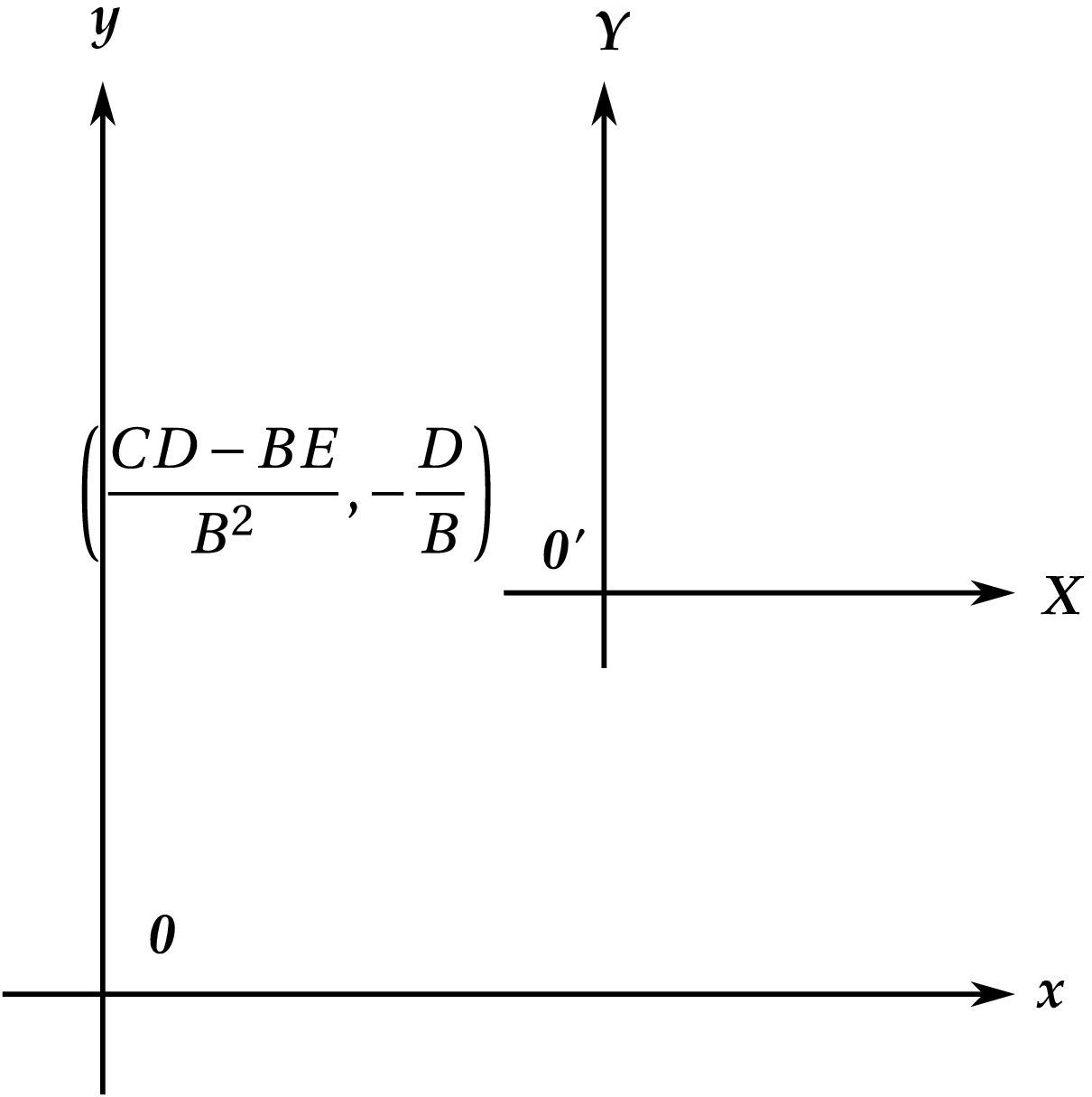}\\
\end{center}
\end{figure}
\newline
\item[i)] Si $\Delta=0, 2BXY+CY^2=0.$
$$Y\left(2BX+CY\right)=0\hspace{0.5cm}\therefore\hspace{0.5cm}Y=0\hspace{0.5cm}\text{ó}\hspace{0.5cm}Y=-\dfrac{2B}{C}X$$
El \underline{lugar consta de dos rectas concurrentes en $0'$}. y de ecuación $\diagup XY: Y=0\hspace{0.5cm}\text{y}\hspace{0.5cm}Y=-\dfrac{2B}{C}X$\\
\item[2)] Si $\Delta\neq 0$, la ecuación del lugar$\diagup XY$ es $2BXY+CY^2=-\dfrac{\Delta}{\delta}.$\\
El paso que sigue en la reducción es eliminar el término mixto. Pero eso se hará en la sección 1.5.\\
Regresamos al caso $i)$.\\
El lugar consta de las rectas de ecuación $Y=0$ y $Y=-\dfrac{2B}{C}X$ concurrentes en $0'$. Como
\begin{align*}
x&=X+\dfrac{CD-BE}{B^2}\\
\text{y}\hspace{0.5cm}y&=Y-\dfrac{D}{B},\hspace{0.5cm}X=x-\dfrac{CD-BE}{B^2}\hspace{0.5cm}Y=y+\dfrac{D}{B}
\end{align*}
y las ecuaciones de las rectas$\diagup xy$ son: $$y+\dfrac{C}{B}=0\hspace{0.5cm}\text{y}\hspace{0.5cm}y+\dfrac{D}{B}=-\dfrac{2B}{C}\left(x-\dfrac{CD-BE}{B^2}\right)$$
que luego de simplificar podemos escribir como $$2Bx+Cy+\dfrac{2BE-CD}{B}=0$$
Definimos $$\mathscr{L}_1=y+\dfrac{D}{B}$$ Y $$\mathscr{L}_2=2Bx+Cy+\dfrac{2BE-CD}{B}$$
Vamos a demostrar que la ecuación de la cónica $$f(x,y)=2Bxy+Cy^2+2Dx+2Ey+F$$ <<factoriza>> como el producto de dos rectas, i.e., que
\begin{equation}\label{24}
f(x,y)=2Bxy+Cy^2+2Dx+2Ey+F=\left(y+\dfrac{D}{B}\right)\left(2Bx+Cy+\dfrac{2BE-CD}{B}\right)
\end{equation}
Como
\begin{align*}
\Delta&=\left|\begin{array}{ccc}0&B&D\\ B&C&E\\ D&E&F\end{array}\right|=0,\hspace{0.5cm}-B\left(BF-DE\right)+D\left(BE-CD\right)=0\\
&\therefore F=\dfrac{2BDE-CD^2}{B^2}
\end{align*}
y $$f(x,y)=2Bxy+Cy^2+2Dx+2Ey+F=2Bxy+Cy^2+2Dx+2Ey+\dfrac{2BDE-CD^2}{B^2}$$
Así que para tener [\ref{24}], bastará con demostrar que
\begin{align*}
\left(y+\dfrac{D}{B}\right)\left(2Bx+Cy+\dfrac{2BE-CD}{B}\right)&=2Bxy+Cy^2+2Dx+2Ey+\dfrac{2BDE-CD}{B^2}\\
\left(y+\dfrac{D}{B}\right)\left(2Bx+Cy+\dfrac{2BE-CD}{B}\right)&=2Bxy+Cy^2+\cancel{\dfrac{2BE-CD}{B}}y+2Dx+\cancel{\dfrac{CD}{B}}y+\dfrac{2BDE-CD^2}{B^2}\\
&=Bxy+Cy^2+2Dx+2Ey+\dfrac{2BDE-CD^2}{B^2}
\end{align*}
\item[ii)]
$\begin{tabular}{|c|c|c|}A&0&C\\ \hline &&\end{tabular}$ La cónica es $Ax^2+Cy^2+2Dx+2Ey+F=0$.\\
$\left\{\dbinom{A}{B},\dbinom{B}{C}\right\}=\left\{\dbinom{A}{0},\dbinom{0}{C}\right\}:\text{Base de $\mathbb{R}^2$}$
\begin{figure}[ht!]
\begin{center}
  \includegraphics[scale=0.5]{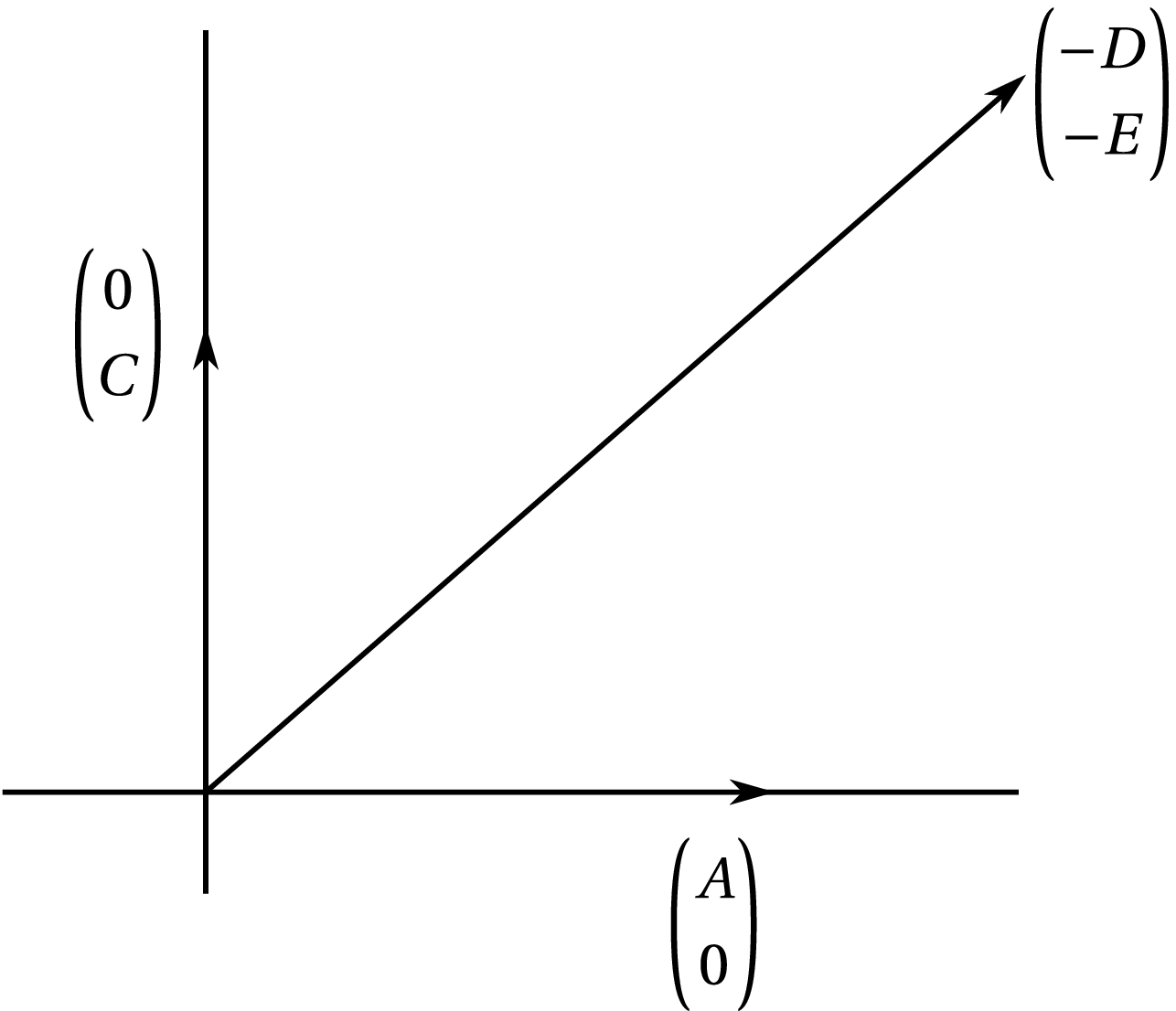}\\
\end{center}
\end{figure}
\newline
El sistema
\begin{align*}
Ah+Bk&=-D\\
Bh+Ck&=-E
\end{align*}
es en este caso
$$\begin{cases}
Ah=-D\\
Ck=-E
\end{cases}$$
\underline{Hay centro único}.\\
La ausencia del término mixto $(Bxy)$ permite que podamos realizar la reducción y no aplazarla para la sección 1.5.\\
Al realizar la traslación de ejes al punto $0'\left(-\underset{\overset{\parallel}{\widetilde{h}}}{\dfrac{D}{A}}, -\underset{\overset{\parallel}{\widetilde{k}}}{\dfrac{E}{C}}\right)$, se eliminan los términos lineales y la ecuación de la cónica$\diagup XY$ con origen en $0'$ es $AX^2+CY^2+f(\widetilde{h},\widetilde{k})=0$.\\
Ahora, $$\delta=\left|\begin{array}{cc}A&0\\0&C\end{array}\right|=AC\neq 0.\text{ya que $A$ y $C$ son $\neq 0$}$$
$$\Delta=\left|\begin{array}{ccc}A&0&0\\0&C&0\\0&0&f(\widetilde{h},\widetilde{k})\end{array}\right|=ACf(\widetilde{h},\widetilde{k})\delta f(\widetilde{h},\widetilde{k})\hspace{0.5cm}\text{y como $\delta\neq 0$},$$
$f(\widetilde{h},\widetilde{k})=\dfrac{\Delta}{\delta}.$ Así que la ecuación de la cónica$\diagup XY$ es: $$AX^2+CY^2=-\dfrac{\Delta}{\delta}.$$ Se tiene varios casos.\\
\item[1)] Si $\Delta>0$ y $\delta=AC<0,\,\,-\dfrac{\Delta}{\delta}>0.$\\
Como $AC<0,$ $A$ y $C$ tienen signos diferentes el \underline{lugar es una hipérbola}.\\
Si $\Delta>0$ y $\delta=AC>0,\,\,-\dfrac{\Delta}{\delta}<0$.\\
Como $AC>0$, $A$ y $C$ tienen el mismo signo.\\
Si $A>0$ Y $C>0$, $\omega=A+C>0$. \underline{El lugar es $\emptyset$}.\\
Si $A<0$ y $C<0$, $\omega=A+C<0$. El \underline{lugar es una elipse o una circunferencia}.\\
\item[2)] Si $\Delta=0$ y $\delta=AC>0,\,\,-\dfrac{\Delta}{\delta}=0$.\\
Pero si $AC>0,$ y $A$ y $C$ tienen el mismo signo.\\
El lugar es el punto $0'$ (el centro).\\
Si $\Delta=0$ y $\delta=AC<0,\hspace{0.5cm}-\dfrac{\Delta}{\delta}=0.$\\
Pero si $AC<0,$ $A$ y $C$ tienen signos contrarios.\\
\underline{El lugar son dos rectas concurrentes en $0'$}.\\
\item[3)] Si $\Delta<0$ y $\delta=AC<0,\hspace{0.5cm}-\dfrac{\Delta}{\delta}<0$.\\
Pero si $AC<0$, $A$ y $C$ tienen signos contrarios.\\
\underline{El lugar es una hipérbola}.\\
Si $\Delta<0$ y $\delta=AC>0,\hspace{0.5cm}-\dfrac{\Delta}{\delta}>0$.\\
Pero si $AC>0,$ y $A$ y $C$ tienen el mismo signo.\\
\item[$\bullet$] Si $A>0$ y $C>0$, $\omega=A+C>0$.\\
\underline{El lugar es una elipse o una circunferencia}.\\
\item[$\bullet$] Si $A<0$ y $C<0$, $\omega=A+C<0$.\\
\underline{El lugar es $\emptyset$}.
En resumen, la naturaleza del lugar se determina a través de la siguiente tabla\\
$\Delta>0\begin{cases}\delta<0:\text{hipérbola}\\
\delta>0\begin{cases}\omega>0:\emptyset\\ \omega<0:\text{elipse o circunferencia}\end{cases}\end{cases}$\\\\
$\Delta=0\begin{cases}\delta>0:\text{el punto $0'$ (el centro)}\\ \delta<0:\text{dos rectas concurrentes en $0'$}\end{cases}$\\\\
$\Delta<0\begin{cases}\delta<0:\text{hipérbola}\\ \delta>0\begin{cases}\omega>0:\text{elipse ó circunferencia}\\
\omega<0:\emptyset\end{cases}\end{cases}$\\\\
\item[iii)] $\begin{tabular}{|c|c|c|}A&B&0\\ \hline &&\end{tabular}$. La cónica es $Ax^2+2Bxy+2Dx+2Ey+F=$.\\
$\left\{\dbinom{A}{B},\dbinom{B}{C}\right\}=\left\{\dbinom{A}{B},\dbinom{B}{0}\right\}:\text{Base de $\mathbb{R}^2$}$
\begin{figure}[ht!]
\begin{center}
  \includegraphics[scale=0.5]{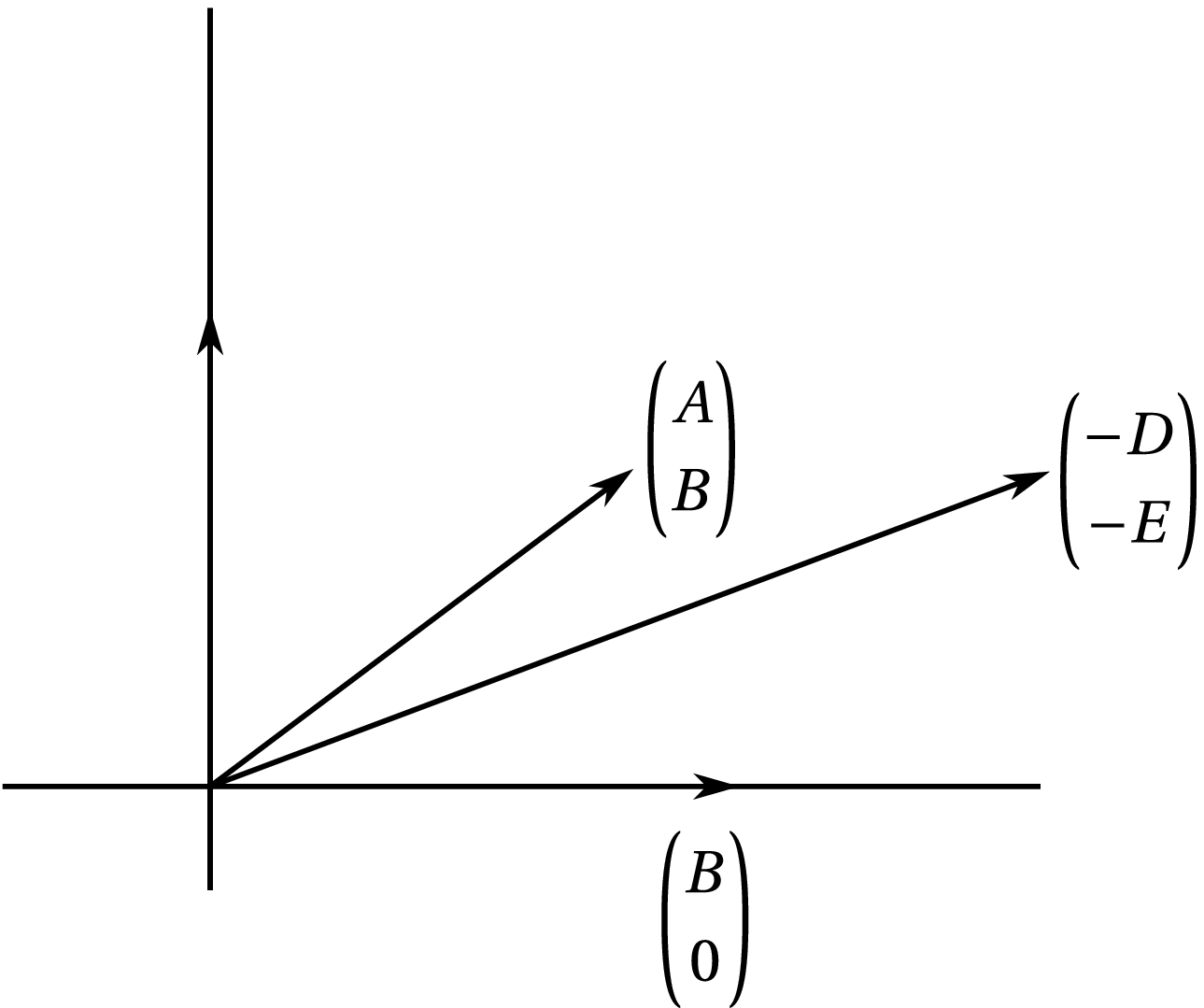}\\
\end{center}
\end{figure}
\newline
El sistema
\begin{align*}
Ah+Bk&=-D\\
Bh+Ck&=-E
\end{align*}
es ahora\\
$$\begin{cases}
Ah+Bk=-D\\
Bh=-E
\end{cases}$$\\
Hay \underline{centro único}.\\
Las coordenadas del centro se obtiene así:\\
De la $2^a$ ecuación, $\hat{h}=-\dfrac{E}{B}$ que llevamos a la $1^a$:
\begin{align*}
B\hat{k}&=-D-A\hat{h}\\
&=-D-A\left(-\dfrac{E}{B}\right)=\dfrac{AE}{B}-D=\dfrac{AE-BD}{B}\\
\therefore\hspace{0.5cm}\hat{k}=\dfrac{AE-BD}{B^2}
\end{align*}
Al trasladar los ejes al punto $0'(\hat{h},\hat{k}),$ la ecuación de la cónica$/XY$ es: $AX^2+2BXY+f(\hat{h},\hat{k})=0$ $$\Delta=\left|\begin{array}{ccc}A&B&0\\B&0&0\\0&0&f(\hat{h},\hat{k})\end{array}\right|=-B^2f(\hat{h},\hat{k})\hspace{0.5cm}\text{y como $B\neq 0$},$$
$$f(\hat{h},\hat{k})=-\dfrac{\Delta}{B^2}.$$ La ecuación de la cónica$\diagup XY$ es:
\begin{align*}
AX^2+2BXY=\dfrac{\Delta}{B^2}&\underset{\uparrow}{=}-\dfrac{\Delta}{\delta}\\
&\delta=AC-B^2=-B^2\neq 0
\end{align*}
\begin{figure}[ht!]
\begin{center}
  \includegraphics[scale=0.5]{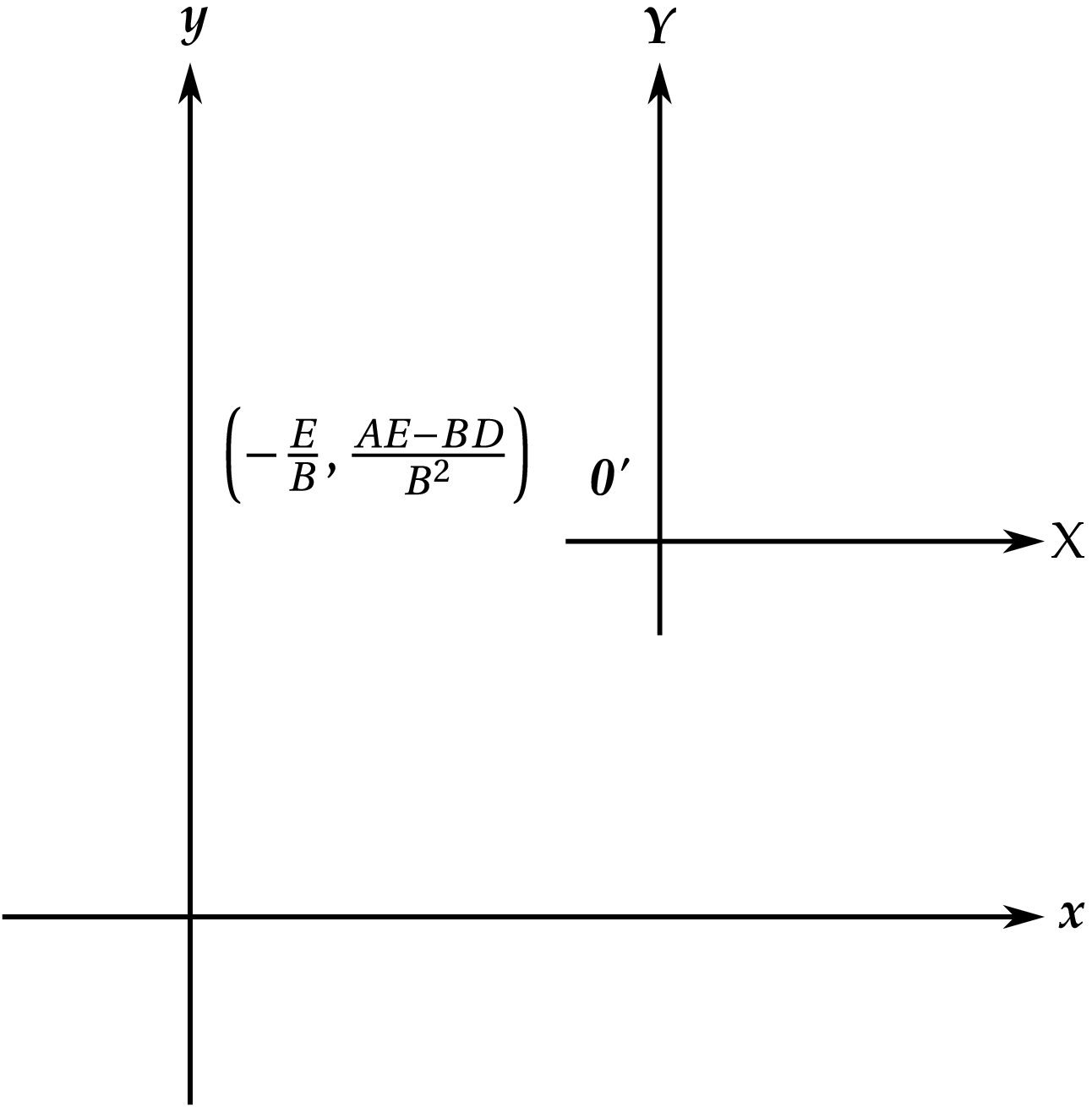}\\
\end{center}
\end{figure}
\newline
\item[1)] Si $\Delta=0,$
\begin{align*}
AX^2+2BXY=0\\
x\left(AX+2BY\right)=0\hspace{0.5cm}\therefore\hspace{0.5cm}X=0\hspace{0.5cm}\text{ó}\hspace{0.5cm}Y=-\dfrac{A}{2B}X
\end{align*}
El lugar consta de dos rectas concurrenter en $0'$ y de ecuación$\diagup XY;X=0\,\,\text{y}\,\,Y=-\dfrac{A}{2B}X$.\\
Demuestre como se hizo en el caso $\begin{tabular}{|c|c|c|}&&\\\hline
0&B&C
\end{tabular}$ que la cónica <<factoriza>> como el producto de dos rectas definidas al sistema $x-y.$\\
\item[2)] Si $\Delta\neq 0$ la ecuación del lugar$\diagup XY$ es $AX^2+2BXY=-\dfrac{\Delta}{\delta}$.\\
Lo que sigue en la reducción es eliminar el término mixto y esto se hace en la sección 1.5.
\end{enumerate}
\end{ejer}
\begin{ejer}
Vamos a estudiar el problema de los centros en dos casos en que ninguno de $A,B,C$ es cero.\\
Puede ocurrir:
\begin{enumerate}
\item[I)] que $\left\{\dbinom{A}{B}, \dbinom{B}{C}\right\}\subset\mathbb{R}^2$ sea L.I.\\
Entonces $\left\{\dbinom{A}{B}, \dbinom{B}{C}\right\}$ es base de $\mathbb{R}^2,\hspace{0.5cm}\left(\begin{array}{cc}A&B\\B&C\end{array}\right)$ es No singular y como los vectores $\dbinom{A}{B}$ y $\dbinom{B}{C}$ definen un paralelogramo, el area de este es $\left|\begin{array}{cc}A&B\\B&C\end{array}\right|=AC-B^2\neq 0$. Además, $\exists!(\widetilde{h},\widetilde{k})$ tal que
\begin{align*}
A\widetilde{h}+B\widetilde{k}&=-D\\
B\widetilde{h}+C\widetilde{k}&=-E
\end{align*}
$\widetilde{h}$ y $\widetilde{k}$ se calculan así:\\
$$\dbinom{\widetilde{h}}{\widetilde{k}}=\left(\begin{array}{cc}A&B\\B&C\end{array}\right)\dbinom{-D}{-E}=\dfrac{1}{AC-B^2}\left(\begin{array}{cc}C&-B\\-B&A\end{array}\right)\dbinom{-D}{-E}$$
\begin{figure}[ht!]
\begin{center}
  \includegraphics[scale=0.4]{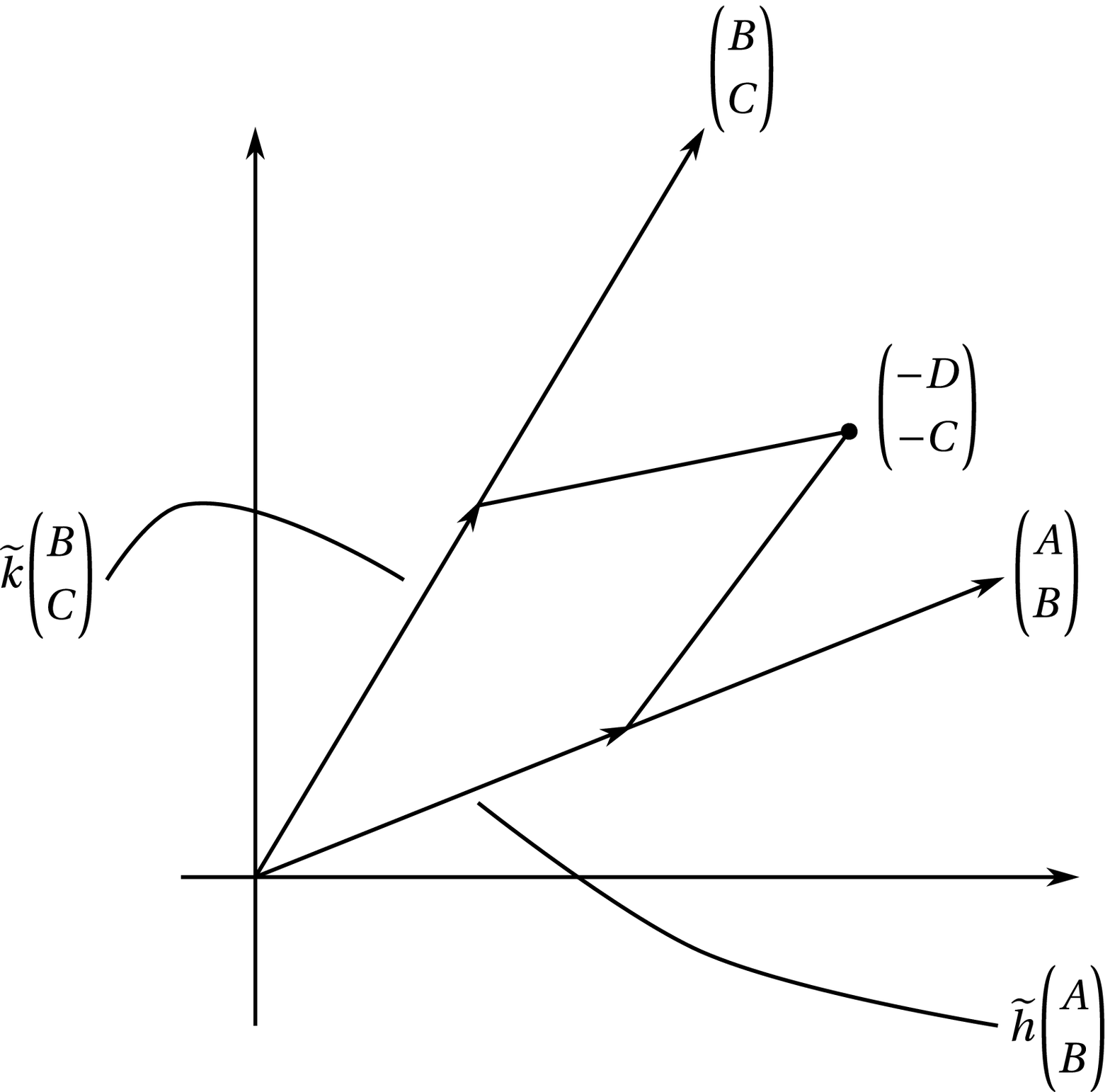}\\
\end{center}
\end{figure}
La cónica \underline{tiene centro único}; $(\widehat{h}, \widehat{k}).$\\
\item[II)] que $\left\{\dbinom{A}{B}, \dbinom{B}{C}\right\}$ sea L.D.\\
En este caso ambos vectores están aplicados sobre una misma línea que no puede ser ni el eje $x$ ni el eje $y$.\\
Así que lo que se tiene es que $\dbinom{B}{C}=\alpha\dbinom{A}{B}$
\begin{figure}[ht!]
\begin{center}
  \includegraphics[scale=0.5]{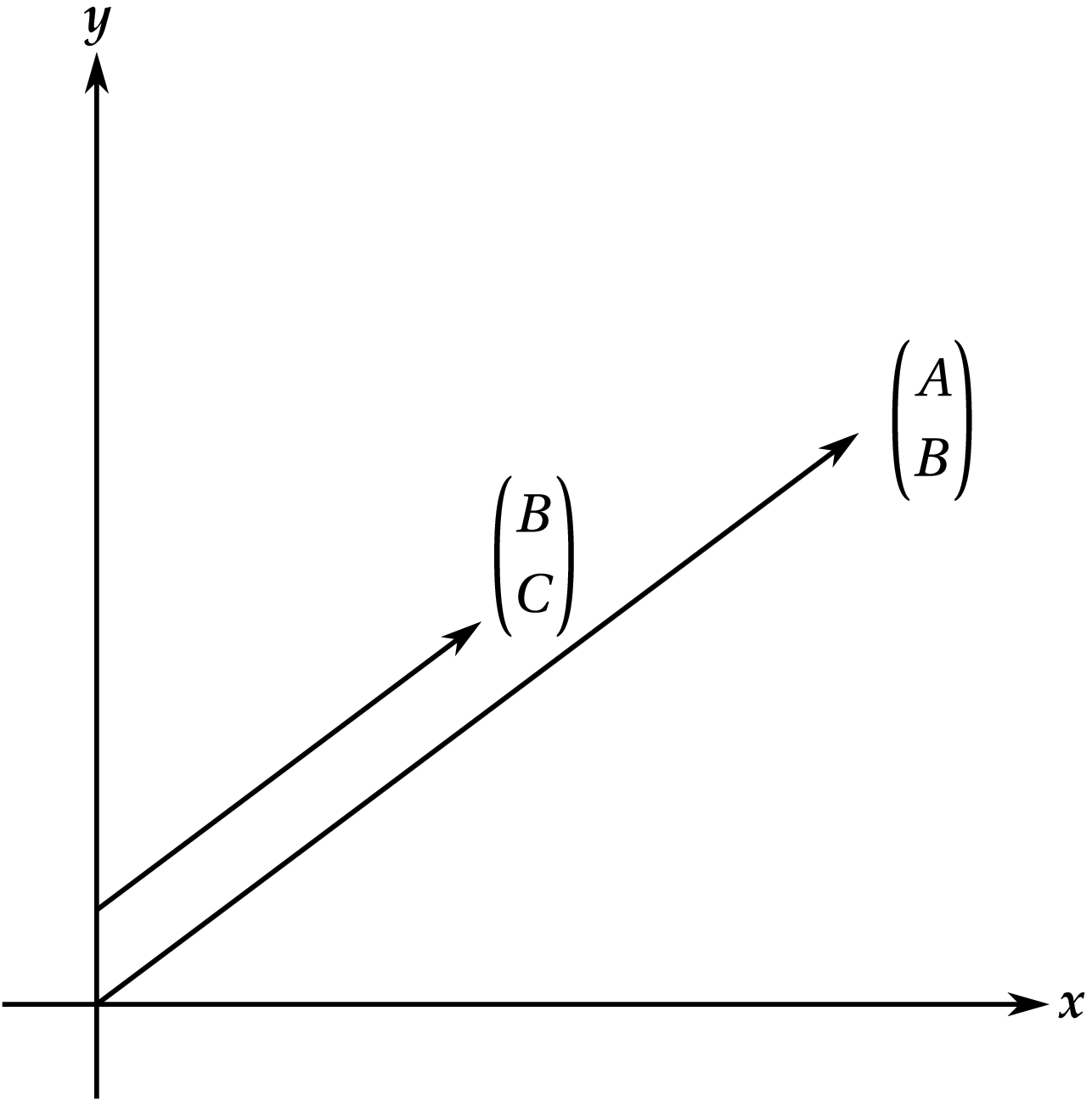}\\
\end{center}
\end{figure}
Como $\dbinom{A}{B}$ y $\dbinom{B}{C}$ son $\parallel$s, el area del paralelogramo de $\dbinom{A}{B}$ y $\dbinom{B}{C}$ es cero y esto quiero decir de $\delta=\left|\begin{array}{cc}A&B\\B&C\end{array}\right|=0$.\\
Ahora, puede tenerse\\
\item[i)] que\\
$\dbinom{-D}{-E}\notin\text{Sg}\left\{\dbinom{A}{B}\right\}=\text{Sg}\left\{\dbinom{B}{C}\right\};\dbinom{-D}{-E}\in\text{Sg}\dbinom{A}{B}\Longleftrightarrow\left|\begin{array}{cc}A&-D\\B&-E\end{array}\right|=0$
\begin{figure}[ht!]
\begin{center}
  \includegraphics[scale=0.5]{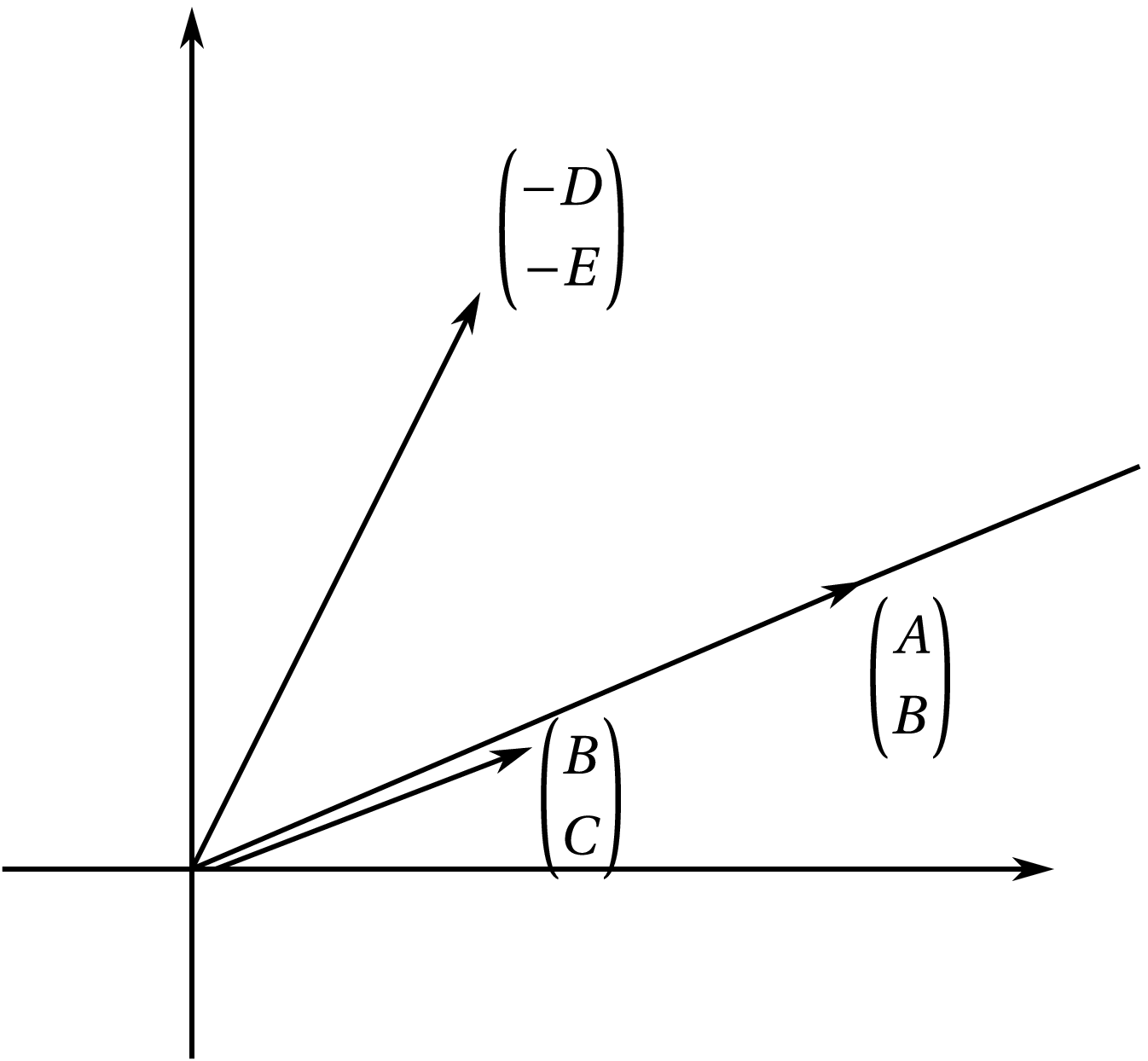}\\
\end{center}
\end{figure}
\newpage
Entonces $\delta=\left|\begin{array}{cc}A&B\\B&C\end{array}\right|=0$ pero $\nexists(h,k)$ tal que $h\dbinom{A}{B}+k\dbinom{B}{C}=\dbinom{-D}{-E}$.\\
O se aque el sistema
\begin{align*}
Ax+By&=-D\\
Bx+Cy&=-E\hspace{0.5cm}\text{No tiene solución}.
\end{align*}
Así que $\left|\begin{array}{cc}A&B\\B&C\end{array}\right|=0$ y  \underline{la cónica No tiene centro}.\\
\item[ii)] que $\dbinom{-D}{-E}\in Sg\left\{\dbinom{A}{B}\right\}=Sg\left\{\dbinom{B}{C}\right\}$.\\
Entonce
\begin{equation}
\dbinom{-D}{-E}=\beta_1\dbinom{A}{B}
\end{equation}
Ahora, como
\begin{equation}
\left.
\begin{split}
Sg\left\{\dbinom{A}{B}\right\}=Sg\left\{\dbinom{B}{C}\right\}, \dbinom{B}{C}&=\alpha\dbinom{A}{B}\\
\therefore\hspace{0.5cm}B&=\alpha A\\
C&=\alpha B
\end{split}
\right\}
\end{equation}
\begin{figure}[ht!]
\begin{center}
  \includegraphics[scale=0.5]{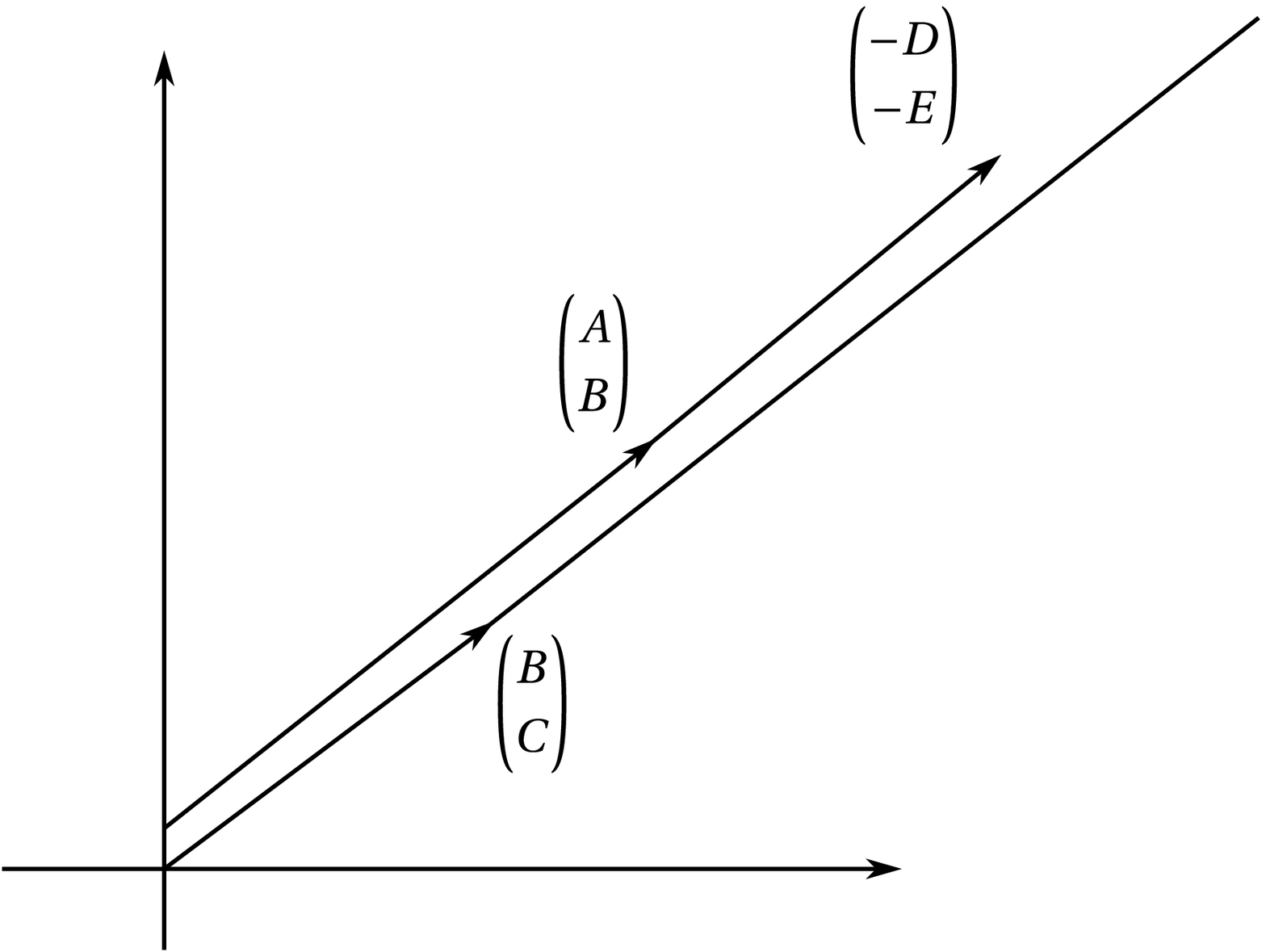}\\
\end{center}
\end{figure}
\newline
O sea que
\begin{align*}
\dbinom{-D}{-E}&=\beta_1\dbinom{A}{B}+\gamma_1\cancel{\dbinom{B}{C}}-\gamma_1\cancel{\dbinom{B}{C}},\hspace{0.5cm}\gamma_1\in\mathbb{R}\\
&=\beta_1\dbinom{A}{B}+\gamma_1\alpha\dbinom{A}{B}-\gamma_1\dbinom{B}{C}\\
&\overset{\nearrow}{(2)}\\
&=\left(\beta_1+\gamma_1\alpha\right)\dbinom{A}{B}-\gamma_1\dbinom{B}{C},\hspace{0.5cm}\text{lo que nos demuestra que}
\end{align*}
$\star\star\star\begin{cases}
Ah+Bk=-D\\
Bh+Ck=-E\hspace{0.5cm}\text{tiene $\infty$s soluciones y la \underline{cónica tiene $\infty$s centros}}
\end{cases}$\\
\item[a)] Vamos a demostrar que en este caso la $2^a$ ecuación del sistema
\begin{align*}
Ah+Bk&=-D\\
Bh+Ck&=-E\hspace{0.5cm}\text{es múltiplo de la $1^a$}.
\end{align*}
Si multilplicamos la $1^a$ por $\alpha$, $\alpha Ah+\alpha Bk=-\alpha D$.\\
Veamos que la esta ecuación es $Bh+Ck=-E$.\\
Como $\alpha A=B$ y $\alpha B=C,$ la ecuación anterior se convierte en
\begin{figure}[ht!]
\begin{center}
  \includegraphics[scale=0.5]{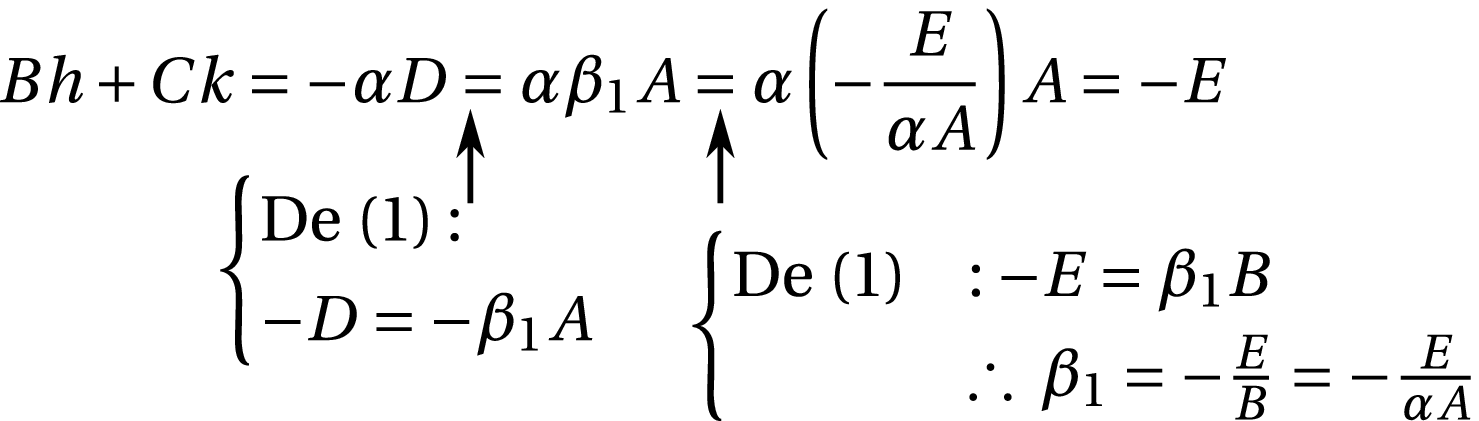}\\
\end{center}
\end{figure}
\newline
\item[b)] La segunda ecuación del sistema $\star\star\star$ es redundante y el conjunto de soluciones de
$\star\star\star$ se obtiene al resolver la \\$1^a: Ah+Bk=-D$. Entonces todos los centros de la cónica están sobre la recta $Ax+By=-D$ que se llama <<\underline{el eje de centros de la cónica}>>.\\
De esta ecuación, $k=-\dfrac{A}{B}h-\dfrac{D}{B}$.\\
Vamos a demostrar que $\forall h\in\mathbb{R}, \left(h,-\dfrac{A}{B}h-\dfrac{D}{B}\right)$ es solución al sistema.\\
Es claro que $\left(h,-\dfrac{A}{B}h-\dfrac{D}{B}\right)$ es solución a la $1^a$ ecuación.\\
Solo resta demostrar que es solución a la $2^a$, o sea que
\begin{align*}
Bh+C\left(-\dfrac{A}{B}h-\dfrac{D}{B}\right)=-E\\
Bh+C\left(-\dfrac{A}{B}h-\dfrac{D}{B}\right)&\underset{\uparrow}{=}Bh+\dfrac{B^2}{A}\left(-\dfrac{A}{B}h-\dfrac{D}{B}\right)\\
&\begin{cases}
\text{Como}\hspace{0.5cm}&AC-B^2=0,\\
&C=\dfrac{B^2}{A}
\end{cases}
\end{align*}
\end{enumerate}
\begin{figure}[ht!]
\begin{center}
  \includegraphics[scale=0.5]{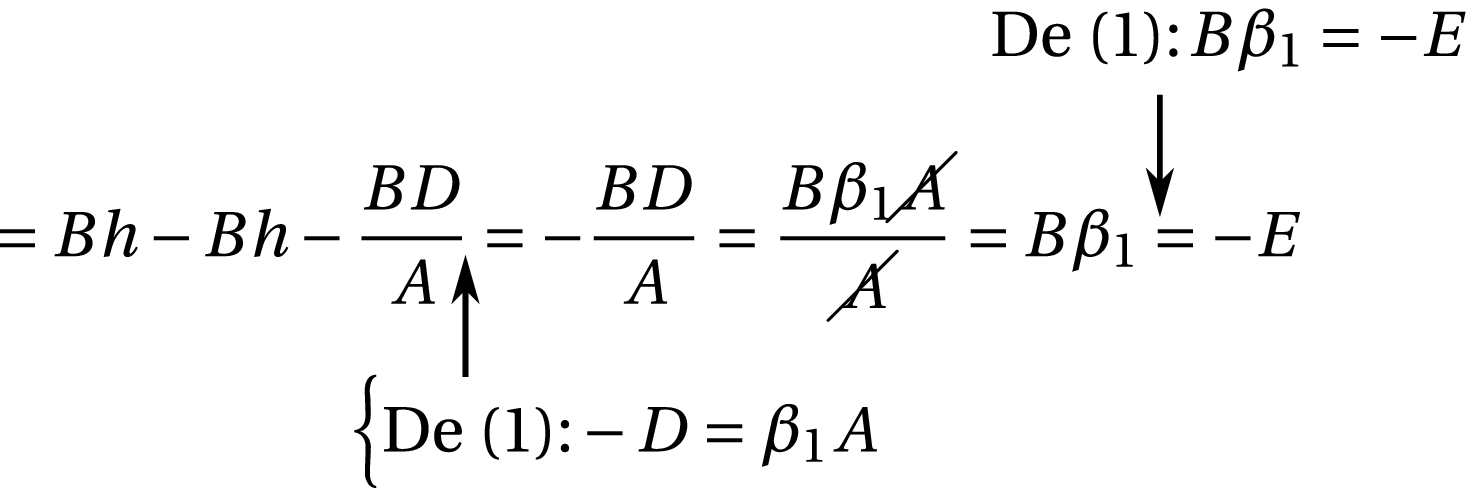}\\
\end{center}
\end{figure}
Toda la discusión nos ha demostrado que en el supuesto de que A,B,C sean cero,
$$\begin{cases}
\text{Si}\,\,\delta=AC-B^2\neq 0,\,\,\text{la cónica tiene centro único}.\\
\text{Si}\,\,\delta=AC-B^2=0,\,\,\text{la cónica no tiene centro, o tiene $\infty$s centros.}
\end{cases}$$
\end{ejer}
\begin{ejem}
Consideremos la cónica $3x^2+4x+2y+1=0$.\\
\begin{tabular}{cccc}
$A=3$&&$2D=4;$&$D=2$\\
$2B=0;$&$B=0$&\\
$C=0$&&$2E=2;$&$E=1$
\end{tabular}
\newline
El sistema
$$\begin{cases}
Ah+Bk=-D\\
Bh+Ck=-E
\end{cases}$$
es ahora
$$\begin{cases}
3h+0k=-2\\
0h+ok=-1
\end{cases}$$
La cónica \underline{No tiene centro}.
\end{ejem}
\begin{ejem}
$9x^2-12xy+4y^2+9x-6y+2=0$.\\
\begin{tabular}{cccc}
$A=9$&&&\\
$2B=-12; B=-6$&&$\dbinom{A}{B}=\dbinom{9}{-6}=3\dbinom{3}{-2}$\\
$C=4$&&&\\
$2D=9; D=\frac{9}{2}$&&$\dbinom{B}{C}=\dbinom{-6}{4}=-2\dbinom{3}{-2}$\\
$2E=-6; E=-3$&&$\dbinom{-D}{-E}=\dbinom{-9/2}{3}=-3/2\dbinom{3}{-2}$
\end{tabular}
\begin{figure}[ht!]
\begin{center}
  \includegraphics[scale=0.5]{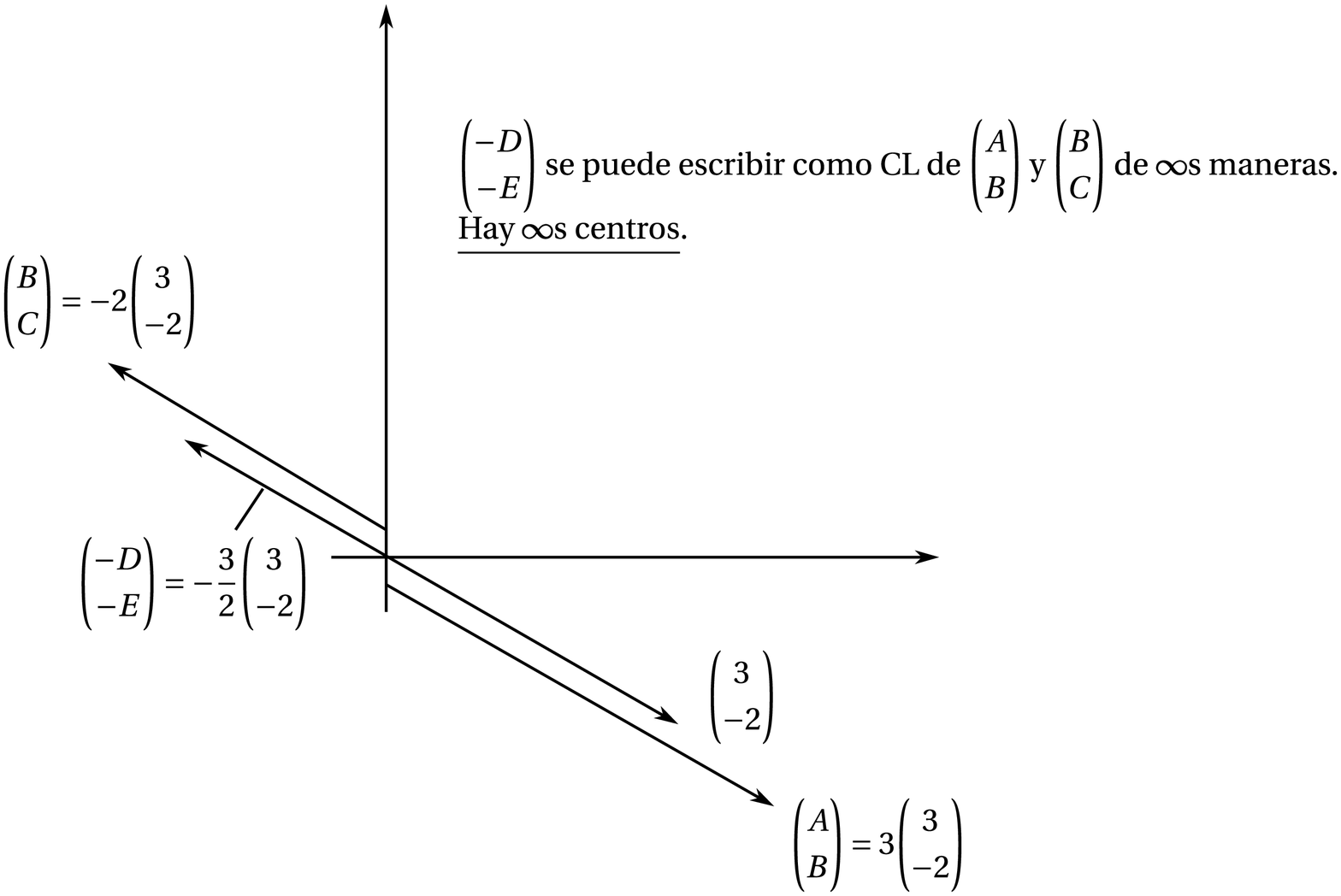}\\
\end{center}
\end{figure}
\end{ejem}
\begin{ejem}
$12x^2-12xy+3y^2+2x+4x+1=0$.\\
\begin{tabular}{cccc}
$A=12$&&&\\
$2B=-12; B=-6$&&$\dbinom{A}{B}=\dbinom{12}{-6}=-6\dbinom{-2}{1}$\\
$C=3$&&&\\
$2D=2; D=1$&&$\dbinom{B}{C}=\dbinom{-6}{3}=3\dbinom{-2}{1}; \dbinom{-D}{-E}=\dbinom{-1}{-2}$\\
$2E=4; E=2$&&&
\end{tabular}
\begin{figure}[ht!]
\begin{center}
  \includegraphics[scale=0.5]{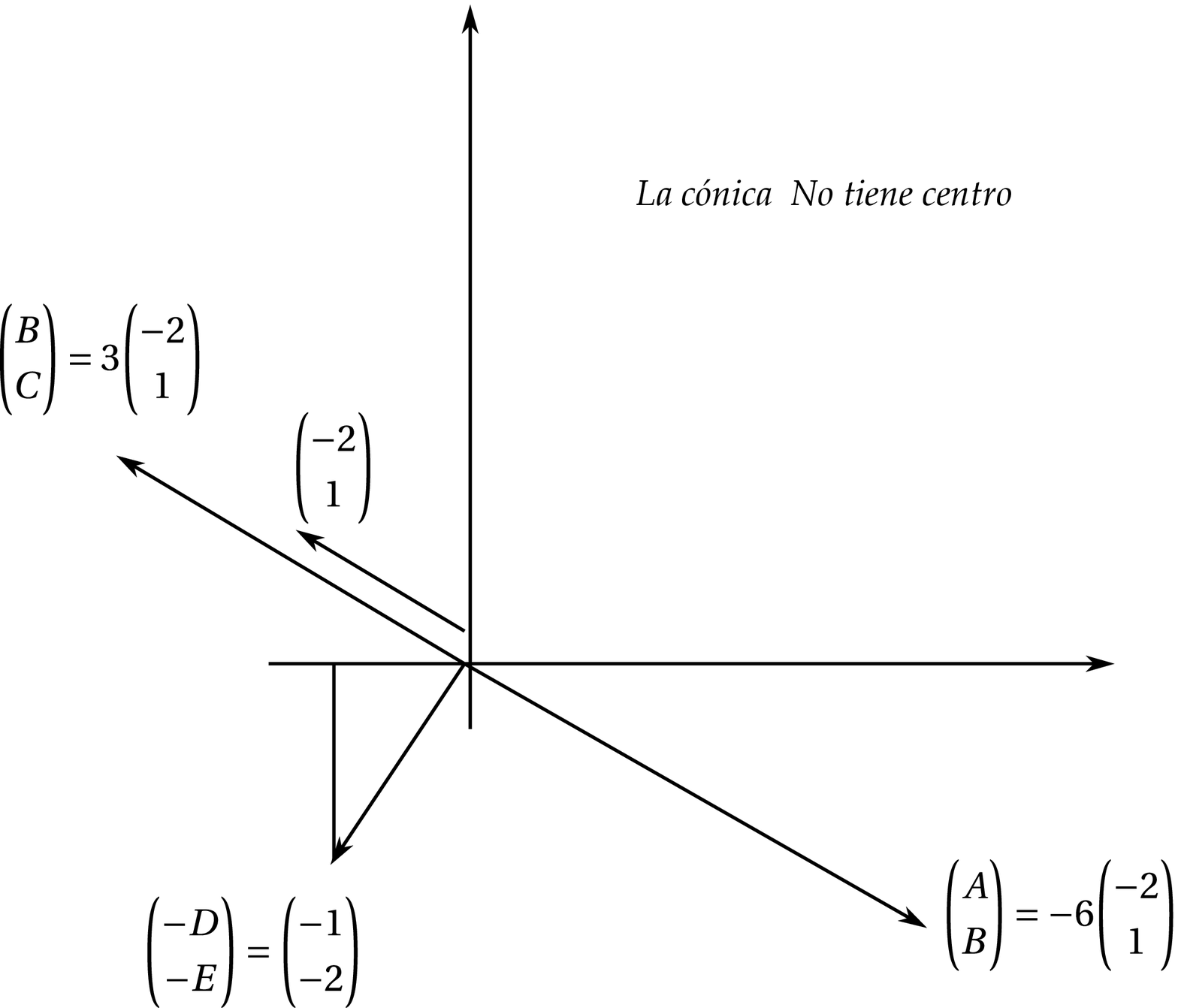}\\
\end{center}
\end{figure}
\end{ejem}
\begin{ejem}
$12x^2-12xy+3y^2+2x-y-\frac{1}{4}=0$. (Véase la diferencia con el ejemplo anterior.)\\
\begin{tabular}{cccc}
$A=12$&&&\\
$2B=-12; B=-6$&&$\dbinom{-D}{-E}=\dbinom{-1}{1/2}=\frac{1}{2}\dbinom{-2}{1}$\\
$C=3$&&&\\
$2D=2; D=1$&&&\\
$2E=-1; E=-1/2$
\end{tabular}
\newpage
\begin{figure}[ht!]
\begin{center}
  \includegraphics[scale=0.5]{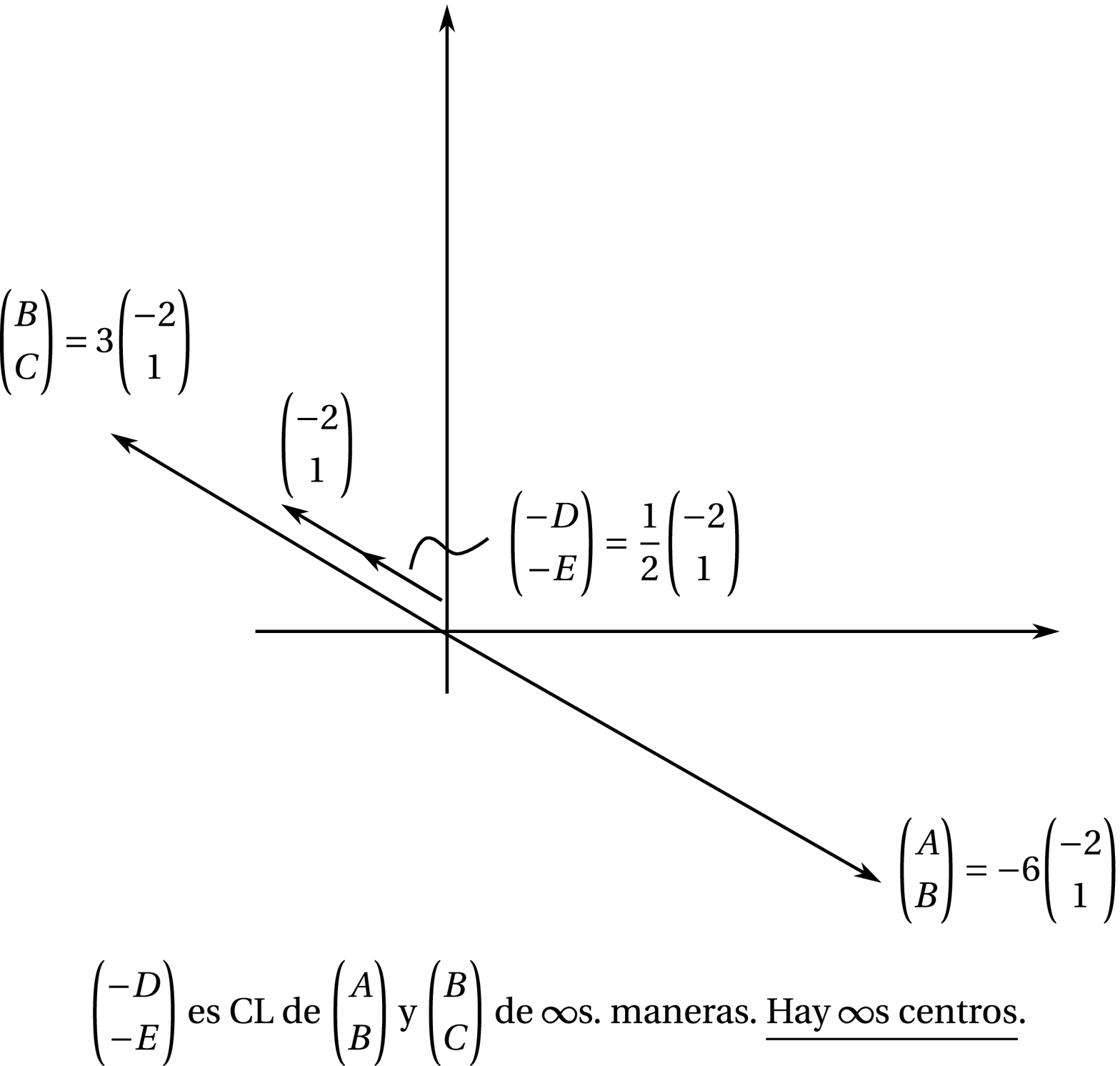}\\
\end{center}
\end{figure}
\end{ejem}
\begin{prop}
Todo centro de una cónica es centro de simetría de la curva.
\end{prop}
\begin{proof}
Sea $O$ un punto de coordenadas $(h,k)\diagup xy,$ $O:$ centro de simetría de la ecuación. O sea que
\begin{equation}\label{27}
\left.\dfrac{\partial f}{\partial x}\right)_{h,k}=\left.\dfrac{\partial f}{\partial y}\right)_{h,k}=0
\end{equation}
Sea $P(X,Y)$ un punto de la cónica (Fig.) y llamamos $Q(X',Y')$ a su eje simétrico$\diagup O$. Se debe demostrar que $Q(X',Y')$ está en la cónica, o lo que es lo mismo, que $$q(X',Y')+f(h,k)=0$$
\begin{figure}[ht!]
\begin{center}
\includegraphics[scale=0.3]{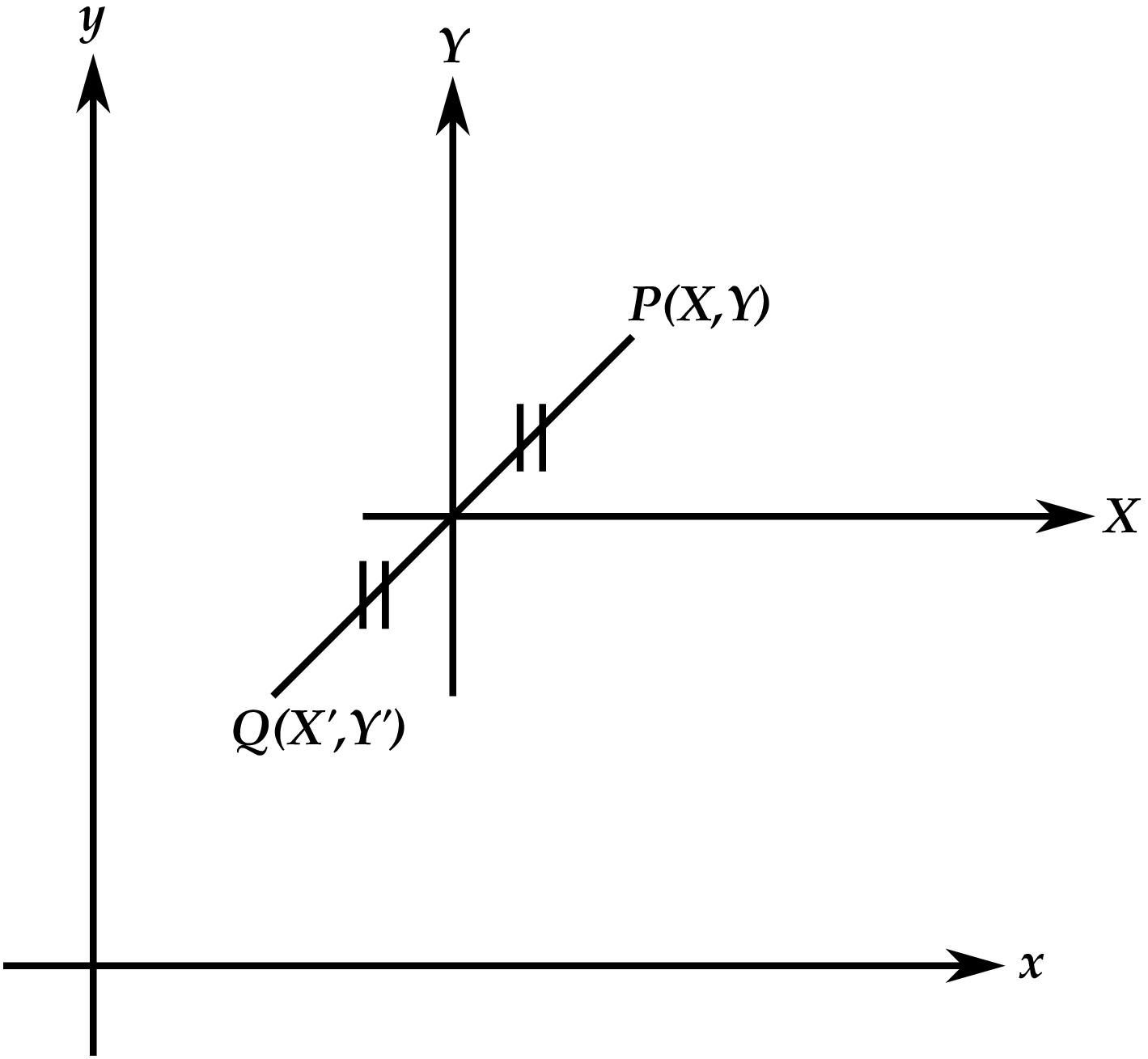}\\
\caption{}\label{}
\end{center}
\end{figure}
\newline
Como $Q(X',Y')$ es simétrico de $P(X,Y)\diagup 0$, $X'=X$ y $Y'=Y$.\\
Luego
\begin{align*}
q(X',Y')+f(h,k)&=\left(X' Y'\right)\left(\begin{array}{cc}A&B\\B&C\end{array}\right)\dbinom{X'}{Y'}+(h,k)\\
&=\left(-X -Y\right)\left(\begin{array}{cc}A&B\\B&C\end{array}\right)\dbinom{-X}{-Y}+(h,k)\\
&=\left(X Y\right)\left(\begin{array}{cc}A&B\\B&C\end{array}\right)\dbinom{X'}{Y'}+(h,k)\\
&=q(X,Y)+f(h,k)\underset{\overset{\uparrow}{[\ref{27}]}}{=}0
\end{align*}
\end{proof}
\begin{prop}
Si $(h,k)$ es un centro de simetría de la cónica, $(h,k)$ es un centro de ella.
\end{prop}
\begin{proof}
Supongamos que $(h,k)$ es centro de simetría de la curva.\\
La ecuación de la cónica referida a los eje $XY\parallel$s a $xy$ y con origen en $O'(h,k)$ es:
$$\left(\begin{array}{cc}X& Y\end{array}\right)\left(\begin{array}{cc}A&B\\B&C\end{array}\right)+\left(\begin{array}{cc}\left.\dfrac{\partial f}{\partial x}\right)_{h,k}&\left.\dfrac{\partial f}{\partial x}\right)_{h,k}\end{array}\right)\dbinom{X}{Y}+f(h,k)=0$$
\begin{figure}[ht!]
\begin{center}
  \includegraphics[scale=0.4]{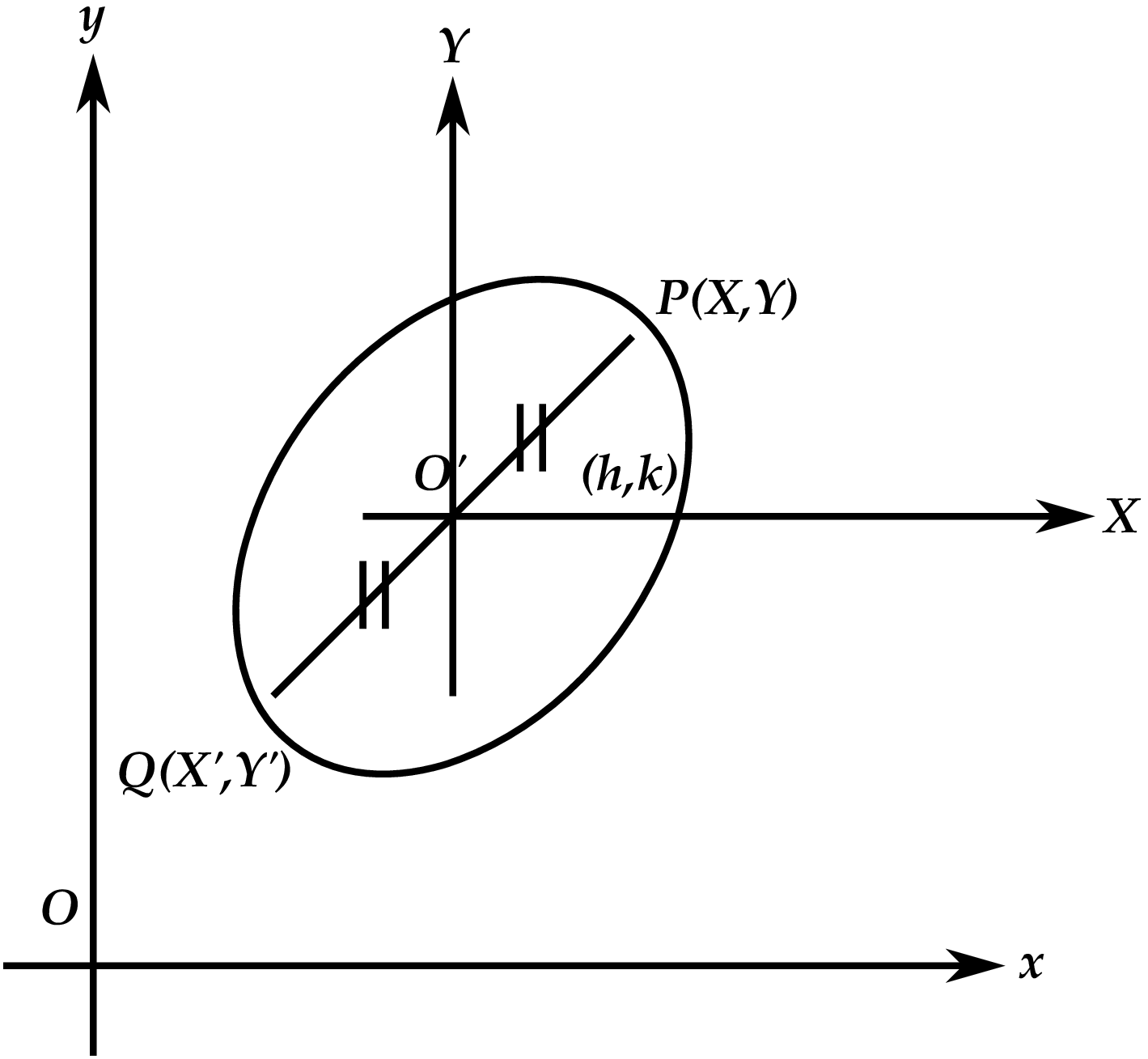}\\
\end{center}
\end{figure}
\newline
Sea $P(X,Y)$ un punto de la cónica y $Q(X',Y')$ el simétrico de $P\diagup O'$, o sea que $X'=-X,\,\,Y'=Y$.\\
Como $P(X,Y)$ está en la curva,
\begin{equation}\label{28}
\left(\begin{array}{cc}X& Y\end{array}\right)\left(\begin{array}{cc}A&B\\B&C\end{array}\right)+\left(\begin{array}{cc}\left.\dfrac{\partial f}{\partial x}\right)_{h,k}&\left.\dfrac{\partial f}{\partial x}\right)_{h,k}\end{array}\right)\dbinom{X}{Y}+f(h,k)=0
\end{equation}
Como $Q(-X,-Y)$ también está en la curva,
\begin{equation*}
\left(\begin{array}{cc}-X& -Y\end{array}\right)\left(\begin{array}{cc}A&B\\B&C\end{array}\right)+\left(\begin{array}{cc}\left.\dfrac{\partial f}{\partial x}\right)_{h,k}&\left.\dfrac{\partial f}{\partial x}\right)_{h,k}\end{array}\right)\dbinom{-X}{-Y}+f(h,k)=0
\end{equation*}
i.e.,
\begin{equation}\label{29}
\left(\begin{array}{cc}-X& -Y\end{array}\right)\left(\begin{array}{cc}A&B\\B&C\end{array}\right)\dbinom{X}{Y}-\left(\begin{array}{cc}\left.\dfrac{\partial f}{\partial x}\right)_{h,k}&\left.\dfrac{\partial f}{\partial x}\right)_{h,k}\end{array}\right)\dbinom{X}{Y}+f(h,k)=0
\end{equation}
De [\ref{28}] y [\ref{29}]: $$2\left(\left.\dfrac{\partial f}{\partial x}\right|_{h,k},\left.\dfrac{\partial f}{\partial y}\right|_{h,k}\right)\dbinom{X}{Y}=0,$$ cualquiera sea el punto $(X,Y)$ en la cónica.\\
Luego el vector $$\left(\left.\dfrac{\partial f}{\partial x}\right|_{h,k},\left.\dfrac{\partial f}{\partial y}\right|_{h,k}\right)$$ es $\perp$ a $\vec{O'P}$ cualquiera sea $P$ en la cónica y esto es posible si $$\left(\left.\dfrac{\partial f}{\partial x}\right|_{h,k},\left.\dfrac{\partial f}{\partial y}\right|_{h,k}\right)=(0,0),$$ i.e.,
$\begin{cases}
\left.\dfrac{\partial f}{\partial x}\right|_{h,k}=0\\
\left.\dfrac{\partial f}{\partial y}\right|_{h,k}=0
\end{cases}$\\
lo que significa que $(h,k)$ es un centro de la cónica.
\end{proof}
\textbf{Ejemplos.}\\
\begin{enumerate}
\item Si la cónica consta de dos rectas $\parallel$s, todo punto de la $\parallel$ media $m$ de ellas es centro de simetría de la cónica.\\
Luego la cónica tiene $\infty$s centros:
\begin{figure}[ht!]
\begin{center}
  \includegraphics[scale=0.5]{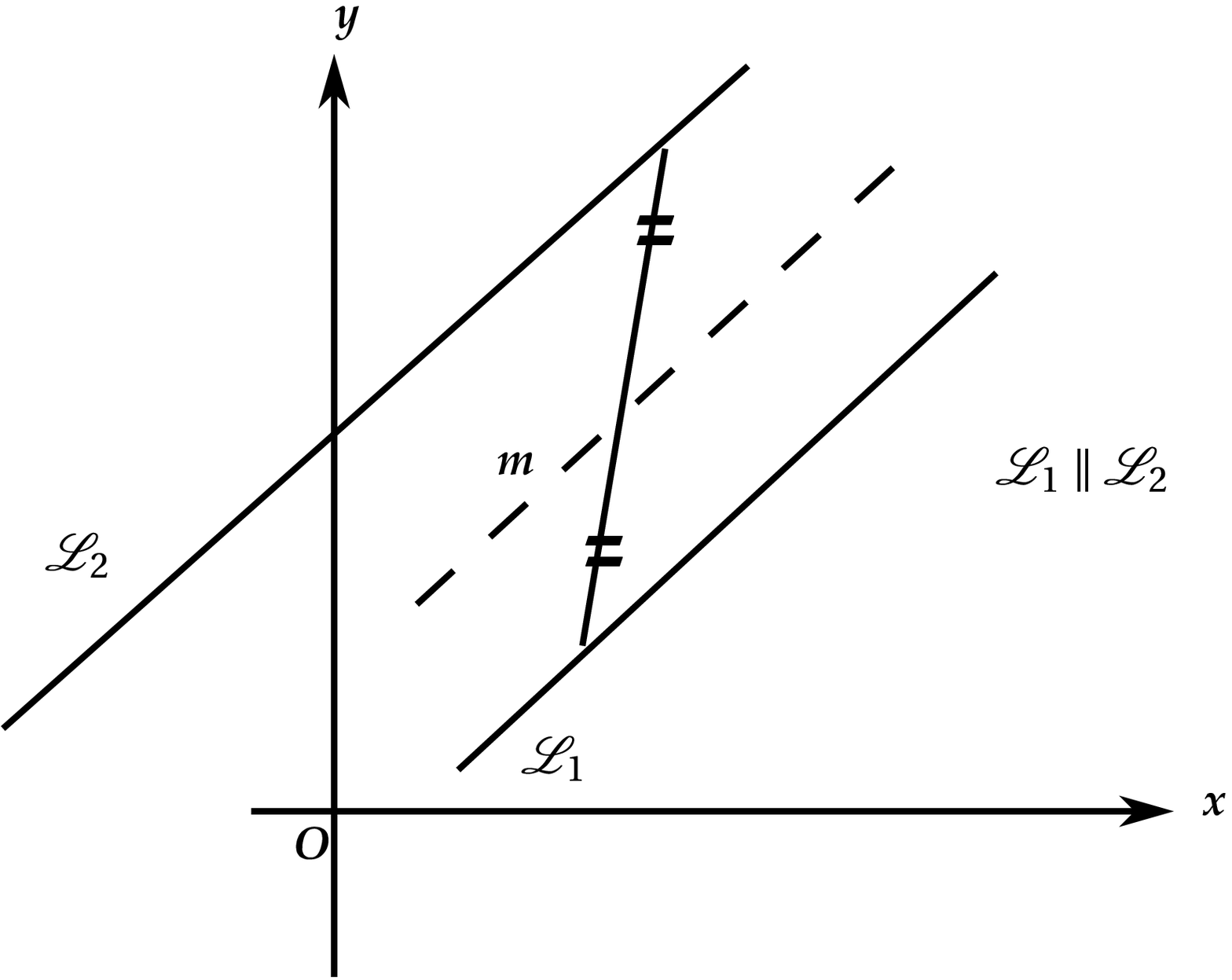}\\
\end{center}
\end{figure}
\item Si la cónica es una recta (una recta doble se dice a veces), la cónica tiene $\infty$s centros ya que todo punto de la recta es centro de simetría de ella.\\
\item Si la cónica no tiene centro de simetría, no tiene centro.\\
Por ejemplo, la parábola no tiene centro de simetría y por lo tanto, la parábola no tiene centro, pero sí tiene vértice.\\
Como se comprende, la noción de centro permite clasificar las cónicas en tres categorías\\
$\begin{cases}
\text{Cónicas con centro único}.\\
\text{Cónicas sin centro}.\\
\text{Cónicas con dos centros}.
\end{cases}$
\end{enumerate}
\section{Cónicas con centro único}
En este caso $$\delta=\left|\begin{array}{cc}A&B\\B&C\end{array}\right|=AC-B^2\neq 0,$$ $\left\{\dbinom{A}{B},\dbinom{B}{C}\right\}$ es Base de y forma un paralelogramo de área $\left|\begin{array}{cc}A&B\\B&C\end{array}\right|=AC-B^2.$ El sistema
\begin{align*}
Ah+Bk&=-D\\
Bh+Ck&=-E
\end{align*}
tiene solución única $\dbinom{\widehat{h}}{\widehat{k}}$ dada por
\begin{align*}
\dbinom{\widetilde{h}}{\widetilde{k}}&={\left(\begin{array}{cc}A&B\\B&C\end{array}\right)}^{-1}\dbinom{-D}{-E}=\dfrac{1}{AC-B^2}\left(\begin{array}{cc}C&-B\\-B&A\end{array}\right)\dbinom{-D}{-E}\\
&=\dfrac{a}{\delta}\left(\begin{array}{cc}-CD+BE\\BD-AE\end{array}\right).
\end{align*}
O sea que\\
\begin{equation}\label{30}
\left.
\begin{split}
\widetilde{h}&=\dfrac{-CD+BE}{AC-B^2}\\
\widetilde{k}&=\dfrac{BD-AE}{AC-B^2}
\end{split}
\right\}
\end{equation}
La cónica tiene centro único $(\widetilde{h},\widetilde{k})$ y éste punto es el centro de simetría de la curva. No estamos afirmando que el centro $(\widetilde{h},\widetilde{k})$ sea un punto de la cónica.\\
Como ya hemos dicho, si trasladamos los ejes $xy$ al punto $O'$ de coordenadas $(\widetilde{h},\widetilde{k})$ dadas por [\ref{30}], obtenemos otro sistema ortogonal de coordenadas $XY$, Fig.1.6, con $X\parallel x$, $Y\parallel y$ y respecto al cual la ecuación de la cónica es
\begin{equation}\label{31}
\left(\begin{array}{cc}X&Y\end{array}\right)\left(\begin{array}{cc}A&B\\C&D\end{array}\right)\dbinom{X}{Y}+f(\widetilde{h},\widetilde{k})=0
\end{equation}
\begin{figure}[ht!]
\begin{center}
  \includegraphics[scale=0.5]{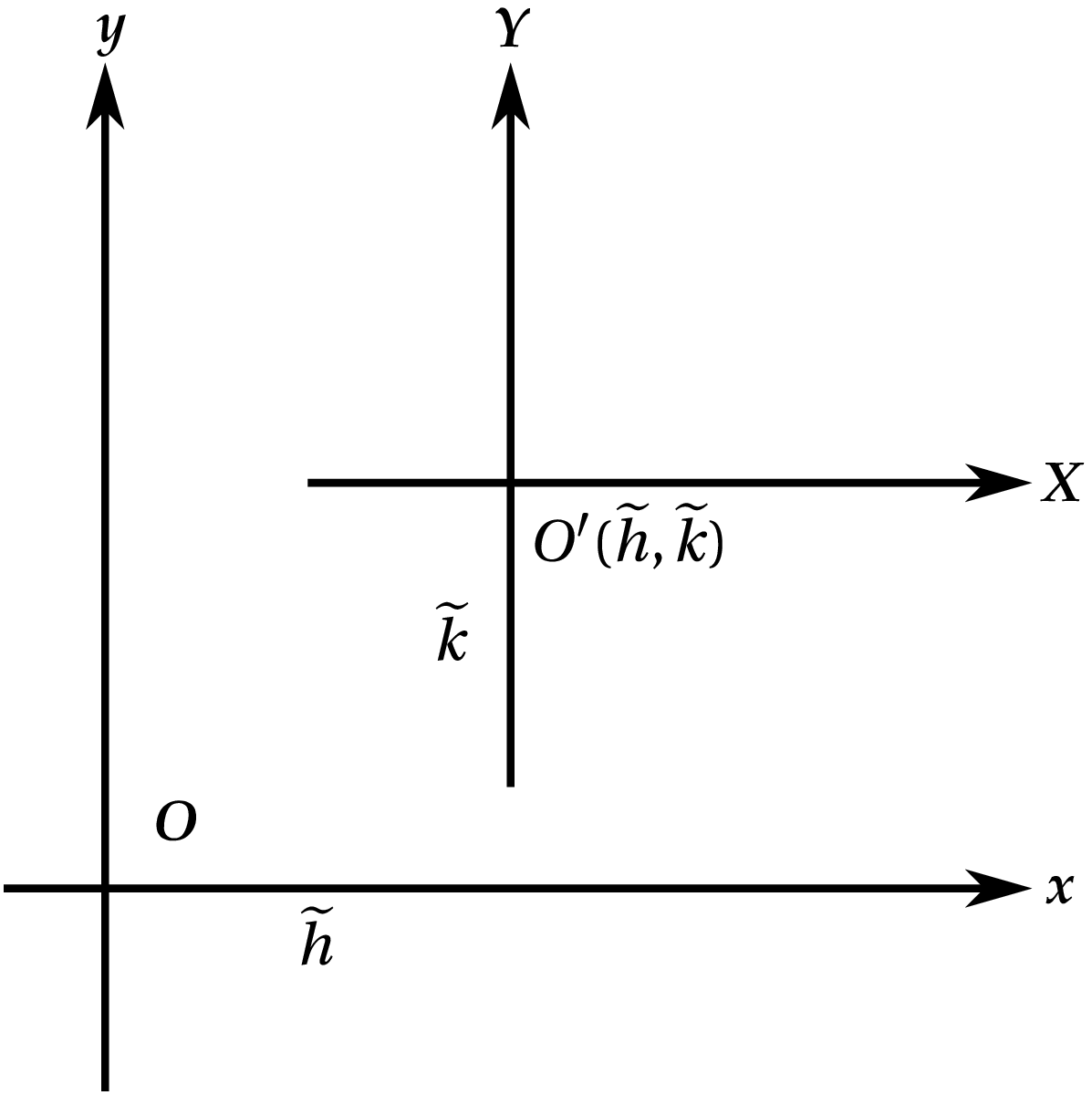}\\
\end{center}
\end{figure}
\newline
La ecuación [\ref{31}] no tiene términos en $X$ ni en $Y$. O sea, hemos conseguido eliminar la parte lineal de la ecuación [\ref{1}].\\
Hallados $(\widetilde{h},\widetilde{k})$, para tener definida [\ref{31}] debemos calcular
\begin{equation}\label{32}
f(\widetilde{h},\widetilde{k})=A\widetilde{h}^2+2B\widetilde{h}\widetilde{k}+C\widetilde{k}^2+2D\widetilde{h}+2E\widetilde{k}+F
\end{equation}
Sin embargo podemos obtener $f(\widetilde{h},\widetilde{k})$ es una manera más simple. Si tenemos en cuenta las ecuaciones [\ref{14}],
\begin{equation}\label{33}
f(\widetilde{h},\widetilde{k})=\left(A\widetilde{h}+B\widetilde{k}\right)\widetilde{h}+\left(C\widetilde{k}+B\widetilde{h}\right)\widetilde{k}+2D\widetilde{h}+2E\widetilde{k}+F
\end{equation}
Pero $$\left(\begin{array}{cc}A&B\\B&C\end{array}\right)\dbinom{\widetilde{h}}{\widetilde{k}}=\dbinom{-D}{-E}$$
Luego
\begin{equation}\label{34}
\left.
\begin{split}
A\widetilde{h}+B\widetilde{K}&=-D\\
B\widetilde{h}+C\widetilde{k}&=-E
\end{split}
\right\}\hspace{0.5cm}\text{que llevamos a [\ref{33}] obteniendo}
\end{equation}
\begin{align}\label{35}
f(\widetilde{h},\widetilde{k})&=-D\widetilde{h}-E\widetilde{k}+2D\widetilde{h}+2E\widetilde{k}+F\notag\\
&=D\widetilde{h}+E\widetilde{k}+F
\end{align}
Así que una vez hallados $\widetilde{h}$ y $\widetilde{k}$ obtenemos $f(\widetilde{h},\widetilde{h})$
a través de [\ref{35}] y no a través de [\ref{33}] que resulta más tedioso.\\ Existe otra forma de obtener $f(\widetilde{h},\widetilde{k})$ más conveniente para nuestros propósitos y sin pasar por la obtención de $(\widetilde{h},\widetilde{k})$ a través de [\ref{32}] ó de [\ref{35}]. No olvidemos que estamos analizando el caso de una cónica con centro único $(\delta\neq 0)$.\\
Según [\ref{34}]:
\begin{align*}
A\widetilde{h}+B\widetilde{k}+D&=0\\
B\widetilde{h}+C\widetilde{k}+E&=0
\end{align*}
Segun [\ref{35}]: $$D\widetilde{h}+E\widetilde{k}+F-f(\widetilde{h},\widetilde{k})=0$$
O sea que $$\widetilde{h}\left(\begin{array}{c}A\\B\\D\end{array}\right)+\widetilde{k}\left(\begin{array}{c}B\\C\\E\end{array}\right)+1\left(\begin{array}{c}D\\E\\F-f(\widetilde{h},\widetilde{k})\end{array}\right)=\left(\begin{array}{c}0\\0\\0\end{array}\right)$$
Lo que nos demuestra que el conjunto de vectores de $\mathbb{R}^3:$ $$\left\{\left(\begin{array}{c}A\\B\\D\end{array}\right)\left(\begin{array}{c}B\\C\\E\end{array}\right)\left(\begin{array}{c}D\\E\\F-f(\widetilde{h},\widetilde{k})\end{array}\right)\right\}$$ es L.D. y por lo tanto, $$\left|\begin{array}{ccc}A&B&D\\B&C&E\\D&E&F-f(\widetilde{h},\widetilde{k})\end{array}\right|=0$$
Este determinante se puede descomponer como la suma de dos determinantes así:
\begin{figure}[ht!]
\begin{center}
  \includegraphics[scale=0.5]{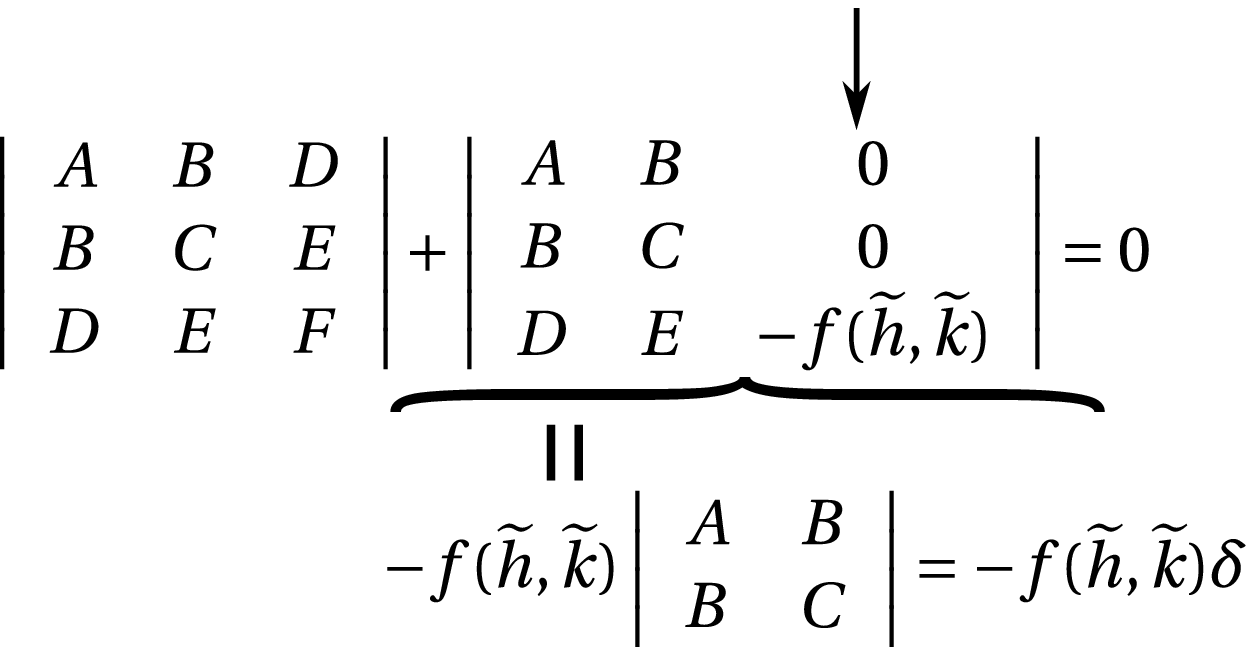}\\
\end{center}
\end{figure}
\newline
Como hemos llamado $$\Delta=\left|\begin{array}{ccc}A&B&D\\B&C&E\\D&E&F\end{array}\right|$$ se tendrá que
\begin{align*}
\Delta-f(\widetilde{h},\widetilde{k})\cdot\delta&=0\\
\therefore\hspace{0.5cm}f(\widetilde{h},\widetilde{k})&\underset{\uparrow}{=}\dfrac{\Delta}{\delta}\\
&(\delta\neq 0)
\end{align*}
De la ecuación [\ref{31}] de la cónica se escribe ahora así:
\begin{align}\label{36}
\left(\begin{array}{cc}X&Y\end{array}\right)\left(\begin{array}{cc}A&B\\C&D\end{array}\right)\dbinom{X}{Y}+\dfrac{\Delta}{\delta}&=0\hspace{0.5cm}\text{ó}\notag\\
AX^2+2BXY+CY^2+\dfrac{\Delta}{\delta}=0
\end{align}
La ecuación [\ref{36}] es la ecuación de la cónica con centro único referida a los ejes $XY$ con origen en el centro una vez realizada la traslación que permite eliminar los términos lineales.\\
Los escalares $\Delta$ y $\delta$ son invariantes $$\delta=\left|\begin{array}{cc}A&B\\B&C\end{array}\right|=AC-B^2\neq 0\hspace{0.5cm}\text{y}\hspace{0.5cm}\Delta=\left|\begin{array}{ccc}A&B&D\\B&C&E\\D&E&F\end{array}\right|$$
Nótes que
\begin{align*}
\Delta&=D\left|\begin{array}{cc}B&C\\D&E\end{array}\right|-E\left|\begin{array}{cc}A&B\\D&E\end{array}\right|+F\left|\begin{array}{cc}A&B\\B&C\end{array}\right|\\
&=D(BE-CD)-E(AE-BD)+F\delta\\
&=D\delta\widetilde{h}+E\delta\widetilde{k}+F\delta\\
&\overset{\nearrow}{[\ref{31}]}\\
&=\delta(D\widetilde{h}+E\widetilde{k}+F)\\
&\therefore\hspace{0.5cm}\dfrac{\Delta}{\delta}=D\widetilde{h}+E\widetilde{k}+F
\end{align*}
Nótese a demás que los coeficientes de la parte cuadrática de [\ref{1}] no se afectaron.
\begin{ejem}
Consideremos la cónica $2Bxy+2Dx+2Ey+F=0$ con $B\neq 0, A=C=0$.\\
$\delta=\left|\begin{array}{cc}0&B\\B&0\end{array}\right|=-B^2\neq 0$\\
\begin{figure}[ht!]
\begin{center}
  \includegraphics[scale=0.4]{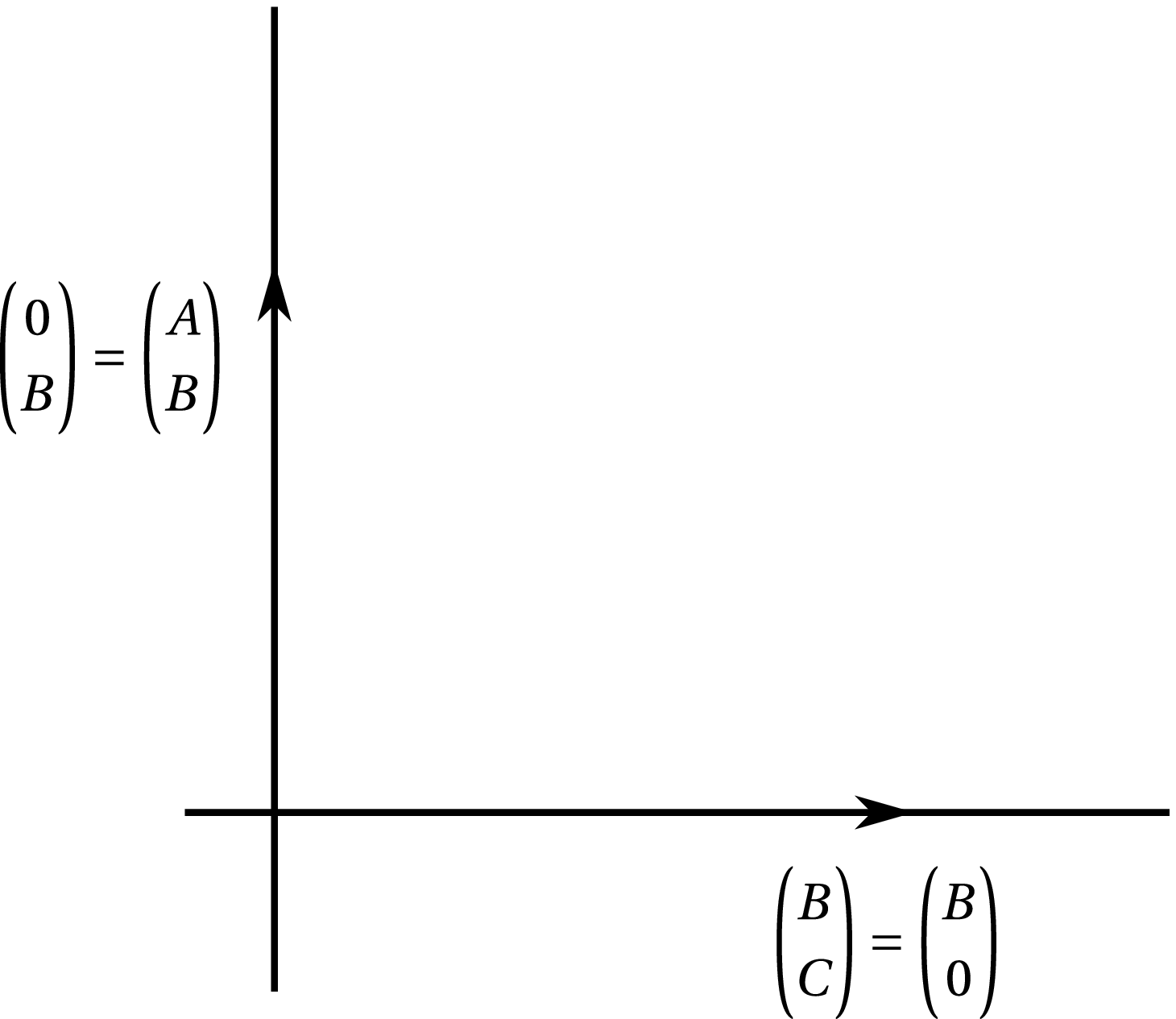}\\
\end{center}
\end{figure}
\newline
La cónica tiene centro único $O'(\widetilde{h},\widetilde{h})$.
$$\dbinom{\widetilde{h}}{\widetilde{k}}=\dfrac{1}{\delta}\left(\begin{array}{cc}-CD+BE\\BD-AE\end{array}\right)=-\dfrac{1}{B^2}\left(\begin{array}{c}BE\\BD\end{array}\right)=-\dfrac{1}{B}\dbinom{E}{D}$$
\begin{align*}
\Delta=\left|\begin{array}{ccc}0&B&D\\B&O&E\\D&E&F\end{array}\right|&=-B(BF-DE)+DBE\\
&=2BDE-B^2F
\end{align*}
Luego $$\dfrac{\Delta}{\delta}=\dfrac{2BDE-B^2F}{-B^2}=F-\dfrac{2DE}{B}.$$
La ecuación de la cónica respecto a los ejes $XY\parallel$s a $x-y$ y con origen en $\left(-\dfrac{E}{B}, -\dfrac{D}{B}\right)$ es en este caso: $$2BXY+F-\dfrac{2DF}{B}.$$
Ahora se consideraran dos casos en la ecuación [\ref{36}].\\
\begin{cas}
$$\begin{cases}
\delta&=AC-B^2\neq 0\\
B&=0
\end{cases}$$\\
Entonces $\delta=AC, A\neq 0$ y $C\neq 0$. $\left\{\dbinom{A}{B}, \dbinom{B}{C}\right\}=\left\{\dbinom{A}{0}, \dbinom{0}{C}\right\}$
\begin{figure}[ht!]
\begin{center}
  \includegraphics[scale=0.5]{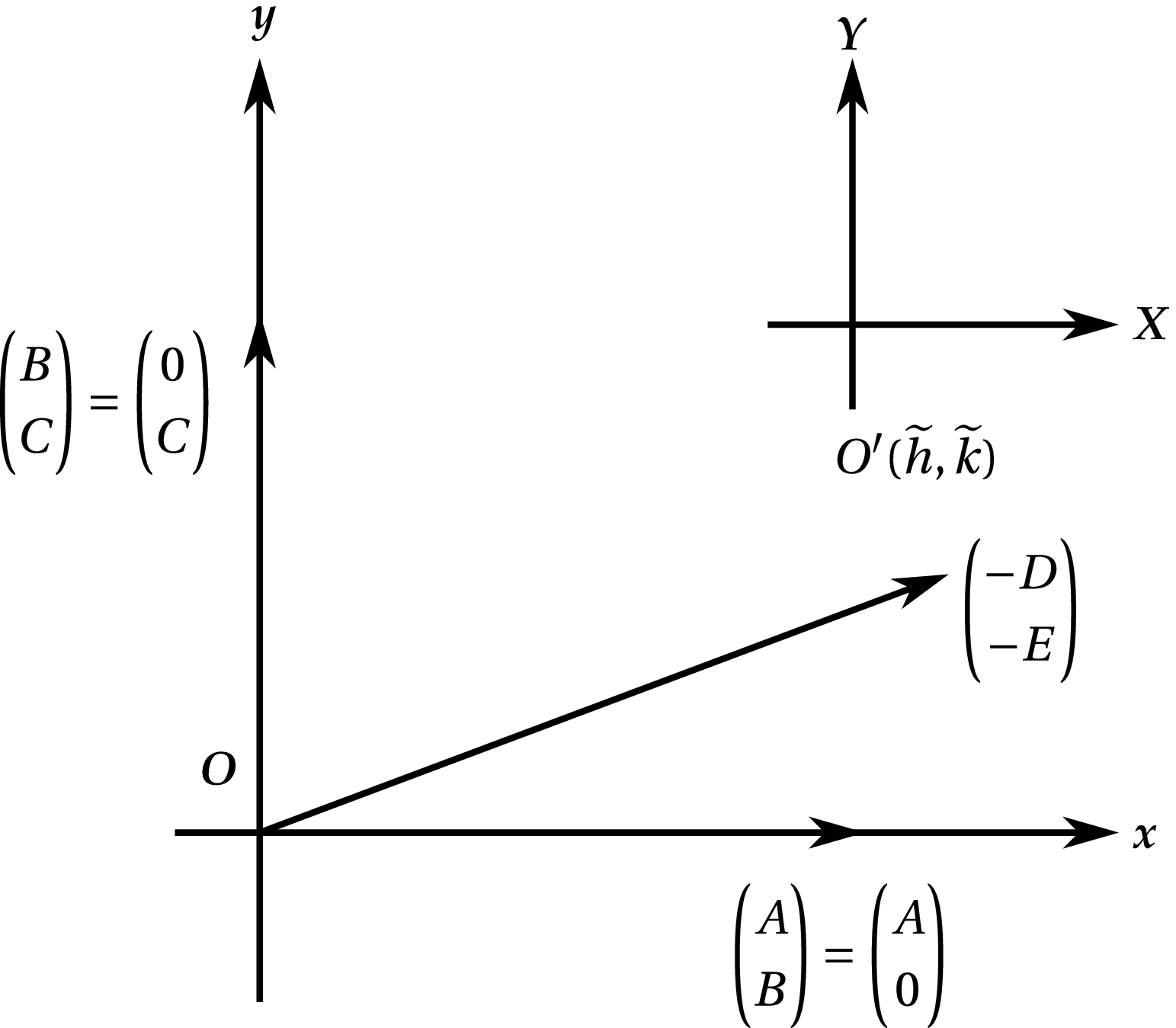}\\
\end{center}
\end{figure}
\newline
\begin{equation}\label{37}
AX^2+CY^2=-\dfrac{\Delta}{\delta}\hspace{0.5cm}\text{Nótese que está ausente el término mixto.}
\end{equation}
$$\delta=AC\neq 0,\,\, A\,\,\text{y}\,\,C\neq 0.$$
Vamos a utilizar los invariantes para definir la naturaleza del lugar representdo por [\ref{37}].
\begin{enumerate}
\item[(I)] Supongamos $\Delta>0$.\\
\item[a)] Si $\delta=AC<0,$ y $A$ y $C$ son de signo contrario y $\left(-\dfrac{\Delta}{\delta}\right)>0.$\\
El \underline{lugar es una hipérbola} de centro $O'$ y de ejes sobre $X$ y $Y$. Como en [\ref{37}] los coeficientes de $X^2$ y $Y^2$ son de signo contrario, la hipérbola puede aparecer así:
\item[i)] $AX^2-CY^2=\dfrac{\Delta}{\delta}$ y con $A$ y $C$ positivos.\\
La hipérbola tiene ramas que se abren según en eje $X$:
\newpage
\begin{figure}[ht!]
\begin{center}
  \includegraphics[scale=0.5]{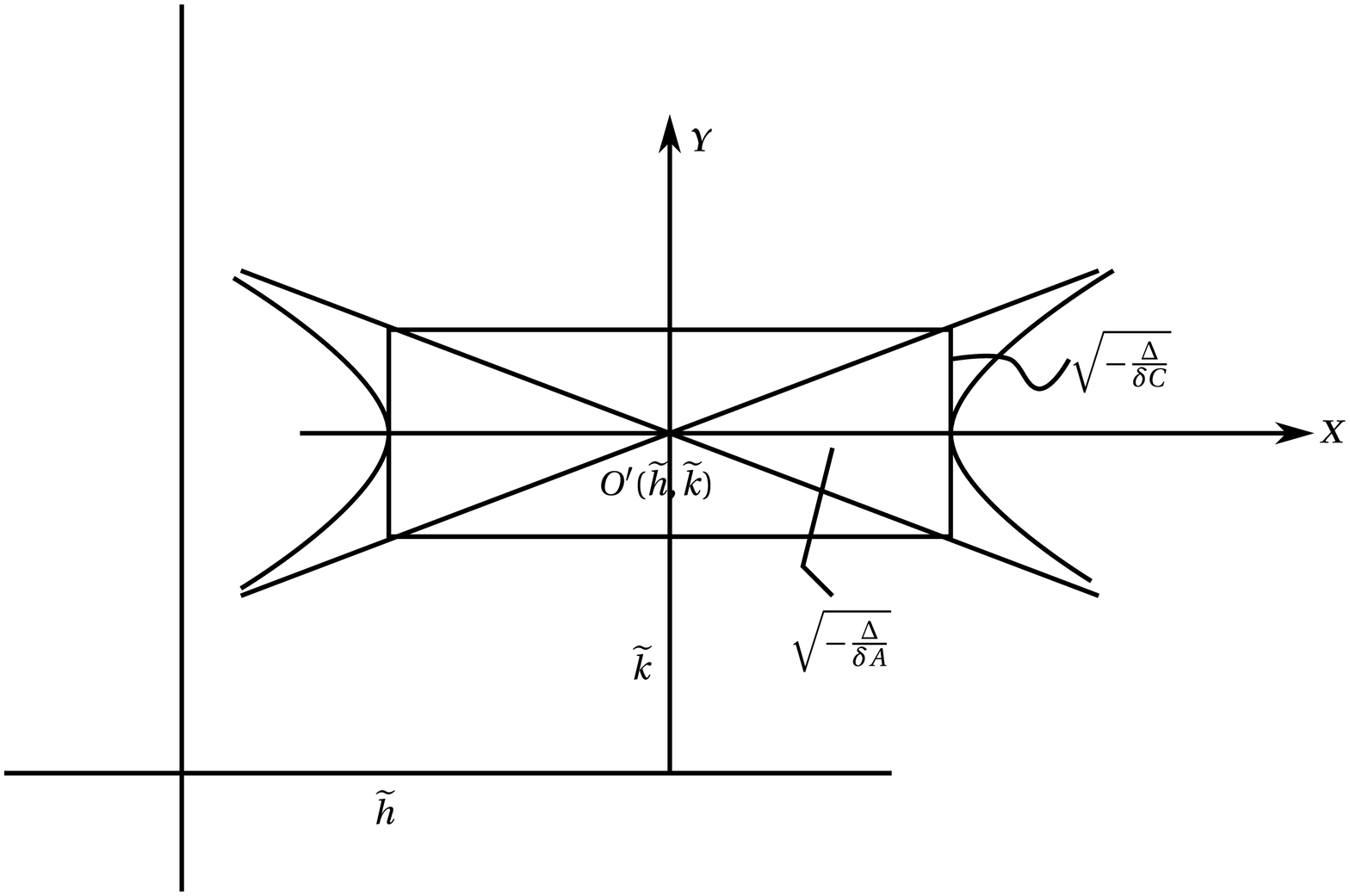}\\
\end{center}
\end{figure}
\item[ii)] $CY^2-AX^2=-\dfrac{\Delta}{\delta}$ con $A$ y $C$ positivos. La hipérbola tiene ramas que se abren según el eje $Y:$\\
\begin{figure}[ht!]
\begin{center}
  \includegraphics[scale=0.4]{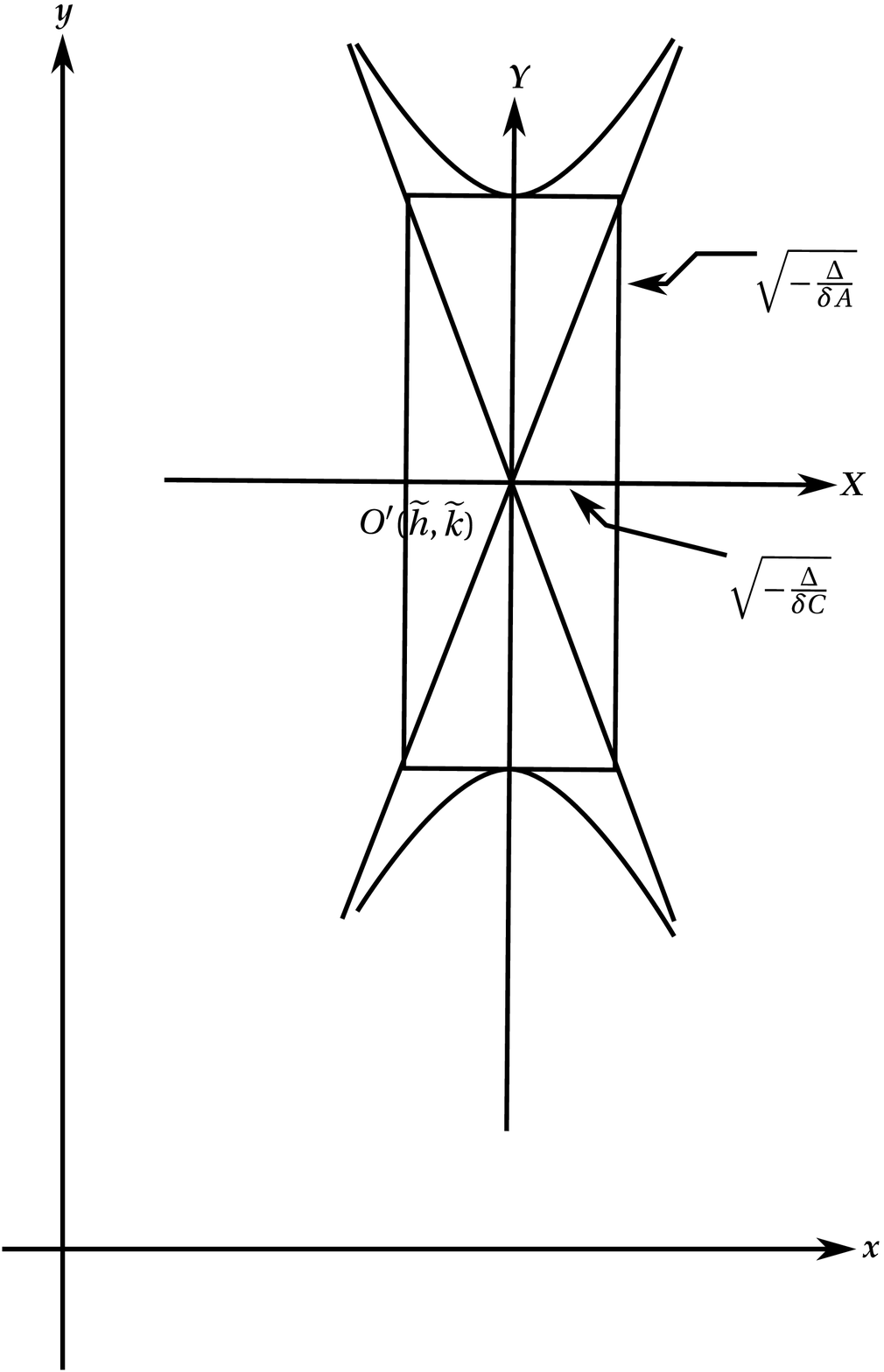}\\
\end{center}
\end{figure}
\newline
\item[b)] Si $\delta=AC>0,$ A y C tienen el mismo signo y
$\left(-\dfrac{\Delta}{\delta}\right)<0;\,\omega=A+C\neq0.$\\
\item[i)] Si $\omega>0$, A y C son ambos positivas. Así que $AX^2+CY^2=-\dfrac{\Delta}{\delta}$ con $A>0,\,C>0$
$-\dfrac{\Delta}{\delta}<0.$ El \underline{lugar es $\infty$}.\\
\item[ii)] Si $\omega<0$, A y C son ambos negativos. Luego $AX^2+CY^2=-\dfrac{\Delta}{\delta}$ con $A<0,\,C<0$ y
$-\dfrac{\Delta}{\delta}<0.$\\
\item[ii-1)] Si $A\neq C$, el \underline{lugar es una elipse} de centro $O'$ y ejes sobre $X$ y $Y$.
\begin{figure}[ht!]
\begin{center}
  \includegraphics[scale=0.5]{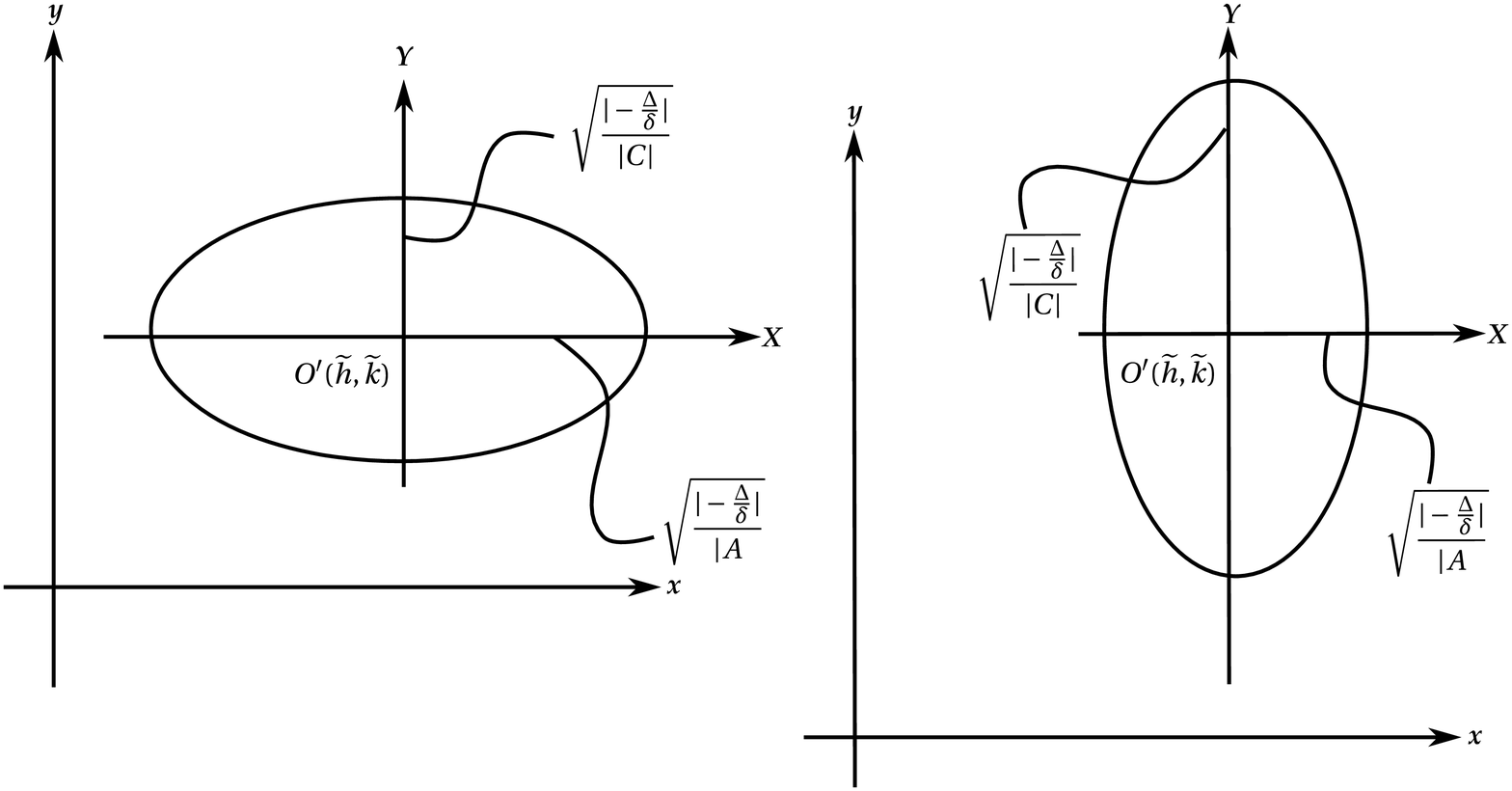}\\
\end{center}
\end{figure}
\newline
\item[ii-2)] Si $A=C$, el \underline{lugar es una circunferencia} de centro en $O'$ y radio
$$\sqrt{\dfrac{|-\frac{\Delta}{\delta}|}{|A|}}$$
\end{enumerate}
\end{cas}
\end{ejem}
En síntesis:
\begin{figure}[ht!]
\begin{center}
  \includegraphics[scale=0.5]{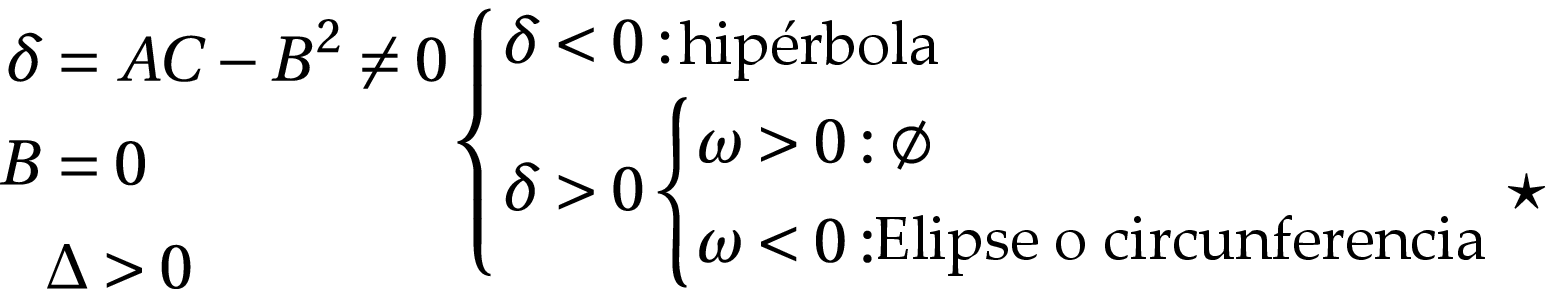}\\
\end{center}
\end{figure}
\newline
\begin{ejem}
Considere la cónica
\begin{equation}\label{38}
2x^2+y^2+4x+2y+10=0
\end{equation}
$A=2; B=0; C=1; \delta=AC-B^2=2; \omega=3.$\\
$$\Delta=\left|\begin{array}{ccc}2&0&2\\0&1&1\\2&1&10\end{array}\right|=14\hspace{1.0cm}\dbinom{A}{B}=\dbinom{2}{0}; \dbinom{B}{C}=\dbinom{0}{1}$$
\begin{figure}[ht!]
\begin{center}
  \includegraphics[scale=0.5]{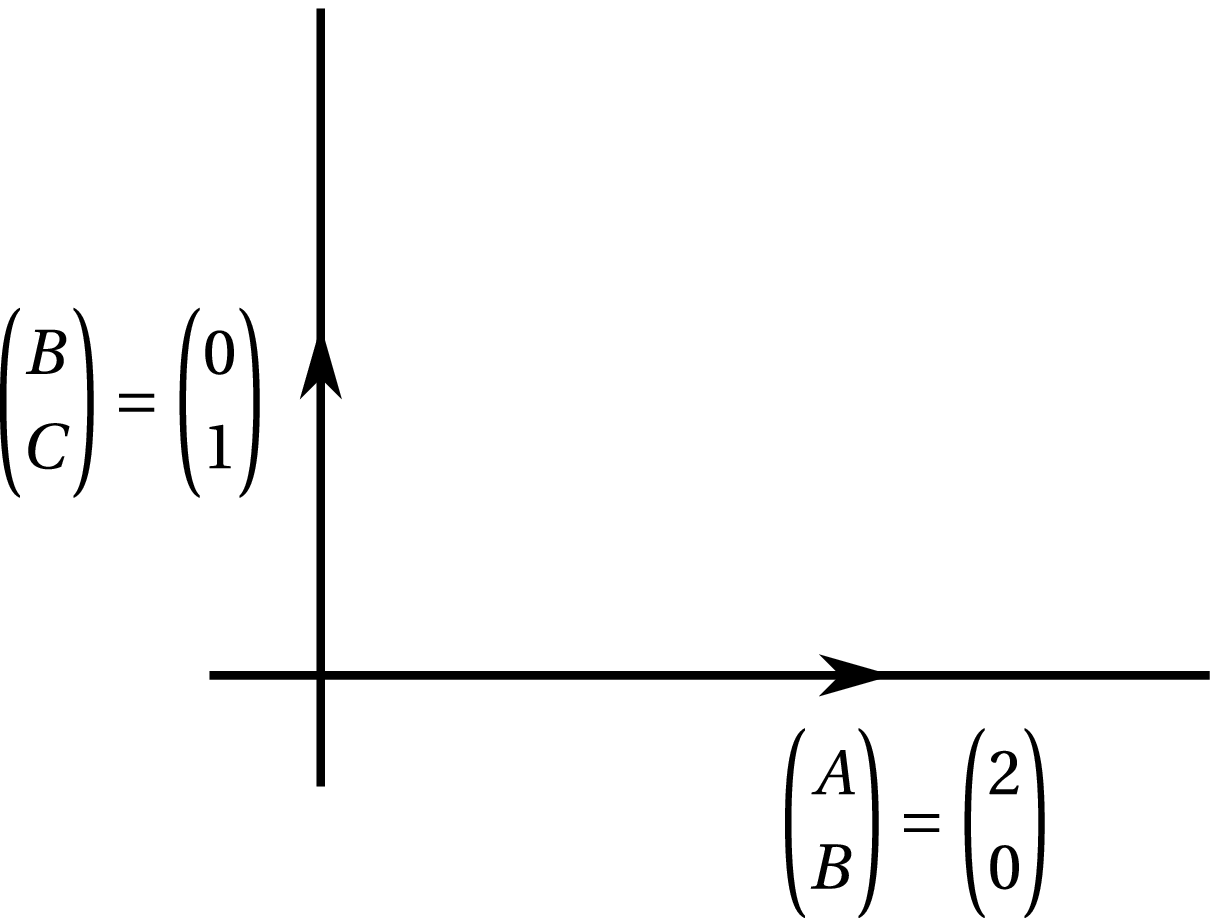}\\
\end{center}
\end{figure}
\newline
Así que $\delta\neq 0, B=0, \Delta>0, \delta>0$ y $\omega>0$.\\
En virtud de $\star$, \underline{el lugar es $\emptyset$}. O sea que no existe un punto del plano $xy$ que satisfaga [\ref{38}].\\
Vamos a comprobarlo.\\
Completando trinomios cuadrados perfectos en [\ref{38}], $$\left(2x^2+4x+2\right)+\left(y^2+2y+1\right)+10-2-1=0.$$
O sea que la ecuación de la cónica puede escribirse así: $${\left(\sqrt{2}x+\dfrac{2}{\sqrt{2}}\right)}^2+{\left(y+1\right)}^2=-7$$ lo que nos prueba que efectivamente el lugar es $\emptyset$.
\end{ejem}
\begin{enumerate}
\item[II)] Si $\Delta=0,$ [\ref{37}] se escbribe así: $AX^2+CY^2=0$.\\
Como $\delta=AC\neq 0,$ puede ocurrir:\\
\item[a)] $\delta>0$. En este caso A y C tienen el mismo signo. El \underline{lugar es el punto $O'$}.\\
\item[b)] $\delta<0$. Entonces A y C tienen signos contrarios. El lugar consta de dos rectas concurrentes en $O'$\\
Resumiendo
\begin{figure}[ht!]
\begin{center}
  \includegraphics[scale=0.5]{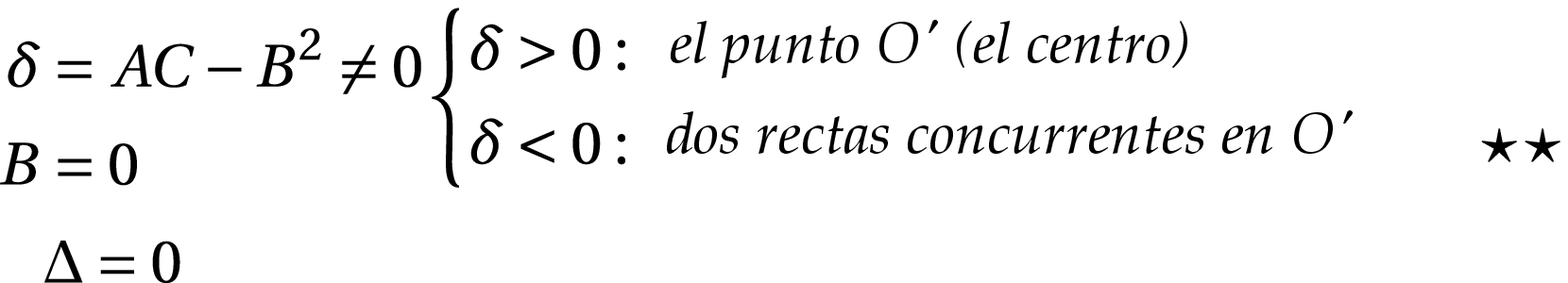}\\
\end{center}
\end{figure}
\newline
\begin{ejem}
Consideremos la cónica $x^2+y^2-4x+2y+5=0$ que podemos escribir ${\left(x-2\right)}^2+{\left(y+1\right)}=0.$\\
El lugar es el punto $(2,-1).$\\
$B=0, \delta=AC-B^2=AC=1\neq 0.$\\
$\Delta=\left|\begin{array}{ccc}1&0&-2\\0&1&1\\-2&1&5\end{array}\right|=0$\\
Como $\delta\neq 0, B=0, \Delta=0$ y $\delta>0$, la cónica, en virtud de $\star\star$ es un punto como ya sabíamos.
\end{ejem}
\begin{ejem}
Sea $f(x,y)=(3x-y+2)(3x+y-5)=p(x,y)\cdot q(x,y)=0.$\\
$f(x,y)=0\Leftrightarrow p(x,y)=0$ ó $q(x,y)=0$.\\
Las rectas $p(x,y)=3x-y+2=0$ y $q(x,y)=3x+y-5=0$ se cortan en el punto $(h,k)$ que es la solución del sistema
\begin{align*}
3x-y+2&=0\\
3x+y-5&=0
\end{align*}
Las dos rectas se corten en $\left(1/2, 7/2\right).$\\
El lugar, o sea $\left\{(x,y)\diagup f(x,y)=0\right\}$ consta de dos rectas concurrentes en $\left(1/2, 7/2\right).$\\
Vamos a verificarlo.\\
La cónica es $$(3x-y+2)(3x+y-5)=0.$$
O sea $$9x^2-y^2-9x+7y-10=0$$
Hallemos su centro. El sistema
\begin{align*}
Ax+By&=-D\\
Bx+Cy&=-E
\end{align*}
es\\
$\begin{cases}
9x&=\dfrac{9}{2}\\
-y&=-\dfrac{7}{2}\hspace{0.5cm}\therefore\hspace{0.5cm}x=1/2, \, y=7/2
\end{cases}$\\
Como $B=0, \delta=AC-B^2=-9<0$ y $$\Delta=\left|\begin{array}{ccc}9&0&-9/2\\0&-1&7/2\\-9/2&7/2&-10\end{array}\right|,$$
se tiene, en virtud de $\star\star$ que el lugar son dos rectas concurrentes en $\left(1/2, 7/2\right).$
\end{ejem}
\item[III)] Supongamos $\Delta<0.$\\
Como $\delta=AC\neq 0$, puede ocurrir:
\item[a)] $\delta<0.$ Entonces A y C tienen signos contrarios y $\left(-\dfrac{\Delta}{\delta}\right)<0.$\\
El \underline{lugar es una hipérbola}. La ecuación [\ref{37}] puede parecerce a\\
\item[i)] $AX^2-CY^2=-\dfrac{\Delta}{\delta}$ con $A>0, C>0, -\dfrac{\Delta}{\delta}<0.$\\
O sea que $CY^2-AX^2=\dfrac{\Delta}{\delta}$ con $\dfrac{\Delta}{\delta}>0$ y la cónica es una hipérbola con centro en $O'$ cuyas ramas se abren según el eje Y:
\newpage
\begin{figure}[ht!]
\begin{center}
  \includegraphics[scale=0.3]{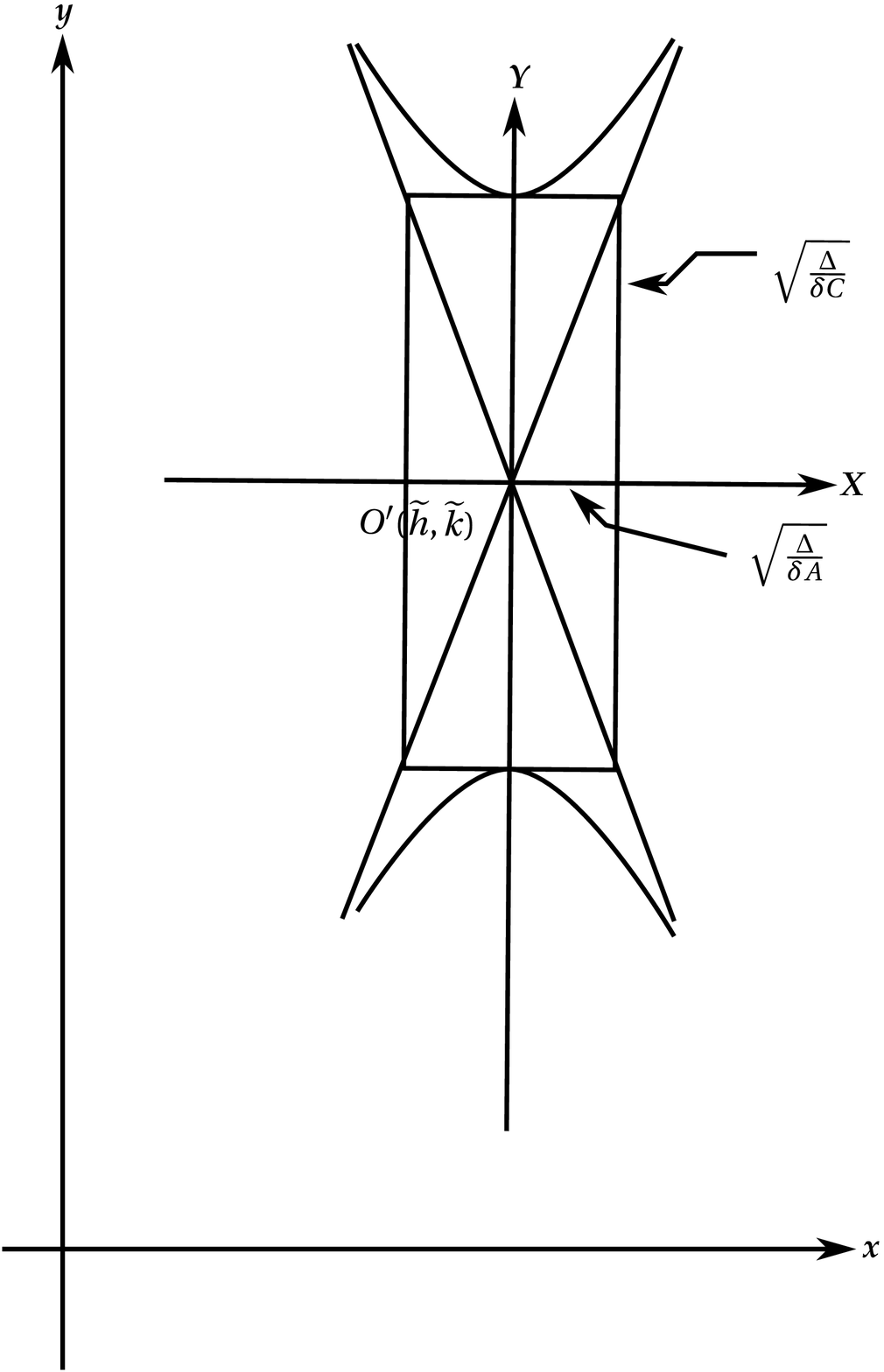}\\
\end{center}
\end{figure}
\item[ii)] $CY^2-AX^2=-\dfrac{\Delta}{\delta}$ con $A>0, C>0, -\dfrac{\Delta}{\delta}<0.$\\
Luego $AX^2-CY^2=\dfrac{\Delta}{\delta}$ con $\dfrac{\Delta}{\delta}>0$ y la cónica es una hipérbola con centro en $O'$ y cuyas ramas se abren según el eje $X:$
\begin{figure}[ht!]
\begin{center}
  \includegraphics[scale=0.5]{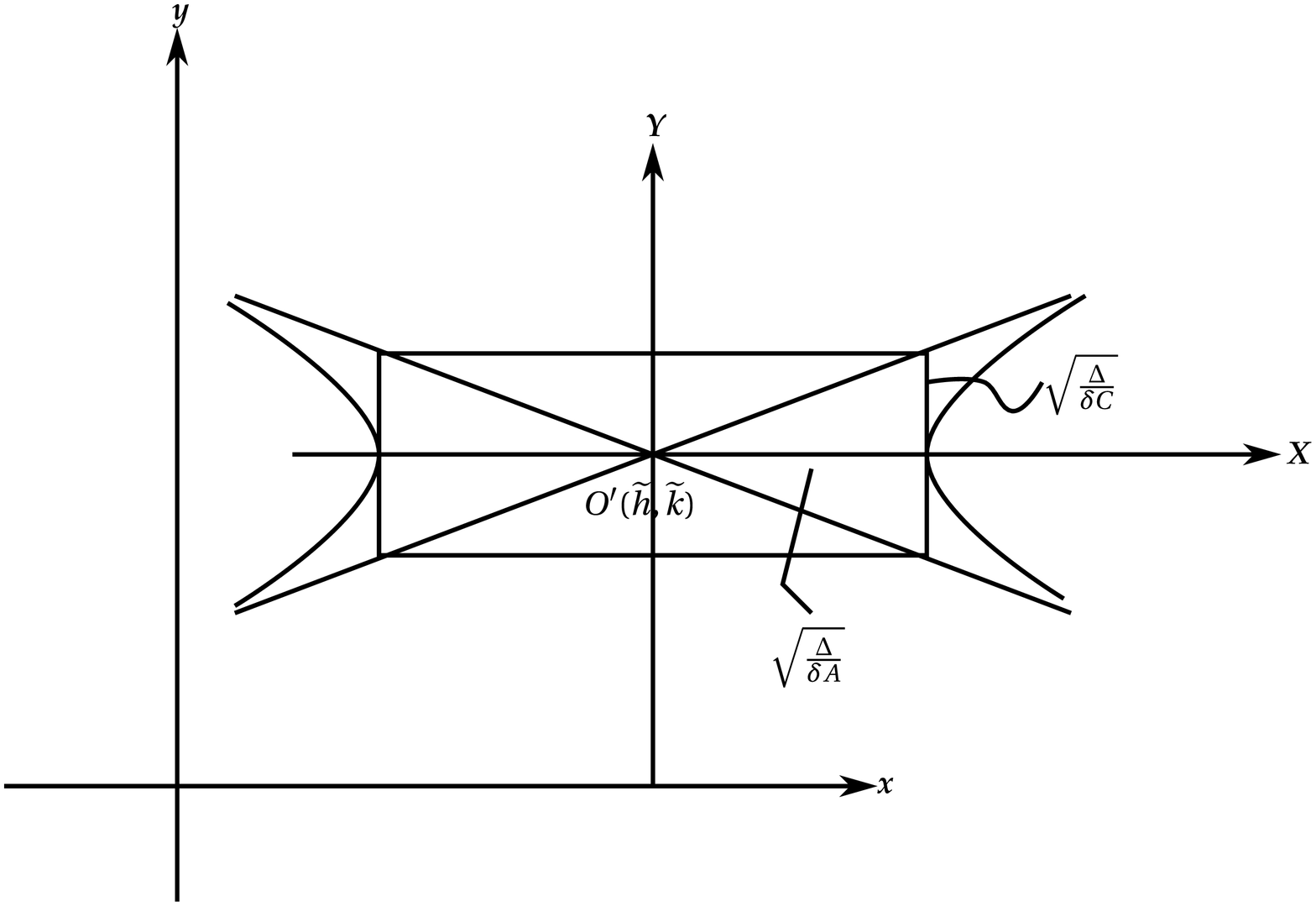}\\
\end{center}
\end{figure}
\newline
\item[b)] $\delta>0.$ En este caso A y C tienen el mismo signo y $-\dfrac{\Delta}{\delta}>0;\,\,\omega=A+C\neq 0.$
\item[i)] Si $\omega=A+C>0,$ A y C son ambos positivos y la ecuación [\ref{37}]\\ \underline{represeta una elipse ó una circunferencia}.\\
\item[ii)] Si $\omega=A+C<0$, A y C son ambos negativos. \underline{El lugar es $\emptyset$}.\\
En resumen
\begin{figure}[ht!]
\begin{center}
  \includegraphics[scale=0.5]{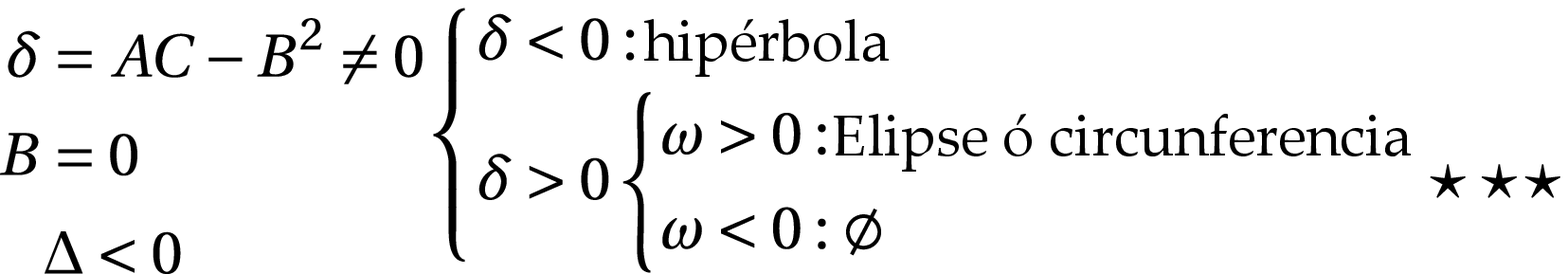}\\
\end{center}
\end{figure}
\newline
El Caso 1. puede presentarse en la siguiente tabla:
\begin{figure}[ht!]
\begin{center}
  \includegraphics[scale=0.6]{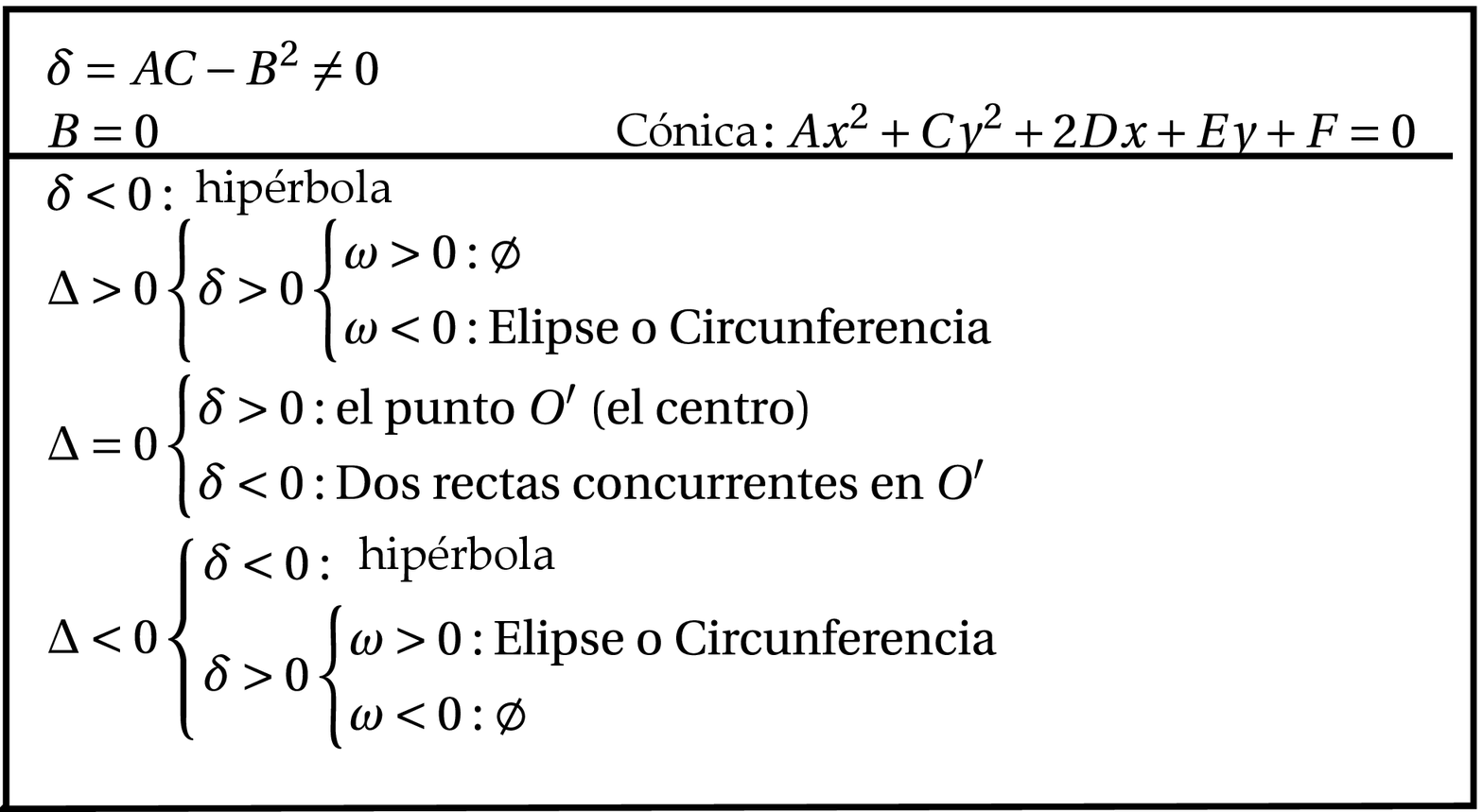}\\
\end{center}
\end{figure}
\begin{cas}
$\begin{cases}
\delta=AC-B^2\neq 0\\
B\neq 0
\end{cases}$\\
Una vez hecha la traslación que elimine los términos lineales en [\ref{1}], la ecuación de la cónica$\diagup XY$ con origen en $O'(\widetilde{h}, \widetilde{k})$ es
\begin{equation}\label{39}
AX^2+BXY+CY^2+\dfrac{\Delta}{\delta}=0\hspace{1.0cm}\text{Nótese que está presente el término mixto.}
\end{equation}
(Recuerdese que cuando se hace una traslación de ejes los coeficientes de la componente cuadratica de [\ref{1}] no se transforman).\\
Procedemos ahora a eliminar el término mixto en [\ref{39}].\\
Sea $Q(X,Y)$ un punto de la cónica y $O'(\widetilde{h}, \widetilde{k})$ su centro:
\newpage
\begin{figure}[ht!]
\begin{center}
  \includegraphics[scale=0.5]{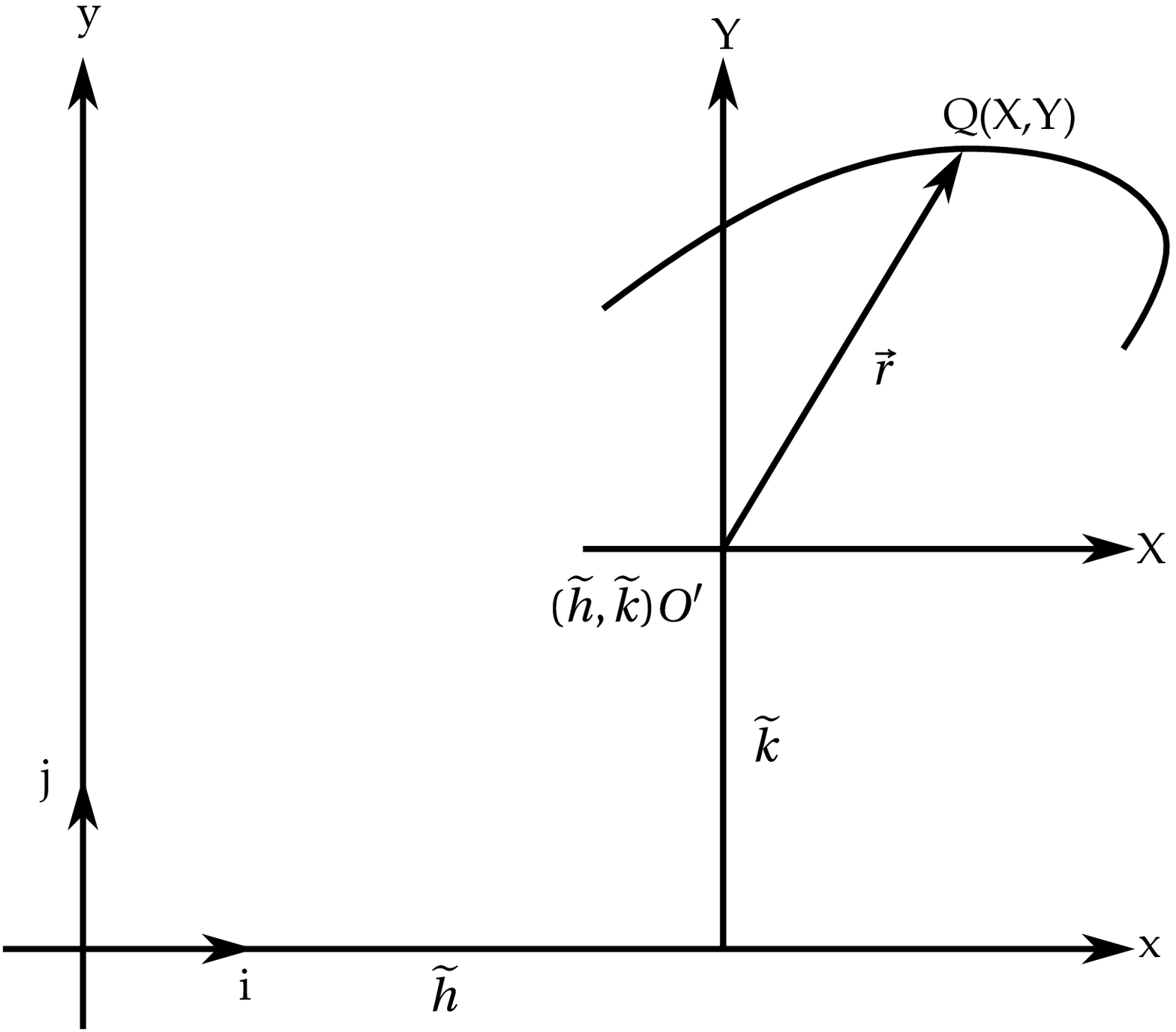}\\
\end{center}
\end{figure}
Vale la pena señalar los casos posibles que pueden presentarse con los coeficientes $A, B, C.$
\begin{figure}[ht!]
\begin{center}
  \includegraphics[scale=0.4]{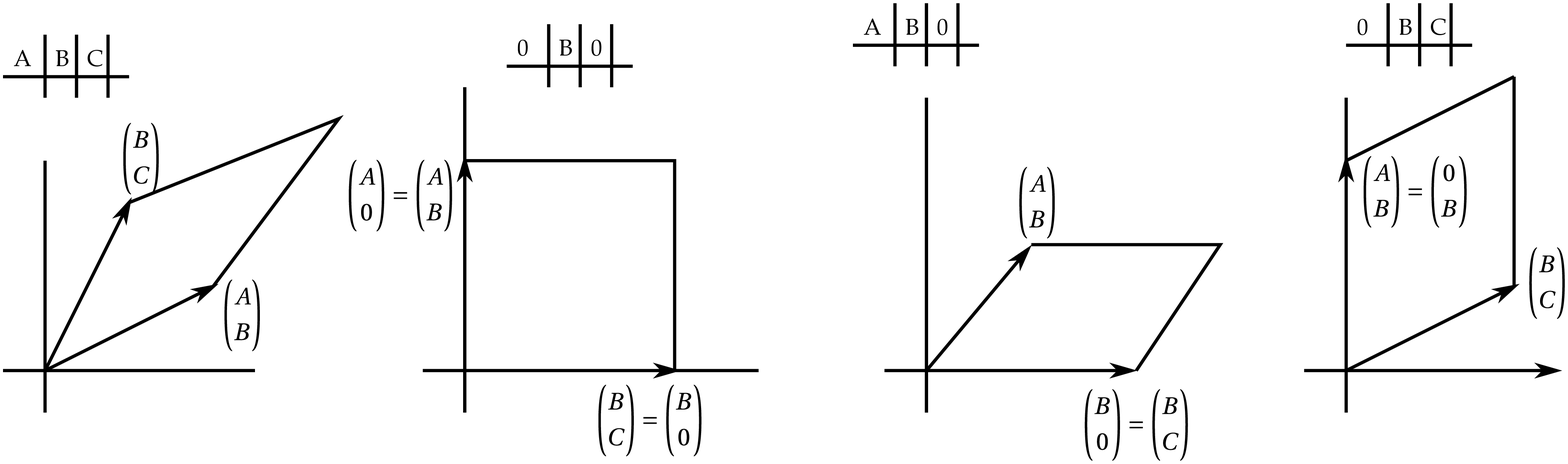}\\
\end{center}
\end{figure}
\newline
Entonces $$\vec{O'Q}=\vec{r}=X\vec{i}+Y\vec{j}.$$
O sea que $\left[\vec{r}\right]_{ij}=\dbinom{X}{Y}$ y la ecuación [\ref{39}] puede escribirse:
\begin{equation}\label{40}
\left[\vec{r}\right]_{ij}^tM\left[\vec{r}\right]_{ij}+\dfrac{\Delta}{\delta}=0\hspace{0.5cm}\text{con}\hspace{0.5cm}M=\left(\begin{array}{cc}A&B\\B&C\end{array}\right)
\end{equation}
$\left\{\dbinom{A}{B};\dbinom{B}{C}\right\}$ Base de $\mathbb{R}^2$ con $B\neq 0.$\\
Como $M$ es simétrica, sabemos, por el Teorema Espectral (ó teorema de los Ejes Principales) que $\exists P=\left(\begin{array}{cc}\underset{\downarrow}{\overset{\uparrow}{p_1}}&\underset{\downarrow}{\overset{\uparrow}{p_2}}\end{array}\right)$
matriz ortogonal, o sea $P^{-1}=P^t$, tal que $P^tMP=\left(\begin{array}{cc}\lambda_1&0\\ 0&\lambda_2\end{array}\right)$ donde $\lambda_1, \lambda_2$ son los valores propios de $M$. A demás, $\lambda_1, \lambda_2\in\mathbb{R}.$\\
Siendo $P$ ortogonal, las columnas $p_1,p_2$ de $P$ definen una base ortonormal de $\mathbb{R}^2/\langle\hspace{0.5cm}\rangle$: producto interno usual en $\mathbb{R}^2,$ base formada por vectores propios de $M.$\\
O sea que\\
$$\begin{cases}
p_i\in EP_{\lambda_i}^M,\hspace{0.5cm}\text{i.e.,}\hspace{0.5cm}M_{p_i}=\lambda_ip_i\\
\langle p_i,p_j\rangle=\delta_i^j,\,\,i,j=1,2
\end{cases}$$
Los vectores $p_1,p_2$ definen un nuevo sistema de coordenadas ortogonal $x'y'$ con origen en $O'$ y rotado respecto a $XY:$
\begin{figure}[ht!]
\begin{center}
  \includegraphics[scale=0.5]{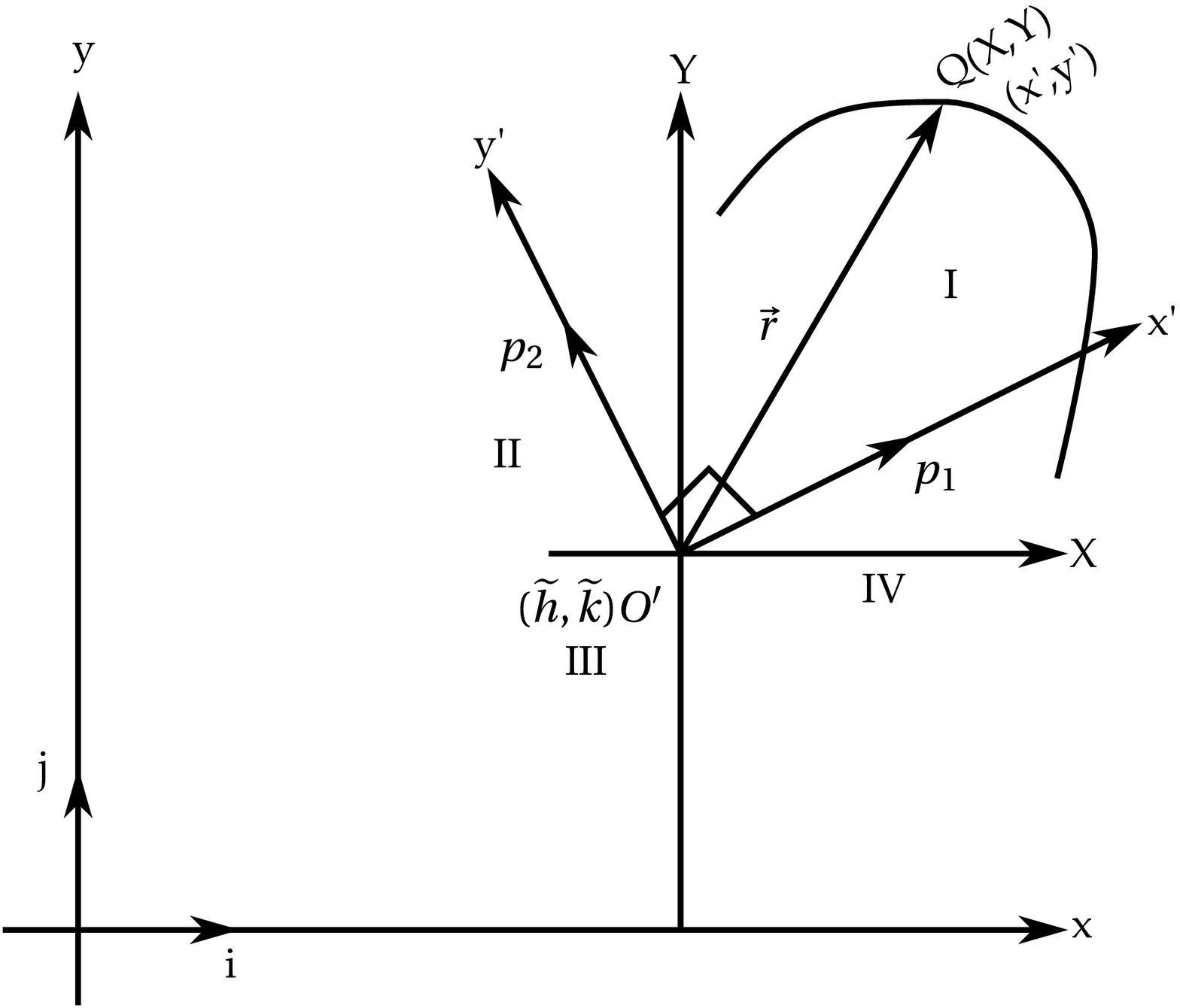}\\
\end{center}
\end{figure}
\newline
Los vectores $p_1$ y $p_2$ están en el semiplano $(I)-(II)$ que determina el eje $X$. El eje $x'$ está en el cuadrante $I$ y lo define aquel valor propio para el cual el vector propio se coloca en el cuadrante $I$. El otro vector propio está $90^\circ\curvearrowleft$ y en el cuadrante $II.$\\
Los ejes $x'y'$ se llaman los \underline{ejes principales de la cónica}.\\
Llamemos $$P=\left(\begin{array}{cc}p_1^1&p_2^1\\p_1^2&p_2^2\end{array}\right).$$
Entonces $$\left[p_1\right]_{ij}=\dbinom{p_1^1}{p_1^2},\,\,\left[p_2\right]_{ij}=\dbinom{p_2^1}{p_2^2}.$$
O sea que $$P=\left(\begin{array}{cc}\underset{\downarrow}{\overset{\uparrow}{\left[p_1\right]_{ij}}}&\underset{\downarrow}{\overset{\uparrow}{\left[p_2\right]_{ij}}}\end{array}\right)=\left[I\right]_{ij}^{p_\alpha}$$
y por lo tanto se tiene que $$\left[\vec{r}\right]_{ij}=\left[I\right]_{ij}^{p_\alpha}\left[\vec{r}\right]_{p_\alpha}=P\left[\vec{r}\right]_{p_\alpha},\hspace{0.5cm}\text{i.e.,}$$
\begin{equation}\label{41}
\dbinom{X}{Y}=P\dbinom{x'}{y'}
\end{equation}
donde $$\left[\vec{r}\right]_{p_\alpha}=\dbinom{x'}{y'}$$ siendo $(x',y')$ las coordenadas de $Q$ respecto a los ejes $x'y'$. O también,
\begin{equation}\label{42}
\dbinom{X}{Y}=\left(\begin{array}{cc}p_1^1&p_2^1\\p_1^2&p_2^2\end{array}\right)\dbinom{x'}{y'}
\end{equation}
\begin{equation}\label{43}
\begin{split}
\therefore\hspace{0.5cm}X&=p_1^1x'+p_2^1y'\\
Y&=p_1^2x'+p_2^2y'
\end{split}
\end{equation}
Siendo $P$ ortogonal, $P^{-1}=P^t$ y de [\ref{41}], $\dbinom{x'}{y'}=P^t\dbinom{X}{Y}.$\\
O sea que $$\dbinom{x'}{y'}=\left(\begin{array}{cc}p_1^1&p_2^1\\p_1^2&p_2^2\end{array}\right)\dbinom{X}{Y},\hspace{0.5cm}\text{i.e.,}$$
\begin{equation}\label{44}
\begin{split}
x'&=p_1^1X+p_1^2Y\\
y'&=p_2^1X+p_2^2Y
\end{split}
\end{equation}
Las ecuaciones [\ref{43}] y [\ref{44}] son las ecuaciones de la transformación ortogonal de coordenadas definida por la matriz $P$ que diagonaliza a M.\\
Al regresar a la ecuación [\ref{40}] se tiene que ${\left(P\left[\vec{r}\right]_{p_\alpha}\right)}^tMP\left[\vec{r}\right]_{p_\alpha}+\dfrac{\Delta}{\delta}=0.$\\
O sea que $$\left[\vec{r}\right]_{p_\alpha}^t\left(P^tMP\right)\left[\vec{r}\right]_{p_\alpha}+\dfrac{\Delta}{\delta}=0$$
Pero $$\left[\vec{r}\right]_{p_\alpha}=\dbinom{x'}{y'}$$ y $$P^tMP=\left(\begin{array}{cc}\lambda_1&0\\0&\lambda_2\end{array}\right).$$ Luego $$\left(\begin{array}{cc}x'&y'\end{array}\right)\left(\begin{array}{cc}\lambda_1&0\\0&\lambda_2\end{array}\right)\dbinom{x'}{y'}+\dfrac{\Delta}{\delta}=0.$$
Lo que nos demuestra que las ecuaciones de la cónica referida a sus ejes principales $x'y'$ con origen en el centro $O'$ de la curva es
\begin{equation}\label{45}
\lambda_1x'^2+\lambda_2y'^2+\dfrac{\Delta}{\delta}=0
\end{equation}
Antes de pasar a estudiar los lugares geométricos representados por [\ref{45}] conviene recordar que siendo $M$ una matriz simétrica, sus valores propios $\lambda_1$ y $\lambda_2$ son números reales y son raíces del
\begin{align*}
PCM(\lambda)&=\lambda^2-\left(trM\right)\lambda+\det M\\
&=\lambda^2-\left(A+C\right)\lambda+\delta=\lambda^2-\omega\lambda+\delta=0
\end{align*}
Así que
\begin{align*}
\lambda_1+\lambda_2&=\omega\\
\lambda_1\lambda_2&=\delta\neq 0
\end{align*}
lo que nos demuestra que $\lambda_1$ y $\lambda_2$ son diferentes de cero. Además, $\lambda_1\neq\lambda_2.$\\
En efecto,
\begin{align*}
\lambda^2-\left(A+C\right)\lambda+\left(AC-B^2\right)=0\\
\therefore\hspace{0.5cm}\lambda=\dfrac{\left(A+C\right)\pm\sqrt{{\left(A+C\right)}^2-4\left(AC-B^2\right)}}{2}
\end{align*}
La cantidad subradical (el discriminante) es
\begin{align*}
\omega^2-4\delta&={\left(A+C\right)}^2-4\left(AC-B^2\right)\\
&=A^2+2AC+C^2-4AC+4B^2\\
&={\left(A-C\right)}^2+2B^2\underset{\overset{\uparrow}{B\neq 0}}{>}0
\end{align*}
Esto demuestra que el discriminante es mayor que cero y por lo tanto $\lambda_1\neq\lambda_2.$ $$\lambda_2^1=\dfrac{\omega\pm\sqrt{\omega^2-4\delta}}{2}$$
Además, recuérdese que $$EP_{\lambda_1=\dfrac{\omega\pm\sqrt{\omega^2-4\delta}}{2}}^M=\mathscr{N}\left(\lambda_1I_2-M\right):\text{espacio nulo de la matriz $\lambda_1I_2-M$}$$
y que
$$EP_{\lambda_2=\dfrac{\omega\pm\sqrt{\omega^2-4\delta}}{2}}^M=\mathscr{N}\left(\lambda_2I_2-M\right):\text{espacio nulo de la matriz $\lambda_2I_2-M$}$$
Regresando a la ecuación [\ref{45}] podemos considerar:
\item[I)] $\Delta>0$. La ecuación [\ref{45}] puede escribir así:
$$\lambda_1x'^2+\lambda_2y'^2=-\dfrac{\Delta}{\delta}$$
\item[a)] Si $\delta=AC-B^2=\lambda_1\lambda_2>0$, $\lambda_1$ y $\lambda_2$ son del mismo signo, $$\omega=A+C=\lambda_1+\lambda_2\neq 0\hspace{0.5cm}\text{y}\hspace{0.5cm}-\dfrac{\Delta}{\delta}<0\hspace{0.5cm}\therefore\hspace{0.5cm}\dfrac{\Delta}{\delta}>0$$
\item[i)] Si $\omega<0,$ $\lambda_1$ y $\lambda_2$ son ambos negativos y el \underline{lugar es una  elipse} de ejes paralelos a $\vec{p_1}$ y $\vec{p_2}$ y de semiejes $\sqrt{\dfrac{\frac{\Delta}{\delta}}{|\lambda_1|}}, \sqrt{\dfrac{\frac{\Delta}{\delta}}{|\lambda_2|}}.$ Pero $$\lambda_{1,2}=\dfrac{\omega\pm\sqrt{\omega^2-4\delta}}{2}.$$ Luego los semiejes valen $$a,b=+\sqrt{\dfrac{2\Delta}{\delta\left|\omega\pm\sqrt{\omega^2-4\delta}\right|}}$$
\begin{figure}[ht!]
\begin{center}
  \includegraphics[scale=0.5]{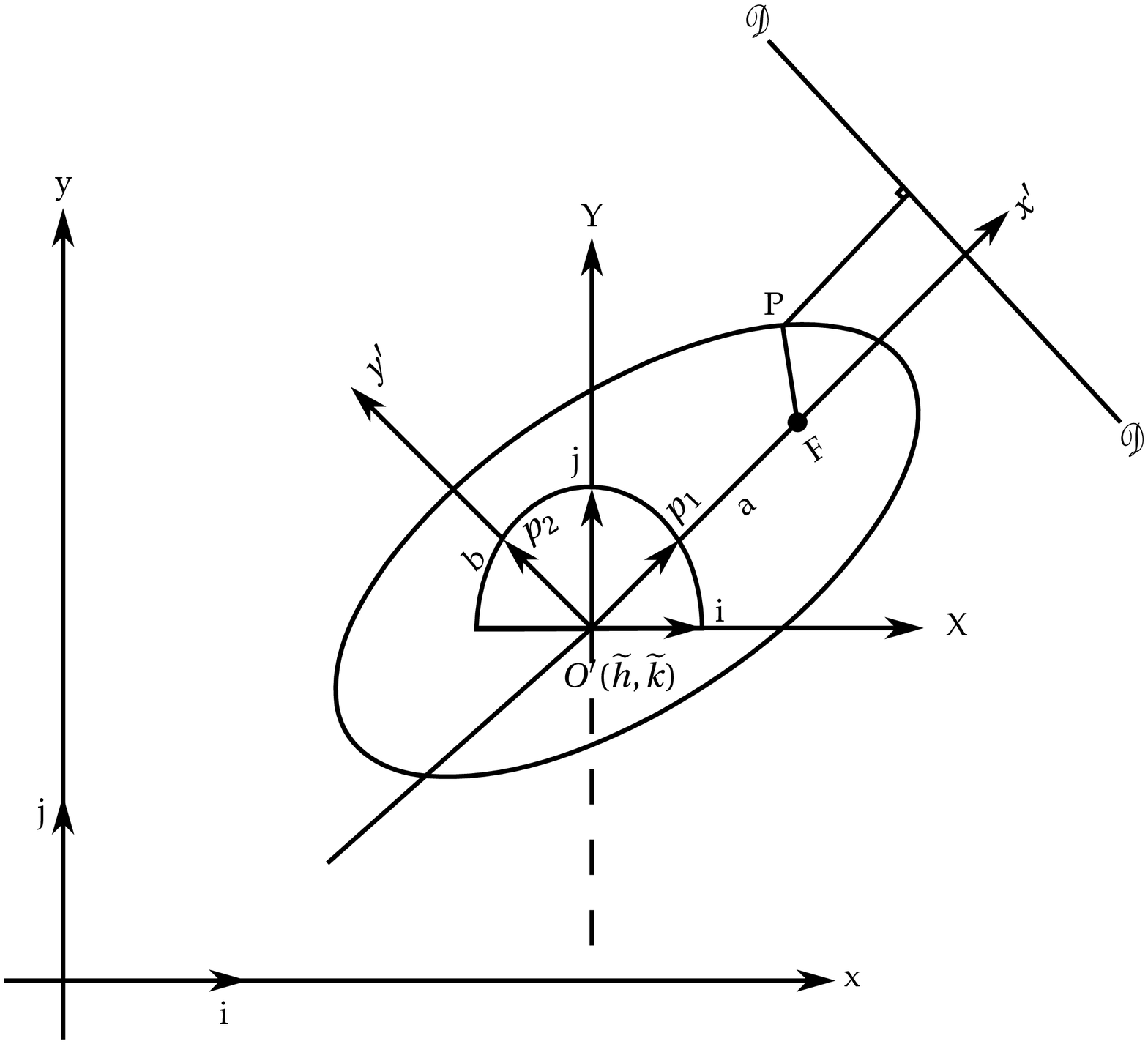}\\
\end{center}
\end{figure}
\newline
Recuérdese que $$a=\dfrac{\epsilon p}{1-\epsilon^2}\hspace{0.5cm}b=\dfrac{\epsilon p}{\sqrt{1-\epsilon^2}}$$
donde $\epsilon$ es la excentricidad y $P=d(F;\mathscr{DD})$ es la distancia foco-directriz.\\
Halle $\epsilon, p$ en términos de los invariantes $\Delta, \delta$ y $\omega.$\\
\item[ii)] Si $\omega>,$ $\lambda_1$ y $\lambda_2$ son ambos positivos y como $-\dfrac{\Delta}{\delta}<0,$ el \underline{lugar es $\emptyset$}.\\
\item[b)] Si $\delta=AC-B^2=\lambda_1\lambda_2<0,$ y $\lambda_1$ y $\lambda_2$ tienen signos contrarios y como
$-\dfrac{\Delta}{\delta}>0,$ el \underline{lugar es una hipérbola} de ejes paralelos a $\vec{p_1}$ y $\vec{p_2}.$ Los semiejes valen $\dfrac{-\frac{\Delta}{\delta}}{|\lambda_1|},\dfrac{-\frac{\Delta}{\delta}}{|\lambda_2|}$ donde $$\lambda_{1,2}=\dfrac{\omega\pm\sqrt{\omega^2-4\delta}}{2}.$$ Las ramas de la hipérbola se abren según el eje $x'$ o $y'$ de acuerdo a los signos de $\lambda_1$ y $\lambda_2.$\\
\item[II)] $\Delta=0.$ La ecuación [\ref{45}] se escribe $\lambda_1x'^2+\lambda_2y'^2=0.$\\
\item[a)] Si $\delta=AC-B^2=\lambda_1\lambda_2<0,$ $\lambda_1$ y $\lambda_2$ tienen signos contrarios.\\
El \underline{lugar consta de dos rectas concurrentes en $O'$}.\\
\item[b)] Si $\delta=AC-B^2=\lambda_1\lambda_2>0,$ $\lambda_1$ y $\lambda_2$ tienen el mismo signo.\\
El \underline{lugar es el punto $O'$}.\\
\item[III)] $\Delta<0.$ La ecuación [\ref{45}] es: $\lambda_1x'^2+\lambda_2y'^2=-\dfrac{\Delta}{\delta}.$\\
\item[a)] Si $\delta=AC-B^2=\lambda_1\lambda_2>0,$ $\lambda_1$ y $\lambda_2$ tienen el mismo signo, $\omega=\lambda_1+\lambda_2\neq 0$ y $-\dfrac{\Delta}{\delta}>0.$\\
\item[i)] Si $\omega<0,$ $\lambda_1$ y $\lambda_2$ son ambos negativos \underline{el lugar es $\emptyset$}.\\
\item[ii)] Si $\omega>0,$ $\lambda_1$ y $\lambda_2$ son ambos positivos. El \underline{lugar es una elipse} de ejes $\parallel$s a $p_1, p_2$ y de semiejes $\sqrt{\dfrac{-\frac{\Delta}{\delta}}{\lambda_1}},\sqrt{\dfrac{-\frac{\Delta}{\delta}}{\lambda_2}}$ donde $\lambda_{1,2}=\dfrac{\omega\pm\sqrt{\omega^2-4\delta}}{2}$.\\
\item[b)] Si $\delta=AC-B^2=\lambda_1\lambda_2<0,$ $\lambda_1$ y $\lambda_2$ tienen signos contrarios y siendo $-\dfrac{\Delta}{\delta}<0$ el \underline{lugar es una  hipérbola} de ejes $\parallel$s a $p_1$ y $p_2$ y semiejes $\sqrt{\dfrac{-\frac{\Delta}{\delta}}{|\lambda_1|}},\sqrt{\dfrac{-\frac{\Delta}{\delta}}{|\lambda_2|}}$ donde $\lambda_{1,2}=\dfrac{\omega\pm\sqrt{\omega^2-4\delta}}{2}$.
\end{cas}
El \textbf{Caso 1.6.2} se resume en la tabla siguiente:
\begin{figure}[ht!]
\begin{center}
  \includegraphics[scale=0.5]{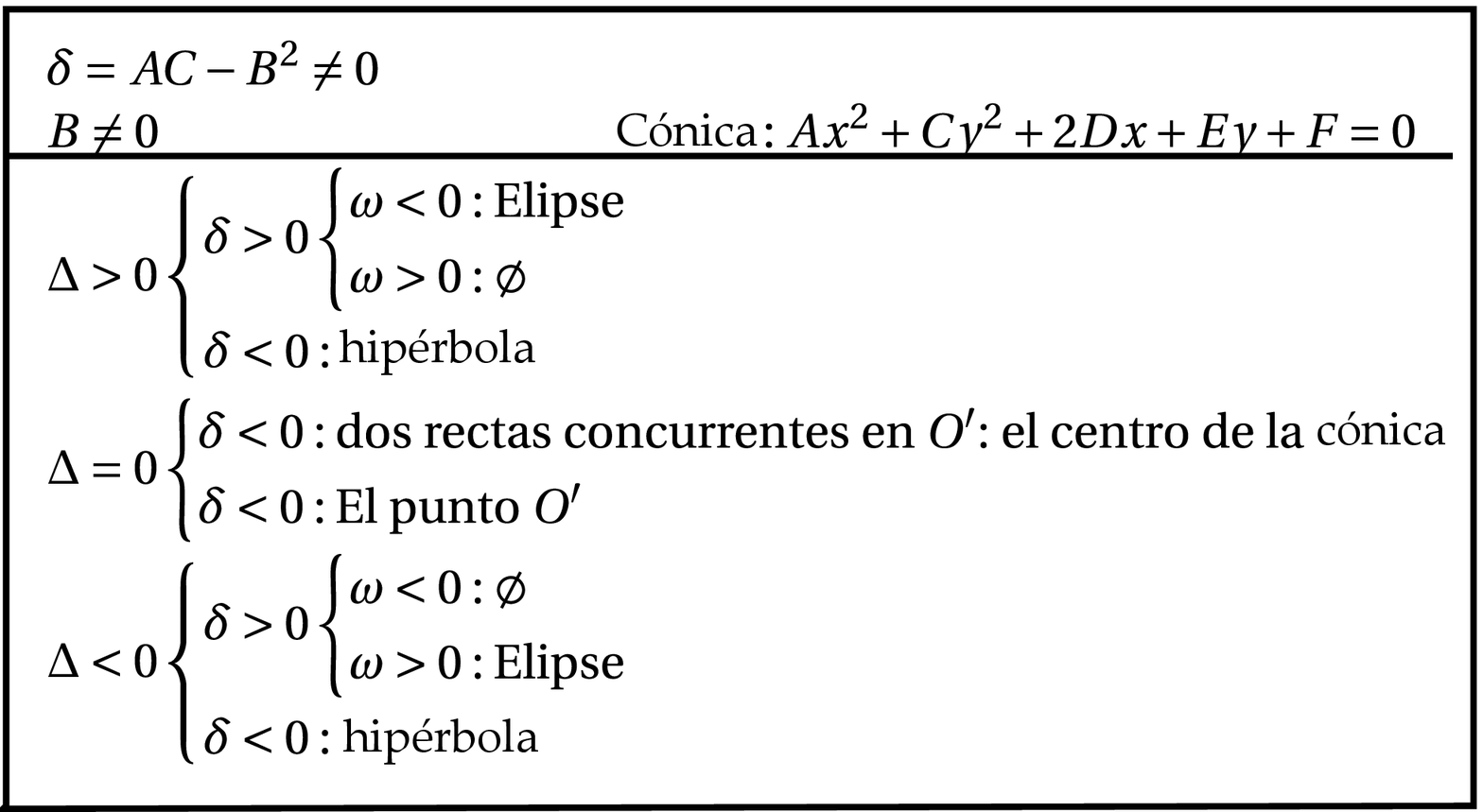}\\
\end{center}
\end{figure}
\newline
Las dos tablas anteriores se juntan en la tabla de página siguiente.\
\begin{obser}
Nótese que la parábola no tiene centro.\\
Los casos en que la cónica representa dos rectas $\parallel$s o una recta son casos en que hay $\infty$s centros. Por esa razón los casos degenerados que se obtienen en la tabla siguiente son dos rectas o un punto.
\end{obser}
\begin{figure}[ht!]
\begin{center}
  \includegraphics[scale=0.5]{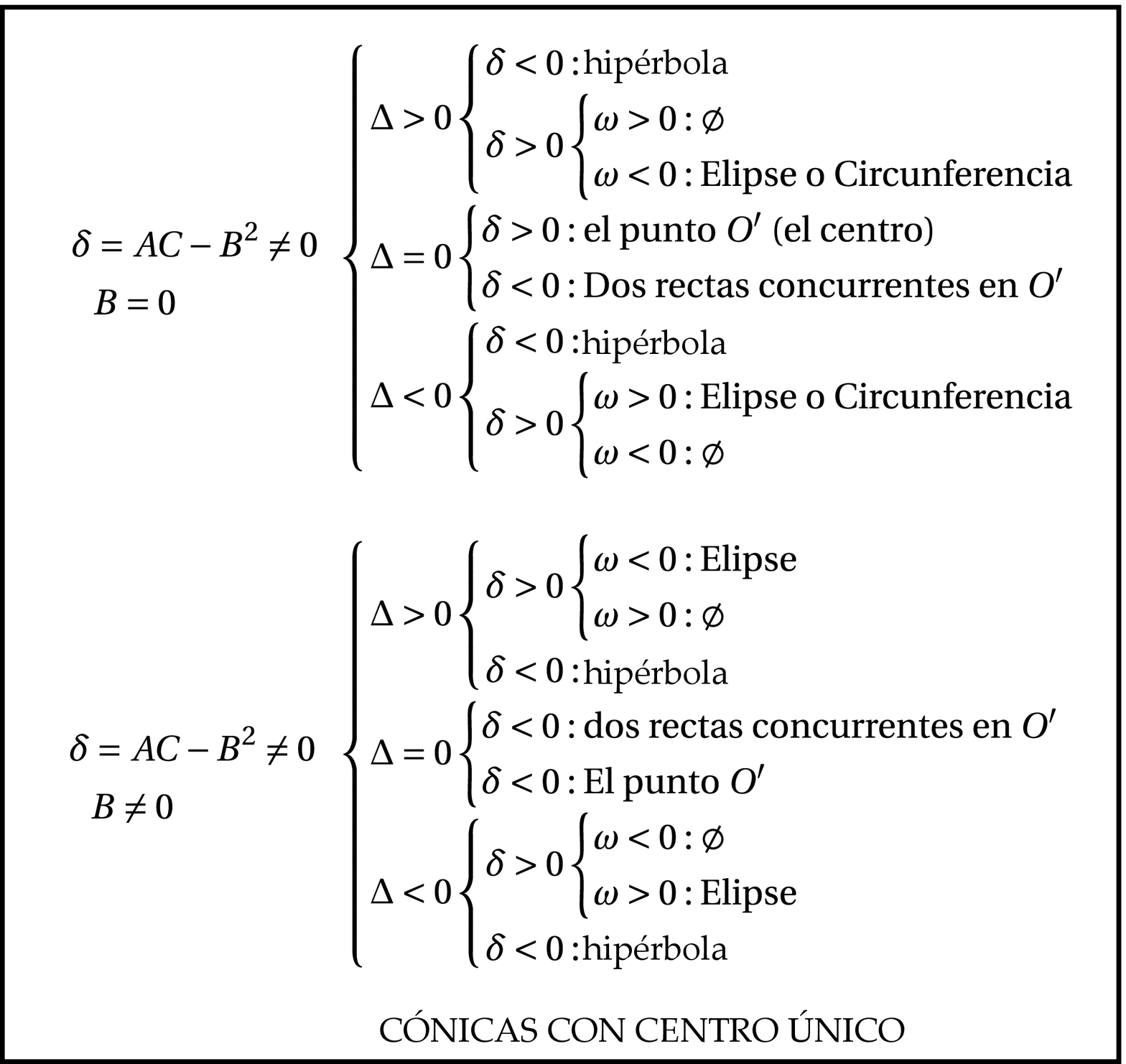}\\
\end{center}
\end{figure}
\end{enumerate}
\newpage
Esta tabla para determinar la naturaleza del lugar representado por las cónicas de centro único se puede reemplazar por la siguiente:\\
$\delta<0\begin{cases}\Delta<0:\text{hipérbola}\\\Delta=0:\text{dos rectas concurrentes en $O'$}\\\Delta>0:\text{hepérbola}\end{cases}\\
\delta>0
\begin{cases}
\Delta<0\begin{cases}\omega<0:\emptyset\\ \omega>0:\text{Elipse o circunferencia}\end{cases}\\
\Delta=0:\text{el punto $O'$}\\
\Delta>0\begin{cases}\omega<0:\emptyset\\ \omega>0:\text{Elipse}\end{cases}
\end{cases}$\\
\textbf{Ejemplos}
\begin{enumerate}
\item Consideremos la cónica $xy=k,\quad k>0,$ ó
\begin{equation}\label{A}
xy-k=0
\end{equation}
\begin{tabular}{cccc}
$A=0$&&&$M=\left(\begin{array}{cc}A&B\\B&C\end{array}\right)=\left(\begin{array}{cc}0&1/2\\1/2&0\end{array}\right),$\\
$C=0$\\
$2B=1,\quad B=1/2$\\
$D=E=0$&&&$tr(M)=0;\quad\delta=-\dfrac{1}{4}$\\
$F=-K$
\end{tabular}
\newline
Para hallar los centros debemos resolver el sistema\\
$\begin{cases}
0x+\dfrac{1}{2}y=0\\
\dfrac{1}{2}x+0y=0
\end{cases}$\\
Hay solución única: $(0,0).$\\
Se trata de una cónica con centro único en $(0,0).$\\
La cónica [\ref{A}] se puede escribir así:
\begin{align*}
&\left(\begin{array}{cc}x&y\end{array}\right)\left(\begin{array}{cc}0&1/2\\1/2&0\end{array}\right)\dfrac{x}{y}-k=0\\
&PCM(\lambda)=\lambda^2-tr(M)\lambda+\delta=\lambda^2-1/4=0\\
&\therefore\quad\left(\lambda+1/2\right)\left(\lambda-1/2\right)=0.\quad\text{Los valores propios de $M$ son $1/2,-1/2$}\\
&EP_{1/2}^M=\mathscr{N}\left(1/2I_2-M\right):\text{espacio nulo de la matriz $1/2\underset{\parallel}{I_2}-M$}\\
&\hspace{6.5cm}\dfrac{1}{2}\left(\begin{array}{cc}1&-1\\-1&1\end{array}\right)
\end{align*}
\begin{align*}
\mathscr{N}\left(\dfrac{1}{2}I_2-M\right)&=\mathscr{N}\left(\dfrac{1}{2}\left(\begin{array}{cc}1&-1\\-1&1\end{array}\right)\right)=\left\{\dbinom{x}{y}\diagup\begin{array}{cc}x-y=0\\-x+y=0\end{array}\right\}\\
&\left\{\dbinom{x}{y}\diagup x=y\right\}=Sg\left\{\dbinom{1}{1}\right\}=Sg\left\{\begin{array}{cc}\dfrac{1}{\sqrt{2}}\\\dfrac{1}{\sqrt{2}}\end{array}\right\}=EP_{1/2}^M\\
\lambda(M)&=\left\{1/2,-1/2\right\}\\
P&=\left(\begin{array}{cc}\overset{\uparrow}{\underset{\downarrow}{p_1}}&\overset{\uparrow}{\underset{\downarrow}{p_2}}\end{array}\right)\left(\begin{array}{cc}\frac{1}{\sqrt{2}}&-\frac{1}{\sqrt{2}}\\ \frac{1}{\sqrt{2}}&\frac{1}{\sqrt{2}}\end{array}\right)
\end{align*}
\begin{figure}[ht!]
\begin{center}
  \includegraphics[scale=0.5]{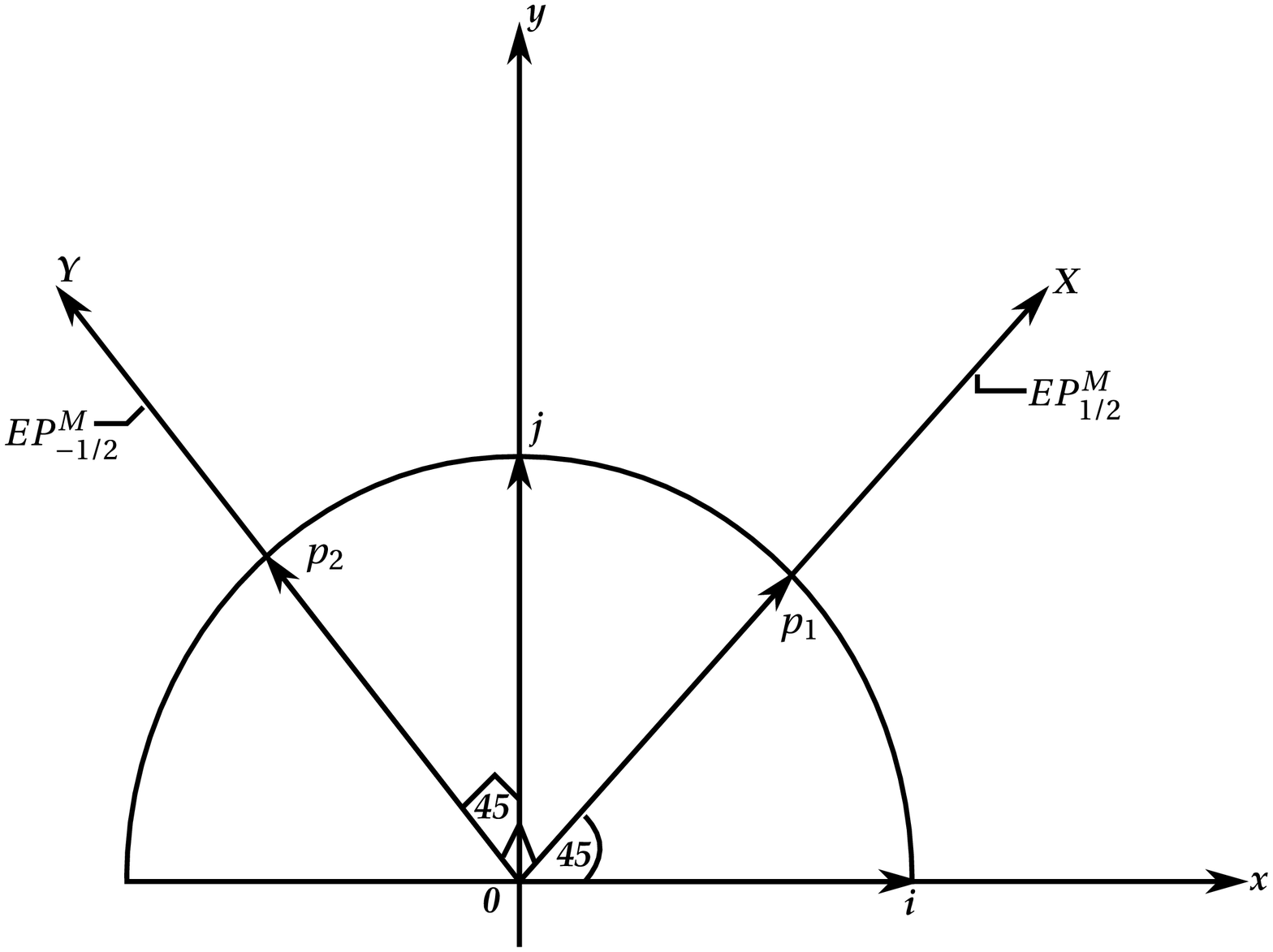}\\
\end{center}
\end{figure}
\newline
Una vez aplicado el Tma. Espectral, se tiene que la ecuación de la cónica$\diagup XY$ es $$\left(\begin{array}{cc}X&Y\end{array}\right)\left(\begin{array}{cc}1/2&0\\0&-1/2\end{array}\right)\dbinom{X}{Y}-k=0.$$
O sea, $$\dfrac{1}{2}X^2-\dfrac{1}{2}Y^2=k;\quad\dfrac{X^2}{2k}-\dfrac{Y^2}{2k}=1,\quad k>0.$$
Se trata entonces de la \underline{hipérbola equilátera} de la figura siguiente: $$a^2=2k;\quad b^2=2k;\quad a=\sqrt{2k};\quad b=\sqrt{2k}$$
\begin{figure}[ht!]
\begin{center}
  \includegraphics[scale=0.5]{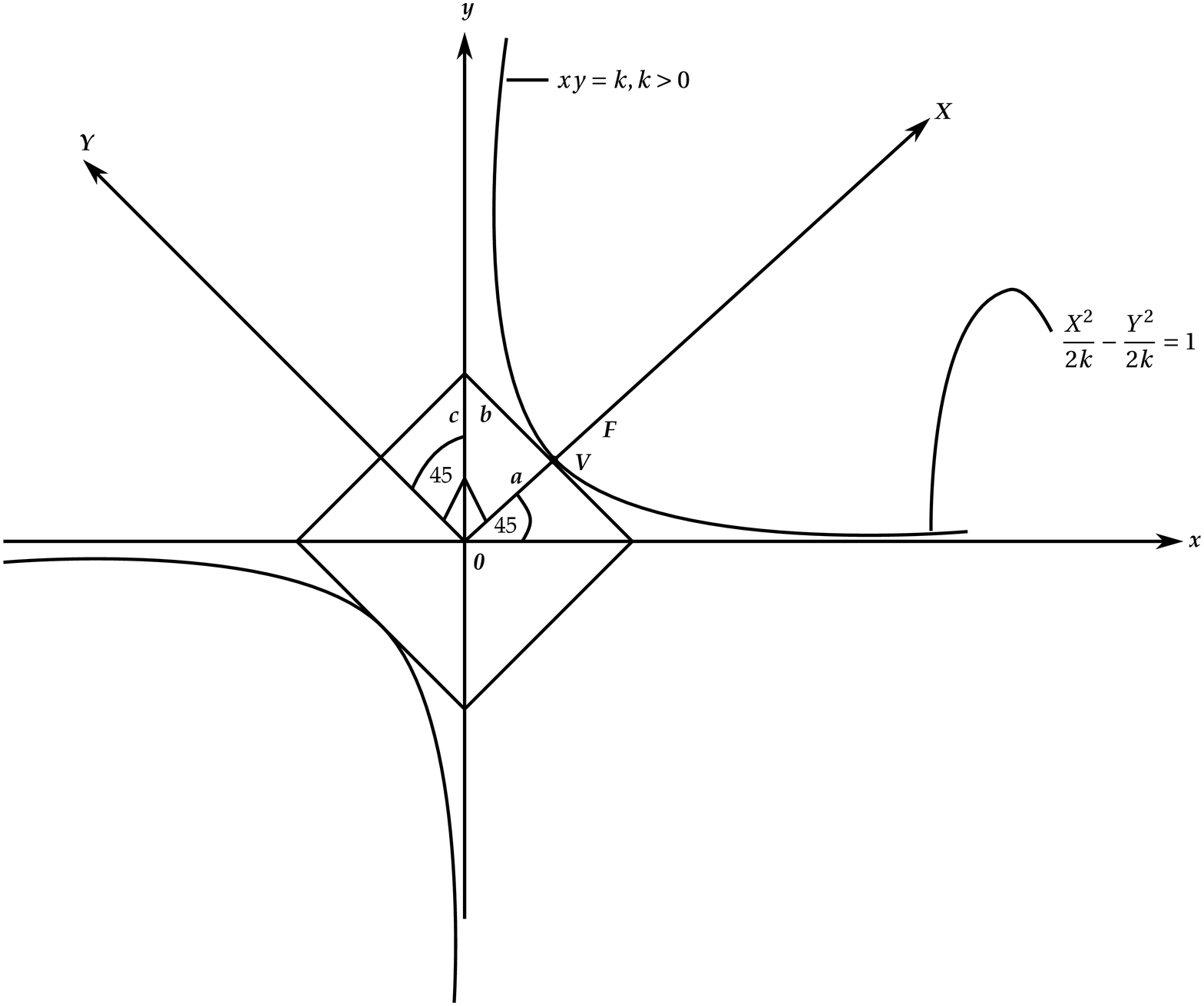}\\
\end{center}
\end{figure}
\newline
Su excentricidad es $$\epsilon=\dfrac{c}{a}=\dfrac{\sqrt{a^2+b^2}}{a}=\dfrac{\sqrt{4k}}{\sqrt{2k}}=\dfrac{2}{\sqrt{2}}=\dfrac{\cancel{2}\sqrt{2}}{\cancel{2}}=\sqrt{2}$$
Los ejes $x,y$ son las asíntotas de la hipérbola.\\
El vértice $V$ de la curva tiene coordenadas $\left(a=0V=\sqrt{2},0\right)\diagup XY$ y se comprende que mientras mayor sea $k$, más alejado está el vértice $V$ de $0$.\\
\item Identificar el lugar geométrico representado por la ecuación $$-2x^2+2xy-y^2+2y-3=0$$
$$\left|\begin{array}{ccc}-2&1&0\\1&-1&1\\0&1&-3\end{array}\right|=-1<0$$\\
$\delta=1\neq 0; \omega=-3.$ Como $\delta\neq 0, B\neq 0, \Delta<0, \delta>0$ y $\omega<0,$ el lugar es $\emptyset.$ Vamos a verificarlo.\\
De la ecuación de la cónica,
\begin{align*}
y^2-\left(2+2x\right)y+\left(3+2x^2\right)=0\\
\therefore\hspace{0.5cm}y=\dfrac{\left(2+2x\right)\pm\sqrt{{\left(2+2x\right)^2-4\left(3+2x^2\right)}}}{2}
\end{align*}
La cantidad sub-radical es
\begin{align*}
4+4x^2+8x-12-8x^2&=-4x^2+8x-8\\
&=-4\left(x^2-2x+2\right)\\
&=-4\left((x^2-2x+1)+2-1\right)\\
&=-4\left[{(x-1)}^2+1\right]<0\hspace{0.5cm}\text{cualquiera sea $x\in\mathbb{R},$}
\end{align*}
lo que demuestra que los valores de $y$ son imaginarios y por lo tanto \underline{el lugar es $\emptyset$}.\\
\item Identificar el luagar representado por la ecuación
\begin{align*}
2x^2+2xy+3y^2+6x+8y+7=0\\
\Delta=\left|\begin{array}{ccc}2&1&3\\1&3&4\\3&4&7\end{array}\right|=0
\end{align*}
$\delta=5\neq 0; B\neq 0.$ Como $\delta\neq 0, B\neq 0, \Delta=0$ y $\delta=0,$\\
\underline{el lugar es el punto $O'$: el centro de la curva.}\\
Hallemos las coordenadas del centro.\\
\begin{align*}
2x+y&=-3\\
x+3y&=-4\hspace{0.5cm}\therefore\hspace{0.5cm}y=-1, x=1.\hspace{0.5cm}\text{Luego} O'(-1,-1).
\end{align*}
$f(-1,-1)=2+2+3-6-8+7=0$ lo que demuestra que $O'$ está en la cónica. Verifiquemos que efectivamente el lugar es el punto $(-1,-1).$\\
De la ecuación de la cónica,
\begin{align*}
3y^2+\left(2x+8\right)y+\left(2x^2+6x+7\right)=0\\
\therefore\hspace{0.5cm}y=\dfrac{-(2x+8)\pm\sqrt{{(2x+8)}^2-12(2x^2+6x+7)}}{6}
\end{align*}
La cantidad sub-radical es $$4x^2+32x+64-24x^2-72x-84=-20{(x+1)}^2.$$
Luego $$y=\dfrac{-(2x+8)\pm\sqrt{-20{(x+1)}^2}}{6}.$$
La cantidad sub-radical se anula en $x=-1$ y en otro caso es negativa. Además, si $x=-1, y=\dfrac{-\left(2(-1)+8\right)}{6}=-1.$\\
El lugar es entonces el centro $O'(-1,-1)$ de la cónica.\\
\item Consideremos la cónica
\begin{align*}
3x^2-4xy+y^2+10x-2y-8=0\\
\Delta&=\left|\begin{array}{ccc}3&-2&5\\-2&1&1\\5&-1&-8\end{array}\right|=0
\end{align*}
$\delta=-3\neq 0.$ La cónica tiene centro único. Como $B\neq 0, \Delta=0$ y $\delta<0,$\\
el \underline{lugar son dos rectas concurrentes} en $O'.$ Las coordenadas del centro se obtienen resolviendo el sistema
\begin{align*}
3h-2k+5&=0\\
-2h+k-1&=0\hspace{0.5cm}\therefore\hspace{0.5cm}(\widetilde{h}, \widetilde{k})=(3,7)
\end{align*}
De la ecuación de la cónica,
\begin{align*}
&y^2-(4x+2)y+(3x^2+10x-8)=0\\
&\therefore\hspace{0.5cm}y=\dfrac{(4x+2)\pm\sqrt{{(4x+2)}^2-4(3x^2+10x-8)}}{2}
\end{align*}
La cantidad sub-radical es $${(4x-2)}^2-4(3x^2+10x-8)=4{(x-3)}^2$$
Luego $$y=\dfrac{(4x+2)\pm\sqrt{4{(x-3)}^2}}{2}=(2x+1)\pm(x-3)$$
Entonce el lugar consiste de dos rectas: $y=3x-2$ y $y=x+4$ que como puede probarse se cortan en $(3,7).$\\
Finalmente nótese que $$(3x-y-2)(x-y+4)=3x^2-4xy+y^2+10x-2y-8$$
\item Identificar y dibujar el lugar representado por la ecuación $$5x^2+6xy+5y^2-4x+4y-4=0.$$
\begin{tabular}{ccccc}
$A=5$&$\dbinom{A}{B}=\dbinom{5}{3}$&$2D=-4$&$D=-2$&$\dbinom{-D}{-E}=\dbinom{2}{-2}$\\
$2B=6; B=3\neq 0$&&$2E=4$&$E=2$&\\
$C=5$&$\dbinom{B}{C}=\dbinom{3}{5}$&&
\end{tabular}
\begin{figure}[ht!]
\begin{center}
  \includegraphics[scale=0.5]{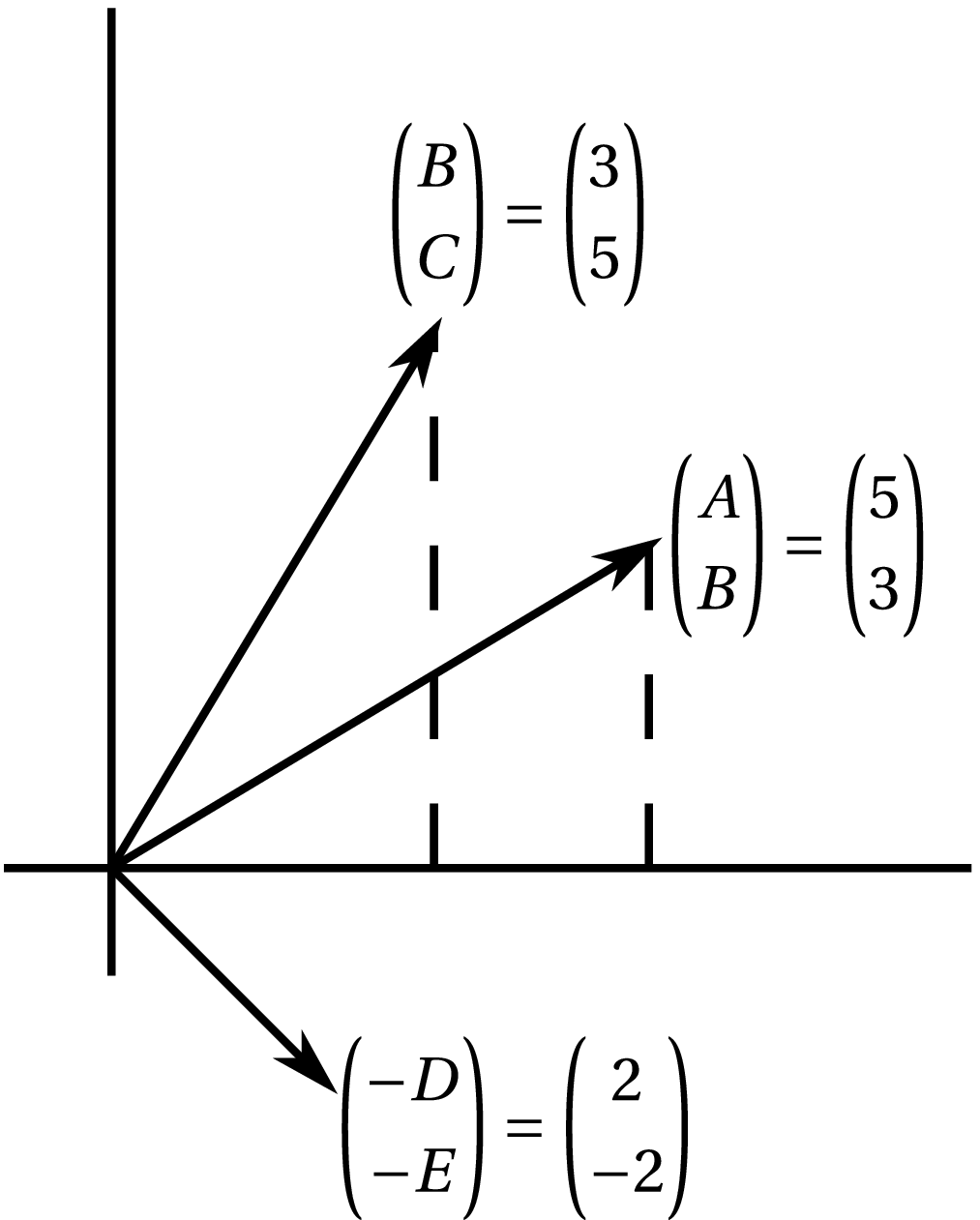}\\
\end{center}
\end{figure}
\newline
\underline{La cónica tiene centro único}.\\
$\Delta=\left|\begin{array}{ccc}5&3&-2\\3&5&2\\-2&2&-4\end{array}\right|=-128;\,\,\delta=\left|\begin{array}{cc}5&3\\3&5\end{array}\right|=16;\,\,\omega=A+C=10.$\\
Así que $\delta\neq 0, B\neq 0, \Delta<0, \delta>0$ y $\omega>0$.\\
\underline{Se trata de una elipse.} El centro es el punto $O'(\widetilde{h}, \widetilde{k})$ donde $$\dbinom{\widetilde{h}}{\widetilde{k}}=\dfrac{1}{16}\left(\begin{array}{cc}5&-3\\-3&5\end{array}\right)\dbinom{2}{-2}=\dbinom{1}{-1}:$$
\begin{figure}[ht!]
\begin{center}
  \includegraphics[scale=0.5]{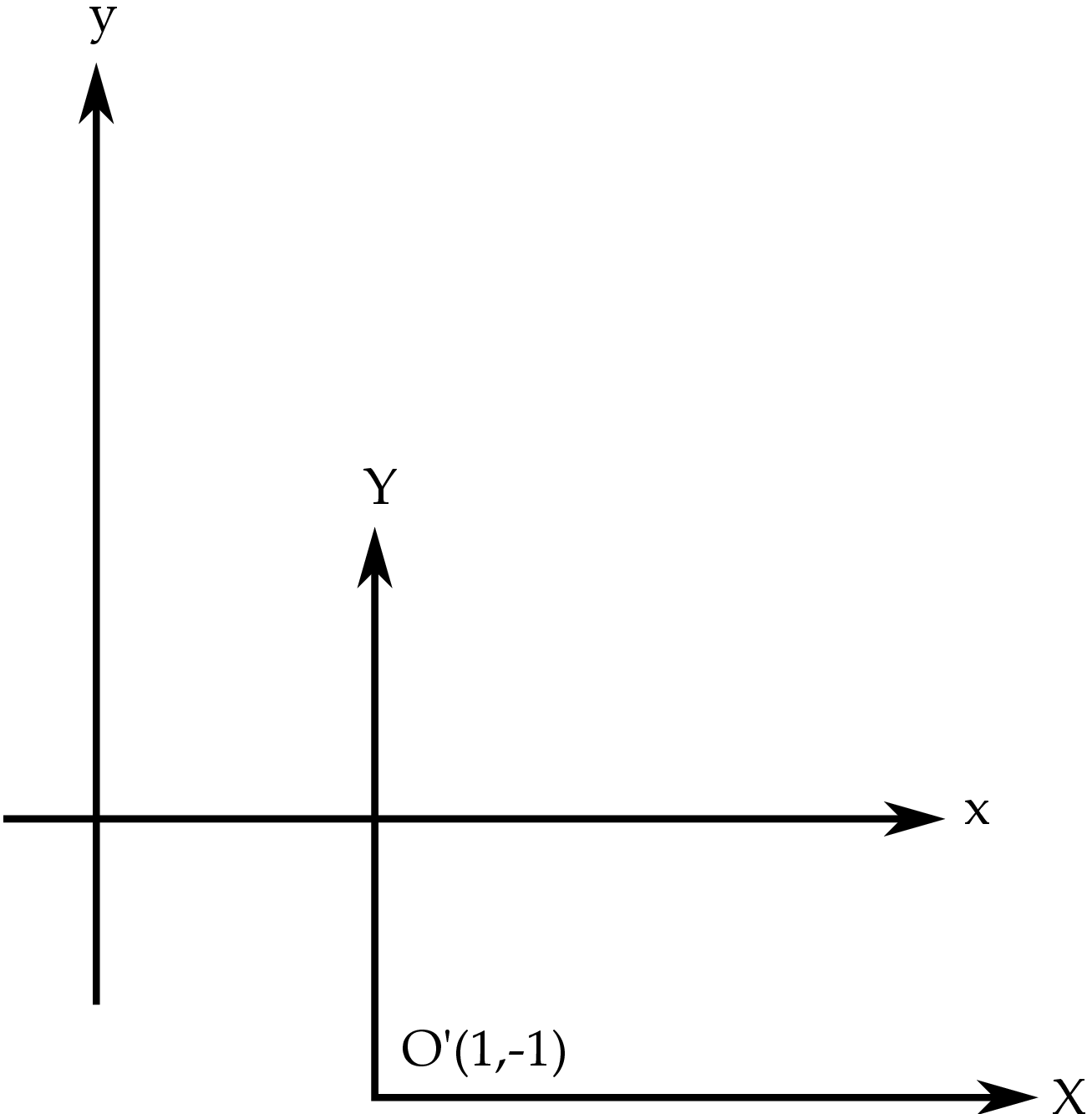}\\
\end{center}
\end{figure}
\newline
La ecuación de la elipse referida a los ejes $XY$ con origen en $O'$ es: $$5X^2+6XY+5Y^2=-\dfrac{\Delta}{\delta}=-\left(-\dfrac{128}{16}\right)=8$$
$$M=\left(\begin{array}{cc}5&3\\3&5\end{array}\right);\,\,PCM(\lambda)_{(\lambda)}=\lambda^2-10\lambda+16=0.$$
Los valores propios de $M$ son 8 y 2.\\
$EP_2^M=\mathscr{N}\left(2I_2-M\right)=\mathscr{N}\left(\begin{array}{cc}-3&-3\\-3&-3\end{array}\right);\,\,-3x-3y=0\,\,\therefore\,\,y=-x.$\\
O sea que todo vector de la forma $\dbinom{x}{-x}$ con $x\in\mathbb{R}$ está en el $EP_2^M.$ Luego, si $x=-1,\,\,\dbinom{-1}{1}\in EP_2^M.$\\
$\therefore\hspace{0.5cm}\left\{\dbinom{-\frac{1}{\sqrt{2}}}{\frac{1}{\sqrt{2}}}\right\}$ es Base del $EP_2^M$ y por lo tanto, $\left\{\dbinom{\frac{1}{\sqrt{2}}}{\frac{1}{\sqrt{2}}}\right\}$ es Base del $EP_8^M.$\\
Como éste vector está en el centro $I$, el espectro de $M$ lo ordenamos así: $\lambda={8,2}.$\\
O sea que definimos $\lambda_1=8,\,\,\lambda_2=2.$\\
$$P=\left(\begin{array}{cc}\underset{\downarrow}{\overset{\uparrow}{p_1}}&\underset{\downarrow}{\overset{\uparrow}{p_2}}\end{array}\right)=\left(\begin{array}{cc}1/\sqrt{2}&-1/\sqrt{2}\\1/\sqrt{2}&1/\sqrt{2}\end{array}\right).$$
Los ejes $x'y'$ se obtienen rotando $\overset{\curvearrowleft}{45^\circ}$ los ejes $XY$ respecto a $O':$
\begin{figure}[ht!]
\begin{center}
  \includegraphics[scale=0.5]{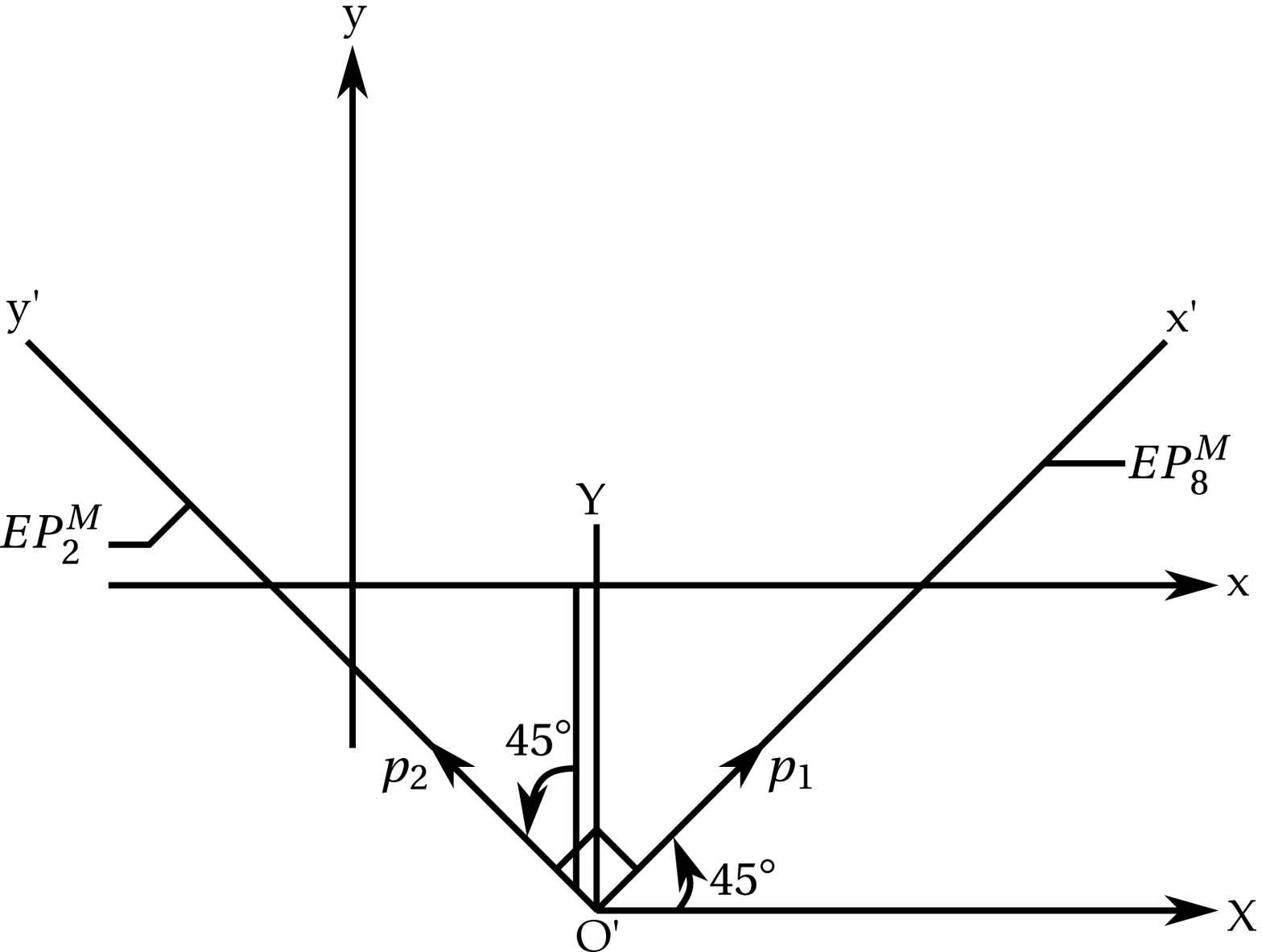}\\
\end{center}
\end{figure}
\newline
Ecuación de la cónica respecto a $x'y':$ $$\lambda_1x'^2+\lambda_2y'^2=8.$$ O sea $8x'^2+2y'^2=8$ y finalmente, $$\dfrac{x'^2}{1}+\dfrac{y'^2}{2^2}=1.$$ Esto nos dice que los semiejes valen 1 y 2 y los focos están en el eje $y':$
\newpage
\begin{figure}[ht!]
\begin{center}
  \includegraphics[scale=0.5]{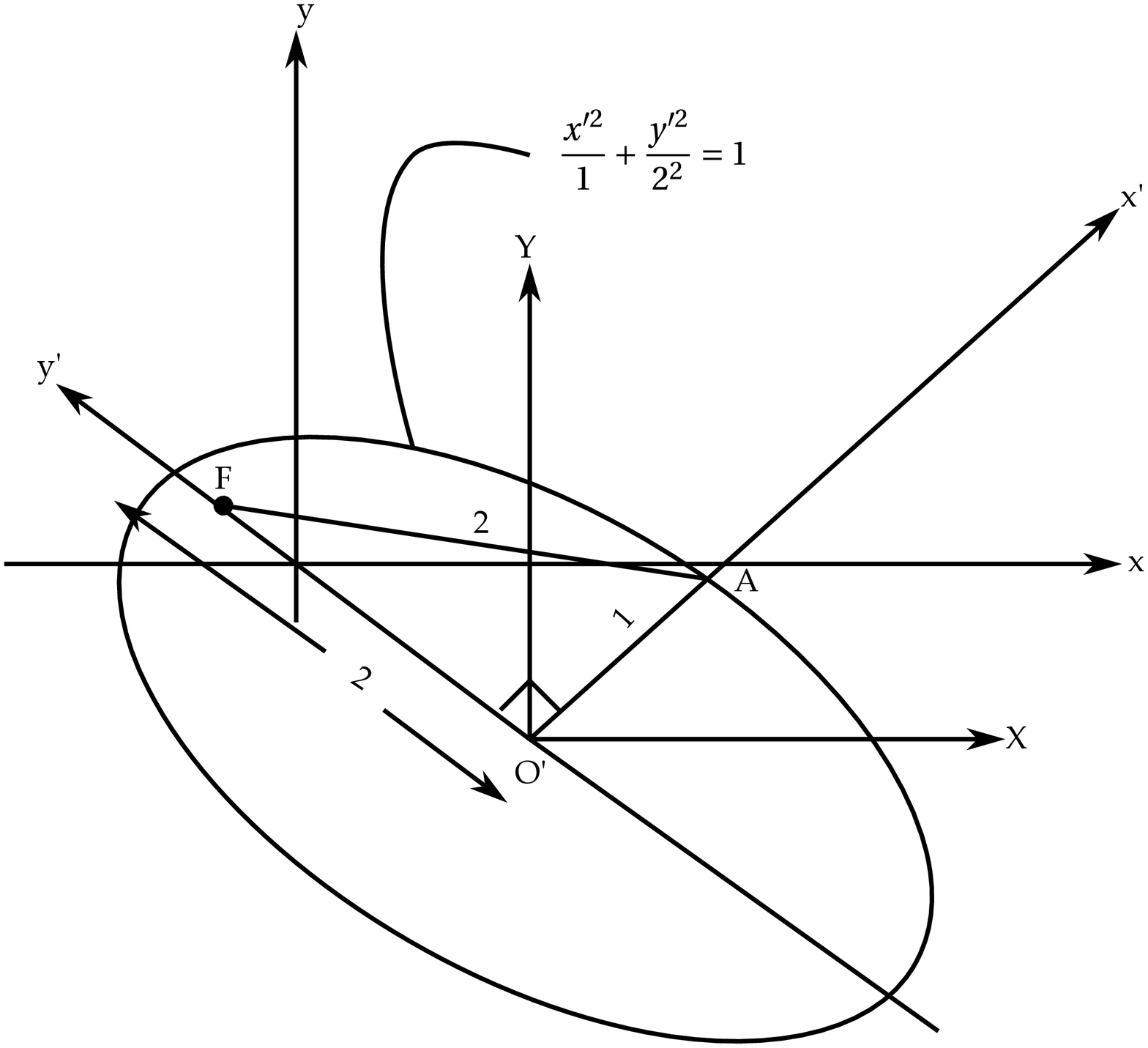}\\
\end{center}
\end{figure}
Localicemos por ejemplo el foco respecto al eje $x'y'.$\\
En el $\underset{O'AF}{\triangle},\,\,O'F=\sqrt{4-1}=\sqrt{3}.$ Luego las coordenadas de $F/x'y'$ son $(0,\sqrt{3}).$\\
Ahora,
\begin{align*}
x&=X+\widetilde{h}\\
y&=Y+\widetilde{k}
\end{align*}
O sea $$\dbinom{x}{y}=\dbinom{X}{Y}+\dbinom{\widetilde{h}}{\widetilde{k}}.$$
Pero $\dbinom{X}{Y}=P\dbinom{x'}{y'}.$ Luego $\dbinom{x}{y}=P\dbinom{x'}{y'}+\dbinom{\widetilde{h}}{\widetilde{k}}=\left(\begin{array}{cc}1/sqrt{2}&-1/sqrt{2}\\1/sqrt{2}&1/sqrt{2}\end{array}\right)\dbinom{x'}{y'}+\dbinom{1}{-1}$\\
Entonces
\begin{align*}
x&=\dfrac{1}{\sqrt{2}}x'-\dfrac{1}{\sqrt{2}}y'+1\\
y&=\dfrac{1}{\sqrt{2}}x'+\dfrac{1}{\sqrt{2}}y'-1
\end{align*}
Localizados los focos de la elipse respecto a $x'y'$, las ecuaciones anteriores nos permiten localizarlos$/xy$.\\
Encuentre la excentricidad y localice el otro foco y las directrices.\\
\item Estudiemos el lugar geométrico representado por la ecuación $$x^2-10xy+y^2+x+y+1=0$$
$$\dbinom{A}{B}=\dbinom{1}{-5};\,\,\dbinom{B}{C}=\dbinom{-5}{1};\,\,\dbinom{-D}{-E}=\dbinom{-1/2}{-1/2}$$
\begin{figure}[ht!]
\begin{center}
  \includegraphics[scale=0.5]{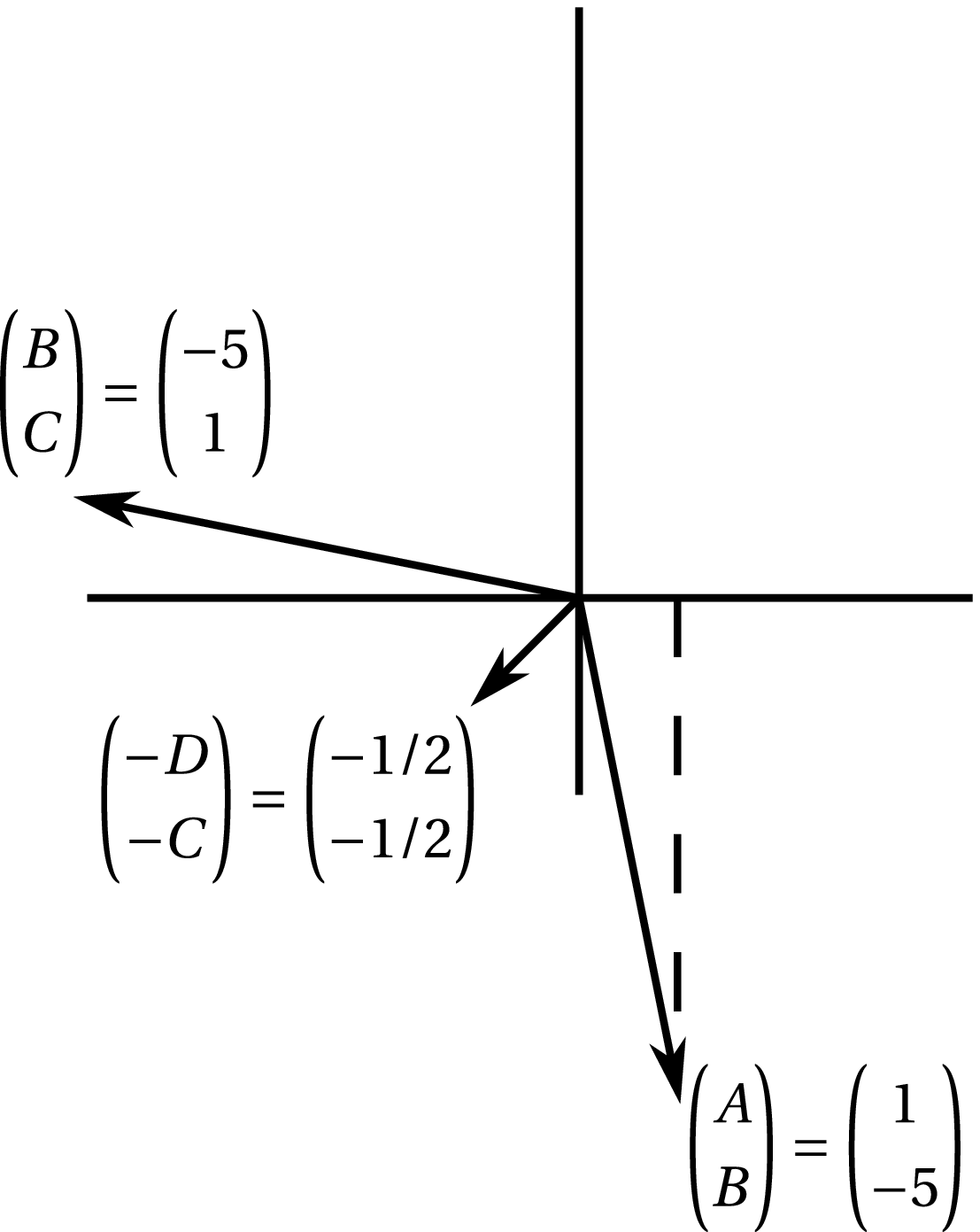}\\
\end{center}
\vspace{0.5cm}
\end{figure}
Se trata de una \underline{cónica con centro único}.\\
$$\Delta=\left|\begin{array}{ccc}1&-5&1/2\\-5&1&1/2\\1/2&1/2&1\end{array}\right|=-27<0;\,\,\delta=\left|\begin{array}{cc}1&-5\\-5&1\end{array}\right|=-24<0;\,\,B=-5\neq 0.$$
Como $B\neq 0, \Delta<0$ y $\delta<0,$ la \underline{cónica es una hipérbola}.\\
Hallemos su centro $O'(\widetilde{h}, \widetilde{k}).$\\
$$\dbinom{\widetilde{h}}{\widetilde{k}}=-\dfrac{1}{24}\left(\begin{array}{cc}1&5\\5&1\end{array}\right)\dbinom{-1/2}{-1/2}=\dbinom{1/8}{1/8}:$$
\begin{figure}[ht!]
\begin{center}
  \includegraphics[scale=0.4]{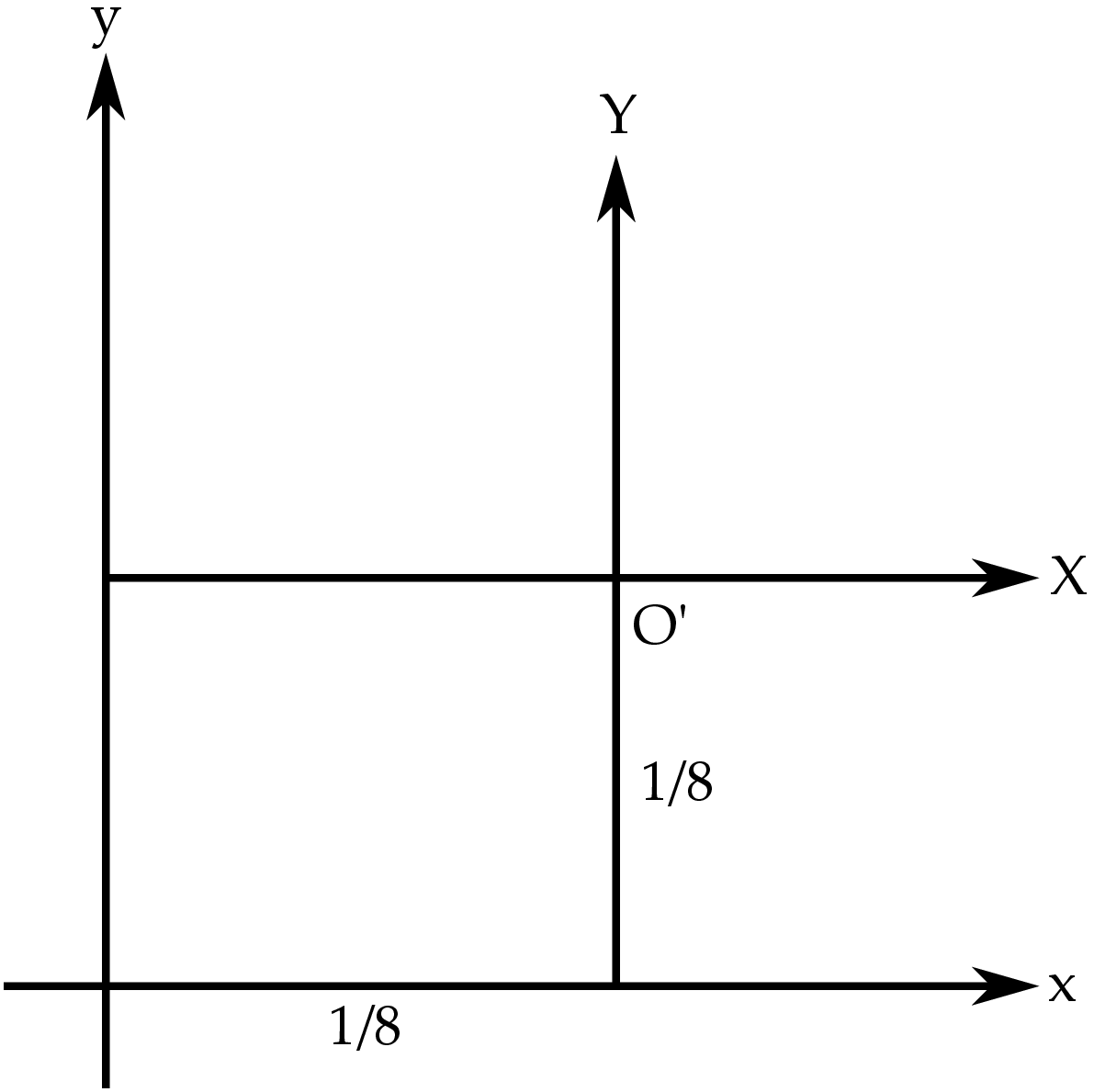}\\
\end{center}
\vspace{0.5cm}
\end{figure}
La ecuación de la hipérbola$/XY$ es: $$X^2-10XY+Y^2=-\dfrac{\Delta}{\delta}=-\dfrac{27}{24}=-\dfrac{9}{8}$$
$$M=\left(\begin{array}{cc}1&-5\\-5&1\end{array}\right);\,\,PCM(\lambda)=\lambda^2-(trM)\lambda+|M|=\lambda^2-2\lambda-24=0$$
Los valores propios de $M$ son $-4$ y $6$.\\
$$EP_{-4}^M=\mathscr{N}\left(\begin{array}{cc}-5&5\\5&-5\end{array}\right)=Sg\left\{\dbinom{1/\sqrt{2}}{1/\sqrt{2}}\right\}$$
Como $\left(\begin{array}{cc}1/\sqrt{2}\\1/\sqrt{2}\end{array}\right)$ está en el cuadrante $I,$ tenemos $\lambda_1=-4$ y $\lambda_2=6.$\\
Luego $$EP_6^M=Sg\left\{\dbinom{-1/\sqrt{2}}{1\sqrt{2}}\right\}.$$
Entonces $$P=\left(\begin{array}{cc}1/\sqrt{2}&-1/\sqrt{2}\\1/\sqrt{2}&1/\sqrt{2}\end{array}\right)$$
lo que nos indica que los ejes $x'y'$ se obtienen rotando los $XY$ un ángulo de $\overset{\curvearrowleft}{45^\circ}$ al rededor de $O':$
\begin{figure}[ht!]
\begin{center}
  \includegraphics[scale=0.5]{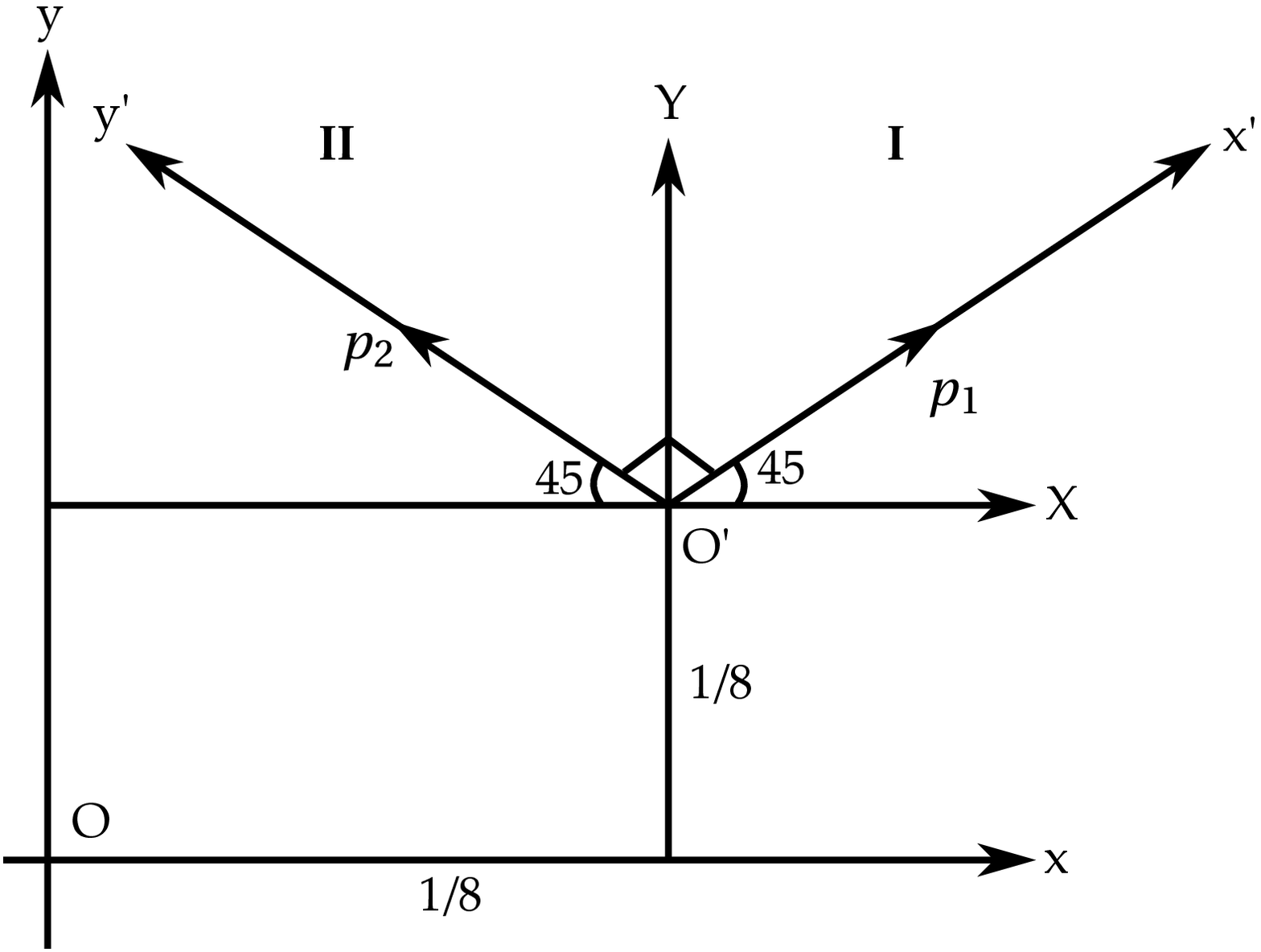}\\
\end{center}
\vspace{0.5cm}
\end{figure}
La ecuación de la hipérbola$/x'y'$ es: $$\lambda_1x'^2+\lambda_2y'^2=-\dfrac{\Delta}{\delta}=-\dfrac{9}{8}.$$
O sea $$-4x'^2+6y'^2=-\dfrac{9}{8}$$ y finalmente $$\dfrac{x'^2}{\frac{9}{32}}-\dfrac{y'^2}{\frac{3}{16}}=1$$
Las ramas de la hipérbola se abren según el eje $x'$ y los focos están sobre el mismo eje.\\
Ahora
\begin{align*}
\dbinom{x}{y}&=\dbinom{X}{Y}+\dbinom{1/8}{1/8}=P\dbinom{x'}{y'}+\dbinom{1/8}{1/8}\\
&=\left(\begin{array}{cc}1/\sqrt{2}&-1/\sqrt{2}\\ 1/\sqrt{2}&1/\sqrt{2}\end{array}\right)\dbinom{x'}{y'}+\dbinom{1/8}{1/8}\\
&=\left(\begin{array}{ccc}1/\sqrt{2}x'&-&1/\sqrt{2}y'+1/8\\\\ 1/\sqrt{2}x'&+&1/\sqrt{2}y'+1/8\end{array}\right)
\end{align*}
\begin{figure}[ht!]
\begin{center}
  \includegraphics[scale=0.5]{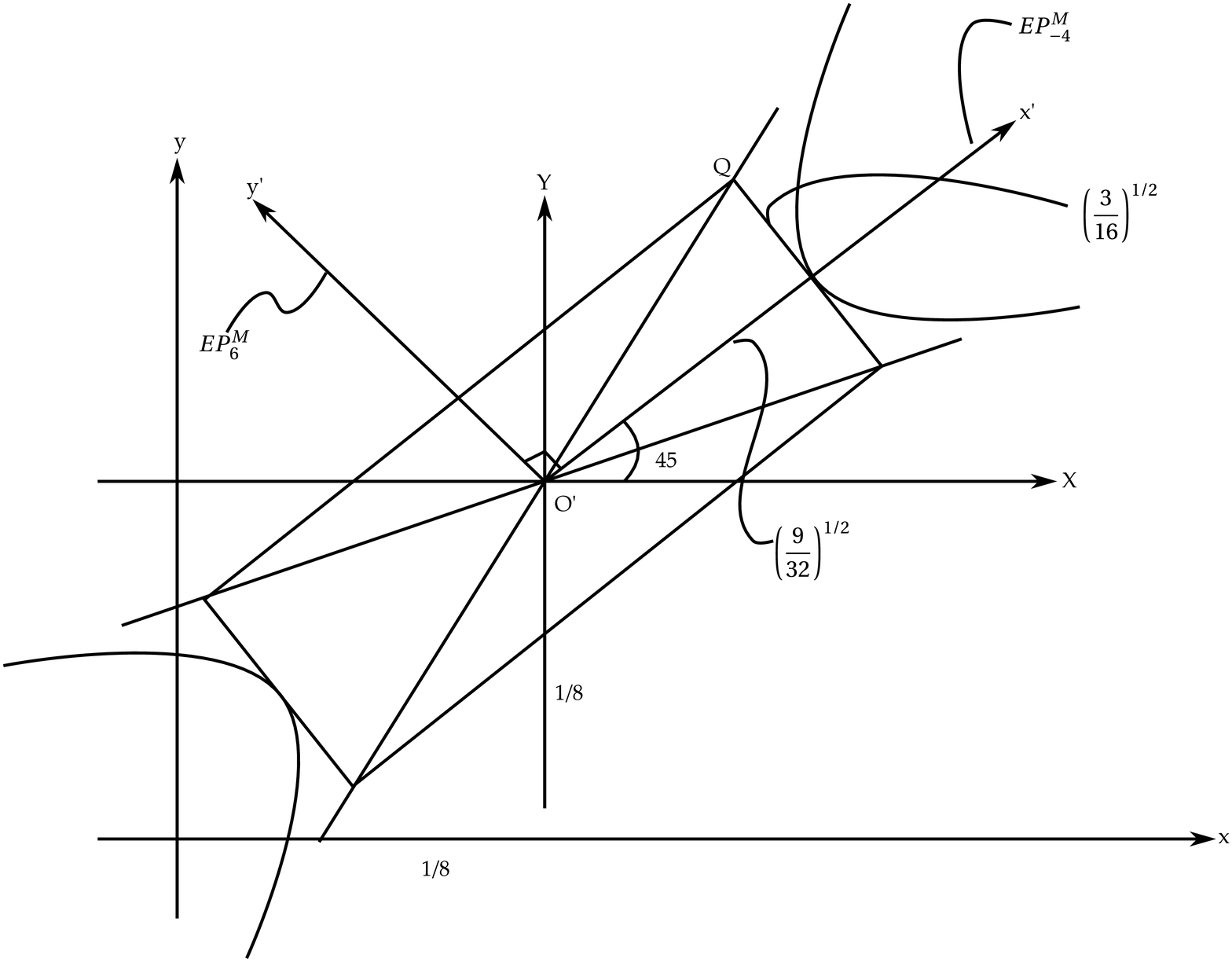}\\
\end{center}
\vspace{0.5cm}
\end{figure}
\newpage
Podemos utilizar las ecuaciones de transformación anteriores para hallar algunos elementos geométricos de la hipérbola.\\
Por ejemplo, hallemos la ecuación de la asíntota $\overset{\leftrightarrow}{OR}/xy.$\\
Esta recta tiene por ecuación$/x'y':y'=\dfrac{{\left(\frac{3}{16}\right)}^{1/2}}{{\left(\frac{9}{32}\right)}^{1/2}}x'={\left(\dfrac{3}{2}\right)}^{1/2}x'.$\\
Así que su ecuación$/xy$ es
\begin{align*}
x&=\dfrac{1}{\sqrt{2}}x'-\dfrac{1}{\sqrt{2}}{\left(\dfrac{2}{3}\right)}^{1/2}x'+\dfrac{1}{8}=\left(\dfrac{1}{\sqrt{2}}-\dfrac{1}{\sqrt{3}}\right)x'+\dfrac{1}{8}\\
y&=\dfrac{1}{\sqrt{2}}x'+\dfrac{1}{\sqrt{2}}{\left(\dfrac{2}{3}\right)}^{1/2}x'+\dfrac{1}{8}=\left(\dfrac{1}{\sqrt{2}}+\dfrac{1}{\sqrt{3}}\right)x'+\dfrac{1}{8}\\
\end{align*}
Eliminando el parámetro $x',$
\begin{align*}
x'=\dfrac{x-\dfrac{1}{8}}{\dfrac{1}{\sqrt{2}}-\dfrac{1}{\sqrt{3}}}&=\dfrac{y-\dfrac{1}{8}}{\dfrac{1}{\sqrt{2}}+\dfrac{1}{\sqrt{3}}}\\
&\overbrace{\text{ecuación de $\overset{\leftrightarrow}{OQ}/xy$}}
\end{align*}
Existe otra manera de encontrar la ecuación de las asíntotas de una hipérbola a partir de su ecuación$/xy.$\\
Esto se analiza en un problema que sigue más adelante.\\
\end{enumerate}
\begin{ejer}
Vamos a estudiar la cónica de ecuación
$$f(x,y)=Ax^2+2Bxy+Cy^2=0\hspace{0.5cm}\text{con}\hspace{0.5cm}A,B,C\neq 0$$ y $$\delta=AC-B^2\neq 0.$$
Hay centro único: el origen $O$ del sistema de coordenadas $x-y.$\\
$\Delta=\left|\begin{array}{ccc}A&B&0\\B&C&0\\0&0&0\end{array}\right|.$ Vamos a la tabla.\\
$\begin{cases}
1)\,\text{Si}\,\delta<0, \text{el lugar son dos rectas concurrentes en O}.\\
2)\,\text{Si}\,\delta<0, \text{el lugar es el punto O}.
\end{cases}$\\
\begin{enumerate}
\item[1)] Supongamos que $\delta<0$.\\
$M=\left(\begin{array}{cc}A&B\\B&C\end{array}\right); PCM(\lambda)=\lambda^2-\omega\lambda+\delta.$ Los valores propios de $M$ son reales y $\neq$s. Como $\lambda_1\lambda_2=\delta<0,$ $\lambda_1$ y $\lambda_2$ tienen signos contrarios.\\
Supongamos $\lambda_1>0$ y $\lambda_2<0.$ Al aplicar el T. Espectral, la ecuación de la cónica$/xy$ es: $\lambda_1x'^2+\lambda_2y'^2=0.$ que podemos escribir así: $$\lambda_1x'^2-|\lambda_2|y'^2=0$$ o $$\left(\sqrt{\lambda_1}x'-\sqrt{|\lambda_2|y'}\right)\left(\sqrt{\lambda_1}x'+\sqrt{|\lambda_2|y'}\right)$$
\begin{figure}[ht!]
\begin{center}
  \includegraphics[scale=0.5]{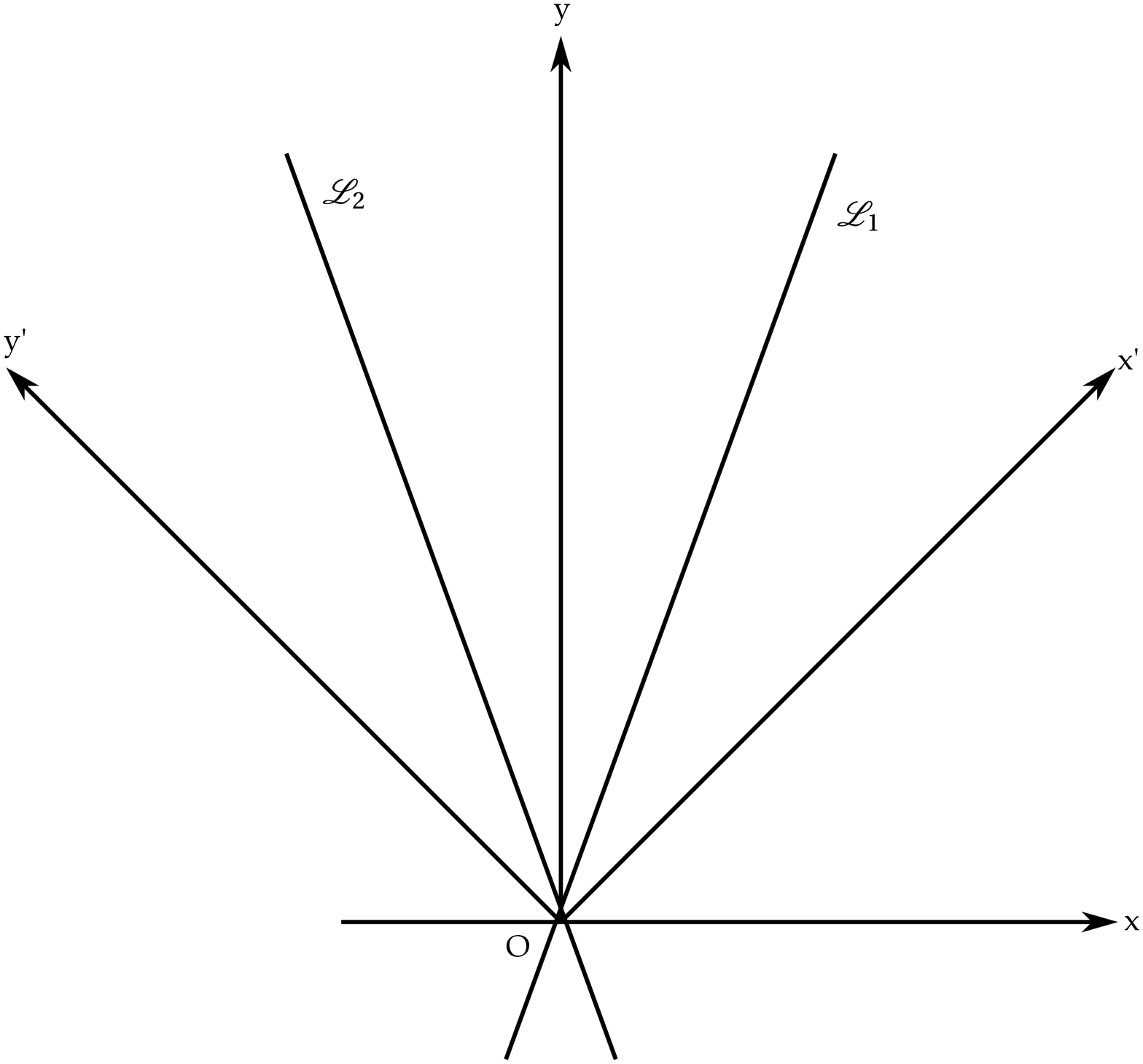}\\
\end{center}
\end{figure}
\newline
Las dos rectas son
\begin{align*}
\sqrt{\lambda_1}x'+\sqrt{|\lambda_2|}y'&=0\\
\intertext{y}\\
\sqrt{\lambda_1}x'-\sqrt{|\lambda_2|}y'&=0\hspace{0.5cm}(\text{que como se ve no son $\perp$s})
\end{align*}
o que podemos escribir
\begin{align*}
\left(\begin{array}{cc}\sqrt{\lambda_1}&\sqrt{|\lambda_2|}\end{array}\right)\dbinom{x'}{y'}=0\\
\intertext{y}\\
\left(\begin{array}{cc}\sqrt{\lambda_1}&-\sqrt{|\lambda_2|}\end{array}\right)\dbinom{x'}{y'}=0
\end{align*}
Pero $$\dbinom{x}{y}=P\dbinom{x'}{y'}\hspace{0.5cm}\therefore\hspace{0.5cm}\dbinom{x'}{y'}=P^t\dbinom{x}{y}.$$ Luego
\begin{align*}
\left(\begin{array}{cc}\sqrt{\lambda_1}&\sqrt{|\lambda_2|}\end{array}\right)P^t\dbinom{x}{y}=0\\
\intertext{y}\\
\left(\begin{array}{cc}\sqrt{\lambda_1}&-\sqrt{|\lambda_2|}\end{array}\right)P^t\dbinom{x}{y}=0
\end{align*}
son las ecuaciones de las dos rectas$/xy.$\\
Vamos a demostrar que $f(x,y)=Ax^2+2Bxy+Cy^2$ factoriza así:
\begin{equation}\label{46}
f(x,y)=Ax^2+2Bxy+Cy^2=\left(\begin{array}{cc}\sqrt{\lambda_1}&\sqrt{|\lambda_2|}\end{array}\right)P^t\dbinom{x}{y}\cdot\left(\begin{array}{cc}\sqrt{\lambda_1}&-\sqrt{|\lambda_2|}\end{array}\right)P^t\dbinom{x}{y}
\end{equation}
\begin{align}
Ax^2+2Bxy+Cy^2=\left(\begin{array}{cc}x&y\end{array}\right)\left(\begin{array}{cc}A&B\\B&C\end{array}\right)\dbinom{x}{y}&\underset{\uparrow}{=}\left(\begin{array}{cc}x&y\end{array}\right)P\left(\begin{array}{cc}\lambda_1&0\\ 0&\lambda_2\end{array}\right)P^t\dbinom{x}{y}\notag\\
&\begin{cases}\text{Por T. Espectral}, P^t\left(\begin{array}{cc}A&B\\B&C\end{array}\right)P=\left(\begin{array}{cc}\lambda_1&0\\ 0&\lambda_2\end{array}\right)\\
\therefore\hspace{0.5cm}\left(\begin{array}{cc}A&B\\B&C\end{array}\right)=P\left(\begin{array}{cc}\lambda_1&0\\ 0&\lambda_2\end{array}\right)P^t
\end{cases}\notag\\
&\underset{\nearrow}{=}\left(\begin{array}{cc}x&y\end{array}\right)P\left(\begin{array}{cc}\lambda_1&0\\ 0&-|\lambda_2|\end{array}\right)P^t\dbinom{x}{y}\label{47}\\
&\notag\begin{cases}\text{Pero}\,\,\lambda_2<0\\
\text{Luego}\,\,\lambda_2=-|\lambda_2|
\end{cases}
\end{align}
\begin{align}
&\underset{\parallel}{\underbrace{\left(\begin{array}{ccc}\sqrt{\lambda_1}&\sqrt{|\lambda_2|}\end{array}\right)P^t\dbinom{x}{y}}}\cdot\left(\begin{array}{ccc}\sqrt{\lambda_1}&-\sqrt{|\lambda_2|}\end{array}\right)P^t\dbinom{x}{y}\notag\\
&=\left(\begin{array}{ccc}x&y\end{array}\right)P\dbinom{\sqrt{\lambda_1}}{\sqrt{|\lambda_2|}}\cdot\left(\begin{array}{ccc}\sqrt{\lambda_1}&-\sqrt{|\lambda_2|}\end{array}\right)P^t\dbinom{x}{y}\notag\\
&=\left(\begin{array}{ccc}x&y\end{array}\right)P\left(\begin{array}{ccc}\lambda_1&-&\sqrt{\lambda_1|\lambda_2|}\\\sqrt{\lambda_1|\lambda_2|}&-&|\lambda_2|\end{array}\right)P^t\dbinom{x}{y}\notag\\
&\underset{\uparrow}{=}\left(\begin{array}{ccc}x&y\end{array}\right)P\left(\begin{array}{ccc}\lambda_1&0\\0&-&|\lambda_2|\end{array}\right)P^t\dbinom{x}{y}\label{48}\\
&\notag\begin{cases}
\left(\begin{array}{ccc}x&y\end{array}\right)\left(\begin{array}{ccc}A&\alpha B\\ \alpha B&C\end{array}\right)\dbinom{x}{y}=\left(\begin{array}{ccc}x&y\end{array}\right)\left(\begin{array}{ccc}A&\dfrac{\alpha+\beta}{2} B\\ \dfrac{\alpha+\beta}{2}B&C\end{array}\right)\dbinom{x}{y}
\end{cases}
\end{align}
De [\ref{47}] y [\ref{48}] se concluye [\ref{46}]
\end{enumerate}
\end{ejer}
\begin{pro}
Considremos la cónica $$2Bxy+Cy^2+2Dx+2Ey+F=0,\,\,B\neq 0,\,\,C\neq 0.$$
(Compárelo con la del problema anterior.)
\begin{figure}[ht!]
\begin{center}
  \includegraphics[scale=0.3]{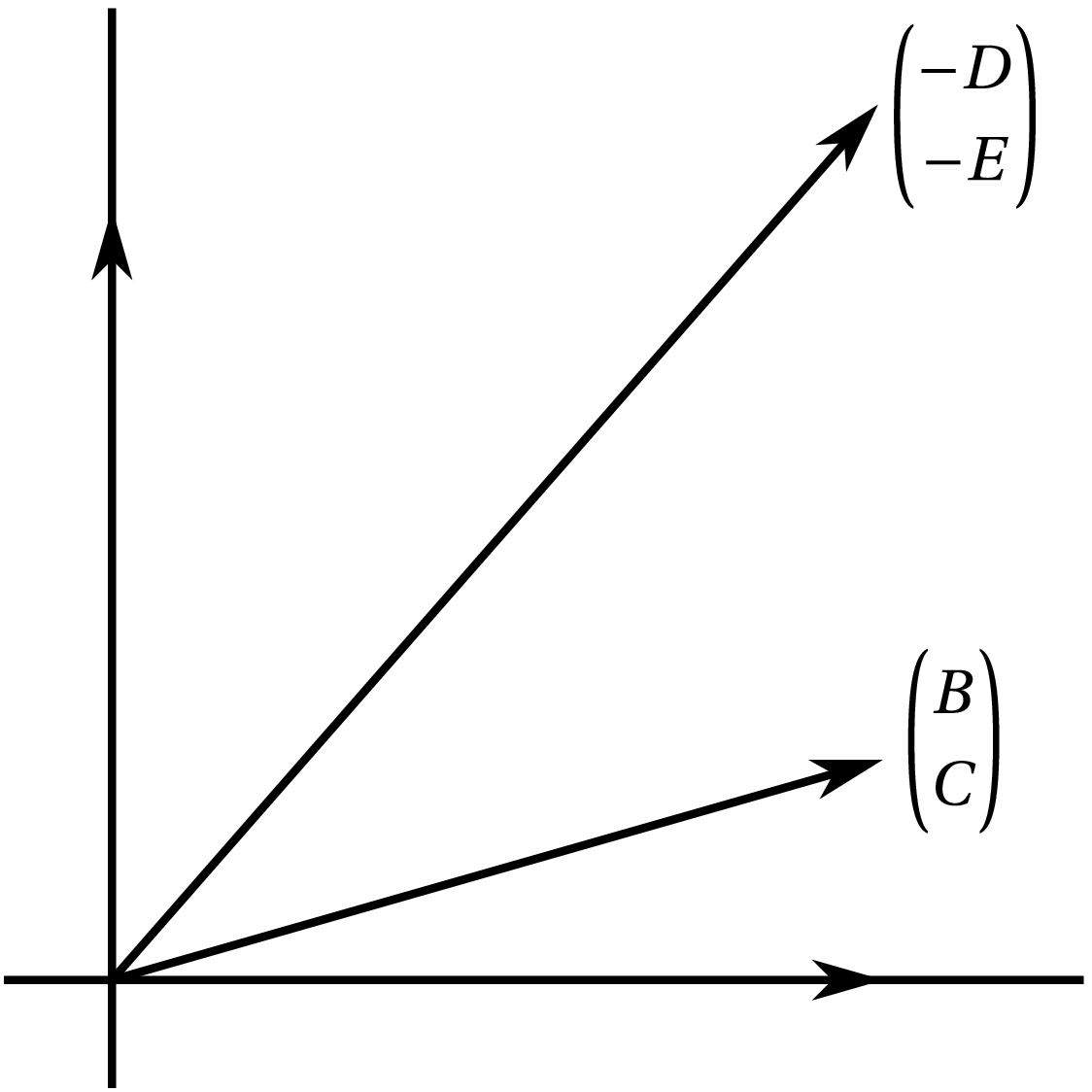}\\
\end{center}
\vspace{0.5cm}
\end{figure}
\begin{align*}
M&=\left(\begin{array}{cc}A&B\\B&C\end{array}\right)=\left(\begin{array}{cc}0&B\\B&C\end{array}\right)\\
\delta&=AC-B^2=-B^2\\
\omega&=A+C=C
\end{align*}
Como $\left\{\dbinom{A}{B}=\dbinom{0}{B}, \dbinom{B}{C}\right\}$ es L.I., \underline{la cónica tiene centro único}.\\
$\begin{cases}
Bk&=-D\\
Bh+Ck&=-E
\end{cases}$\\
De la $1^a$ ecuación, $k-=\dfrac{D}{B}$ que llevamos a la $2^a.$\\
\begin{align*}
Bh&=-Ck-E\\
\therefore\hspace{0.5cm}h&=-\dfrac{C}{D}\left(-\dfrac{D}{B}\right)-\dfrac{E}{B}=\dfrac{CD}{B^2}-\dfrac{E}{B}=\dfrac{CD-BE}{B^2}
\end{align*}
El centro $O'$ es el punto de coordenadas $$\left(\widetilde{h}, \widetilde{k}\right)=\left(\dfrac{CD-BE}{B^2},-\dfrac{D}{B}\right)$$
Una vez trasladado los ejes al punto $O'$, la ecuación de la cónica$/XY$ es $$2BXY+CY^2+f\left(\widetilde{h}, \widetilde{k}\right)=0.$$
ó $$\left(\begin{array}{cc}X&Y\end{array}\right)\left(\begin{array}{cc}0&B\\B&C\end{array}\right)\dbinom{X}{Y}+f\left(\widetilde{h}, \widetilde{k}\right)=0$$
\begin{align*}
M=\left(\begin{array}{cc}0&B\\B&C\end{array}\right);\hspace{0.5cm}PCM(\lambda)&=\lambda^2-\omega\lambda+\delta\\
&=\lambda^2-C\lambda-B^2=0\\
&\lambda=\dfrac{C\pm\sqrt{C^2+4B^2}}{2}\begin{cases}\lambda_1=\dfrac{C+\sqrt{4B^2+C^2}}{2}\\
\lambda_2=\dfrac{C-\sqrt{4B^2+C^2}}{2}\hspace{0.5cm}\therefore\hspace{0.5cm}\lambda_1\neq\lambda_2\end{cases}
\end{align*}
\begin{align*}
\lambda_1+\lambda_2&=C\\
\lambda_1\lambda_2&=-B^2\neq 0\hspace{0.5cm}\therefore\hspace{0.5cm}\lambda_1,\lambda_2\neq 0.
\end{align*}
Supongamos que $$EP_{\lambda_1}^M=\mathscr{N}\left(\lambda_1I_2-M\right)=Sg{p_1}$$ y que $$EP_{\lambda_2}^M=\mathscr{N}\left(\lambda_2I_2-M\right)=Sg{p_2}$$
donde ${p_1,p_2}:$ Base ortonormal.
\begin{figure}[ht!]
\begin{center}
  \includegraphics[scale=0.5]{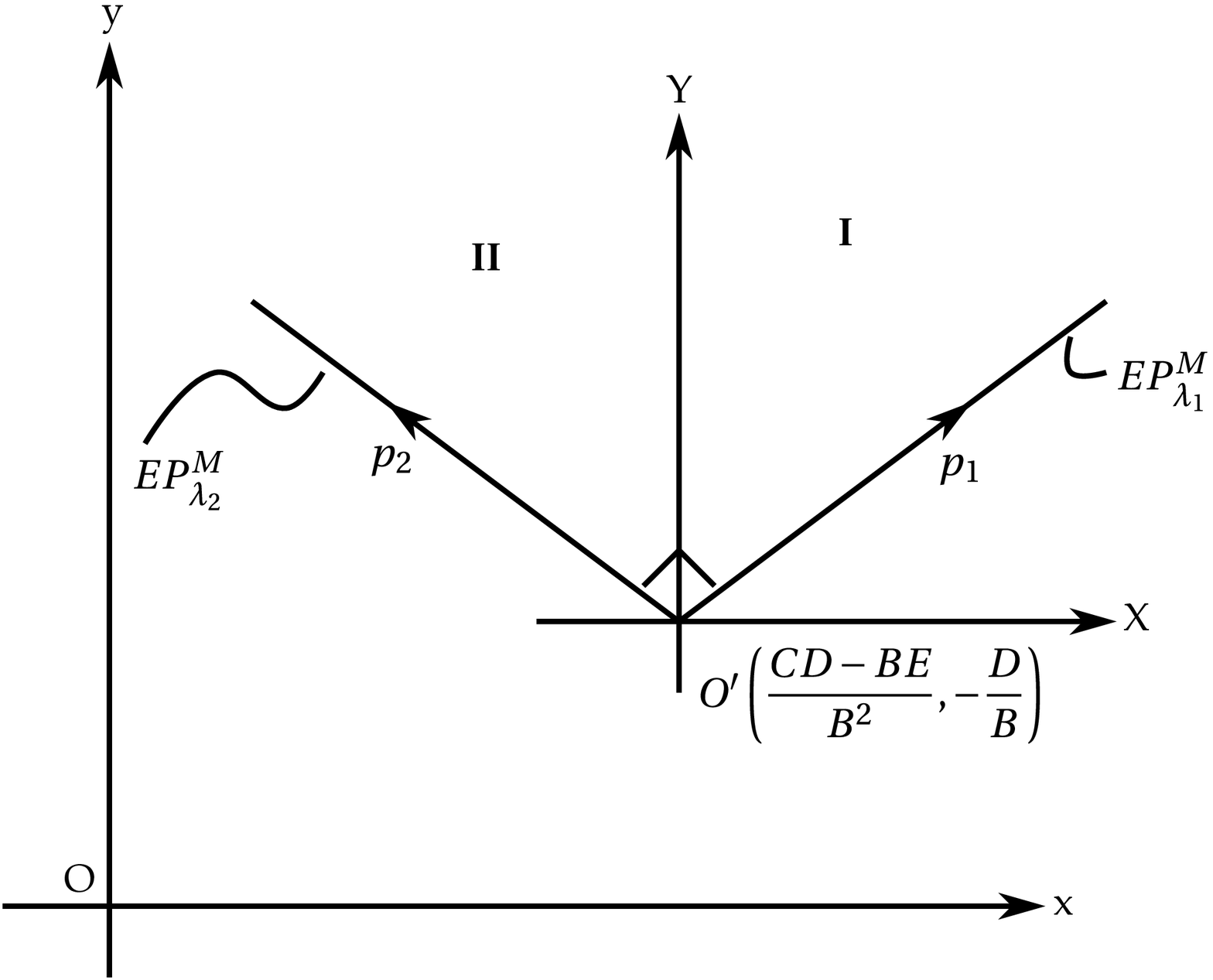}\\
\end{center}
\vspace{0.5cm}
\end{figure}
Una vez aplicado el T. Espectral, la ecuación de la cónica$/x'y'$ es: $$\left(\begin{array}{cc}x'&y'\end{array}\right)\left(\begin{array}{cc}\lambda_1&0\\0&\lambda_2\end{array}\right)\dbinom{x'}{y'}+f(h,k)=0$$
O sea $$\lambda_1x'^2+\lambda_2y'^2+f\left(\widetilde{h},\widetilde{k}\right)=0\hspace{0.5cm}\star\star.$$
Calculemos ahora el invariante $\Delta.$\\
$$\Delta=\left|\begin{array}{ccc}\lambda_1&0&0\\0&\lambda_2&0\\0&0&f\left(\widetilde{h},\widetilde{k}\right)\end{array}\right|=\lambda_1\lambda_2f\left(\widetilde{h},\widetilde{k}\right)=-B^2f\left(\widetilde{h},\widetilde{k}\right).$$
Se presentan dos casos:
\begin{enumerate}
\item[a)] Si $\Delta=0,\, f\left(\widetilde{h},\widetilde{k}\right)=0$ y al regresar a $\star\star$, la ecuación de
la cónica es: $$2BXY+CY^2=0.$$ O sea $$Y\left(2BX+CY\right)=0.$$
$$\therefore\hspace{0.5cm}Y=0\hspace{0.5cm}\text{ó}\hspace{0.5cm}Y=-\dfrac{2B}{C}X.$$ La cónica se compone de dos rectas: $$Y=0:\text{el eje X y la recta de la ecuación}\,\,Y=-\dfrac{2B}{C}X.$$
\item{b)} Si $\Delta\neq 0,\,\,f\left(\widetilde{h},\widetilde{k}\right)\neq 0$ y la ecuación $\star\star$ puede
ponerse así $$\dfrac{x'^2}{-\dfrac{f\left(\widetilde{h},\widetilde{k}\right)}{\lambda_1}}+\dfrac{y'^2}{-\dfrac{f\left(\widetilde{h},\widetilde{k}\right)}{\lambda_2}}=1$$
y el lugar, dependiendo de los signos de $f\left(\widetilde{h},\widetilde{k}\right), \lambda_1, \lambda_2$ puede representar:\\
$\begin{cases}
&\text{\underline{una elipse}}\\
&\text{ó \underline{una hipérbola}}\\
&\text{centradas en $O'$ y de ejes los ejes $x'$ y $y'$.}
\end{cases}$
\end{enumerate}
\end{pro}
\begin{pro}
El caso en que la ecuación es $Ax^2+2Bxy+2Dx+2Ey+F=0,\,\,A\neq0,\,\,B\neq0.$ se analiza de manera análoga.
\end{pro}
\begin{pro}
Consideremos la cónica $12xy-5y^2+48y-36=0.$\\
\begin{tabular}{cccc}
$A=0$&&&\\
$2B=12'\,\,B=6$&&$M=\left(\begin{array}{cc}0&6\\6&-5\end{array}\right);$&$\delta=-36$\\
$C=-5$&&&$\omega=-5$\\
$2D=0\,\,D=0$&&&\\
$2E=48\,\,E=24$&&&\\
$F=-36$&&&
\end{tabular}
\begin{figure}[ht!]
\begin{center}
  \includegraphics[scale=0.4]{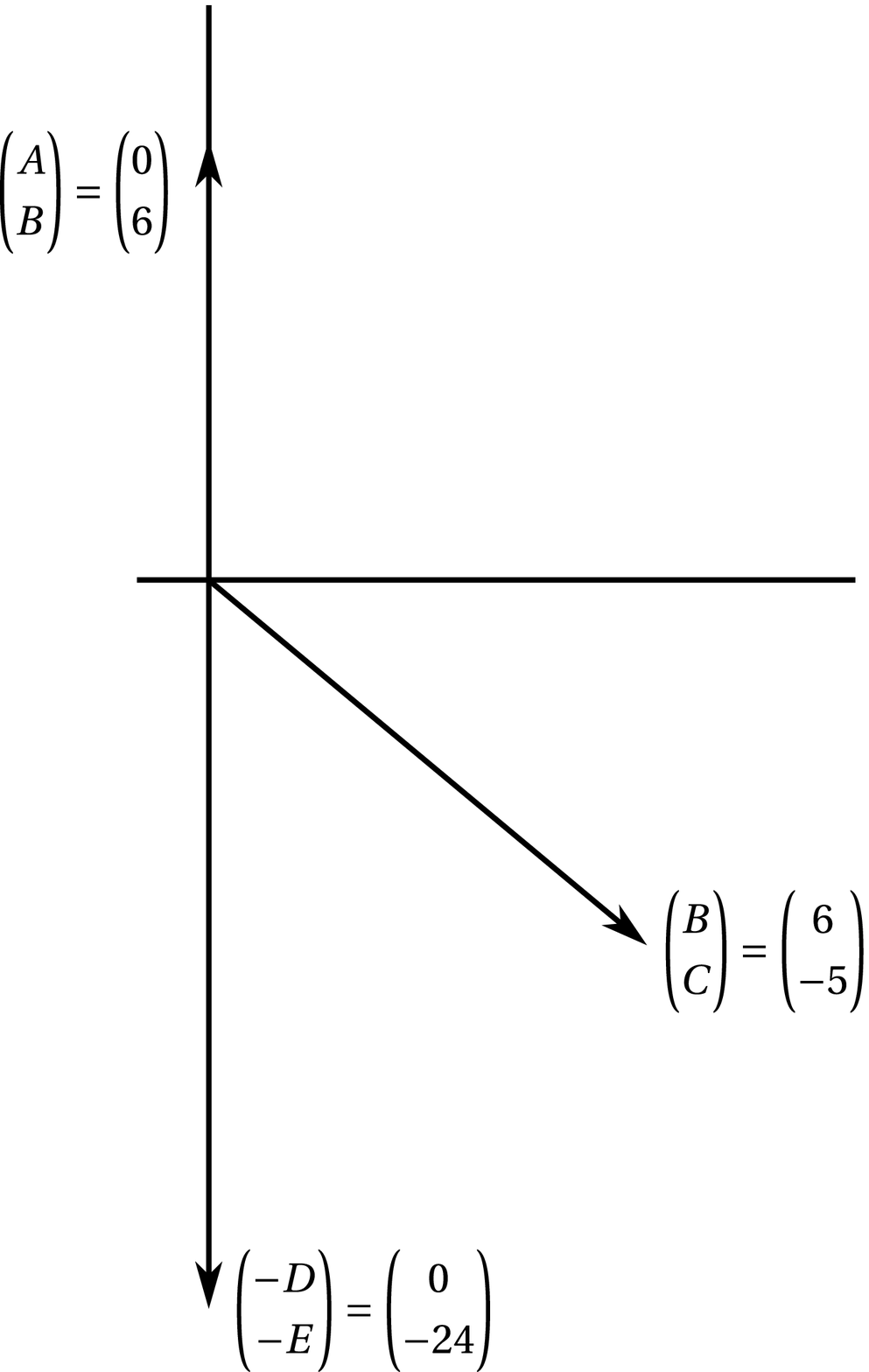}\\
\end{center}
\vspace{0.5cm}
\end{figure}
\underline{La cónica tiene centro único}.\\
$\begin{cases}
Bk=-D\\
Bh+Ck=-E\\
Bk=0\hspace{0.5cm}\therefore\hspace{0.5cm}k=0\\
h=-\dfrac{E}{B}=-\dfrac{24}{6}=-4
\end{cases}$\\
El punto $O'(-4,0)$ es el centro de la cónica.\\
Al trasladar los ejes $xy$ el punto $O'$, la ecuación de la cónica$/XY$ es $12XY-5Y^2+f(-4,0)=0$
\begin{figure}[ht!]
\begin{center}
  \includegraphics[scale=0.4]{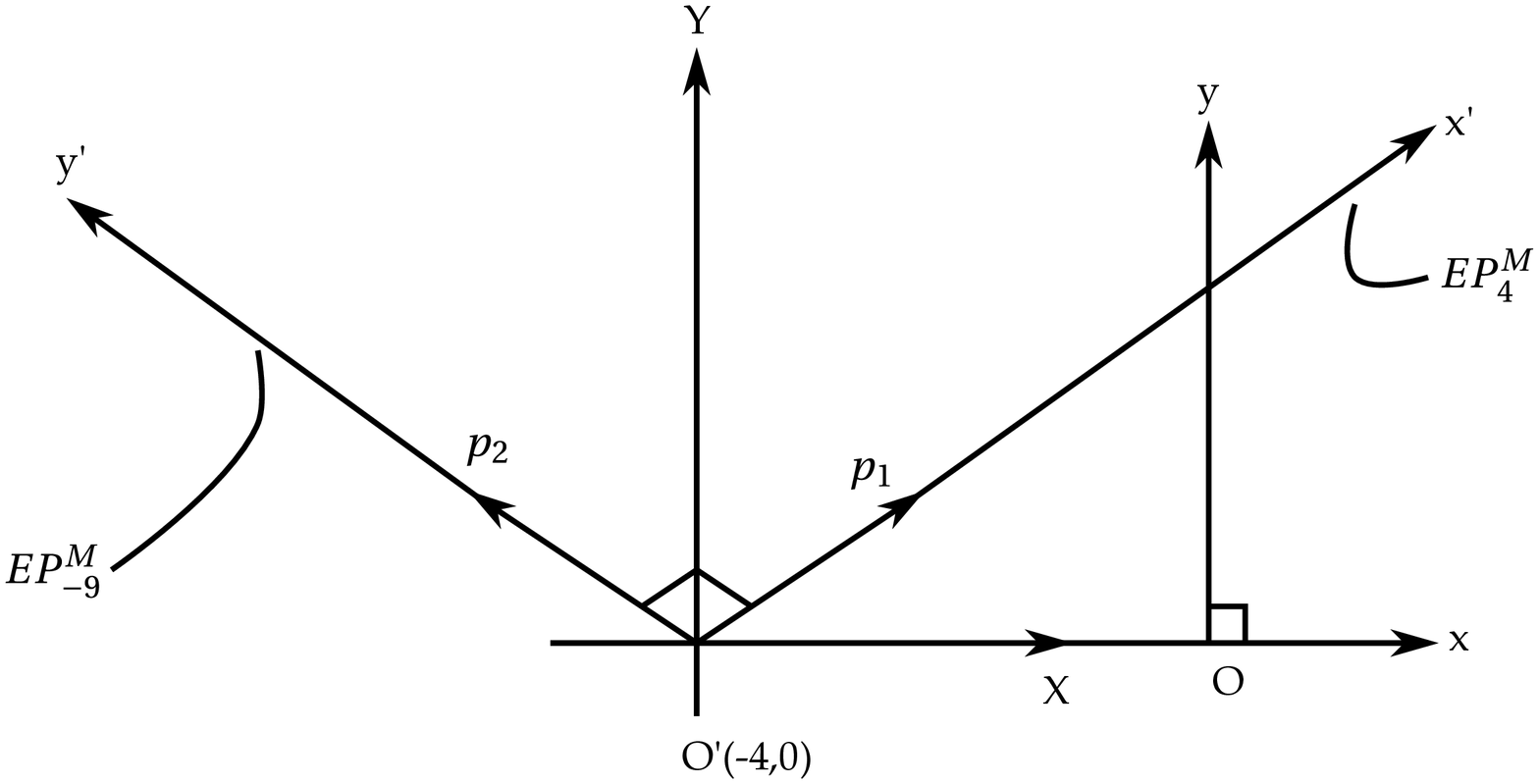}\\
\end{center}
\vspace{0.5cm}
\end{figure}
\begin{align*}
f(-4,0)&=12(-4)\cdot 0-5\cdot 0^2+48\cdot -36\\
&=-36
\end{align*}
Así que la ecuación de la cónica$/XY$ es $12XY-5Y^2-36=0.$\\
\begin{align*}
M&=\left(\begin{array}{cc}0&6\\6&-5\end{array}\right),\,\,PCM(\lambda)=\lambda^2-\omega\lambda+\delta=\lambda^2+5\lambda-36=0\\
\lambda&=\dfrac{-5\pm\sqrt{25+144}}{2}=\dfrac{-5\pm\sqrt{169}}{2}=\dfrac{-5\pm 13}{2}
\end{align*}
Los valores propios de $M$ son $4,-9.$
\begin{align*}
EP_4^M&\underset{\uparrow}{=}\mathscr{N}\left(4I_2-M\right)=\left\{\dbinom{u}{v}\diagup\begin{array}{ccc}4u&-&6v=0\\-6u&+&9v=0\end{array}\right\}\\
&\begin{cases}4I_2-M&=\left(\begin{array}{cc}4&0\\0&4\end{array}\right)-\left(\begin{array}{cc}0&6\\6&-5\end{array}\right)\\
&=\left(\begin{array}{cc}4&-6\\6-&9\end{array}\right)\end{cases}
\end{align*}
$$=\left\{\dbinom{u}{v}\diagup\begin{array}{ccc}2u&-&3v=0\\-2u&+&3v=0\end{array}\right\}=\left\{\dbinom{u}{v}\diagup 2u-3v=0\right\}$$
$v=\dfrac{2}{3}u.$ Todo vector de la forma $\left(u,\dfrac{2}{3}u\right)=u\left(1,\dfrac{2}{3}\right)=\alpha(3,2).$\\
con $\alpha\in\mathbb{R}$ está en el $EP_4^M.$
\begin{figure}[ht!]
\begin{center}
  \includegraphics[scale=0.3]{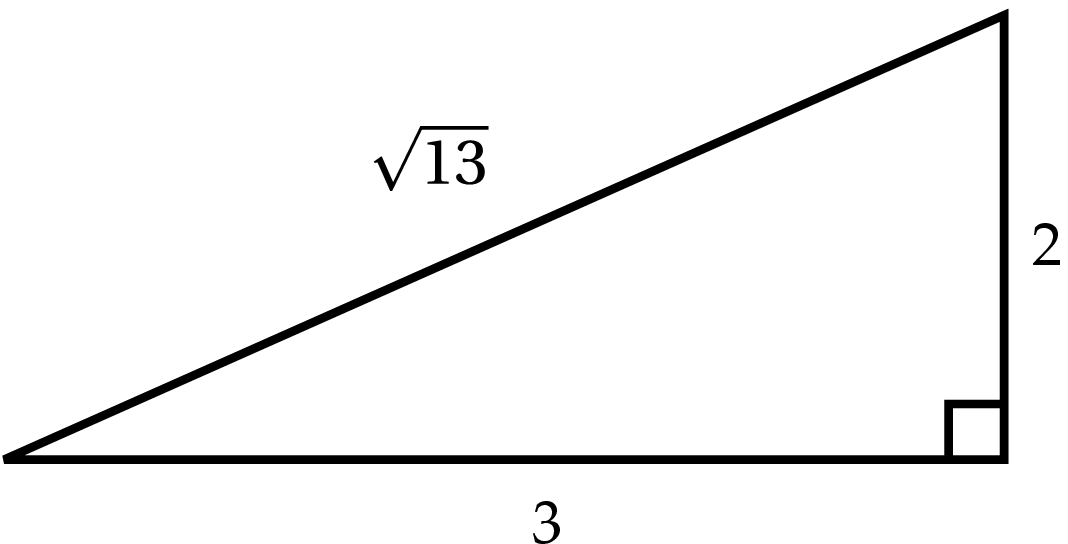}\\
\end{center}
\vspace{0.5cm}
\end{figure}
Luego $$EP_4^M=Sg\left\{\overset{\overset{p_1}{\parallel}}{\left(\begin{array}{cc}\dfrac{3}{\sqrt{13}}\\ \dfrac{2}{\sqrt{13}}\end{array}\right)}\right\};\hspace{0.5cm}EP_{-9}^M=Sg\left\{\overset{\overset{p_2}{\parallel}}{\left(\begin{array}{cc}\dfrac{-2}{\sqrt{13}}\\ \dfrac{3}{\sqrt{13}}\end{array}\right)}\right\}$$
${p_1,p_2}:$ Base ortonormal que señala las direcciones de $x'$ y $y'.$\\
Al aplicar el T. Espectral, la ecuación de la cónica$/x'y'$ es:
\begin{align*}
\left(\begin{array}{cc}x'&y'\end{array}\right)\left(\begin{array}{cc}4&0\\0&-9\end{array}\right)\dbinom{x'}{y'}-36=0\\
4x´^2-9y'^2=36;\hspace{0.5cm}\dfrac{x'^2}{3^2}-\dfrac{y'^2}{2^2}=1.\hspace{0.5cm}\text{\underline{Hipérbola}}
\end{align*}
\begin{figure}[ht!]
\begin{center}
  \includegraphics[scale=0.4]{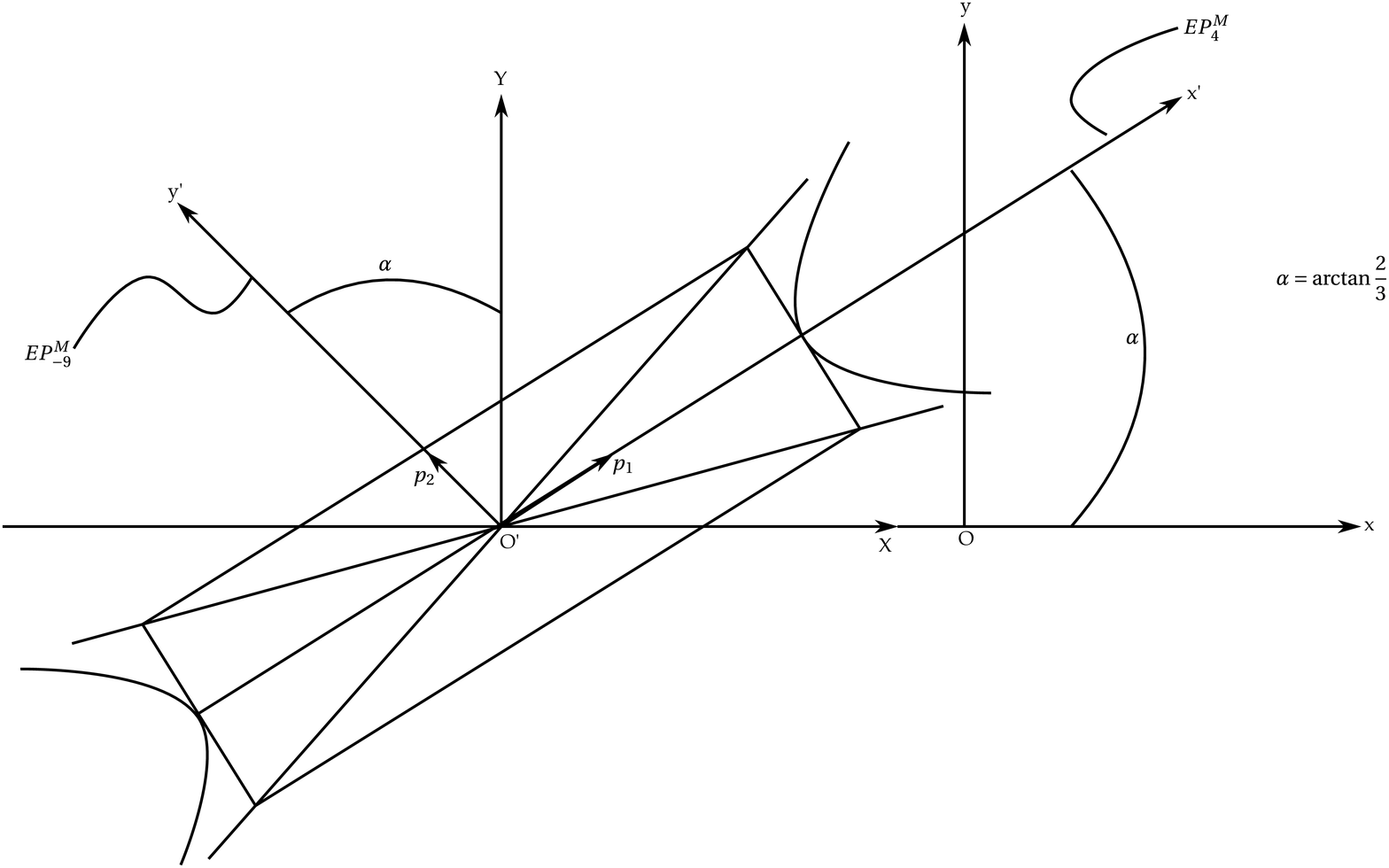}\\
\end{center}
\vspace{0.5cm}
\end{figure}
\end{pro}
\newpage
\begin{pro}
Como encontrar las asíntotas de una hipérbola a partir de su ecuación.\\
Consideremos la hipérbola de ecuación:
\begin{equation}\label{49}
h:Ax^2+2Bxy+Cy^2+2Dx+2Ey+F=0
\end{equation}
Sea $r:y=mx+b$ una asíntota de la curva. Vamos a determinar $m$ y $b$ a partir de $A, B,\cdots, F.$\\
\begin{figure}[ht!]
\begin{center}
  \includegraphics[scale=0.5]{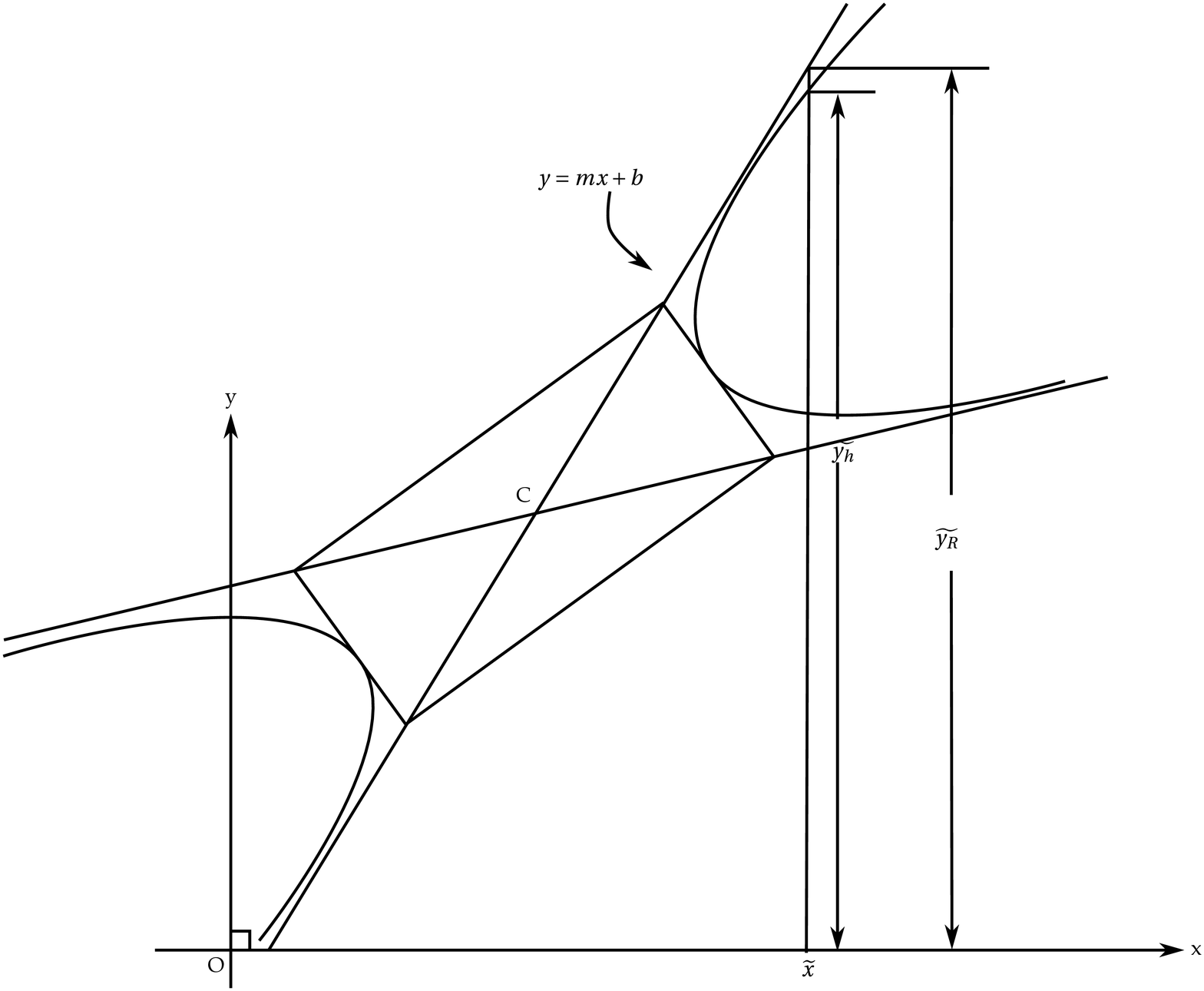}\\
\end{center}
\vspace{0.5cm}
\end{figure}
Si $\widetilde{x}\gg 0,\,\,\widetilde{y}_h\approx\widetilde{y}_R=m\widetilde{x}+b.$\\
Como $\left(x,\widetilde{y}_h\right)=\left(\widetilde{x},m\widetilde{x}+b\right)$ está en la curva, sus coordenadas satisfacen [\ref{49}].\\
O sea que
\begin{align}
A\widetilde{x}^2+2B\widetilde{x}\left(m\widetilde{x}+b\right)+C{\left(m\widetilde{x}+b\right)}^2+2D\widetilde{x}+2E\left(m\widetilde{x}+b\right)+F=0\notag\\
A\widetilde{x}^2+2Bm\widetilde{x}^2+2Bb\widetilde{x}+Cm^2\widetilde{x}^2+Cb^2+2Cmb\widetilde{x}+2D\widetilde{x}+2Em\widetilde{x}+2Eb+F=0\notag\\
\left(A+2Bm+Cm^2\right)\widetilde{x}^2+\left(2Bb+2Cmb+2D+2Em\right)\widetilde{x}+Cb^2+2Eb+F=0\label{50}
\end{align}
Dividiendo por $\widetilde{x}^2,$ $$\left(A+2Bm+Cm^2\right)+\dfrac{2Bb+2Cmb+2D+2Em}{\widetilde{x}^2}+\dfrac{Cb^2+2Eb+F}{\widetilde{x}^2}=0$$
Tomando $\lim\limits_{\widetilde{x}\rightarrow\infty},$
$$\lim\limits_{\widetilde{x}\rightarrow\infty}\left(A+2Bm+Cm^2\right)+\lim\limits_{\widetilde{x}\rightarrow\infty}\dfrac{2Bb+2Cm+2D+2Em}{\widetilde{x}}+\lim\limits_{\widetilde{x}\rightarrow\infty}\dfrac{Cb^2+2Eb+F}{\widetilde{x}^2}=0$$
O sea que $$A+2Bm+Cm^2=0\hspace{0.5cm}\star$$ y al regresar a [\ref{49}]:
$$\left(2Bb+2Cmb+2D+2Em\right)\widetilde{x}+Cb^2+2Eb+F=0.$$
Dividiendo por $\widetilde{x},$
$$\left(2Bb+2Cmb+2D+2Em\right)+\dfrac{Cb^2+2Eb+F}{\widetilde{x}}=0.$$
Tomando nuevamente $\lim\limits_{\widetilde{x}\rightarrow\infty}.$\\
$$\lim\limits_{\widetilde{x}\rightarrow\infty}\left(2Bb+2Cmb+2D+2Em\right)+\lim\limits_{\widetilde{x}\rightarrow\infty}=0.$$
Luego
\begin{align*}
Bb+Cmb+D+Em&=0\\
(B+Cm)b+D+Em&=0\\
\therefore\hspace{0.5cm}b=-\dfrac{D+Em}{B+Cm}\hspace{0.5cm}\star\star
\end{align*}
$\star$ y $\star\star$ resuelven el problema.\\
Al resolver $\star$ para $m$, hallamos las pendientes $m_1$ y $m_2$ de las asíntotas que al reemplazar en $\star\star$ permiten obterner los interseptos $b$ de cada una de ellas.\\
Tenidas las ecuaciones de las dos asíntotas, al interseptarlas es posible hallar las coordenadas del centro de la curva.\\
Regresemos a $\star.$\\
\begin{align*}
m\underset{\overset{\uparrow}{C}\neq 0}{=}\dfrac{-2B\pm\sqrt{4B^2-4AC}}{2C}=\dfrac{-B\pm\sqrt{B^2-AC}}{C}&\underset{\uparrow}{=}\dfrac{-B\pm\sqrt{-\delta}}{C}\\
&\begin{array}{cc}\text{Recuérdese que las}\\ \text{hipérbolas $\delta<0$}\end{array}
\end{align*}
$B+Cm
\begin{cases}
=B+\cancel{C}\dfrac{-B+\sqrt{-\delta}}{\cancel{C}}=\sqrt{-\delta}\hspace{0.5cm}\text{y}\hspace{0.5cm}b=-\dfrac{D+Em}{\sqrt{-\delta}}\\
=B+\cancel{C}\dfrac{-B+\sqrt{-\delta}}{\cancel{C}}=-\sqrt{-\delta}\hspace{0.5cm}\text{y}\hspace{0.5cm}b=\dfrac{D+Em}{\sqrt{-\delta}}\\
\end{cases}$\\
\end{pro}
\begin{ejem}
Consideremos la hipérbola $$x^2-3xy+2y^2-4x=0.$$
\begin{tabular}{ccc}
$A=1$&$F=0$\\
$2B=-3$&$B=-\dfrac{3}{2}$\\
$C=2$&&\\
$2D=-4$&$D=-2$
\end{tabular}
La ecuación $\star$ es
\begin{align*}
1-3m+2m^2=0\\
\therefore\hspace{0.5cm}m=\dfrac{3\pm\sqrt{9-8}}{4}=\dfrac{3\pm 1}{4}\begin{cases}m_1&=1\\ m_2&=\dfrac{1}{2}\end{cases}
\end{align*}
Pero $$m_1=1,\,\, b_1=-\dfrac{-2}{-\dfrac{3}{2}+2\cdot 1}=\dfrac{2}{\dfrac{1}{2}}=4$$
Pero $$m_2=\dfrac{1}{2},\,\, b_2=-\dfrac{-2}{-\dfrac{3}{2}+2\dfrac{1}{2}}=\dfrac{2}{-\dfrac{1}{2}}=-4$$
Las asíntotas son:\\
$\begin{cases}
y&=x+4\\
y&=\dfrac{1}{2}x-4
\end{cases}$\\
Al hacerlas simultaneas obtenemos las coordenadas$/xy$ del centro $C$ de la curva:
\begin{figure}[ht!]
\begin{center}
  \includegraphics[scale=0.3]{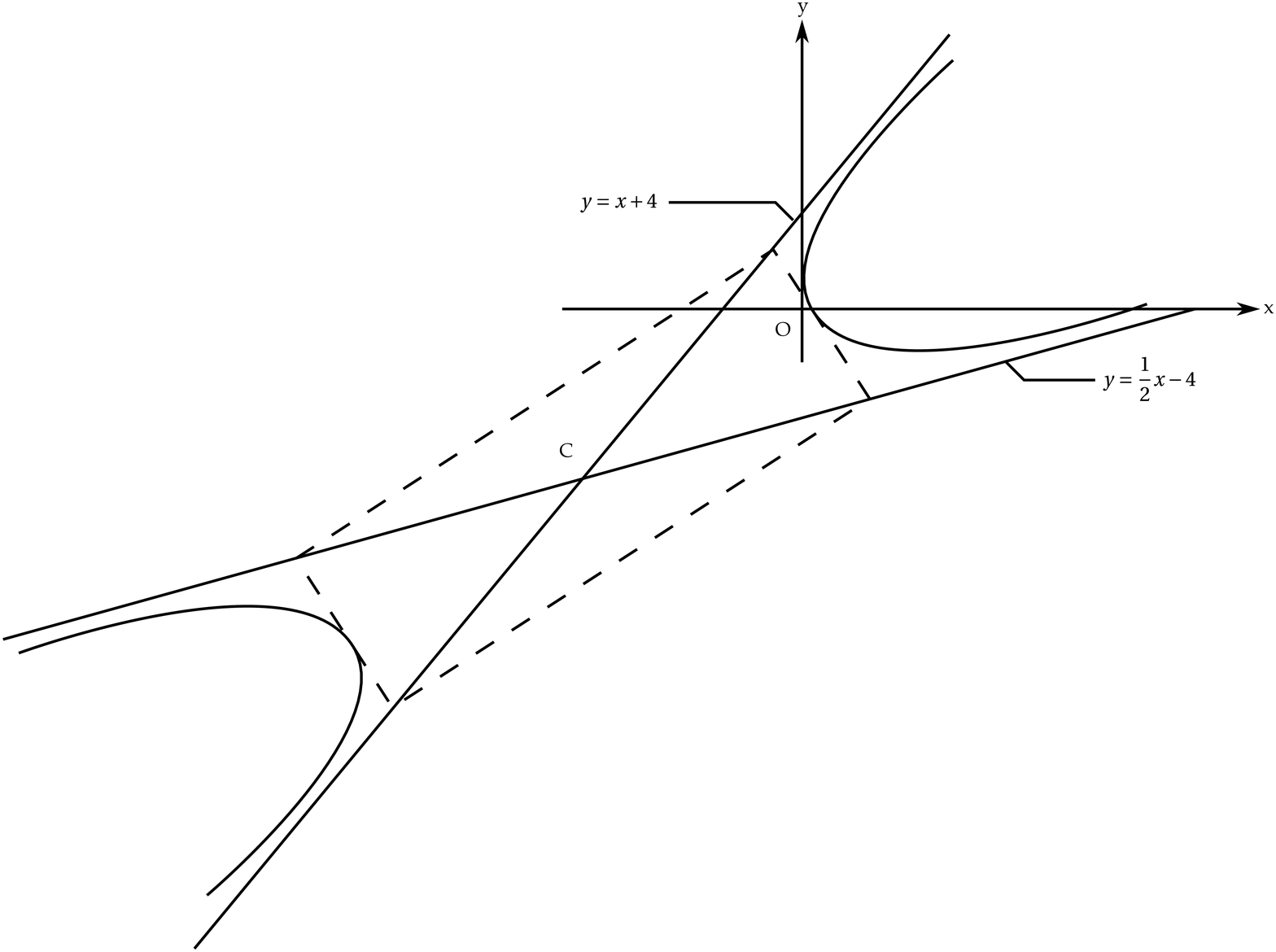}\\
\end{center}
\end{figure}
\newline
Las dos ecuaciones anteriores conducen al sistema
$\begin{cases}
x-y+4&=0\\
x-2y-8&=0
\end{cases}$\\
Cuya solución es $(-16,-12).$\\
Otra forma de conseguir el centro es resolviendo el sistema
$\begin{cases}
Ax+By&=-D\\
Bx+Cy&=-E
\end{cases}$\\
que en este caso es\\
$\begin{cases}
x-\dfrac{3}{2}y=2\\
-\dfrac{3}{2}x+2y=0\hspace{0.5cm};\hspace{0.5cm}\begin{cases}2x-3y=4\\ -3x+4y=0\hspace{0.5cm}\text{cuya solución es $(-16,-12).$}\end{cases}
\end{cases}$\\
\end{ejem}
\section{Cónicas sin centro}
Consideremos la cónica
\begin{equation}\label{51}
Ax^2+2Bxy+Cy^2+2Dx+2Ey+F=0
\end{equation}
Supongamos que los coeficientes $A,B,C\neq 0.$ Los casos en que uno de ellos (o dos) es cero ya fueron discutidos.\\
Recuérdese que en esos casos se tenía:
\begin{figure}[ht!]
\begin{center}
  \includegraphics[scale=0.5]{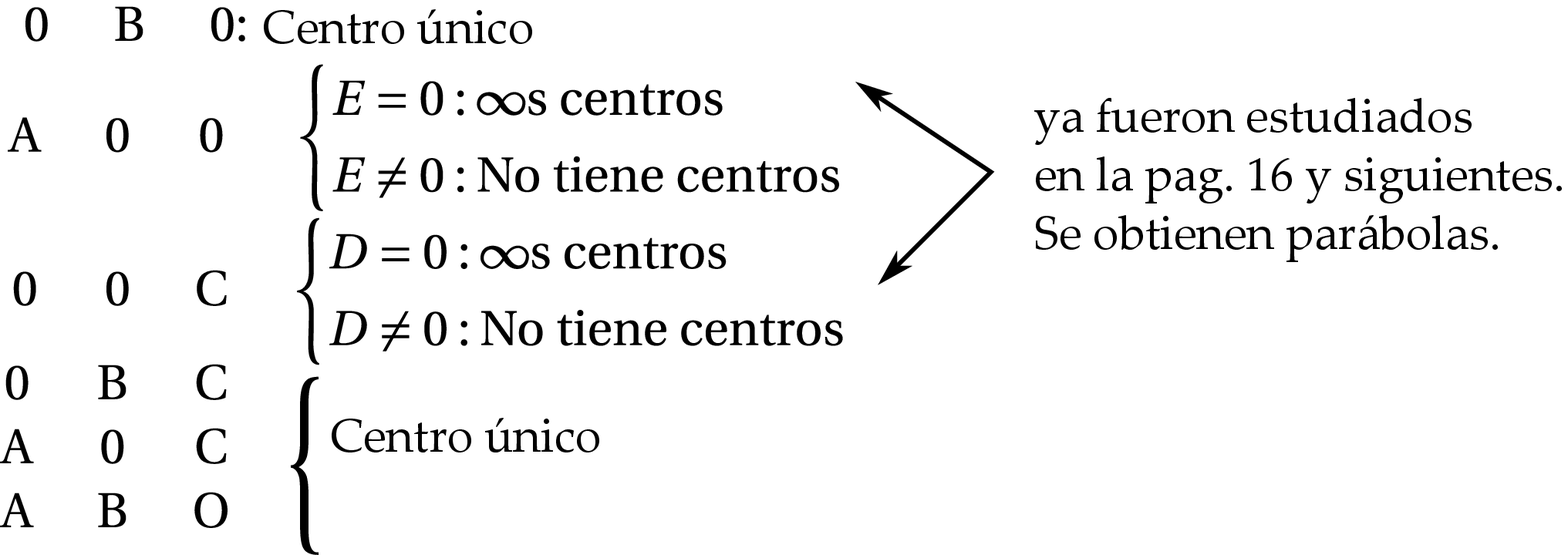}\\
\end{center}
\vspace{0.5cm}
\end{figure}
Esperamos de antemano que los lugares que se obtengan sean parábolas. No pueden ser ni elipses ni hipérbolas ya que éstas son cónicas con centro. No pueden ser rectas paralelas ó una recta ya que en estos casos hay $\infty$s centros.\\
Tampoco la cónica puede reducirse a dos rectas que se cortan ó a un punto ya que en otros casos hay centro único.\\
La situación es ésta:$$\left\{\dbinom{A}{B}, \dbinom{B}{C}\right\}\hspace{0.5cm}\text{es L.D. y}\hspace{0.5cm}\dbinom{-D}{-E}\notin Sg\left\{\dbinom{A}{B}\right\}=Sg\left\{\dbinom{B}{C}\right\}.$$\\
Además, $$\left\{\dbinom{A}{B}, \dbinom{-D}{-E}\right\}\hspace{0.5cm}\text{es L.I.}.$$ y por lo tanto área del paralelogramo de $\dbinom{A}{B}$ y $\dbinom{-D}{-E}$ $$=\left|\begin{array}{cc}A&-D\\B&-C\end{array}\right|=BD-AE\neq 0.$$ Análogamente, $$\left\{\dbinom{B}{C},\dbinom{-D}{-E}\right\}\hspace{0.5cm}\text{es L.I}.$$ y área del paralelogramo de $\dbinom{B}{C}$ y $\dbinom{-D}{-E}$ $$=\left|\begin{array}{cc}B&-D\\C&-E\end{array}\right|=CD-BE\neq 0.$$
$$M=\left(\begin{array}{cc}A&B\\B&C\end{array}\right); \delta=\left|\begin{array}{cc}A&B\\B&C\end{array}\right|=AC-B^2=\text{área del paralelogramo de $\dbinom{A}{B}\, y\,\dbinom{B}{C}$}=0.$$
\begin{figure}[ht!]
\begin{center}
  \includegraphics[scale=0.4]{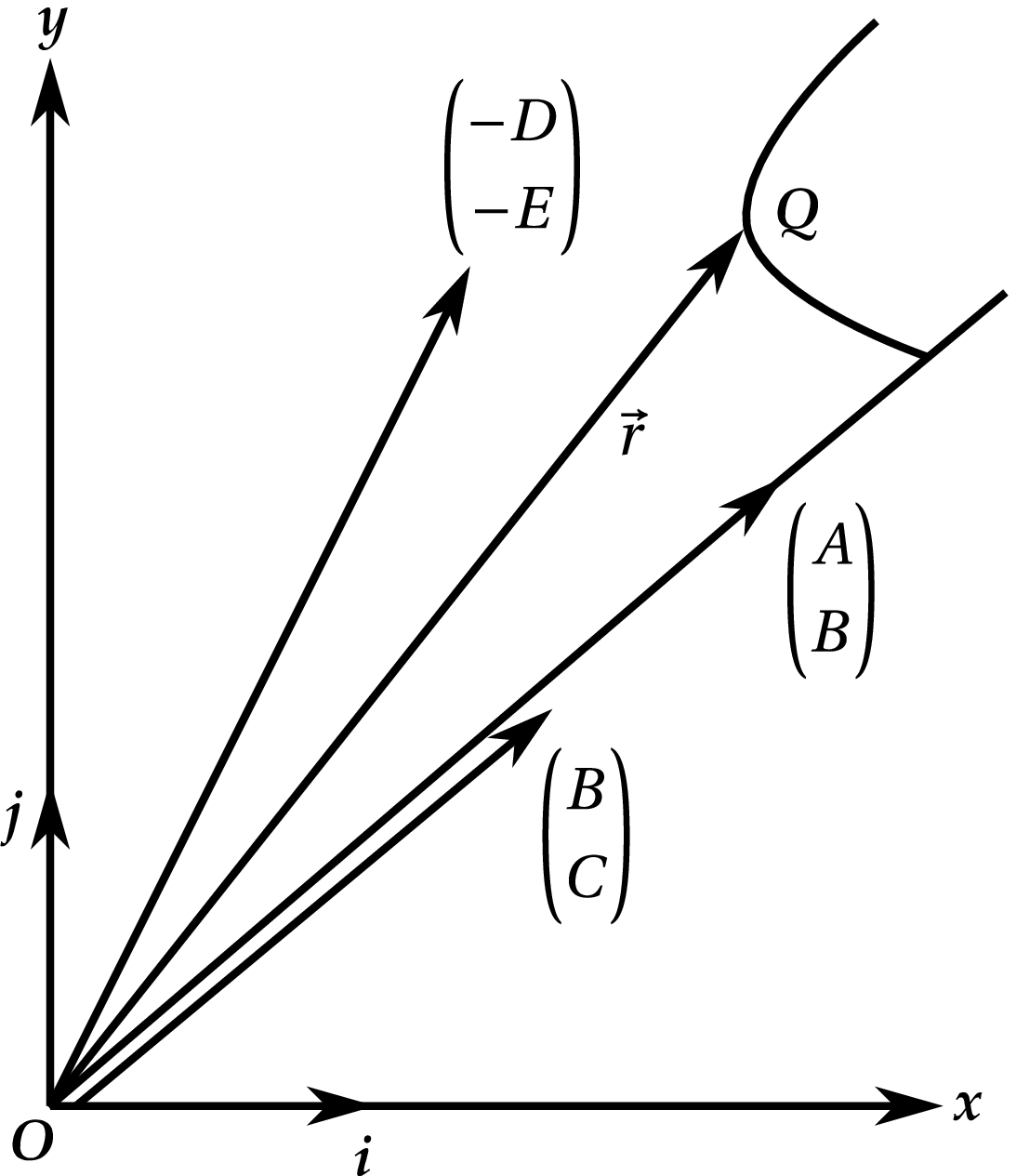}\\
\end{center}
\vspace{0.5cm}
\end{figure}
El sistema
\begin{align*}
Ax+By&=-D\\
Bx+Cy&=-E
\end{align*}
No tiene solución y por lo tanto la cónica [\ref{51}] no tiene centro.\\
Ahora, $\omega=A+C\neq 0$ porque si $\omega=0, A+C=0$ y $C=-A.$\\
Luego $0=\delta=AC-B^2=-A^2-B^2=-\left(A^2+B^2\right)$ y se tendría que $A^2+B^2=0(\rightarrow\leftarrow)$ ya que hemos asumido que ninguno de los $A,B,C$ es cero.\\
Como la cónica no tiene centro, no es posible realizar una traslación de los ejes $xy$ que elimine los términos lineales en [\ref{51}].\\
Entonces la primera transformación que hacemos es una transformación ortogonal que elimine el término mixto en [\ref{51}].\\ Sea $Q(x,y)$ un punto de la cónica.\\
Entonces sus coordenadas satisfacen [\ref{51}]. O sea que $Ax^2+2Bxy+Cy^2+2Dx+2Ey+F=0$ que podemos escribir así:
\begin{equation}\label{52}
\left(\begin{array}{cc}x&y\end{array}\right)M\dbinom{x}{y}+2\left(\begin{array}{cc}2D&2E\end{array}\right)\dbinom{x}{y}+F=0
\end{equation}
donde $$M=\left(\begin{array}{cc}A&B\\B&C\end{array}\right).$$
Llamemos $\vec{r}$ al vector de posición de $Q/O.$\\
Entonces $[\vec{r}]_{ij}=\dbinom{x}{y}$ y al regresar a [\ref{52}]:
\begin{equation}\label{53}
\left[\vec{r}\right]_{ij}^tM\left[\vec{r}\right]_{ij}+\left(\begin{array}{cc}2D&2E\end{array}\right)\left[\vec{r}\right]_{ij}+F=0
\end{equation}
Ahora, como la matriz $M=\left(\begin{array}{cc}A&B\\B&C\end{array}\right)$ es simétrica, por el Teorema Espectral,  $$\exists P=\left(\begin{array}{cc}\underset{\downarrow}{\overset{\uparrow}{p_1}}&\underset{\downarrow}{\overset{\uparrow}{p_2}}\end{array}\right)=\left(\begin{array}{cc}p_1^1&p_2^1\\p_1^2&p_2^2\end{array}\right):\text{ortogonal},$$
i.e., $$P^{-1}=P^t$$ y $$\langle p_i,p_j\rangle=\delta_{ij}$$ donde $\langle\hspace{0.5cm}\rangle:$ P.I.Usual en $\mathbb{R}^2,$ tal que
\begin{equation}\label{54}
P^tMP=\left(\begin{array}{cc}\lambda_1&0\\0&\lambda_2\end{array}\right)
\end{equation}
donde $\lambda_1$ y $\lambda_2$ son los valores propios de $M.$\\
O sea que $p_1\in EP_{\lambda_1}^M$ y $p_2\in EP_{\lambda_2}^M,$ lo que significa que
\begin{align*}
Mp_1&=\lambda_1 p_1\\
Mp_2&=\lambda_2p_2
\end{align*}
$\lambda_1$ y $\lambda_2$ son las raices del polinomio característico de $M:$ $$PCM(\lambda)=\lambda^2-\omega\lambda+\delta=0.$$ Pero $\delta=0.$\\
Luego $\lambda^2-\omega\lambda=0$ lo que dice que $\omega$ y $0$ son los valores propios de $M.$\\
Ahora, $$EP_0^M=\mathscr{N}\left(0I_2-M\right)=\underset{\begin{tabular}{cc}$\text{Espacio}$\\$\text{nulo de M}$\end{tabular}}{\underbrace{\mathscr{N}}}(M)\left\{\dbinom{u}{v}\diagup\begin{tabular}{ccc}$Au$&$+$&$Bv=0$\\$Bu$&$+$&$Cv=0$\end{tabular}\right\}$$
Consideremos el sistema
\begin{align*}
Au+Bv&=0\\
Bu+Cv&=0.
\end{align*}
Como $\left\{\dbinom{A}{B},\dbinom{B}{C}\right\}$ es L.D., $\alpha\dbinom{A}{B}.$ O sea que $B=\alpha A$ y $C=\alpha B.$\\
Ahora, toda solución de la $1^a$ ecuación es de la forma $\left(u,-\dfrac{A}{B}u\right).$\\
Como $$Bu+C\left(-\dfrac{A}{B}u\right)=\alpha Au-\dfrac{A\alpha\cancel{B}}{\cancel{B}}u=0,$$ las soluciones de la $2^a$ ecuación son las mismas de la $1^a.$ Esto explica porque la $2^a$ es redundante y por lo tanto $\mathscr{N}(M)=\left\{\dbinom{u}{v}Au+Bv=0\right\}.$\\
Todo vector de la forma $\left(u, -\dfrac{A}{B}u\right)=u\left(1, -\dfrac{A}{B}\right)=\alpha\left(B,-A\right)$ con $\alpha\in\mathbb{R}$ está en el $\mathscr{N}(M).$\\
Así que
\begin{align*}
EP_0^M=Sg\left\{\dbinom{B}{-A}\right\}&\underset{\uparrow}{=}Sg\left\{\left(\begin{array}{cc}\dfrac{B}{\sqrt{A^2+B^2}}\\-\dfrac{A}{\sqrt{A^2+B^2}}\end{array}\right)\right\}\\
&\begin{cases}
\text{Se normaliza el}\\
\text{vector $\dbinom{B}{-A}$}
\end{cases}
\end{align*}
Puede ocurrir:
\begin{enumerate}
\item[(1)] que
\begin{align*}
B>0\\
A>0
\end{align*}
Lo primero que hacemos es dibujar el vector $\left(\begin{array}{cc}\dfrac{B}{\sqrt{A^2+B^2}}\\-\dfrac{A}{\sqrt{A^2+B^2}}\end{array}\right)$ ya que él señala el $EP_0^M.$\\
El vector $\left(\begin{array}{cc}-\dfrac{B}{\sqrt{A^2+B^2}}\\\dfrac{A}{\sqrt{A^2+B^2}}\end{array}\right)$ está en el $II$ cuadrante y el vector
\begin{figure}[ht!]
\begin{center}
  \includegraphics[scale=0.6]{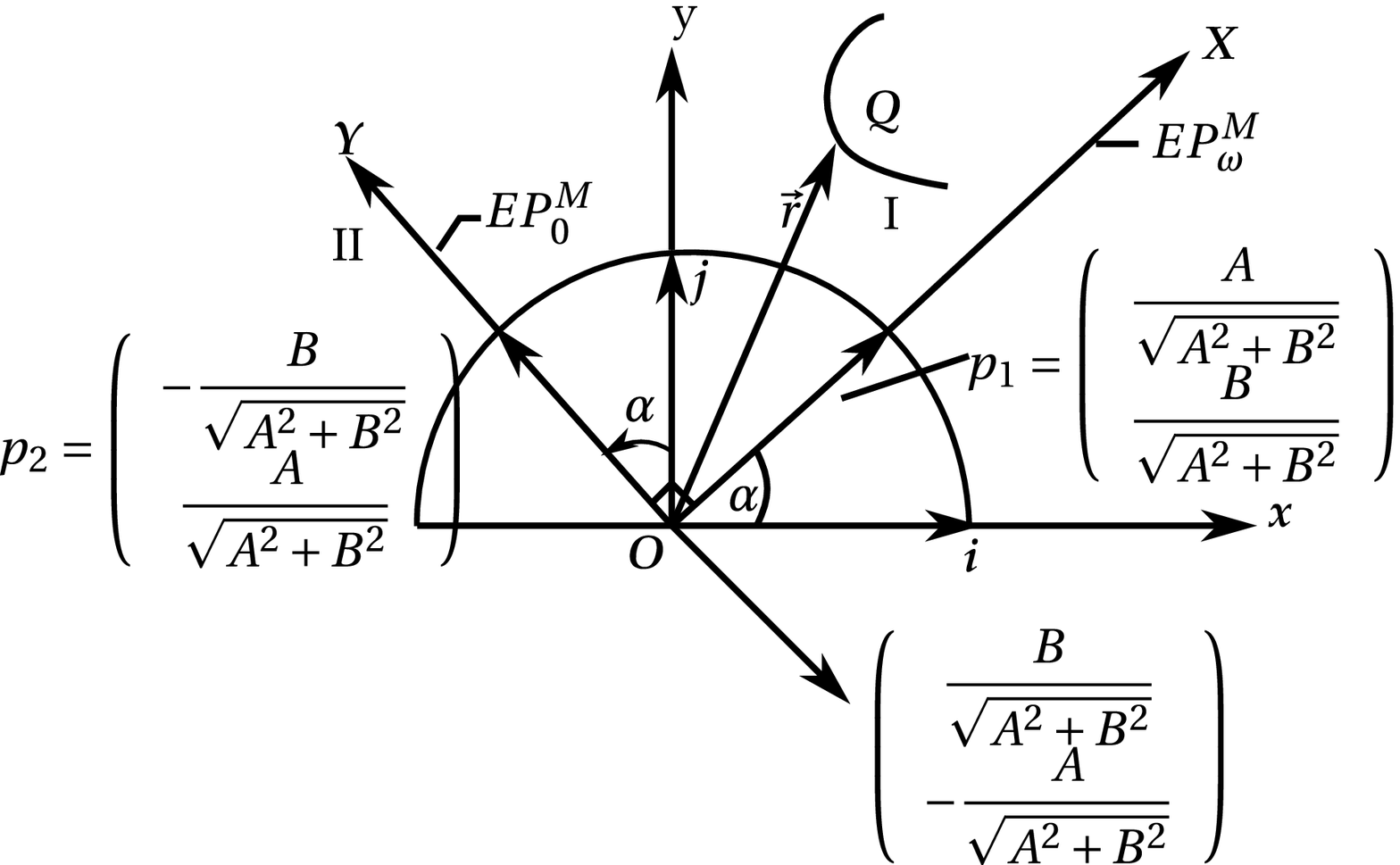}\\
\end{center}
\vspace{0.5cm}
\end{figure}
$\left(\begin{array}{cc}\dfrac{A}{\sqrt{A^2+B^2}}\\\dfrac{B}{\sqrt{A^2+B^2}}\end{array}\right)$ está en el I cuadrante. El espectro de $M$ lo ordenamos así: $$\lambda(M)={\omega,0}$$
$$EP_{\omega\ni p_1}^M=\dbinom{p_1^1}{p_1^2}=\left(\begin{array}{cc}\dfrac{A}{\sqrt{A^2+B^2}}\\\dfrac{B}{\sqrt{A^2+B^2}}\end{array}\right);\hspace{0.5cm}EP_{0\ni p_2}^M=\dbinom{p_2^1}{p_2^2}=\left(\begin{array}{cc}-\dfrac{B}{\sqrt{A^2+B^2}}\\\dfrac{A}{\sqrt{A^2+B^2}}\end{array}\right)$$
$\left\{\vec{p}_1,\vec{p}_2\right\}$ es una base ortogonal.\\
$\alpha=\arctan\dfrac{B}{A}$ es el $\sphericalangle$ que giran los ejes $xy.$\\
$$P=\left(\begin{array}{cc}\underset{\downarrow}{\overset{\uparrow}{[p_1]_{ij}}}&\underset{\downarrow}{\overset{\uparrow}{[p_2]_{ij}}}\end{array}\right)=[I]_{ij}^{p_1,p_2}$$
Los vectores $\vec{p}_1$ y $p_2$ definen los ejes $XY$ con origen en $O$.\\
Si llamamos $[\vec{r}]_{p_1,p_2}=\dbinom{X}{Y}$ se tendrá que $$[\vec{r}]_{ij}=[I]_{ij}^{p_1,p_2}[\vec{r}]_{p_1p_2}=P[\vec{r}]_{p_1p_2}$$ y al regersar a [\ref{54}]:
$${\left(P[\vec{r}]_{p_1p_2}\right)}^tMP[\vec{r}]_{p_1p_2}+\left(\begin{array}{cc}2D&2E\end{array}\right)P[\vec{r}]_{p_1p_2}+F=0$$
O sea que
\begin{equation}\label{55}
[\vec{r}]_{p_1p_2}^t\left(P^tMP\right)[\vec{r}]_{p_1p_2}+\left(\left(\begin{array}{cc}2D&2E\end{array}\right)P\right)[\vec{r}]_{p_1p_2}+F=0.
\end{equation}
Pero
\begin{align*}
P^tMP&\underset{\uparrow}{=}\left(\begin{array}{cc}\omega&0\\0&0\end{array}\right)\\
&\text{Tma. Espectral}
\end{align*}
Llamemos
\begin{align*}
\left(\begin{array}{cc}2D'&2E'\end{array}\right)&=\left(\begin{array}{cc}2D&2E\end{array}\right)P\\
&=\left(\begin{array}{cc}2D&2E\end{array}\right)\left(\begin{array}{cc}\dfrac{A}{\sqrt{A^2+B^2}}&-\dfrac{B}{\sqrt{A^2+B^2}}\\ \dfrac{B}{\sqrt{A^2+B^2}}&\dfrac{A}{\sqrt{A^2+B^2}}\end{array}\right)\\
&=\left(\begin{array}{cc}\dfrac{2AD+2BE}{\sqrt{A^2+B^2}}&\dfrac{-2BD+2AE}{\sqrt{A^2+B^2}}\end{array}\right)
\end{align*}
Osea que
\begin{equation}\label{56}
D'=\dfrac{AD+BE}{\sqrt{A^2+B^2}},\hspace{0.5cm}E'=\dfrac{AE-BD}{\sqrt{A^2+B^2}}\hspace{0.5cm}\text{con $E'\neq 0$ ya que $BD-AE\neq 0.$}
\end{equation}
Volviendo a [\ref{55}] se tiene:
$$\left(\begin{array}{cc}X&Y\end{array}\right)\left(\begin{array}{cc}\omega&0\\0&0\end{array}\right)\dbinom{X}{Y}+2D'X+2E'Y+F=0.$$
O sea que
\begin{equation}\label{57}
\omega X^2+2D'X+2E'Y+F=0
\end{equation}
con $D'$ y $E', E'\neq 0,$ dados por [\ref{56}], es la ecuación de la cónica$\diagup XY.$\\
Ahora, $$\Delta=\left|\begin{array}{ccc}\omega&0&D'\\ 0&0&E'\\ D'&E'&F\end{array}\right|=-\omega E'^2$$
y como $\omega\neq 0$ y $E'\neq 0, \Delta\neq 0.$\\
En [\ref{57}], $E\neq 0.$\\
Vamos a demostrar que es posible definir una traslación de los ejes XY a un punto de coordenadas $(a,b)\diagup XY,$ \underline{a y b a determinar}, de modo que en [\ref{57}] se anulan el término en $X$ y el término independiente:
\begin{figure}[ht!]
\begin{center}
  \includegraphics[scale=0.4]{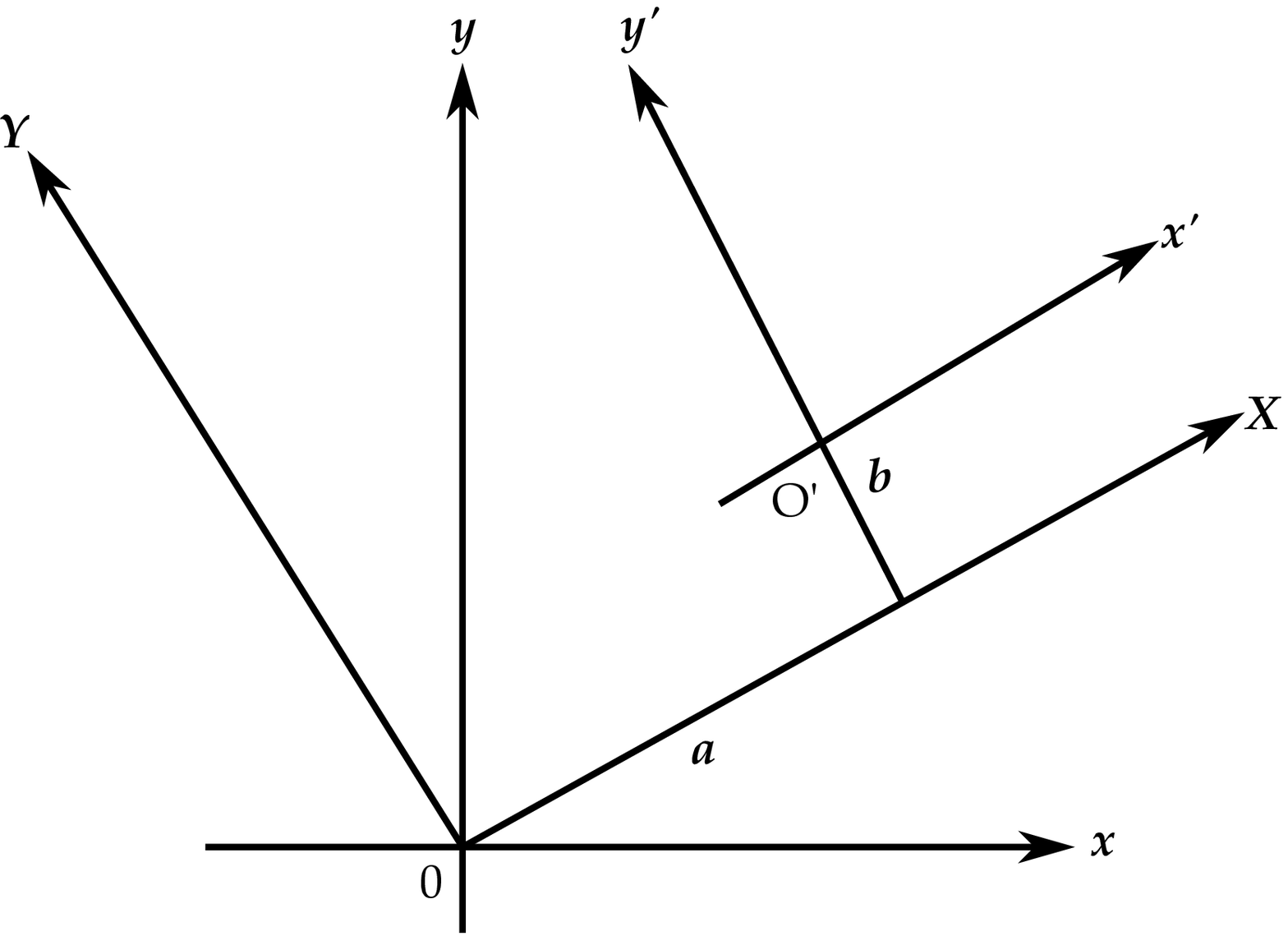}\\
\end{center}
\end{figure}
\newline
Supongamos, pues, que trasladamos los ejes $X-Y$ a un punto $O'$ de coordenadas $(a,b)\diagup XY.$\\
Se definen así dos nuevos ejes $x',y'$ con $x'\parallel X, y'\parallel Y.$\\
Las ecuaciones de la transformación son:
\begin{align*}
X&=x'+a\\
Y&=y'+b
\end{align*}
que llevamos a [\ref{57}]:
\begin{align*}
&\omega{(x'+a)}^2+2D'(x+a)+2E'(y'+b)+F=0\\
&\omega x'^2+2a\omega x'+\omega a^2+2D'x'+2D'a+2E'y'+2E'b+F=0\\
&\omega x'^2+(2a\omega+2D')x'+2E'y'+(\omega a^2+2D'a+2E'b+F)=0
\end{align*}
Para lo que buscamos, debe tenerse que\\
$$\begin{cases}
2a\omega+2D'=0\\
\omega a^2+2D'a+2E'b+F=0
\end{cases}$$
De la $1^a$ ecuación, $a=-\dfrac{D'}{\omega}$ que llevamos a la $2^a$:
\begin{align*}
2E''b&=-(\omega a^2+2D'a+F)\\
\therefore\,\,b&=-\dfrac{\omega a^2+2D'a+F}{2E'}\\
&=-\dfrac{\dfrac{\omega D'^2}{\omega^2}-\dfrac{2D'^2}{\omega}+F}{2E'}=-\dfrac{F-\dfrac{D'^2}{\omega}}{2E'}
\end{align*}
Si trasladamos los ejes $XY$ al punto $O'$ de coordenadas $$\left(-\dfrac{D'}{\omega}, -\dfrac{F-\dfrac{D'^2}{\omega}}{2E'}\right)\diagup XY,$$ donde
\begin{align*}
D'&=\dfrac{AE+BE}{\sqrt{A^2+B^2}}\\
E'&=\dfrac{AE-BD}{\sqrt{A^2+B^2}}
\end{align*}
La ecuación de la cónica$\diagup x'y'$ es: $$\omega x'^2+2E'y'=0\hspace{0.5cm}\therefore\hspace{0.5cm}y'=-\dfrac{\omega}{2E'}x'^2,\hspace{0.5cm}\text{con}\,\, E'=\dfrac{AE-BD}{\sqrt{A^2+B^2}}\neq 0.$$
El \underline{lugar es una parábola} que se abre según el eje $y'.$\\
El punto $O'$ de coordenadas $\left(-\dfrac{D'}{\omega}, -\dfrac{F-\dfrac{D'^2}{\omega}}{2E'}\right)\diagup XY$ es el vértice de la parábola.
\begin{figure}[ht!]
\begin{center}
  \includegraphics[scale=0.4]{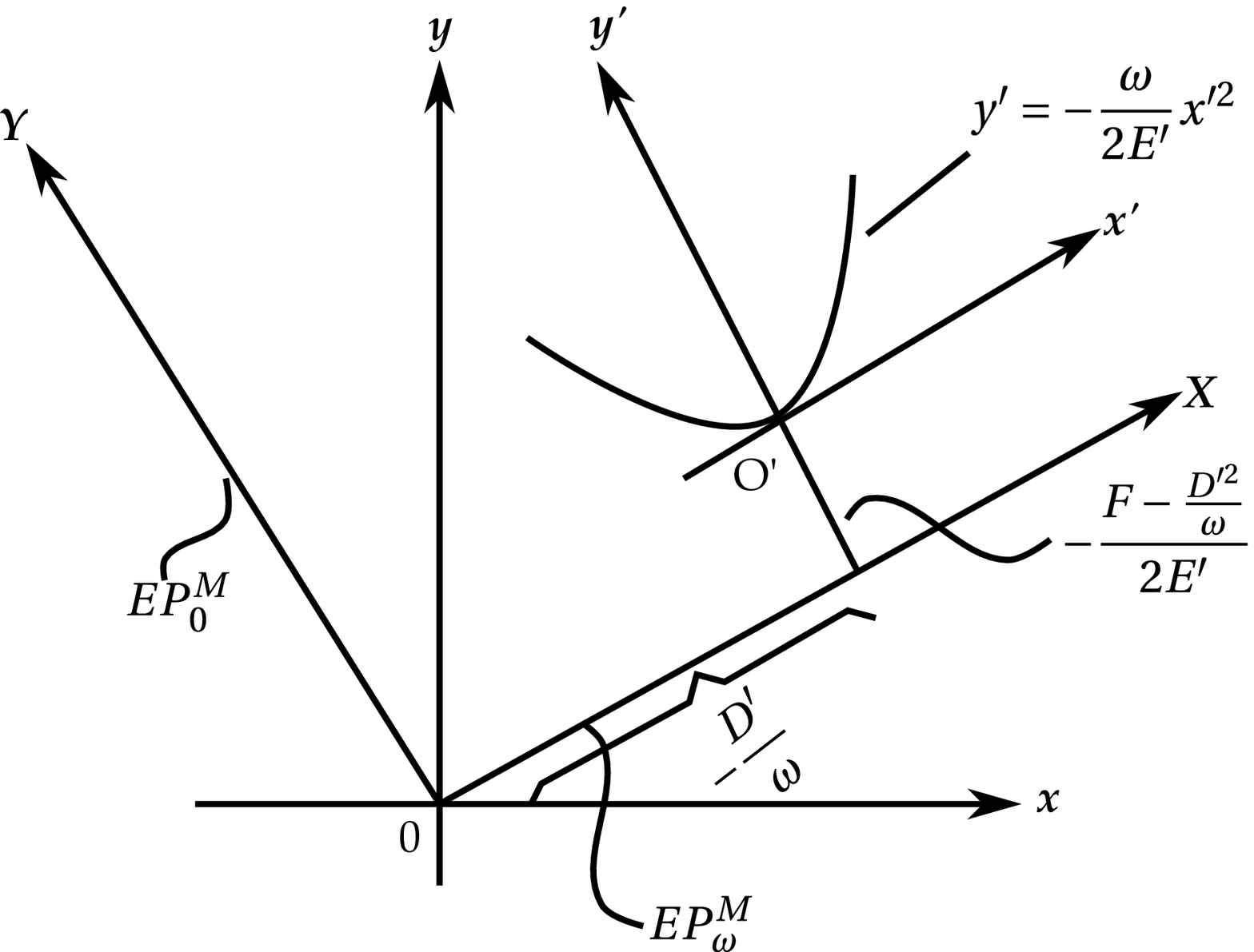}\\
\end{center}
\vspace{0.5cm}
\end{figure}
\item[2)]
\begin{align*}
B>0\\
A<0
\end{align*}
Lo primero que hacemos es dibujar el vector $\left(\begin{array}{cc}\dfrac{B}{\sqrt{A^2+B^2}}\\-\dfrac{A}{\sqrt{A^2+B^2}}\end{array}\right)$ ya que es el que señala el $EP_0^M.$
\begin{figure}[ht!]
\begin{center}
  \includegraphics[scale=0.5]{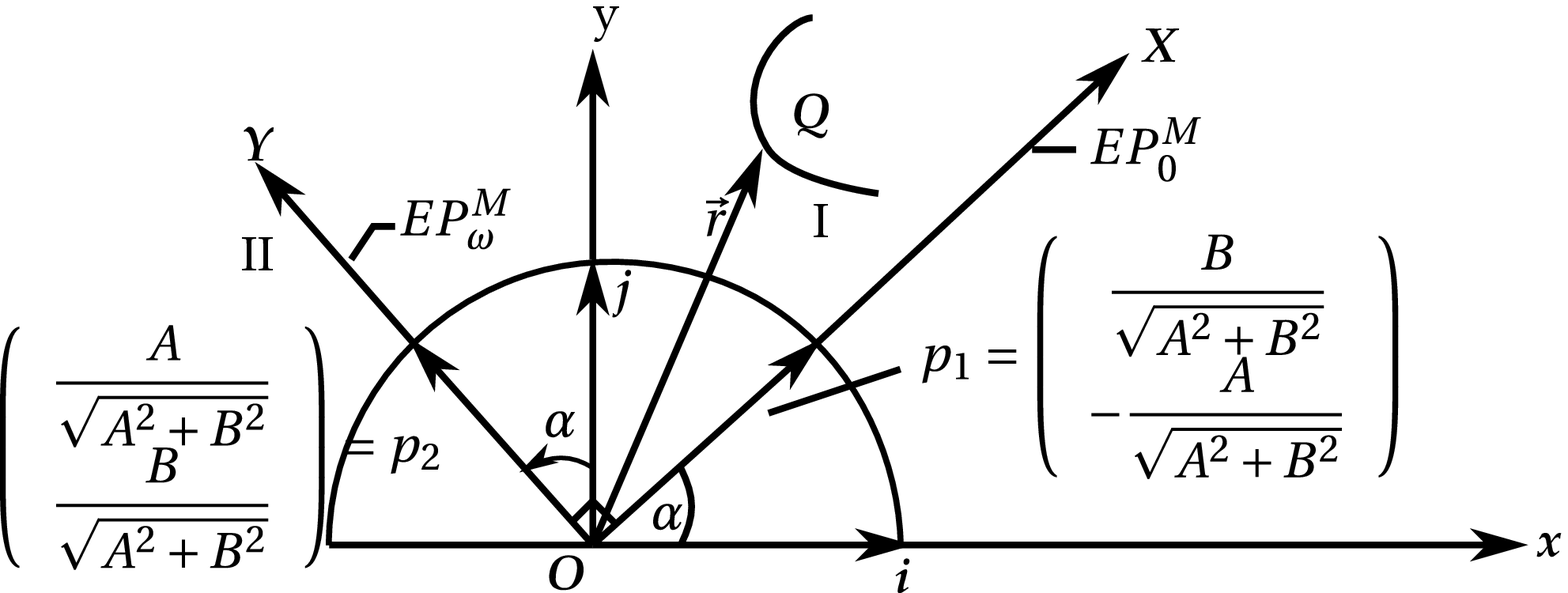}\\
\end{center}
\vspace{0.5cm}
\end{figure}
\newline
El espectro de $M$ lo ordenamos en este caso así: $\lambda(M)=\left\{0,\omega\right\}$
\begin{align*}
EP_0^M\ni p_1=\dbinom{p_1^1}{p_1^2}&=\left(\begin{array}{cc}\dfrac{B}{\sqrt{A^2+B^2}}\\-\dfrac{A}{\sqrt{A^2+B^2}}\end{array}\right)\\
EP_\omega^M\ni p_2\dbinom{p_2^1}{p_2^2}&=\left(\begin{array}{cc}\dfrac{A}{\sqrt{A^2+B^2}}\\ \dfrac{B}{\sqrt{A^2+B^2}}\end{array}\right); {p_1, p_2}:\text{Base ortonormal.}\,\, \alpha=\arctan\left(\dfrac{-A}{B}\right).\\
P=\left(\begin{array}{cc}\underset{\downarrow}{\overset{\uparrow}{p_1}}&\underset{\downarrow}{\overset{\uparrow}{p_2}}\end{array}\right)&=\left(\begin{array}{cc}\dfrac{B}{\sqrt{A^2+B^2}}&\dfrac{A}{\sqrt{A^2+B^2}}\\ -\dfrac{A}{\sqrt{A^2+B^2}}&\dfrac{B}{\sqrt{A^2+B^2}}\end{array}\right)=[I]_{ij}^{p_1,p_2}
\end{align*}
La ecuación de la cónica$\diagup XY$ es ahora
\begin{equation}\label{58}
[\vec{r}]_{p_1p_2}^tP^tMP[\vec{r}]_{p_1p_2}+\left(\left(\begin{array}{cc}2D&2E\end{array}\right)P\right)[\vec{r}]_{p_1p_2}+F=0.
\end{equation}
Por el Teorema Espectral, $$P^tMP=\left(\begin{array}{cc}0&0\\ 0&\omega\end{array}\right)$$
Llamando
\begin{align*}
\left(\begin{array}{cc}2D'&2E'\end{array}\right)&=\left(\begin{array}{cc}2D&2E\end{array}\right)P\\
&=\left(\begin{array}{cc}2D&2E\end{array}\right)\left(\begin{array}{cc}\dfrac{B}{\sqrt{A^2+B^2}}&\dfrac{A}{\sqrt{A^2+B^2}}\\ -\dfrac{A}{\sqrt{A^2+B^2}}&\dfrac{B}{\sqrt{A^2+B^2}}\end{array}\right)
\end{align*}
\begin{equation}\label{59}
D'=\dfrac{DB-AE}{\sqrt{A^2+B^2}}, E'=\dfrac{AD+BE}{\sqrt{A^2+B^2}}
\end{equation}
con $D'\neq 0$ ya que $BD-AE\neq 0.$\\
Al regresar a [\ref{58}] se tiene que
$$\left(\begin{array}{cc}X&Y\end{array}\right)\left(\begin{array}{cc}0&0\\0&\omega\end{array}\right)\dbinom{X}{Y}+2D'X+2E'Y+F=0$$
O sea
\begin{equation}\label{60}
\omega Y^2+2D'X+2E'Y+F=0
\end{equation}
con $D\neq 0$ y $E'$ dados por [\ref{59}] es la ecuación de la cónica$\diagup XY.$\\
$$\left|\begin{array}{ccc}0&0&D'\\ 0&\omega&E'\\ D'&E'&F\end{array}\right|=-\omega D'^2\,\,\text{y como $\omega\neq 0$\,\,y\,\,$D'\neq0$, $\Delta\neq 0.$}$$
En [\ref{60}], $D'\neq0$ ya que $BD-AE\neq0.$\\
Veamos que es posible definir una traslación de los ejes $XY$ a un punto $O'$ de coordenadas $(a,b)\diagup XY,$ \underline{a y b a determinar}, de modo que se elimien en [\ref{60}] el término en $Y$ y el término independiente:
\begin{figure}[ht!]
\begin{center}
  \includegraphics[scale=0.5]{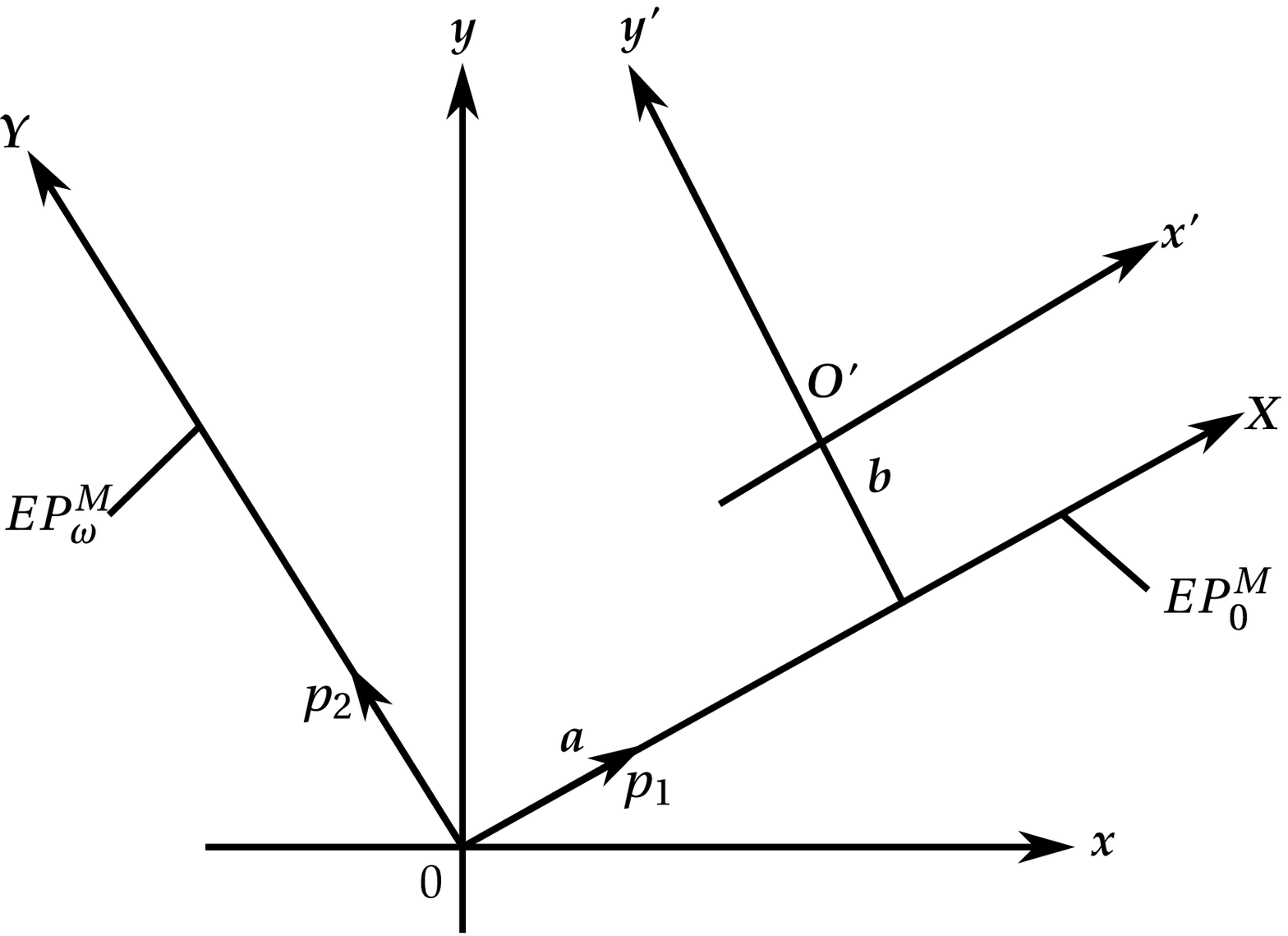}\\
\end{center}
\end{figure}
\newline
Supongamos, pues, que trasladamos los ejes $XY$ a un punto $O'$ de coordenadas $(a,b)\diagup XY.$\\
Se definen así dos nuevos ejes $x'y'\parallel XY.$\\
Las ecuaciones de la trasformación son:
\begin{align*}
X&=x'+a\\
Y&=y'+b\hspace{0.5cm}\text{que llevamos a [\ref{60}]}
\end{align*}
\begin{align*}
\omega{(y'+b)}^2+2D'(x'+a)+2E'(y'+b)+F=0\\
\omega y'^2+2b\omega y'+\omega b^2+2D'x'+2D'a+2E'y'+2E'b+F=0\\
\omega y'^2+(2b\omega+2E')y'+2D'x'+(\omega b^2+2D'a+2E'b+F)=0
\end{align*}
Para lo que buscamos debe tenerse que
$\begin{cases}
2b\omega+2E'=0\\
\omega b^2+2D'a+2E'b+F=0
\end{cases}$\\
De la $1^a$ ecuación, $b=-\dfrac{E'}{\omega}$ que llevamos a la $2^a$
\begin{align*}
2D'a&=-(\omega b^2+2E'b+F)\\
\therefore\hspace{0.5cm}a&=-\dfrac{\omega b^2+2E'b+F}{2D'}\\
&=-\dfrac{\frac{\omega E'^2}{\omega^2}-\frac{2E'^2}{\omega}+F}{2D'}=-\dfrac{F-\frac{E'^2}{\omega}}{2D'}.
\end{align*}
De esta manera, si trasladamos los ejes $XY$ al punto $O'$ de coordenadas $\left(-\dfrac{F-\frac{E'^2}{\omega}}{2D'}, -\dfrac{E'}{\omega}\right)\diagup XY,$ la ecuación de la cónica$\diagup x'y'$ es $$\omega y'^2+2D'x'=0.$$
$\therefore\hspace{0.5cm}x'=-\dfrac{\omega}{2D'}y'^2$ con $D'=\dfrac{BD-AE}{\sqrt{A^2+B^2}}\neq 0.$\\
El lugar es una parábola que se abre según el eje $x':$
\begin{figure}[ht!]
\begin{center}
  \includegraphics[scale=0.5]{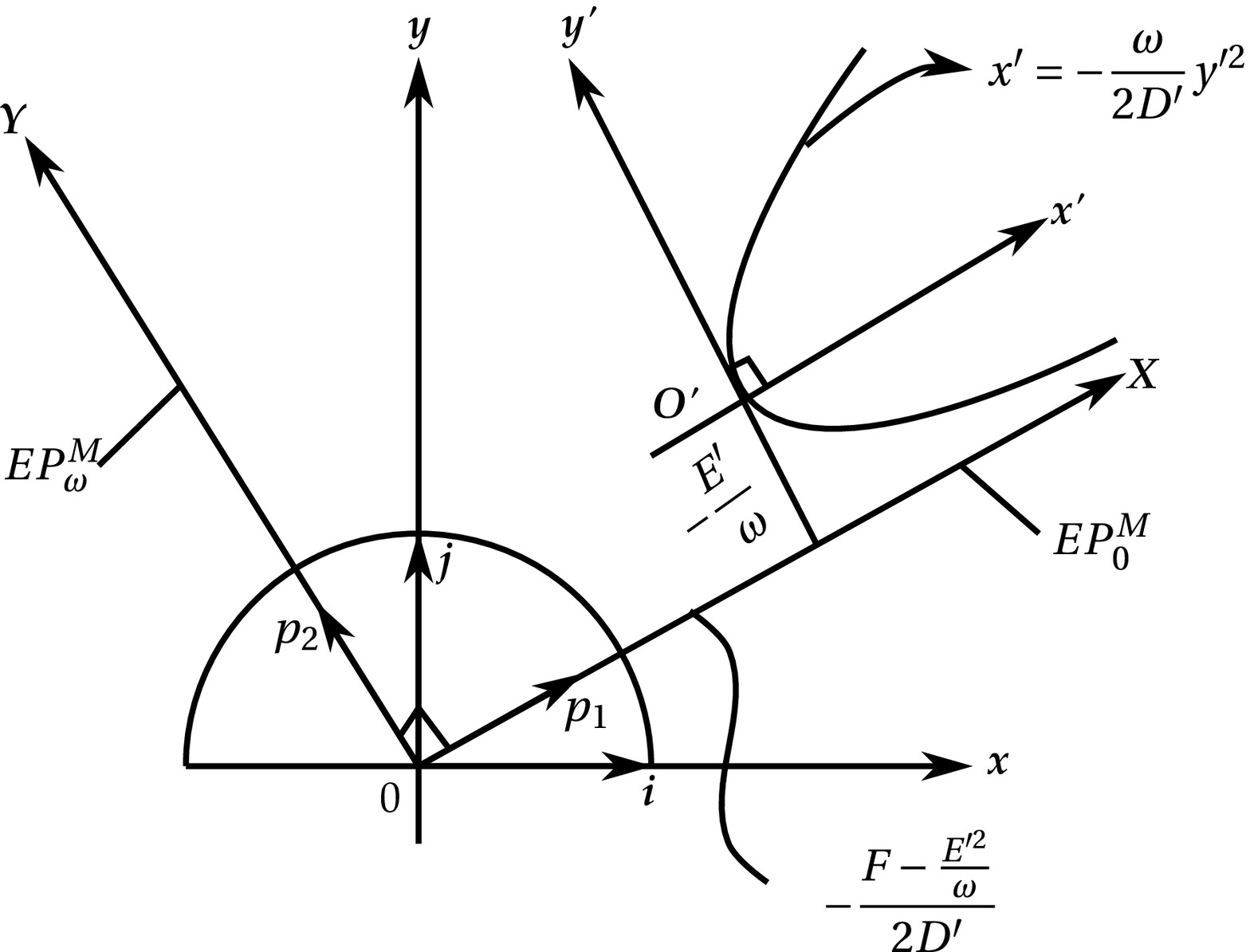}\\
\end{center}
\end{figure}
\newline
El punto $O'$ de coordenadas $\left(-\dfrac{F-\frac{E'^2}{\omega}}{2D'}, -\dfrac{E'}{\omega}\right)\diagup XY$ es el vértice de la parábola.\\
\item[(3)]
\begin{align*}
B<0\\
A<0
\end{align*}
\begin{figure}[ht!]
\begin{center}
  \includegraphics[scale=0.5]{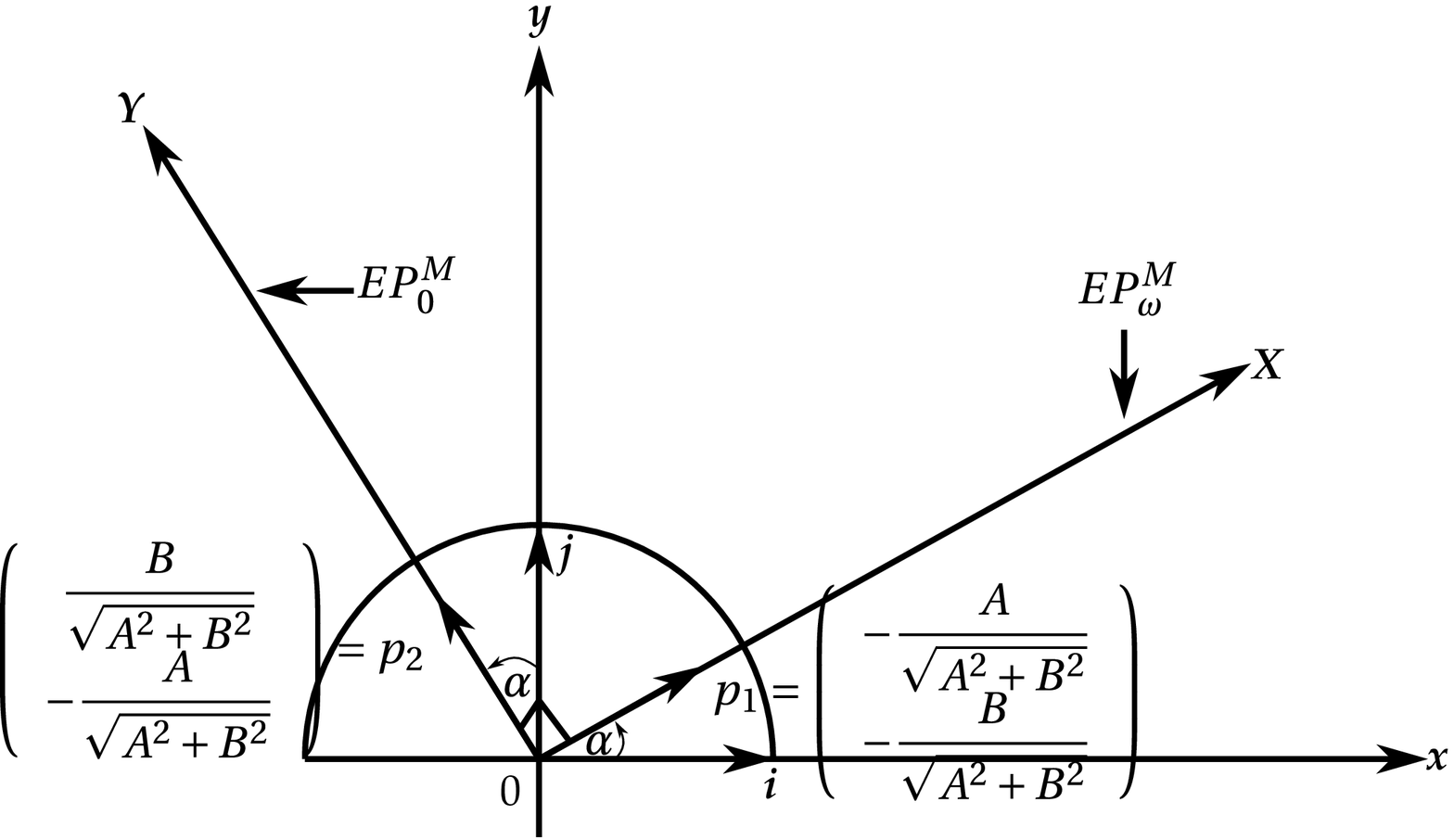}\\
\end{center}
\end{figure}
\newline
Definimos $\lambda(M)=\left\{\omega, 0\right\}.$
\begin{align*}
EP_\omega^M\ni p_1=\dbinom{p_1^1}{p_1^2}&=\left(\begin{array}{cc}-\dfrac{A}{\sqrt{A^2+B^2}}\\ -\dfrac{B}{\sqrt{A^2+B^2}}\end{array}\right);\hspace{0.5cm}{p_1,p_2}:\text{Base ortonormal.}\,\, \alpha=\arctan\dfrac{A}{B}\\
EP_0^M\ni p_2=\dbinom{p_1^1}{p_1^2}&=\left(\begin{array}{cc}\dfrac{B}{\sqrt{A^2+B^2}}\\ -\dfrac{A}{\sqrt{A^2+B^2}}\end{array}\right)\\
p&=\left(\begin{array}{cc}-\dfrac{A}{\sqrt{A^2+B^2}}&\dfrac{B}{\sqrt{A^2+B^2}}\\
-\dfrac{B}{\sqrt{A^2+B^2}}&-\dfrac{A}{\sqrt{A^2+B^2}}\end{array}\right)=\left[I\right]_{ij}^{p_1,p_2}
\end{align*}
La ecuación de la cónica$\diagup XY$ después de aplicar el teorema Espectral es:
\begin{equation}\label{61}
\omega X^2+2D'X+2E'Y+F=0
\end{equation}
con $$\left(\begin{array}{cc}2D'&2E'\end{array}\right)=\left(\begin{array}{cc}2D&2E\end{array}\right)\left(\begin{array}{cc}-\dfrac{A}{\sqrt{A^2+B^2}}&\dfrac{B}{\sqrt{A^2+B^2}}\\
-\dfrac{B}{\sqrt{A^2+B^2}}&-\dfrac{A}{\sqrt{A^2+B^2}}\end{array}\right)$$
O sea,
\begin{align*}
D'&=-\dfrac{AD+BE}{\sqrt{A^2+B^2}}\\
E'&=\dfrac{BD-AE}{\sqrt{A^2+B^2}}\neq 0\hspace{0.5cm}\text{ya que}\,\,BD-AE\neq 0.\\
\Delta&=\left|\begin{array}{ccc}\omega&0&D'\\ 0&0&E'\\ D'&E'&F\end{array}\right|=-\omega E'^2\,\,\text{y como}\,\,\omega\neq 0\,\,\text{y}\,\,E'\neq 0,\,\,\Delta\neq0.
\end{align*}
En [\ref{61}], $E\neq0.$\\
De nuevo se demuestra, como en el caso (1), que si se trasladan los ejes $XY$ al punto $0'$ de coordenadas$\diagup XY:$ $$\left(-\dfrac{D'}{\omega}, -\dfrac{F-\dfrac{D'^2}{\omega}}{2E'}\right),$$ en [\ref{61}] se eliminan el término lineal en $X$ y el término independiente.\\
La cónica referida a los ejes $x'y'$ es: $$\omega x'^2+2E'y'=0\hspace{0.5cm}\therefore\hspace{0.5cm}y'=-\dfrac{\omega}{2E'}\,\,\text{con}\,\,E'=\dfrac{BD-AE}{\sqrt{A^2+B^2}}.$$
\underline{El lugar es una parábola} que se abre según el eje $y'.$\\
\item[(4)]
\begin{align*}
B<0\\
A>0
\end{align*}
Primero ubicamos el vector $\left(\begin{array}{cc}\dfrac{B}{\sqrt{A^2+B^2}}\\ -\dfrac{A}{\sqrt{A^2+B^2}}\end{array}\right)$ que nos da el $EP_0^M.$
\begin{figure}[ht!]
\begin{center}
  \includegraphics[scale=0.5]{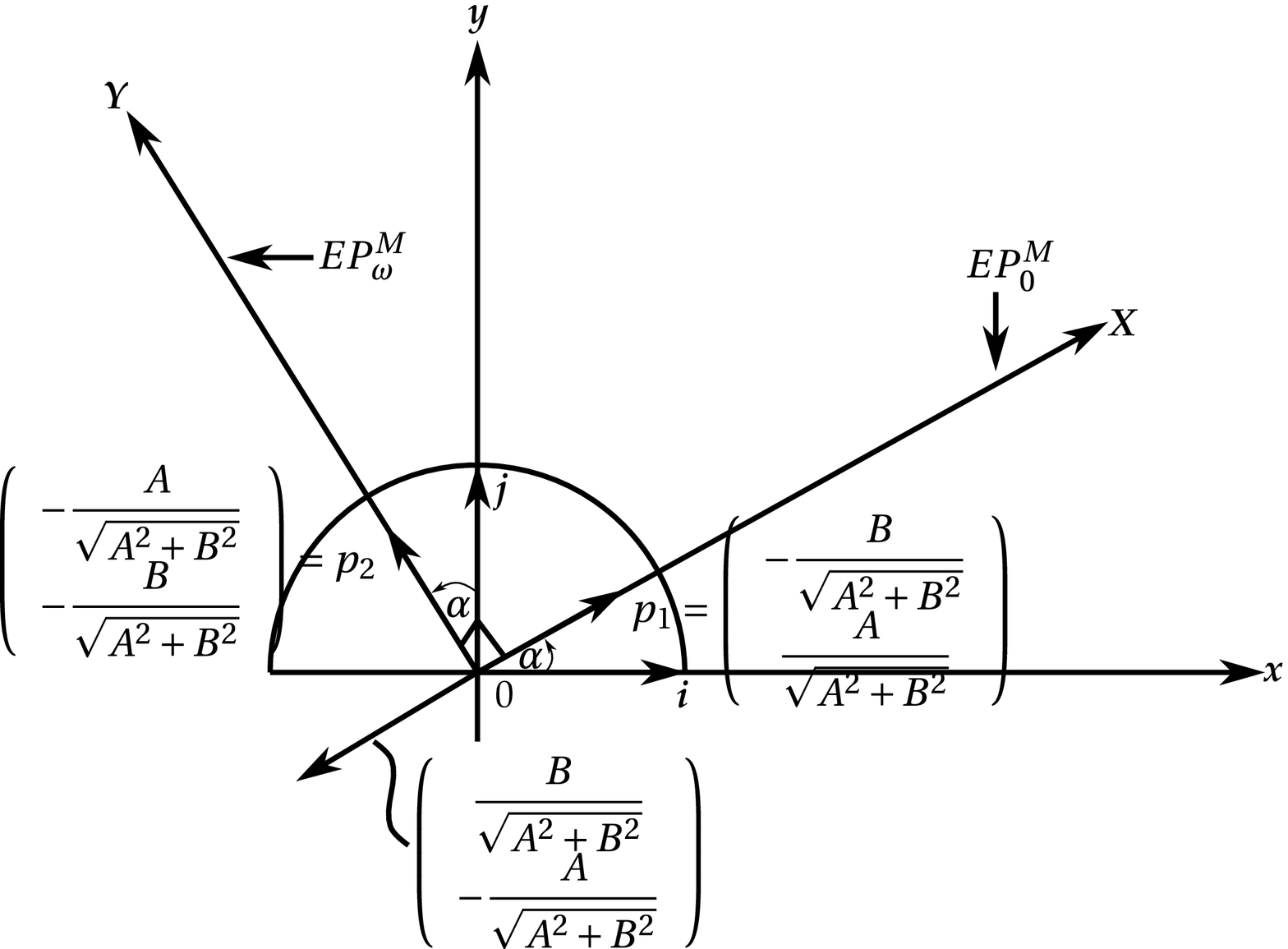}\\
\end{center}
\end{figure}
\newline
Definimos $\lambda(M)=\left\{\omega, 0\right\}.$
\begin{align*}
EP_0^M\ni p_1=&\left(\begin{array}{cc}-\dfrac{B}{\sqrt{A^2+B^2}}\\ \dfrac{A}{\sqrt{A^2+B^2}}\end{array}\right);\hspace{0.5cm}{p_1,p_2}:\text{Base ortonormal.}\,\, \alpha=\arctan\dfrac{A}{B}\\
EP_\omega^M\ni p_2&=\left(\begin{array}{cc}-\dfrac{A}{\sqrt{A^2+B^2}}\\ -\dfrac{B}{\sqrt{A^2+B^2}}\end{array}\right)\\
p&=\left(\begin{array}{cc}-\dfrac{B}{\sqrt{A^2+B^2}}&-\dfrac{A}{\sqrt{A^2+B^2}}\\
\dfrac{A}{\sqrt{A^2+B^2}}&-\dfrac{B}{\sqrt{A^2+B^2}}\end{array}\right)=\left[I\right]_{ij}^{p_1,p_2}
\end{align*}
Después de aplicar el teorema Espectral, la ecuación de la cónica$\diagup XY$ es:
\begin{equation}\label{62}
\omega X^2+2D'X+2E'Y+F=0
\end{equation}
Donde,
\begin{align*}
D'&=\dfrac{AE-BD}{\sqrt{A^2+B^2}}\\
E'&=-\dfrac{AD+BE}{\sqrt{A^2+B^2}}\neq 0\hspace{0.5cm}\text{ya que}\,\,AE-BD\neq 0.\\
\Delta&=\left|\begin{array}{ccc}0&0&D'\\ 0&\omega&E'\\ D'&E'&F\end{array}\right|=-\omega D'^2\,\,\text{y como}\,\,\omega\neq 0\,\,\text{y}\,\,D'\neq 0,\,\,\Delta\neq0.
\end{align*}
En [\ref{62}], $D'\neq0.$\\
como se hizo en el caso (2), si trasladamos los ejes $XY$ al punto $0'$ de coordenadas$\diagup XY:$ $$\left( -\dfrac{F-\dfrac{E'^2}{\omega}}{2D'}, -\dfrac{E'}{\omega}\right)\diagup XY,$$
la ecuación de la cónica$\diagup x'y'$ es: $$\omega y'^2+2D'y'=0\hspace{0.5cm}\therefore\hspace{0.5cm}x'=-\dfrac{\omega}{2D'}\,\,\text{con}\,\,D'=\dfrac{AE-BD}{\sqrt{A^2+B^2}}\neq0.$$
\underline{El lugar es una parábola} que se abre según el eje $x'.$\\
\end{enumerate}
\framebox[1.2\width]{\parbox[2.6\height]{8.1cm}{En resumen, el criterio para identificar las cónicas sin centro es este:\\
Consideremos la cónica $Ax^2+2Bxy+2Dx+2Ey+F=0$ en la que $A,B,C\neq0.$\\
Supongamos que $\left\{\dbinom{A}{B}\dbinom{B}{C}\right\}$ es L.D., i.e., $\delta=\left|\begin{array}{cc}A&B\\B&C\end{array}\right|=0.$ y que $\dbinom{-D}{-E}\notin Sg\left\{\dbinom{A}{B}\right\},$ o sea que $\left\{\dbinom{A}{B},\dbinom{-D}{-E}\right\}$ es L.I., y por lo tanto, $\delta=\left|\begin{array}{cc}A&B\\B&C\end{array}\right|=BD-AE\neq0.$\\
La cónica es una parábola. En este caso $\Delta\neq0.$ }}
\begin{ejem}
Consideremos la cónica de ecuación
\begin{equation}\label{63}
4x^2-4xy+y^2-2x-14y+7=0
\end{equation}
\begin{tabular}{cccc}
$A=4$&&$2D=-2; D=-1$\\
$2B=-4; B=-2$&$2E=-14; -E=7$\\
$C=1$
\end{tabular}
\begin{align*}
M&=\left(\begin{array}{cc}4&-2\\-2&1\end{array}\right)\\
\delta&=\left|\begin{array}{cc}4&-2\\-2&1\end{array}\right|=0
\end{align*}
\begin{figure}[ht!]
\begin{center}
  \includegraphics[scale=0.4]{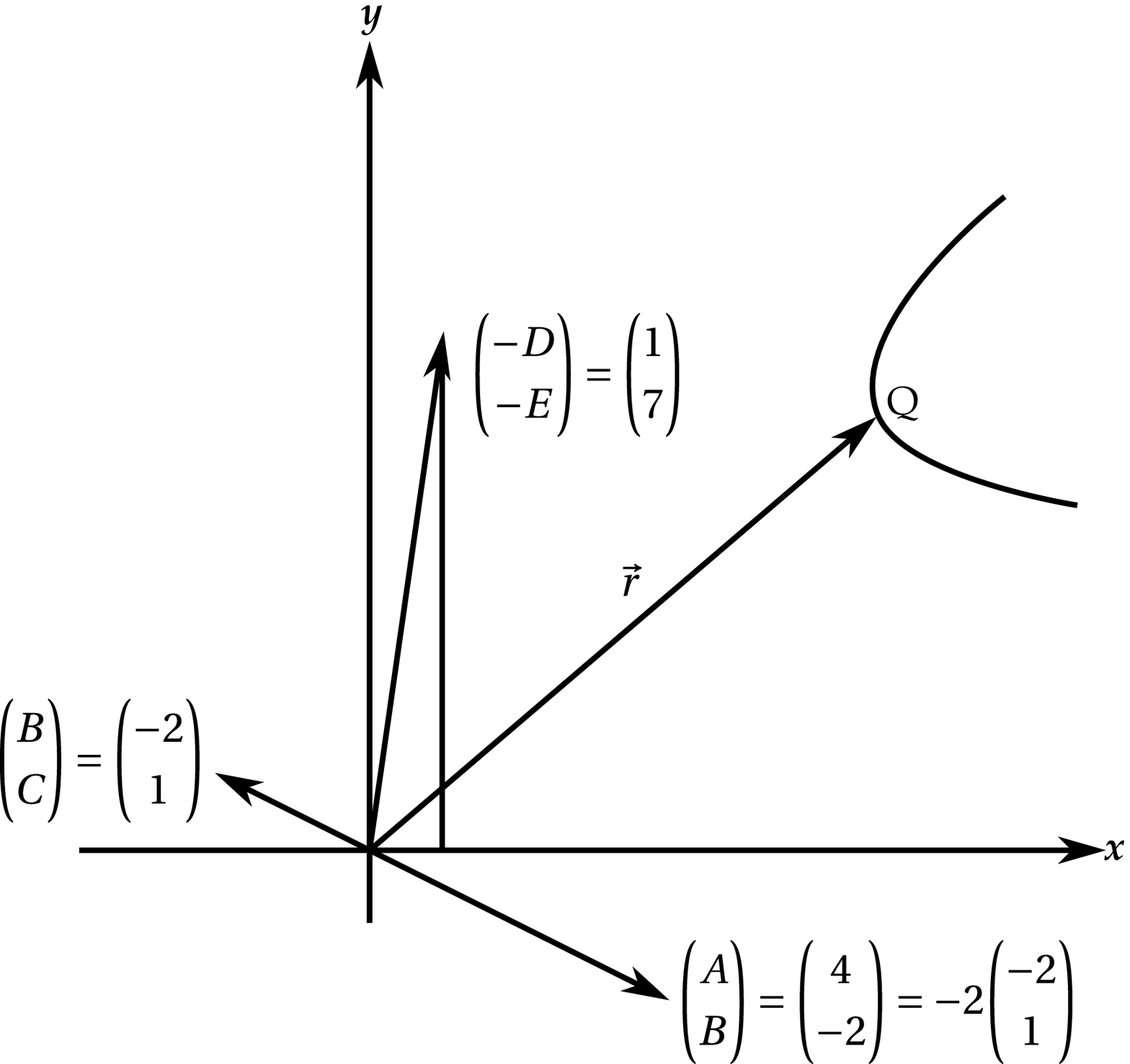}\\
\end{center}
\end{figure}
$\left\{\dbinom{A}{B}=\dbinom{4}{-2}, \dbinom{B}{C}=\dbinom{-2}{1}\right\}$ es L.D.\\
$\dbinom{-D}{-E}=\dbinom{1}{7}\notin Sg\left\{\dbinom{4}{-2}\right\}=Sg\left\{\dbinom{-2}{1}\right\}$\\
\underline{La cónica no tiene centro.} El sistema
\begin{align*}
Ah+Bk&=-D\\
Bh+Ck&=-E,\hspace{0.5cm}\text{o sea,}\\
Ah-2k&=1\\
-2h+k&7\hspace{0.5cm}\text{no tiene solución.}
\end{align*}
En virtud de la tabla anterior, \underline{la cónica es una parábola.}\\
Dado que la cónica [\ref{63}] no tiene centro, no podemos realizar una traslación que elimine los términos lineales en [\ref{63}].\\
Para determinar el lugar debemos realizar primero una transformación ortogonal que elimine el término mixto.\\
Sea $Q(x,y)$ un punto de la cónica y $[\vec{r}]_{ij}=\dbinom{x}{y}$ su vector de posición$\diagup O.$ La ecuación de [\ref{63}]$\diagup xy$ se escribe así:
\begin{equation}\label{64}
\left(\begin{array}{cc}x&y\end{array}\right)M\dbinom{x}{y}+\left(\begin{array}{cc}2D&2E\end{array}\right)\dbinom{x}{y}+F=0
\end{equation}
$PCM(\lambda)=\lambda^2-\omega\lambda+\delta=0.$ Como $\delta=0$ y $\omega=5,$ $\lambda^2-5\lambda=0.$ Luego los valores propios de $M$ son 5 y 0.\\
$$EP_0^M=\mathscr{N}(M)=\left\{(u,v)\diagup\left(\begin{array}{cc}4&-2\\-2&1\end{array}\right)\dbinom{u}{v}=\dbinom{0}{0}\right\}$$
Dado que la $1^a$ ecuación es 2 veces la $2^a,$ $$\mathscr{N}(M)=\left\{(u,v)\diagup -2u+v=0\right\}.$$
De $-2u+v=0, v=2u.$ Luego todo vector de la forma $(u,2u)=u(1-2)$ con $u\in\mathbb{R}$ está en el $EP_0^M=\mathscr{N}(M).$\\
Así que
\begin{align*}
EP_0^M=Sg\left\{\dbinom{1}{2}\right\}&\underset{\uparrow}{=}Sg\left\{\left(\begin{array}{cc}\overset{\underset{\parallel}{p_1}}{\dfrac{1}{\sqrt{5}}}\\ \dfrac{2}{\sqrt{5}}\end{array}\right)\right\}\\
&\begin{cases}
\text{Se normaliza $\dbinom{1}{2}$}\\
\text{para hallar $p_1$.}
\end{cases}\\
EP_5^M&=Sg\left\{\left(\begin{array}{cc}-\overset{\underset{\parallel}{p_2}}{\dfrac{2}{\sqrt{5}}}\\ \dfrac{1}{\sqrt{5}}\end{array}\right)\right\}
\end{align*}
$$P^tMP=\left(\begin{array}{cc}0&0\\0&5\end{array}\right).$$ Al aplicar el Teorema Espectral se encuentra que la ecuación de la cónica$\diagup XY$ es:
\begin{equation}\label{65}
\left(\begin{array}{cc}X&Y\end{array}\right)\left(\begin{array}{cc}0&0\\0&5\end{array}\right)\dbinom{X}{Y}+\left(\begin{array}{cc}2D'&2E'\end{array}\right)\dbinom{X}{Y}+7=0
\end{equation}
con
\begin{align*}
2D'&=\left(\begin{array}{cc}2D&2E\end{array}\right)\dbinom{p_1^1}{p_2^1}=\left(\begin{array}{cc}-2&-12\end{array}\right)\left(\begin{array}{cc}\dfrac{1}{\sqrt{5}}\\\dfrac{2}{\sqrt{5}}\end{array}\right)=-6\sqrt{5}\\
2E'&=\left(\begin{array}{cc}2D&2E\end{array}\right)\dbinom{p_2^1}{p_2^2}=\left(\begin{array}{cc}-2&-12\end{array}\right)\left(\begin{array}{cc}-\dfrac{2}{\sqrt{5}}\\\dfrac{1}{\sqrt{5}}\end{array}\right)=-2\sqrt{5}
\end{align*}
Regresando a [\ref{65}]:
\begin{equation}\label{66}
5Y^2-6\sqrt{5}X-2\sqrt{5}-2\sqrt{5}Y+7=0
\end{equation}
es la ecuación de la cónica$\diagup XY.$
\newpage
\begin{figure}[ht!]
\begin{center}
  \includegraphics[scale=0.5]{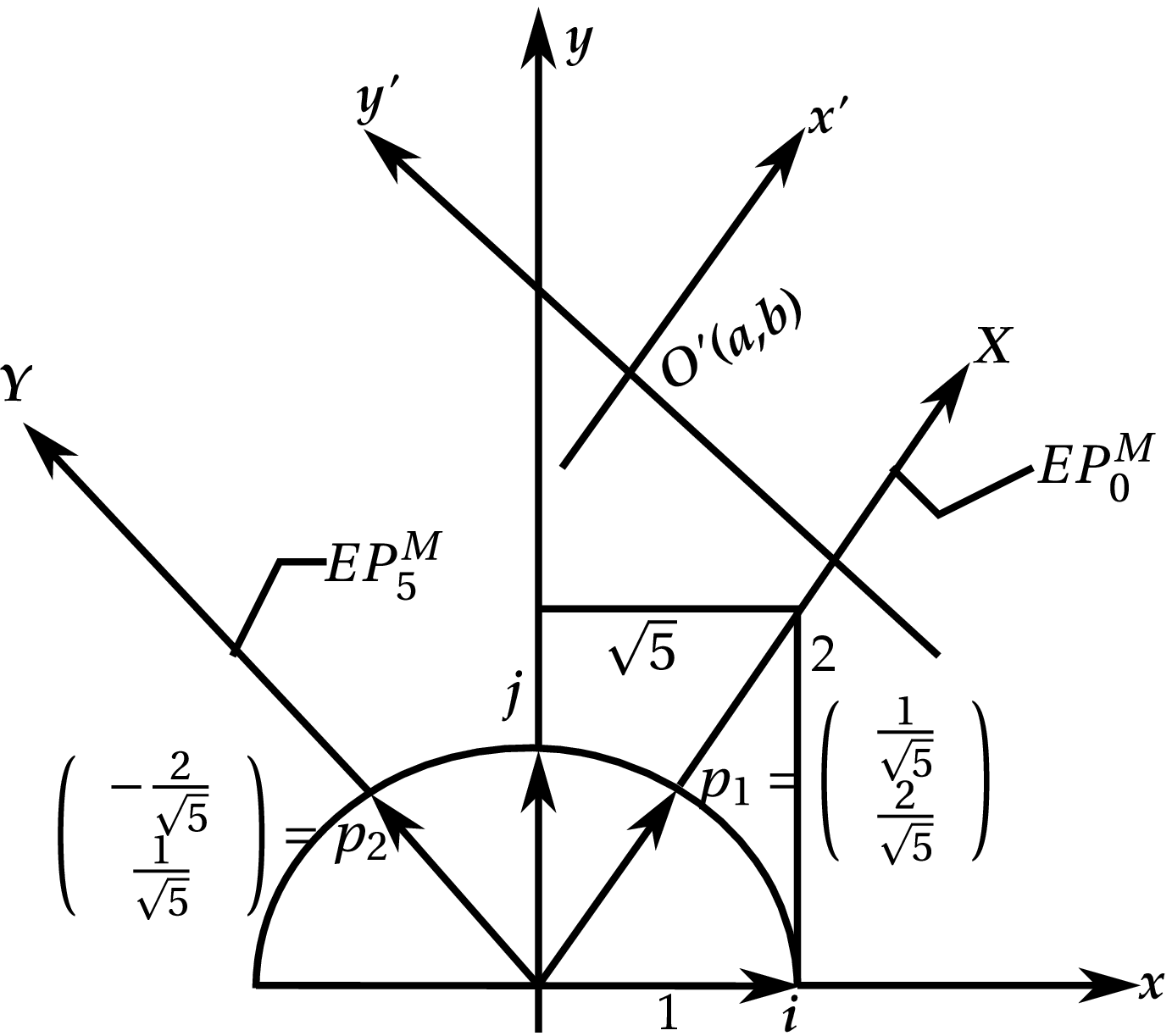}\\
\end{center}
\end{figure}
El próximo paso en la reducción es definir una traslación de los eje $XY$ a un punto $O'$ de coordenadas $(a,b)\diagup XY,$ \underline{a,b a determinar,} de manera que se eliminan en [\ref{66}] el coeficiente en $Y$ y el término independiente.\\
Supongamos que realizamos una traslación de los ejes $XY$ al punto $O'$ de coordenadas $(a,b)\diagup XY.$\\
Ecuaciones de la transformación:
\begin{align*}
X&=x'+a\\
Y&=y'+b\hspace{0.5cm}\text{que llevamos a [\ref{66}]}:
\end{align*}
\begin{align*}
5{(y'+b)}^2-6\sqrt{5}(x'+a)-2\sqrt{5}(y'+b)+7=0\\
5y'^2+10by'+5b^2-6\sqrt{5}x'-6\sqrt{5}a-2\sqrt{5}y'-2\sqrt{5}b+7=0\\
5y'^2-6\sqrt{5}x'+(10b-2\sqrt{5})y'+(5b^2-6\sqrt{5}a-2\sqrt{5}b+7)=0
\end{align*}
Para conseguir lo que se quiere,
$\begin{cases}
10b-2\sqrt{5}&=0\\
5b^2-6\sqrt{5}a-2\sqrt{5}b+7&=0
\end{cases}$\\
De la $1^a$ ecuación, $b=\dfrac{\sqrt{5}}{5}$ que llevamos a la $2^a$
$$6\sqrt{5}a=5b^2-2\sqrt{5}b+7$$
\begin{align*}
\therefore\hspace{0.5cm}a=\dfrac{5b^2-2\sqrt{5}b+7}{6\sqrt{5}}&\underset{\uparrow}{=}\dfrac{\sqrt{5}}{5}\\
&b=\dfrac{\sqrt{5}}{5}
\end{align*}
Luego si trasladamos los ejes $XY$ al punto $O'$ de coordenadas $\left(\dfrac{\sqrt{5}}{5},\dfrac{\sqrt{5}}{5}\right)\diagup XY$ se obtienen los eje $x'y'$ y la ecuación de la cónica$\diagup x'y'$ es  $$5y'2-6\sqrt{6}x'=0,\hspace{0.5cm}x'=\dfrac{5}{6\sqrt{5}}y'^2$$ y finalmente, $x'=\dfrac{\sqrt{5}}{6}y'^2:\text{\underline{parábola} que se abre como se muestra en la figura:}$
\begin{figure}[ht!]
\begin{center}
  \includegraphics[scale=0.5]{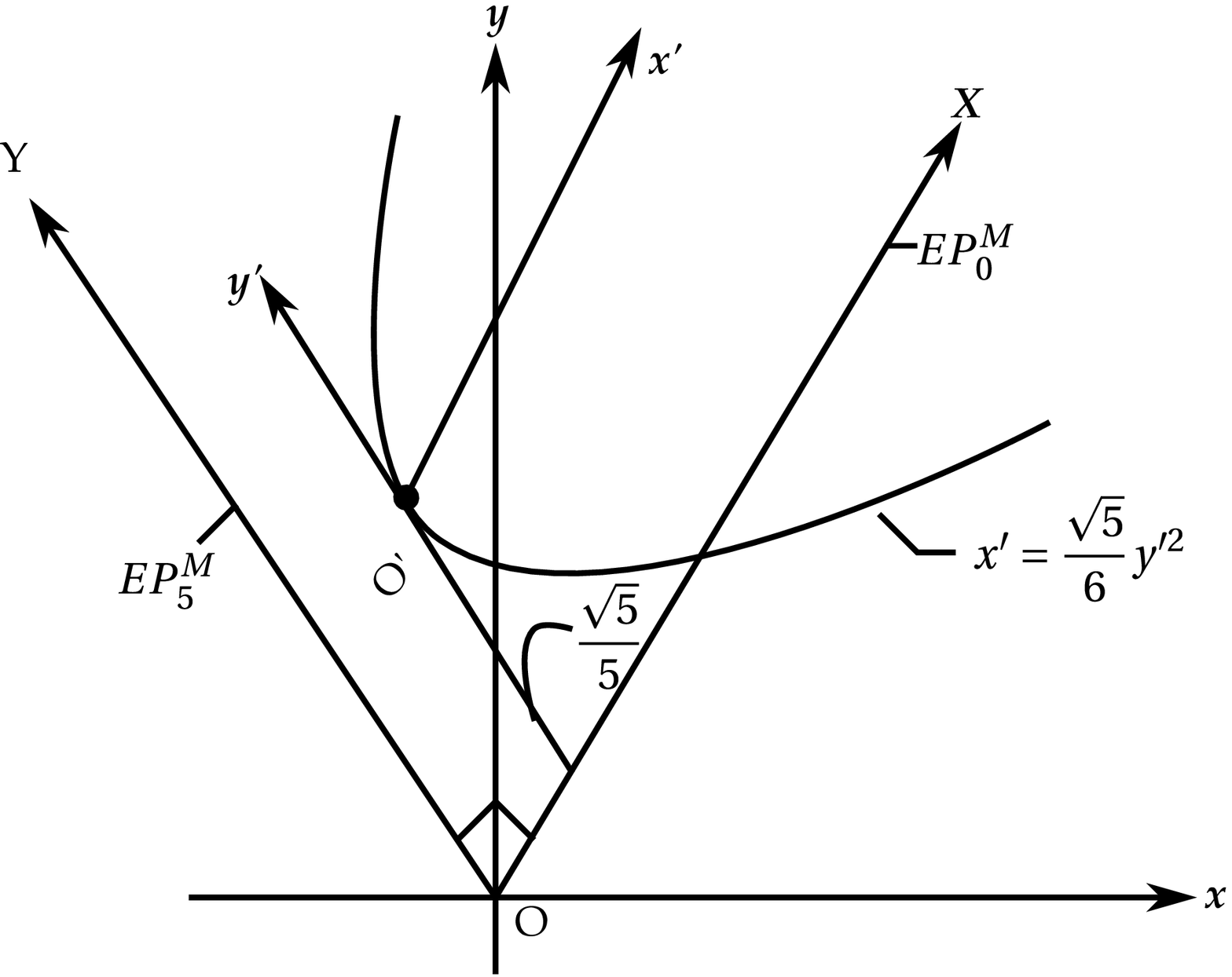}\\
\end{center}
\end{figure}
\end{ejem}
\section{Cónicas con infinitos centros}
Consideremos la cónica de ecuación
\begin{equation}\label{67}
Ax^2+2Bxy+Cy^2+2Dx+2Ey+F=0
\end{equation}
Como en el caso de las cónicas sin centro, vamos a asumir que los coeficientes A,B,C son $\neq 0$ ya que el en que uno o dos de ellos sea cero y conduzca a una cónica con infinitos centros ya fue analizado.\\
Esperamos que el estudio nos lleve a establecer que el lugar consta de dos rectas $\parallel$s, de una recta, o que sea $\emptyset.$\\
Si la cónica se reduce a dos rectas $\parallel$s, todo punto de la $\parallel$ media de ellos es centro de simetría.\'
Si la cónica se reduce a una recta, todo punto de ella es un centro de simetría de la recta.\\
Si la cónica se reduce a dos rectas concurrentes, tiene centro único: el punto donde concurren ambas rectas.\\
En las cónicas con infinitos centros, $\left\{\dbinom{A}{B},\dbinom{B}{C}\right\}$ es L.D. y $\dbinom{-D}{-E}\in Sg\left\{\dbinom{A}{B}\right\}=Sg\left\{\dbinom{B}{C}\right\}.$\\
O sea que $\left\{\dbinom{A}{B},\dbinom{-D}{-E}\right\}$ L.D. y por tanto $\left|\begin{array}{cc}A&-D\\B&-E\end{array}\right|=BD-AE=0\hspace{0.5cm}\therefore\hspace{0.5cm}BD=AE.$\\
Tomemos $\left\{\dbinom{B}{C}, \dbinom{-D}{-E}\right\}$ L.D. y por tanto $\left|\begin{array}{cc}B&-D\\C&-E\end{array}\right|=CD-BE=0\hspace{0.5cm}\therefore\hspace{0.5cm}CD=BE.$\\
\begin{figure}[ht!]
\begin{center}
  \includegraphics[scale=0.5]{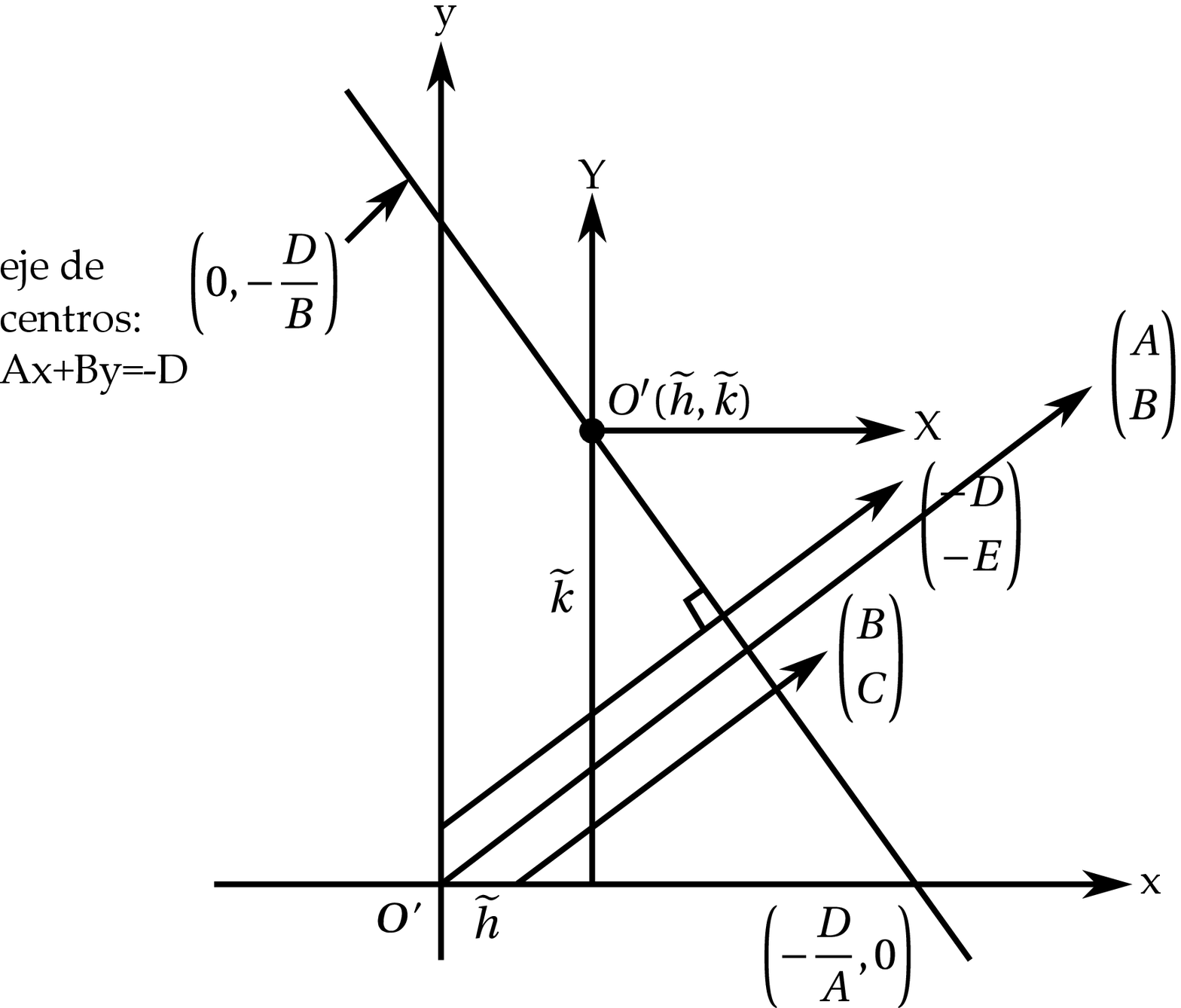}\\
\end{center}
\end{figure}
\newline
\begin{align*}
M=\left(\begin{array}{cc}A&B\\B&C\end{array}\right);\,\,\delta&=AC-B^2=0.\\
&\therefore\,\,AC=B^2\hspace{0.5cm}\text{y como $B\neq 0$ A y C son $\neq0$ y tienen el mismo signo.}
\end{align*}
$\omega=A+C\neq0$ ya que si $\omega=A+C=0, C=-A$ y $0=\delta=AC-B^2=-\left(A^2+B^2\right).$ O sea que $A^2+B^2=0, (\rightarrow\leftarrow),$ ya que $A$ y $B\neq0.$\\
Además, signo de $A=\text{signo de C=signo de $\omega$}.$\\
La cónica tiene infinitos centros porque el vector $\dbinom{-D}{-E}$ se puede escribir de infinitas maneras como una C.L. de $\dbinom{A}{B}$ y $\dbinom{B}{C}.$\\
Hay un eje de centros: la recta
\begin{equation}\label{68}
Ax+By=-D
\end{equation}
Probemos esto. Sea $\left(h,-\dfrac{A}{B}h-\dfrac{D}{B}\right)$ una solución de [\ref{68}]. Veamos que también es solución de $Bx+Cy=-E.$
\begin{align*}
Bh+C\left(-\dfrac{A}{B}h-\dfrac{D}{B}\right)=Bh+\dfrac{C}{B}\left(-Ah-D\right)=\dfrac{(\overset{\overset{0}{\parallel}}{B^2-AC})h-CD}{B}&\underset{\uparrow}{=}-\dfrac{CD}{B}=-\dfrac{\cancel{B}E}{\cancel{B}}=-E\\
&CD=BE
\end{align*}
Esto demuestra que $\forall h\in\mathbb{R},$ el punto de coordenadas $\left(h,-\dfrac{A}{B}h-\dfrac{D}{B}\right)\diagup xy$ es un centro de la cónica.\\
Tomemos un punto $O'$ de coordenadas $(\widetilde{h},\widetilde{k})\diagup xy$ en el eje de centros, $O'$ fijo y realicemos la traslación de los ejes $x-y$ al punto $O'.$\\
Se definen así los ejes $X-Y\parallel$s a $x-y.$\\
Ecuaciones de la transformación:
\begin{align*}
X&=x+\widetilde{h}\\
Y&=y+\widetilde{k}\hspace{0.5cm}\text{con}\,\,\widetilde{k}=-\dfrac{A}{B}\widetilde{h}-\dfrac{D}{B}.
\end{align*}
Como $O'$ es un centro de la cónica, al realizar dicha traslación se eliminan en [\ref{67}] los términos lineales y el término constante se transforma en $f(\widetilde{h},\widetilde{k}).$\\
La ecuación de la cónica$\diagup XY$ es:
\begin{align}
\left(\begin{array}{cc}X&Y\end{array}\right)M\dbinom{X}{Y}+f(\widetilde{h},\widetilde{k})&=0,\notag\\
\left(\begin{array}{cc}A&B\\B&C\end{array}\right)M\dbinom{X}{Y}+f(\widetilde{h},\widetilde{k})&=0,\notag\\
\intertext{ó}\notag\\
AX^2+2BXY+CY^2+f(\widetilde{h},\widetilde{k})\label{69}
\end{align}
y hemos conseguido anular los términos lineales en [\ref{67}].\\
Vamos ahora a demostrar que $\Delta=0.$\\
Recuerdese que hemos trasladado la cónica [\ref{67}] a un punto $O'(\widetilde{h},\widetilde{k})$ que es un centro de la misma y por tanto, la ecuación de la cónica$\diagup XY$ es:
\begin{align*}
\left(\begin{array}{cc}A&B\\B&C\end{array}\right)\dbinom{X}{Y}+f(\widetilde{h},\widetilde{k})=0\hspace{0.5cm}\text{con}\\
f(\widetilde{h},\widetilde{k})=\left(A\widetilde{h}+B\widetilde{k}\right)\widetilde{h}+\left(C\widetilde{k}+B\widetilde{h}\right)\widetilde{k}+2D\widetilde{h}+2E\widetilde{k}+F\hspace{0.5cm}\star.
\end{align*}
Como $O'(\widetilde{h},\widetilde{k})$ es centro de la cónica,
\begin{align*}
A\widetilde{h}+B\widetilde{k}&=-D\\
B\widetilde{h}+C\widetilde{k}&=-E\hspace{0.5cm}\text{que llevamos a $\star$\,\, nos da:}\\
f(\widetilde{h},\widetilde{k})&=-D\widetilde{h}-E\widetilde{k}+2D\widetilde{h}+2D\widetilde{k}+F\\
&=D\widetilde{h}+E\widetilde{k}+F.
\end{align*}
De esta manera,\\
$\begin{cases}
A\widetilde{h}+B\widetilde{k}+D=0\\
B\widetilde{h}+C\widetilde{k}+E=0\\
D\widetilde{h}+E\widetilde{k}+F-f(\widetilde{h},\widetilde{k})=0\hspace{0.5cm}\text{y por lo tanto, pag. 39, 40}\\
\Delta-f(\widetilde{h},\widetilde{k})\delta=0\hspace{0.5cm}\text{y como $\delta=0, \Delta=0$.}
\end{cases}$\\
El siguiente paso es aplicar el Teorema Espectral para eliminar el término mixto en la ecuación [\ref{69}].\\
$PCM(\lambda)=\lambda^2-\omega\lambda+\delta=0.$\\
Como $\delta=0, \lambda^2-\omega\lambda=0.$ Luego los valores propios de $M$ en $O$ son 0 y $\omega.$
\begin{align*}
EP_0^M=\mathscr{N}(M)=\left\{\dbinom{u}{v}\diagup\begin{array}{cc}Au+Bv=0\\ Bu+Cv=0\end{array}\right\}&\underset{\uparrow}{=}\left\{\dbinom{u}{v}\diagup Au+Bv=0\right\}\\
&\begin{cases}
\text{Como}\,\,\left\{\dbinom{A}{B},\dbinom{B}{C}\right\}\,\,\text{es L.D.},\\
\text{la $2^a$\,\, ecuación del sistema}\\
\text{es redundante (veáse dm. pag. 74.)}
\end{cases}
\end{align*}
De $Au+Bv=0,\,\,v=-\dfrac{A}{B}u.$\\
Luego todo vector de la forma $\left(u,-\dfrac{A}{B}u\right)=u\left(1,-\dfrac{A}{B}\right)=\alpha\left(B,-A\right)$ con $\alpha\in\mathbb{R}$ está en el $\mathscr{N}(M).$\\
Esto demuestra que
\begin{align*}
EP_0^M=Sg\left\{\dbinom{B}{-A}\right\}&\underset{\uparrow}{=}Sg\left\{\left(\begin{array}{c}\dfrac{B}{\sqrt{A^2+B^2}}\\-\dfrac{A}{\sqrt{A^2+B^2}}\end{array}\right)\right\}\\
&\begin{cases}
\text{Se normaliza}\\
\text{el vector $\dbinom{-B}{A}$}
\end{cases}
\end{align*}
Supongamos $D>0.$ Puede tenerse:
\begin{enumerate}
\item[(1)] que
\begin{align*}
B<0\\
A<0
\end{align*}
Lo primero es dibujar el eje de centros y luego el vector $\left(\begin{array}{c}\dfrac{B}{\sqrt{A^2+B^2}}\\-\dfrac{A}{\sqrt{A^2+B^2}}\end{array}\right)$ que nos da el $EP_0^M.$ Nótese que el vector es paralelo al eje de centros $\left(\begin{array}{c}\dfrac{B}{\sqrt{A^2+B^2}}\\-\dfrac{A}{\sqrt{A^2+B^2}}\end{array}\right)$
\begin{figure}[ht!]
\begin{center}
  \includegraphics[scale=0.4]{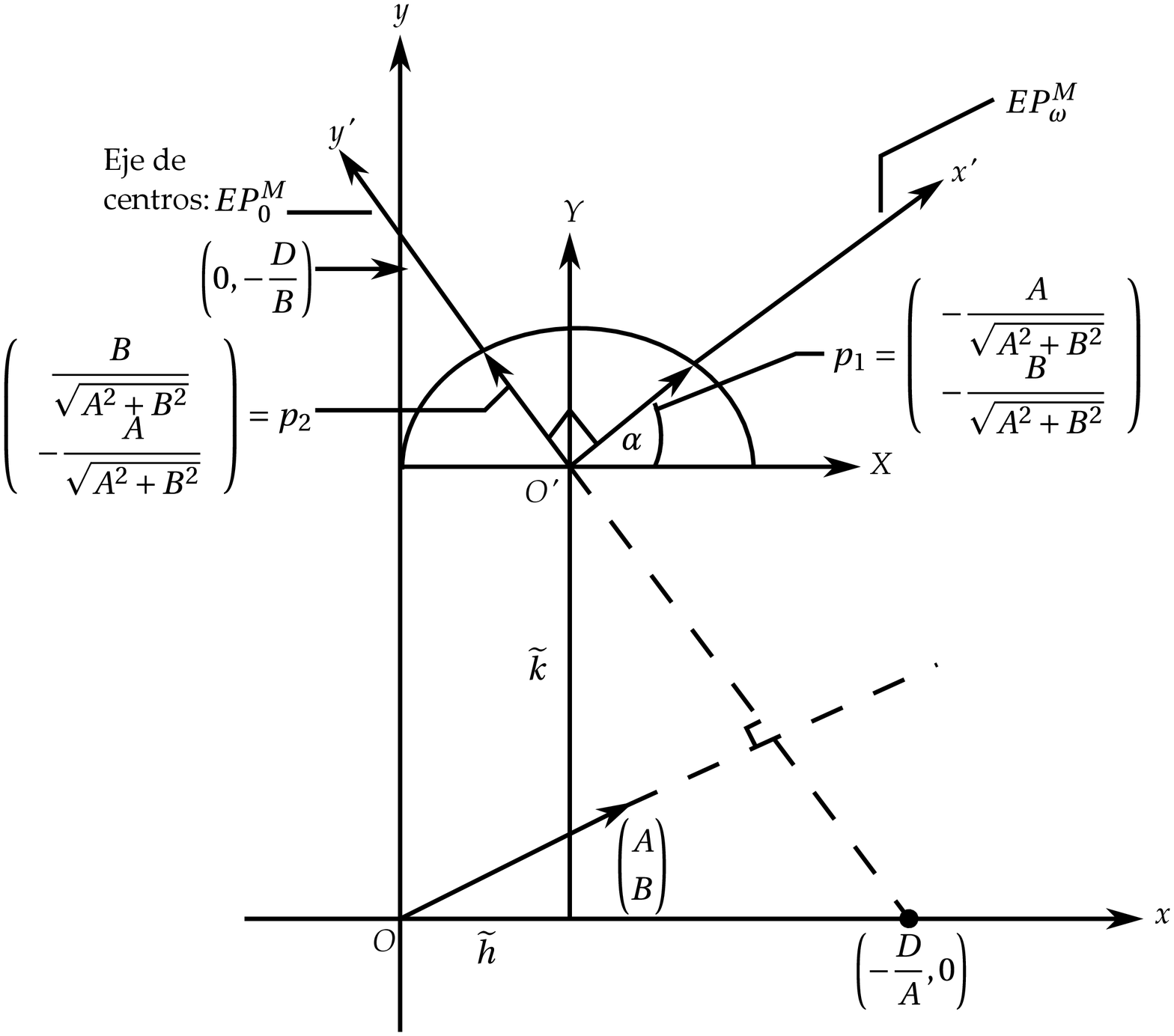}\\
\end{center}
\end{figure}
\newpage
El espectro de $M$ lo ordenamos así:
\begin{align*}
\lambda(M)&={\omega,0}\\
EP_\omega^M\ni p_1&=\left(\begin{array}{c}-\dfrac{A}{\sqrt{A^2+B^2}}\\-\dfrac{B}{\sqrt{A^2+B^2}}\end{array}\right)\\
EP_0^M\ni p_2&=\left(\begin{array}{c}\dfrac{B}{\sqrt{A^2+B^2}}\\-\dfrac{A}{\sqrt{A^2+B^2}}\end{array}\right); {p_1,p_2}:\text{Base ortonormal},\,\, \alpha=\arctan\dfrac{B}{A}=\text{ángulo que giran los ejes}\\
P&=\left(\begin{array}{ccc}-\dfrac{A}{\sqrt{A^2+B^2}}&+&\dfrac{B}{\sqrt{A^2+B^2}}\\-\dfrac{B}{\sqrt{A^2+B^2}}&-&-\dfrac{A}{\sqrt{A^2+B^2}}\end{array}\right)=[I]_{ij}^{p_1,p_2};\,\,\dbinom{X}{Y}=P\dbinom{X'}{Y'}:\text{ecuación de la rotación}
\end{align*}
Después de aplicar el Teorema Espectral a la cónica definida por [\ref{69}], la ecuación de esta$\diagup x'y'$ es: $$\left(\begin{array}{cc}x'&y'\end{array}\right)\left(\begin{array}{cc}\omega&0\\0&0\end{array}\right)\dbinom{x'}{y'}+f(\widetilde{h},\widetilde{k})=0.$$
O sea
\begin{equation}\label{70}
\omega x'^2+f(\widetilde{h},\widetilde{k})=0\hspace{0.5cm}\therefore\hspace{0.5cm}x'^2=-\dfrac{f(\widetilde{h},\widetilde{k})}{\omega}
\end{equation}
Surge ahora una pregunta:
?`Si hubiesemos elegido otro punto $O''(\widetilde{\widetilde{h}},\widetilde{\widetilde{k}})$ en el eje de centros y se hubiese realizado la reducción se obtendría que la ecuación del lugar$\diagup x''y''$ sería: $x''^2=-\dfrac{\widetilde{\widetilde{h}},\widetilde{\widetilde{k}}}{\omega}$?
\begin{figure}[ht!]
\begin{center}
  \includegraphics[scale=0.5]{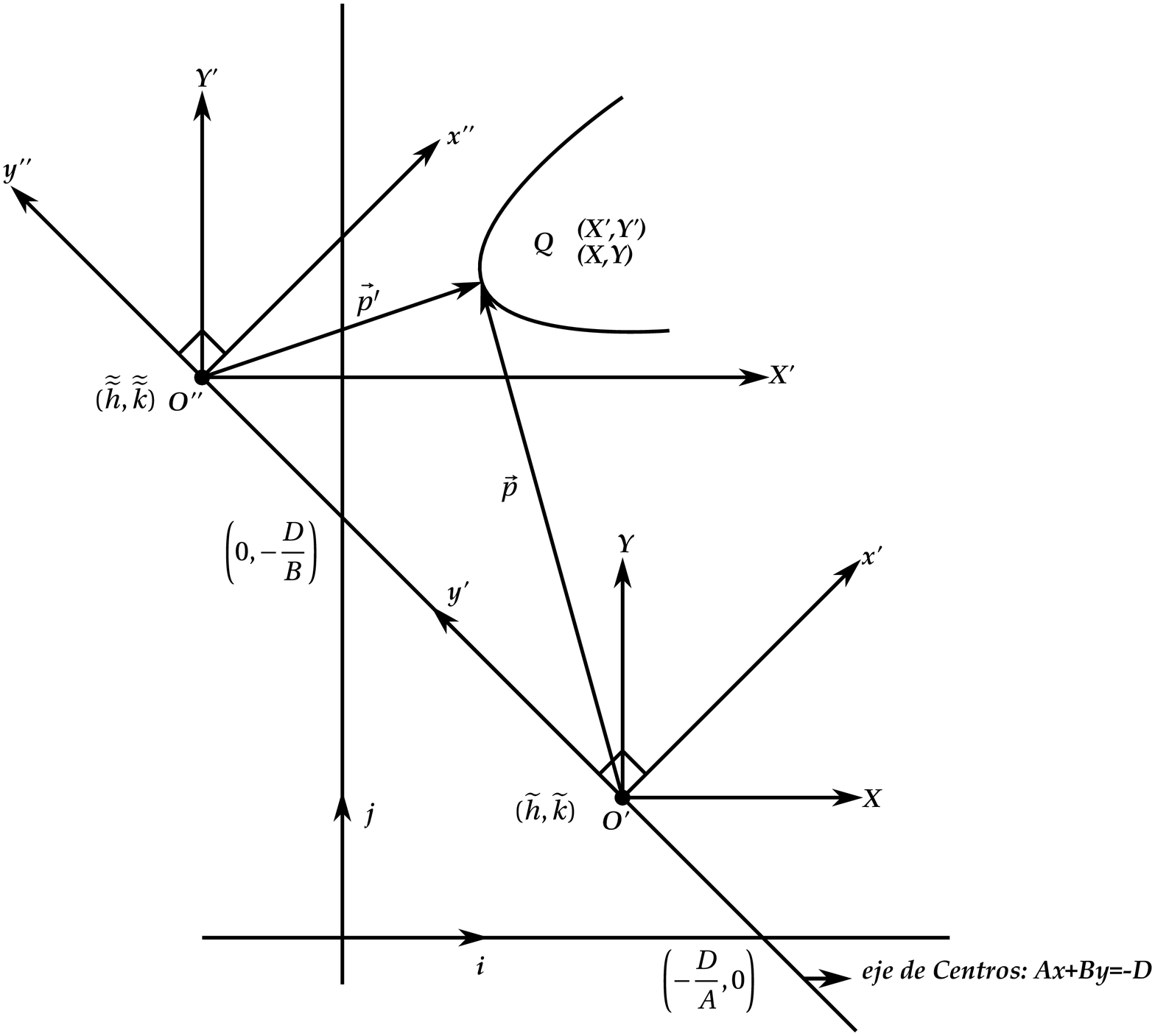}\\
\end{center}
\end{figure}
\newline
Vamos a demostrar que $f(\widetilde{\widetilde{h}},\widetilde{\widetilde{k}})=f(\widetilde{h},\widetilde{k})$ con lo que quedará probado que la ecuación [\ref{70}] del lugar no depende del punto elegido en el eje de centros lo cual sugiere que, por comodidad en los cálculos, para encontrar la ecuación del lugar debe utilizarse preferiblemente los puntos donde el eje de centros corte al eje $x$ o $y.$ En las cónicas con infinitos centros hay un invariante adicional ya que $\forall(\widetilde{h},\widetilde{k})$ que esté en el eje de centros, $f(\widetilde{h},\widetilde{k}):\text{cte.}$\\
Ecuación del lugar$\diagup X'Y'$ con origen en $O'':$
\begin{equation}\label{71}
\left(\begin{array}{cc}X'&Y'\end{array}\right)\left(\begin{array}{cc}A&B\\B&C\end{array}\right)\dbinom{X'}{Y'}+f(\widetilde{\widetilde{h}},\widetilde{\widetilde{k}})=0
\end{equation}
Ecuación del lugar$\diagup XY$ con origen en $O':$
\begin{equation}\label{72}
\left(\begin{array}{cc}X&Y\end{array}\right)\left(\begin{array}{cc}A&B\\B&C\end{array}\right)\dbinom{X}{Y}+f(\widetilde{h},\widetilde{k})=0
\end{equation}
Si $Q$ es un punto del lugar, $\vec{p}=\vec{O'O''}+\vec{p'}$ y $[\vec{p}]_{ij}=[\vec{p'}]_{ij}+[\vec{O'O''}]_{ij}$. O sea que $$\dbinom{X}{Y}=\dbinom{X'}{Y'}+\lambda p_2$$ y al transponer, $$\left(\begin{array}{cc}X&Y\end{array}\right)=\left(\begin{array}{cc}X'&Y'\end{array}\right)+\lambda\left(\begin{array}{ccc}\dfrac{B}{\sqrt{A^2+B^2}}&-&\dfrac{A}{\sqrt{A^2+B^2}}\end{array}\right)$$ que llevamos a [\ref{72}] $$\left(\left(\begin{array}{cc}X'&Y'\end{array}\right)+\lambda\left(\begin{array}{ccc}\dfrac{B}{\sqrt{A^2+B^2}}&-&\dfrac{A}{\sqrt{A^2+B^2}}\end{array}\right)\right)\left(\begin{array}{cc}A&B\\B&C\end{array}\right)\left(\dbinom{X'}{Y'}+\lambda\left(\begin{array}{ccc}\dfrac{B}{\sqrt{A^2+B^2}}\\-\dfrac{A}{\sqrt{A^2+B^2}}\end{array}\right)\right)f(\widetilde{h},\widetilde{k})=0$$
Teniendo en cuenta [\ref{71}]
\begin{align*}
&-f(\widetilde{\widetilde{h}},\widetilde{\widetilde{k}})+\lambda\left(\begin{array}{cc}X&Y\end{array}\right)\cancel{\left(\begin{array}{cc}A&B\\B&C\end{array}\right)}\left(\begin{array}{ccc}\dfrac{B}{\sqrt{A^2+B^2}}\\-\dfrac{A}{\sqrt{A^2+B^2}}\end{array}\right)+\\
&+\lambda\left(\begin{array}{ccc}\dfrac{B}{\sqrt{A^2+B^2}}&-&\cancel{\dfrac{A}{\sqrt{A^2+B^2}}}\end{array}\right)\left(\begin{array}{cc}A&B\\B&C\end{array}\right)\dbinom{X'}{Y'}+\\
&+\lambda^2\left(\begin{array}{ccc}\dfrac{B}{\sqrt{A^2+B^2}}&-&\cancel{\dfrac{A}{\sqrt{A^2+B^2}}}\end{array}\right)\left(\begin{array}{cc}A&B\\B&C\end{array}\right)\left(\begin{array}{ccc}\dfrac{B}{\sqrt{A^2+B^2}}\\-\dfrac{A}{\sqrt{A^2+B^2}}\end{array}\right)+f(\widetilde{h},\widetilde{k})=0
\end{align*}
Los tres últimos sumandos se anulan porque $\left(\begin{array}{cc}\dfrac{B}{\sqrt{}}\\-\dfrac{A}{\sqrt{}}\end{array}\right)$ está en el $EP_0^M.$\\
Así que $$-f(\widetilde{\widetilde{h}},\widetilde{\widetilde{k}})+f(\widetilde{h},\widetilde{k})=0\hspace{0.5cm}\therefore\hspace{0.5cm}f(\widetilde{\widetilde{h}},\widetilde{\widetilde{k}})=f(\widetilde{h},\widetilde{k}).$$
Consideremos el eje de centros $Ax+By=-D$ y sus interceptos con los ejes $x-y.$
\begin{figure}[ht!]
\begin{center}
  \includegraphics[scale=0.5]{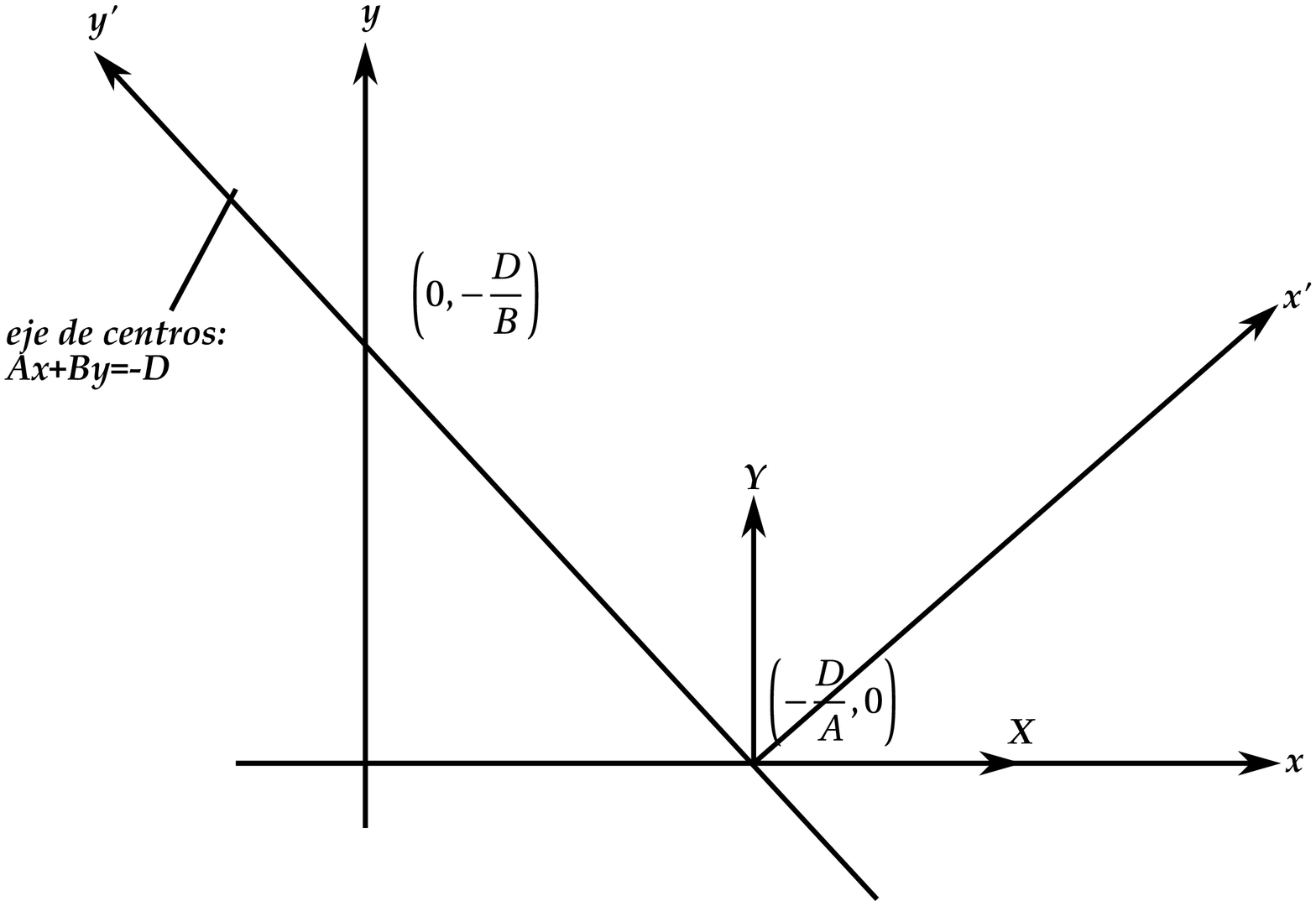}\\
\end{center}
\end{figure}
\newline
Según acaba de probarse, $f\left(-\dfrac{D}{A},0\right)=f\left(0,-\dfrac{D}{B}\right)=f(\widetilde{h},\widetilde{k}),$ cualquiera sea el punto $(\widetilde{h},\widetilde{k})$ que se tiene sobre el eje de centros.\\
Si utilizamos los eje $XY$ y $x'y'$ con origen en $\left(-\dfrac{D}{A},0\right)$ la ecuación de la cónica se escribre así $$x'^2=-\dfrac{f\left(-\dfrac{D}{A},0\right)}{\omega}$$
Ahora,
\begin{align*}
f\left(-\dfrac{D}{A},0\right)&\underset{\uparrow}{=}A{\left(-\dfrac{D}{A}\right)}^2+2D\left(-\dfrac{D}{A}\right)+F\\
&(1)
\end{align*}
\begin{align*}
=\dfrac{D^2}{A}-\dfrac{2D^2}{A}+F=-\dfrac{D^2}{A}+F=\dfrac{AF-D^2}{A}&\underset{\uparrow}{=}\dfrac{M_{22}}{A}\\
&\begin{cases}
\Delta=\left|\begin{array}{ccc}A&B&D\\ B&C&E\\ D&E&F\end{array}\right|; M_{22}=AF-D^2
\end{cases}
\end{align*}
Luego $$x'^2=-\dfrac{f\left(-\dfrac{D}{A},0\right)}{\omega}=-\dfrac{M_{22}}{A\omega}$$
Como $A$ y $\omega$ tienen el mismo signo, $A\omega>0$ y por tanto,\\
$\begin{cases}
\text{Si}\,\,M_{22}<0,\,\,\textit{el lugar consta de dos rectas paralelas al eje $y'$ (el eje de centros).}\\
\text{Si}\,\,M_{22}>0,\,\,\textit{el lugar es $\emptyset$}.\\
\text{Si}\,\,M_{22}=0\,\,\textit{el lugar es el eje $y'$ ó eje de centros.}
\end{cases}$\\
Así que el lugar ó es el eje de centros ó consta de dos rectas $\parallel$s al eje de centros siendo este la $\parallel$ media de las dos rectas.\\
\begin{obser}
Nótese que no hemos demostrado que $M_{22}$ es un invariante. $M_{22}$ es el menor principal $2-2$ de la matriz $\left(\begin{array}{ccc}A&B&D\\ B&C&E\\ D&E&F\end{array}\right)$ y no de otra.\\
La ecuación también puede escribirse así: $x'^2=-f\left(0,-\dfrac{D}{B}\right)$ respecto a los ejes de la figura siguiente:
\begin{figure}[ht!]
\begin{center}
  \includegraphics[scale=0.5]{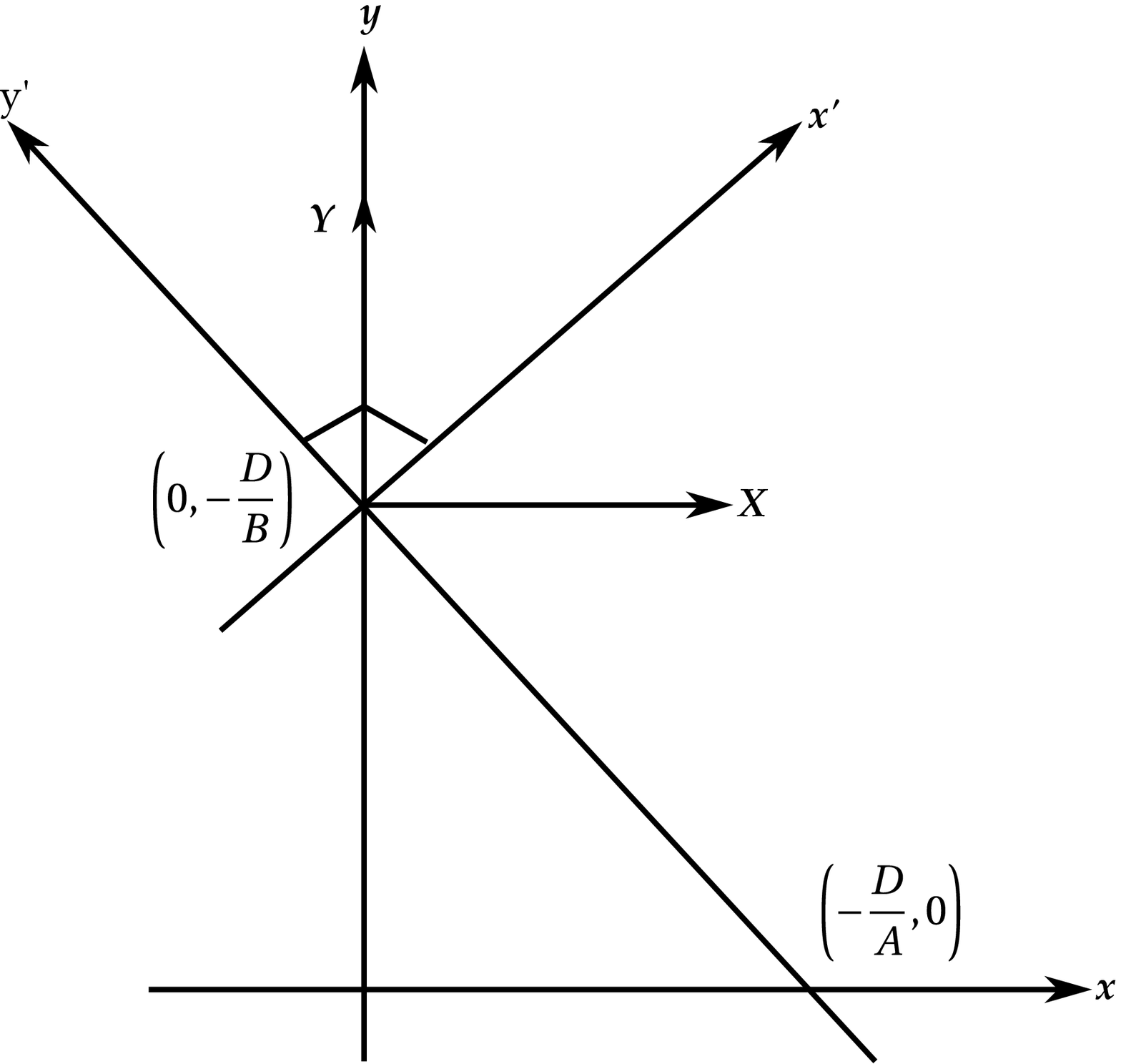}\\
\end{center}
\end{figure}
\newline
$$f\left(0,-\dfrac{D}{B}\right)=C{\left(-\dfrac{D}{B}\right)}^2+2E\left(-\dfrac{D}{B}\right)+F=C{\left(\dfrac{D}{B}\right)}^2-2\dfrac{D}{B}E+F\hspace{0.5cm}\star$$
Ahora, como eje de centros puede tomarse la recta $Ax+By=-D$ ó la recta $Bx+Cy=-E$ ya que son la misma curva y por lo tanto tienen los mismos interceptos con los ejes.\\
Si $x=0, y=-\dfrac{D}{B}$ con la $1^a$ ecuación.\\
Si $x=0, y=-\dfrac{F}{C}$ con la $2^a$ ecuación. Luego $\dfrac{D}{B}=\dfrac{E}{C}$ y al regresar a $\star$,
\begin{align*}
f\left(0,-\dfrac{D}{B}\right)&=C{\left(\dfrac{E}{C}\right)}^2-2\dfrac{E}{C}E+F=\dfrac{E^2}{C}-\dfrac{2E^2}{C}+F\\
&=F-\dfrac{E^2}{C}=\dfrac{CF-E^2}{C}=\dfrac{M_{11}}{C}:\text{menor principal $1-1$ de $\Delta$.}\\
\end{align*}
Así que la ecuación del lugar puede escribirse, respecto a los ejes de la figura anterior así:\\
$x'^2=-\dfrac{M_{11}}{C\omega}$ y como signo de C es el mismo de $\omega, C\omega>0.$\\
Si $M_{11}<0,$ \textit{el lugar consta de dos rectas $\parallel$s al eje $y'$ de la fig. anterior.}\\
Si $M_{11}>0$ \textit{el lugar es $\emptyset$.}\\
Si $M_{11}=0$\textit{el lugar es el eje $y'$ de la fig. anterior.}\\
Para el cálculo de $M_{11}$ téngase en cuenta la obs. anterior.\\
Osea que $M_{11}$ debe tomarse como el menor $1-1$ de la matriz $\left(\begin{array}{ccc}A&B&D\\ B&C&E\\ D&E&F\end{array}\right).$
\end{obser}
Consideremos el caso:
\item[(2)]
\begin{align*}
&B>0\\
&A<0\,\,\text{(Recuérdese que hemos asumido $D>0$.)}
\end{align*}
Primero dibujemos el eje de centros: $Ax+By=-D.$\\
Luego tomemos un punto $O'(\widetilde{h},\widetilde{k})$ en dicho eje.
\begin{figure}[ht!]
\begin{center}
  \includegraphics[scale=0.4]{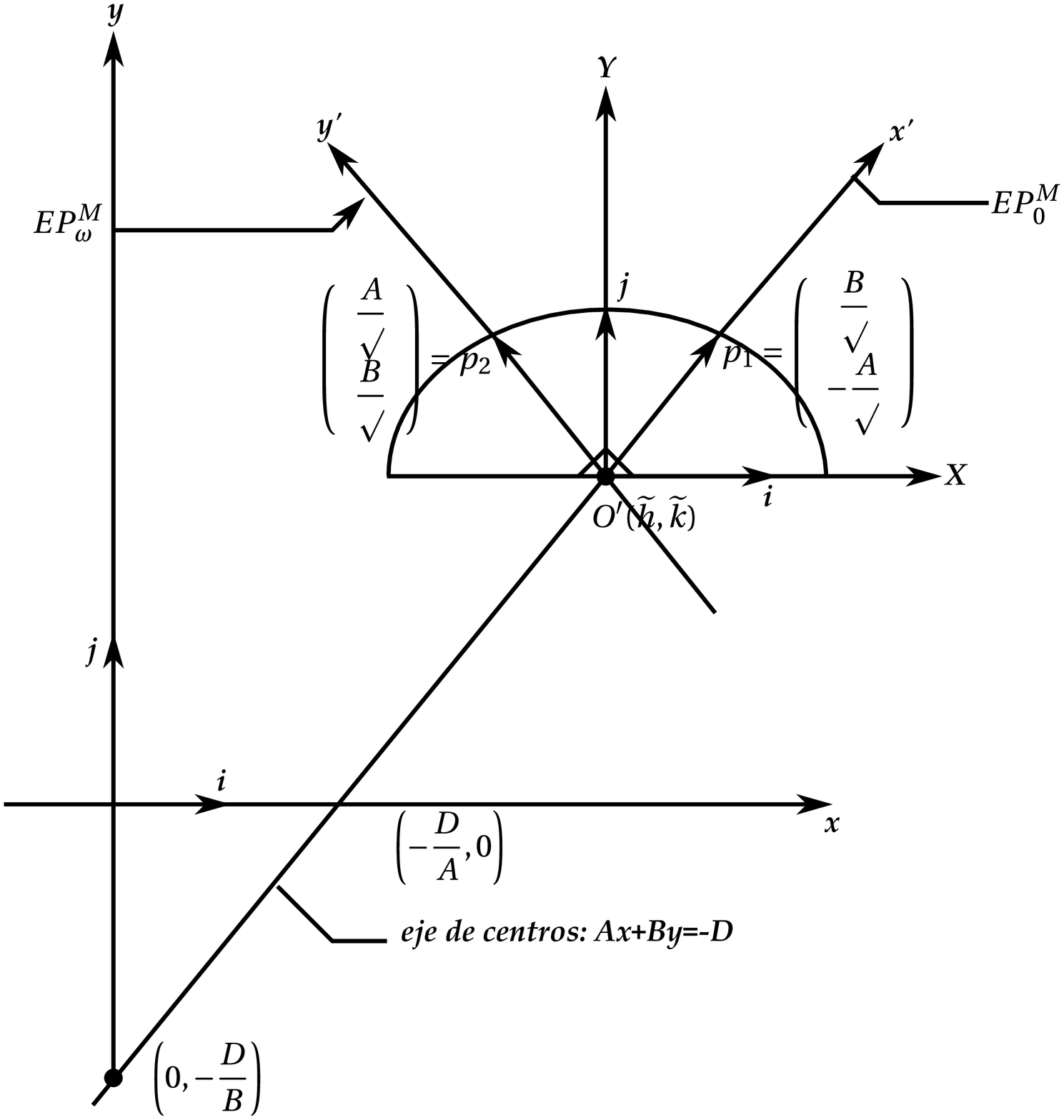}\\
\end{center}
\end{figure}
\newline
Una vez hecha la traslación al punto $O'$, la ecuación de la cónica$\diagup XY$ es $$AX^2+2BXY+CY^2+f(\widetilde{h},\widetilde{k})=0.$$
Dibujamos el vector $$[p_1]_{ij}=\left(\begin{array}{cc}\dfrac{B}{\sqrt{}}\\ -\dfrac{A}{\sqrt{}}\end{array}\right)\in EP_0^M;\hspace{0.5cm}[p_2]_{ij}=\left(\begin{array}{cc}\dfrac{A}{\sqrt{}}\\ \dfrac{B}{\sqrt{}}\end{array}\right)\in EP_\omega^M$$
Nótese que $\vec{p_1}$ es $\parallel$ al eje de centros.\\
El espectro de $M$ lo ordenamos así:
\begin{align*}
\lambda(M)&={0,\omega}\\
P&=\left(\begin{array}{cc}\dfrac{B}{\sqrt{}}&\dfrac{A}{\sqrt{}}\\ -\dfrac{A}{\sqrt{}}&\dfrac{B}{\sqrt{}}\end{array}\right)=[I]_{ij}^{p_1,p_2}:\text{ecuación para la rotación.}\\
\dbinom{X}{Y}&=P\dbinom{x'}{y'}
\end{align*}
Después de aplicar el T. Espectral, la ecuación de la cónica$\diagup x'y'$ es: $$\left(\begin{array}{cc}x'&y'\end{array}\right)\left(\begin{array}{cc}0&0\\0&\omega\end{array}\right)\dbinom{x'}{y'}+f(\widetilde{h},\widetilde{k})=0$$
O sea
\begin{align*}
&\omega y'^2+f(\widetilde{h},\widetilde{k})=0\\
&\therefore\,\,y'^2=-\dfrac{f(\widetilde{h},\widetilde{k})}{\omega}\hspace{0.5cm}\star.
\end{align*}
Nos hacemos la misma pregunta que nos hicimos en el caso $(1)$.\\
?`Si hubiesemos elegido otro punto $O''(\widetilde{\widetilde{h}},\widetilde{\widetilde{k}})$ en el eje de centros y se hubiese realizado la reducción la ecuación de la cónica sería $y''^2=-\dfrac{f(\widetilde{\widetilde{h}},\widetilde{\widetilde{k}})}{\omega}$?
\begin{figure}[ht!]
\begin{center}
  \includegraphics[scale=0.5]{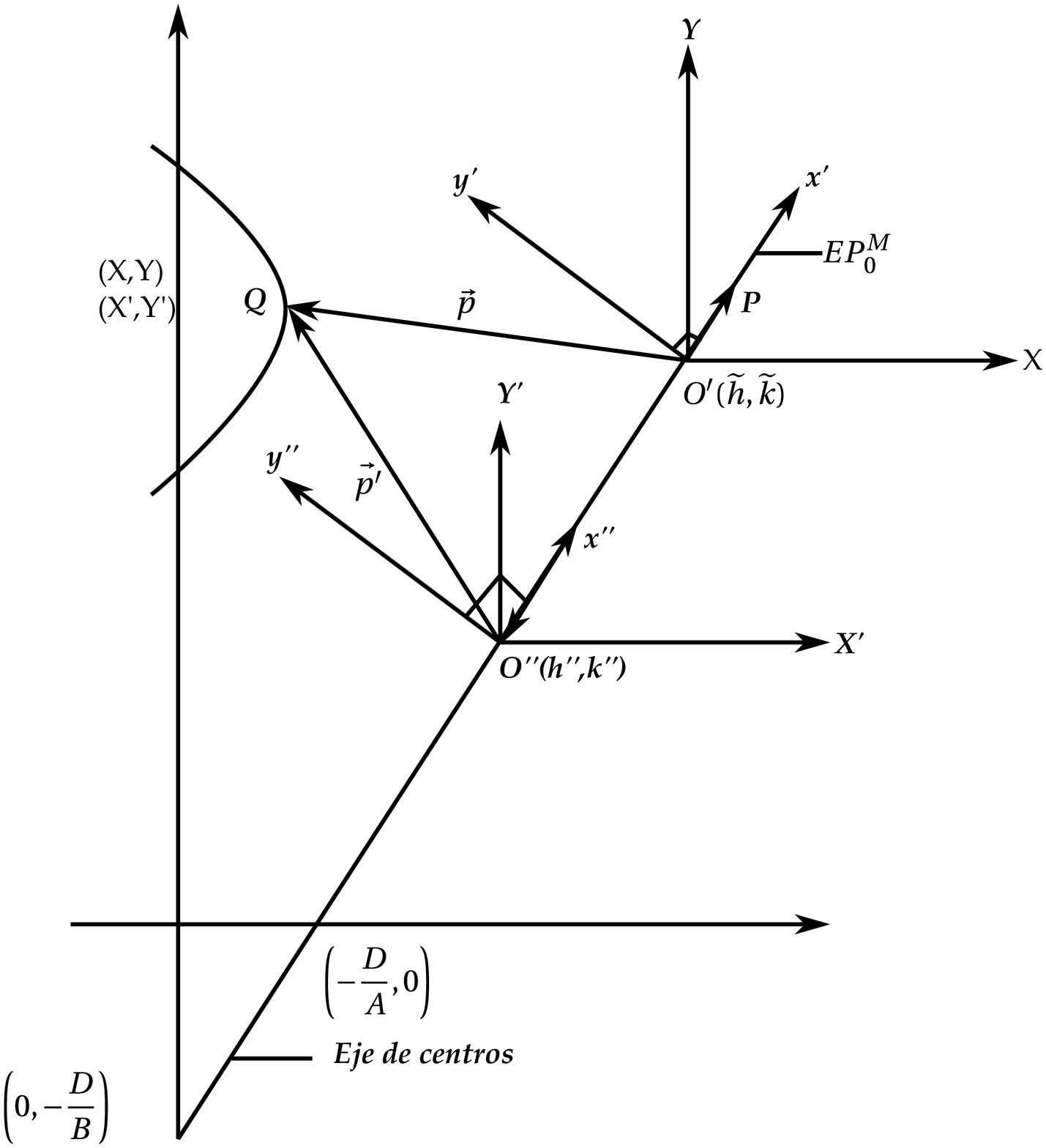}\\
\end{center}
\end{figure}
\newpage
Vamos a demostrar que $f(\widetilde{\widetilde{h}},\widetilde{\widetilde{k}})=f(\widetilde{h},\widetilde{k})$ con lo que quedará demostrado que la ecuación $\star$ del lugar es independiente del punto elegido en el eje de centros.\\
Ecuación del lugar$\diagup X'Y'$ con origen en $O'':$
\begin{equation}\label{73}
\left(\begin{array}{cc}X'&Y'\end{array}\right)\left(\begin{array}{cc}A&B\\B&C\end{array}\right)\dbinom{X'}{Y'}+f(\widetilde{\widetilde{h}},\widetilde{\widetilde{k}})=0
\end{equation}
Ecuación del lugar$\diagup XY$ con origen en $O':$
\begin{equation}\label{74}
\left(\begin{array}{cc}X'&Y'\end{array}\right)\left(\begin{array}{cc}A&B\\B&C\end{array}\right)\dbinom{X'}{Y'}+f(\widetilde{h},\widetilde{k})=0
\end{equation}
Si $Q$ es un punto del lugar, $\vec{\rho}=\vec{O'O''}+\vec{\rho_1}$ y $$[\vec{\rho}]_{ij}=[\vec{\rho_1}]_{ij}+[\vec{O'O''}]_{ij}.$$ O sea que $$\dbinom{X}{Y}=\dbinom{X'}{Y'}+\lambda[p_1]_{ij}$$ con $p_1\in EP_0^M,$ y al transponer, $$\left(\begin{array}{cc}X&Y\end{array}\right)=\left(\begin{array}{cc}X'&Y'\end{array}\right)+\lambda\left(\begin{array}{cc}\dfrac{B}{\sqrt{}}&-\dfrac{A}{\sqrt{}}\end{array}\right)$$ que llevamos a [\ref{74}]:
\begin{align*}
\left(\left(\begin{array}{cc}X'&Y'\end{array}\right)+\lambda\left(\begin{array}{cc}\dfrac{B}{\sqrt{}}&-\dfrac{A}{\sqrt{}}\end{array}\right)\right)\left(\begin{array}{cc}A&B\\B&C\end{array}\right)\left(\dbinom{X'}{Y'}+\lambda\left(\begin{array}{cc}\dfrac{B}{\sqrt{}}\\-\dfrac{A}{\sqrt{}}\end{array}\right)\right)+f(\widetilde{\widetilde{h}},\widetilde{\widetilde{k}})&=0\\
\left(\left(\begin{array}{cc}X'&Y'\end{array}\right)\left(\begin{array}{cc}A&B\\B&C\end{array}\right)+\lambda\left(\begin{array}{cc}\dfrac{B}{\sqrt{}}&-\dfrac{A}{\sqrt{}}\end{array}\right)\left(\begin{array}{cc}A&B\\B&C\end{array}\right)\right)\left(\dbinom{X'}{Y'}+\lambda\left(\begin{array}{cc}\dfrac{B}{\sqrt{}}\\-\dfrac{A}{\sqrt{}}\end{array}\right)\right)+f(\widetilde{\widetilde{h}},\widetilde{\widetilde{k}})&=0
\end{align*}
\begin{align*}
&\underset{\parallel\leftarrow[\ref{73}]}{\underbrace{\left(\begin{array}{cc}X'&Y'\end{array}\right)\left(\begin{array}{cc}A&B\\B&C\end{array}\right)\dbinom{X'}{Y'}}}+\lambda\left(\begin{array}{cc}X'&Y'\end{array}\right)\left(\begin{array}{cc}A&B\\B&C\end{array}\right)\cancel{\left(\begin{array}{cc}\dfrac{B}{\sqrt{}}\\-\dfrac{A}{\sqrt{}}\end{array}\right)}+\lambda\cancel{\left(\begin{array}{cc}\dfrac{B}{\sqrt{}}&-\dfrac{A}{\sqrt{}}\end{array}\right)}\left(\begin{array}{cc}A&B\\B&C\end{array}\right)\dbinom{X'}{Y'}+\\
&-f(\widetilde{\widetilde{h}},\widetilde{\widetilde{k}})\\
&+\lambda^2\left(\begin{array}{cc}\dfrac{B}{\sqrt{}}&-\dfrac{A}{\sqrt{}}\end{array}\right)\cancel{\left(\begin{array}{cc}A&B\\B&C\end{array}\right)}\left(\begin{array}{cc}\dfrac{B}{\sqrt{}}\\-\dfrac{A}{\sqrt{}}\end{array}\right)+f(\widetilde{h},\widetilde{k})=0
\end{align*}
Los 3 sumandos intermedios se anulan porque $\left(\begin{array}{cc}\dfrac{B}{\sqrt{}}&-\dfrac{A}{\sqrt{}}\end{array}\right)\in EP_0^M.$\\
Luego
\begin{align*}
&-f(\widetilde{\widetilde{h}},\widetilde{\widetilde{k}})+f(\widetilde{h},\widetilde{k})=0\\
&\therefore\hspace{0.5cm}f(\widetilde{\widetilde{h}},\widetilde{\widetilde{k}})=f(\widetilde{h},\widetilde{k}).
\end{align*}
Según acaba de probarse, $$f\left(-\dfrac{D}{A},0\right)=f\left(0,-\dfrac{D}{B}\right)=f(\widetilde{h},\widetilde{k})\hspace{0.5cm}\text{cualquiera sea el punto $(\widetilde{h},\widetilde{k})$}.$$ que se tome en el eje de centros.\\
Si utilizamos eje $XY$ y $x'y'$ con origen en $\left(-\dfrac{D}{A},0\right),$ figura que sigue, la ecuación de la cónica$\diagup x'y'$ se escribe así: $$y'^2=-\dfrac{f(-\dfrac{D}{A},0)}{\omega}.$$
\newpage
\begin{figure}[ht!]
\begin{center}
  \includegraphics[scale=0.5]{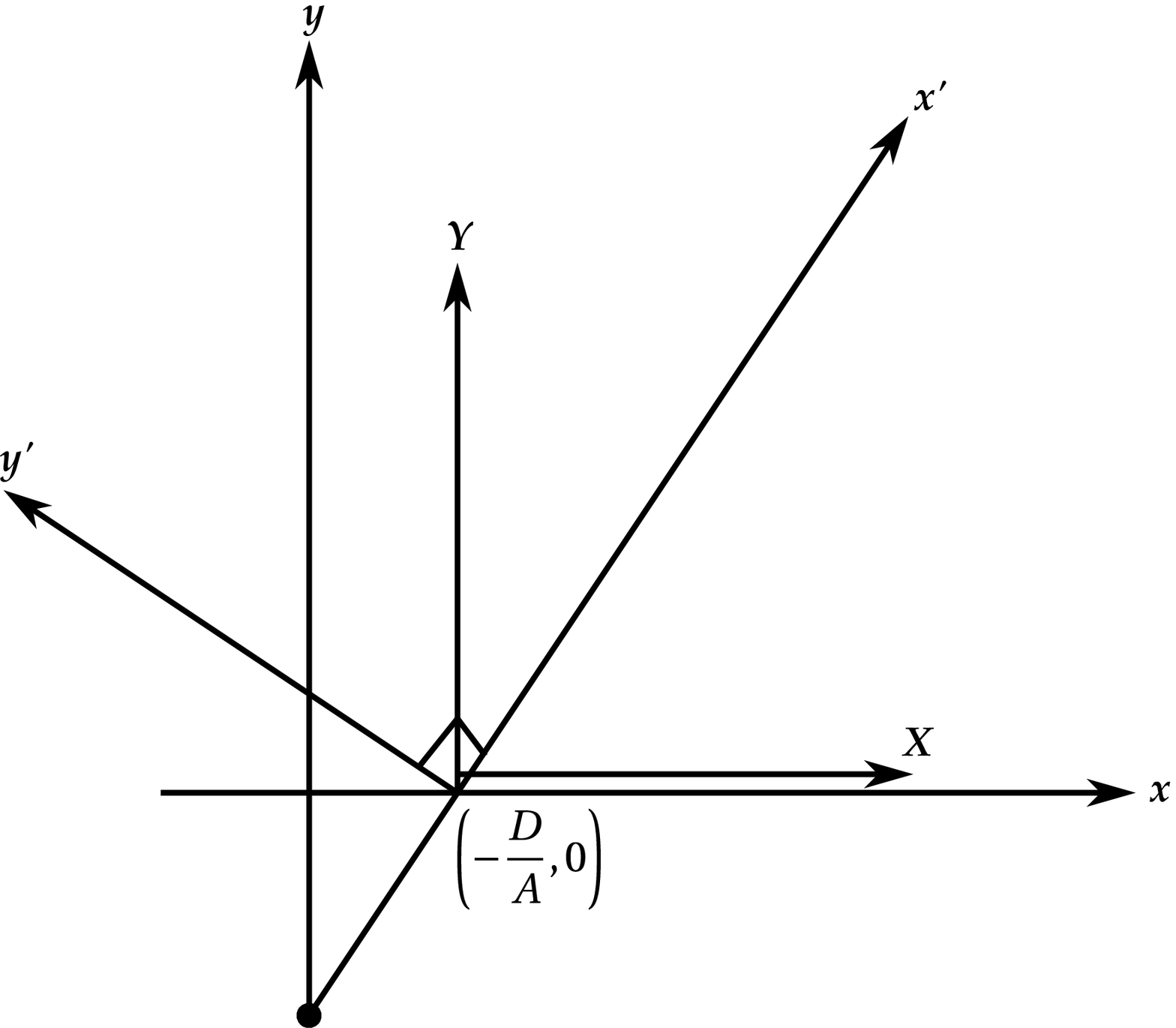}\\
\end{center}
\end{figure}
Ahora, $f\left(-\dfrac{D}{A},0\right)=\dfrac{M_{22}}{A}$ donde $M_{22}=\dfrac{AF-D^2}{A}:\text{menor pal. de orden dos de la matriz}$$$\left(\begin{array}{ccc}A&B&D\\B&C&E\\D&E&F\end{array}\right)$$
Luego $$y'^2=-\dfrac{M_{22}}{\omega}.$$
Como $A$ y $\omega$ tienen el mismo signo, $A\omega>0$ y por tanto,\\
Si $M_{22}<0,$ el lugar consta de dos rectas $\parallel$s al eje $x'$ de la figura anterior (o eje de centros).\\
Si $M_{22}>0,$ el lugar es $\emptyset.$\\
Si $M_{22}=0,$ el lugar es el eje $x'.$ El lugar consta entonces de dos rectas $\parallel$s el eje de centros siendo el eje de centros la paralela media de dos rectas, o eje de centros.\\
La ecuación de la cónica también puede escribirse así: $$y'^2=-f\left(0,-\dfrac{D}{B}\right)$$ respecto a los ejes de la figura siguiente:
\begin{figure}[ht!]
\begin{center}
  \includegraphics[scale=0.4]{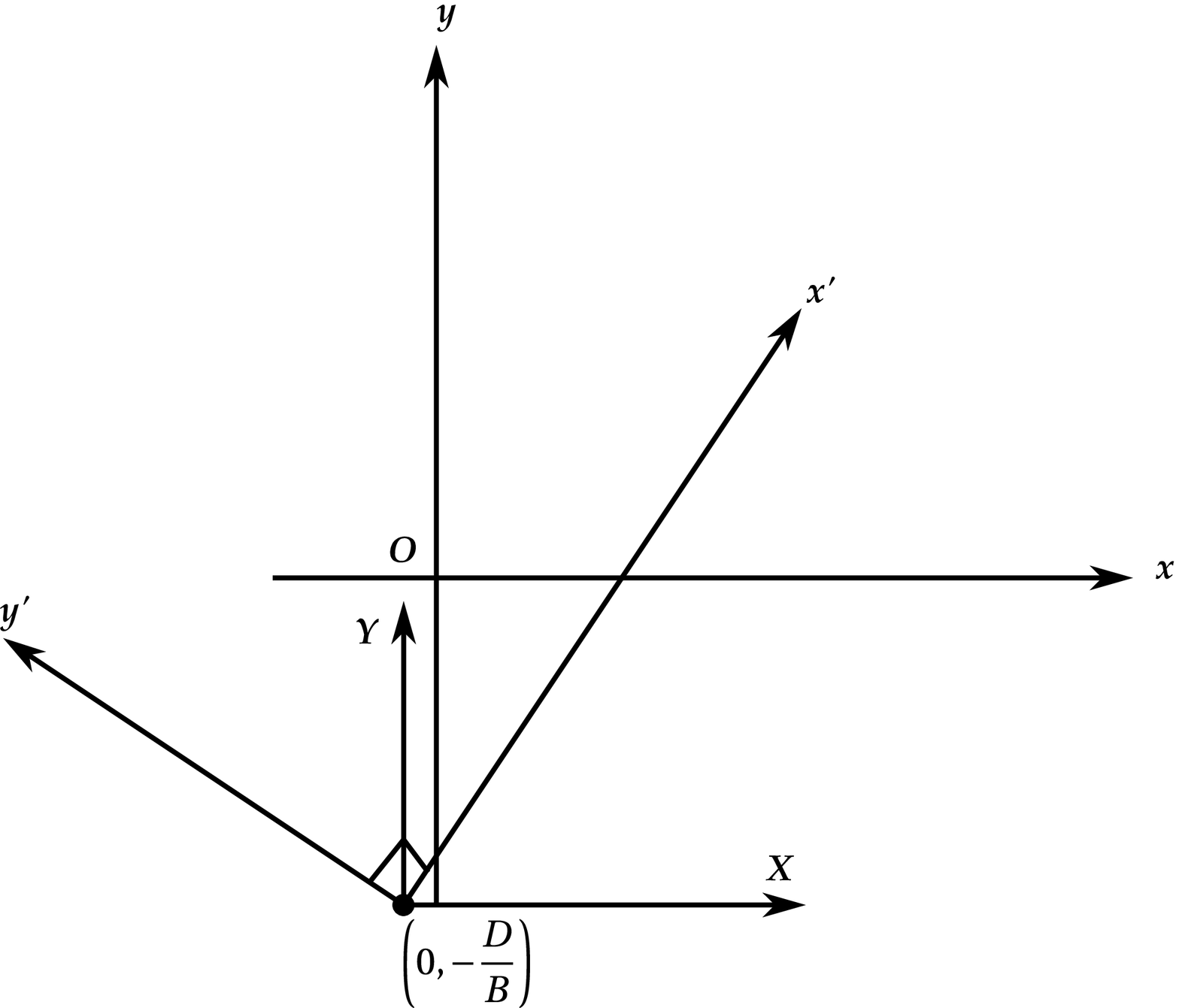}\\
\end{center}
\end{figure}
$$f\left(0,-\dfrac{D}{B}\right)=C{\left(\dfrac{D}{B}\right)}^2-{\left(\dfrac{D}{B}\right)}^2E+F\hspace{0.5cm}\star$$
Ahora, como eje de centros puede tomarse la recta $Ax+By=-D$ o la recta $Bx+Cy=-E$ ya que ambas rectas son la misma curva y por lo tanto tienen los mismos interceptos con los ejes.\\
Si $x=0,$ y $y=-\dfrac{D}{B}$ en la $1^a$ ecuación.\\
Si $x=0,$ y $y=-\dfrac{E}{C}$ en la $2^a$ ecuación. Luego $\dfrac{D}{B}=\dfrac{E}{C}$ y al regresar a $\star,$
\begin{align*}
f\left(0,-\dfrac{D}{B}\right)&=C{\left(\dfrac{E}{C}\right)}^2-2\dfrac{E}{C}\cdot E+F=\dfrac{E^2}{C}-2\dfrac{E^2}{C}+F\\
&=F-\dfrac{E^2}{C}=\dfrac{CF-E^2}{C}=\dfrac{M_{11}}{C}
\end{align*}
donde $M_{11}$ es el menor pal. $1-1$ de orden 2 de la matriz $\left(\begin{array}{ccc}A&B&D\\B&C&E\\D&E&F\end{array}\right).$\\
Así que la ecuación del lugar puede escribirse, respecto a los ejes $x'y'$ de la figura anterior así: $$y'^2=-\dfrac{M_{11}}{C\omega}$$ y como $\text{signo C\,\,$=$\,\,signo\,\,$\omega$,\,\,C$\omega>0$}.$ Luego,\\
Si $M_{11}<0,$ el lugar consta de dos rectas $\parallel$s al eje $x'.$\\
Si $M_{11}>0,$ el lugar es $\emptyset.$\\
Si $M_{11}=0,$ el lugar es el eje $x'.$\\
Regresemos a la página 89.\\
Los otros casos, $3)\hspace{0.5cm}\begin{array}{cc}B<0\\A>0\end{array}\hspace{0.5cm}\text{y}\hspace{0.5cm}\begin{array}{cc}B>0\\A>0\end{array}$\\
se analizan como los casos $1)$ y $2).$\\
Luego de realizar una rotación de los ejes $xy$ a un punto $O'$ de la línea de centros, se consigue la ecuación del lugar$\diagup XY.$\\
Seguidamente se aplica el T. Espectral y se obtiene la ecuación del lugar$\diagup x'y'$. Se halla de nuevo que el lugar\\
$\begin{cases}
\text{Son dos rectas $\parallel$s al eje de centros.}\\
\emptyset.\\
\text{ó el eje de centros.}
\end{cases}$\\
Regresemos al caso (1).\\
Si $M_{22}<0,$ el lugar consta de dos rectas paralelas $\mathscr{L}_1$ y $\mathscr{L}_2$ de ecuaciones$\diagup x'y'$\\
$\begin{cases}
\mathscr{L}_1: x'=+\sqrt{-\dfrac{M_{22}}{A\omega}}\\
\mathscr{L}_2: x'=-\sqrt{-\dfrac{M_{22}}{A\omega}}\hspace{0.5cm}(\text{Recuérdese que el Sign A$=$Sign$\omega$.})
\end{cases}$\\
donde $M_{22}=Af\left(-\dfrac{D}{A},0\right).$
\begin{figure}[ht!]
\begin{center}
  \includegraphics[scale=0.5]{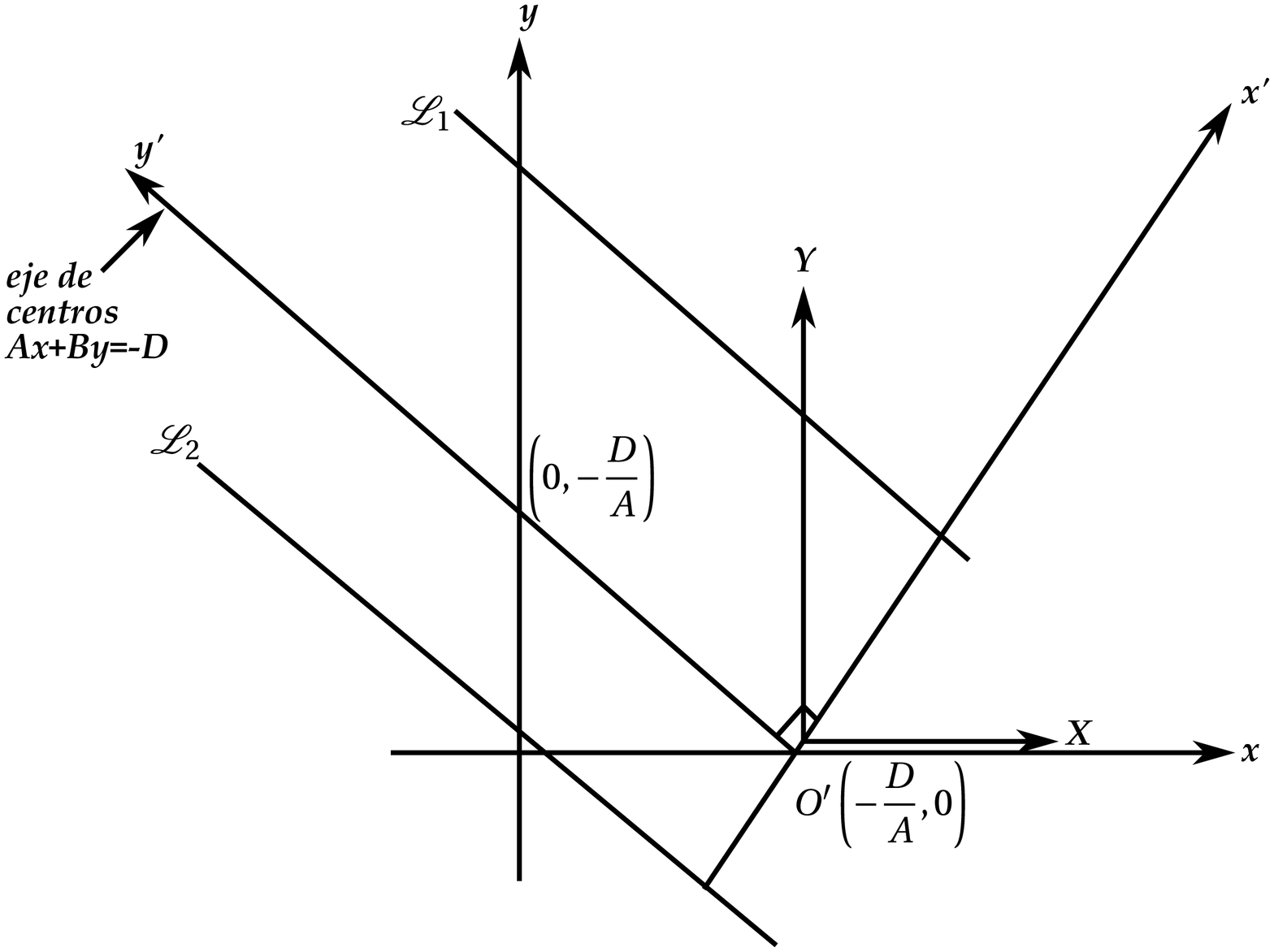}\\
\end{center}
\end{figure}
\item[$\bullet$] Ecuación de la cónica$\diagup XY:$ $$AX^2+2BXY+CY^2+f(-\dfrac{D}{A},0)=0.$$
Al multiplicar por $A,$
$$A^2X^2+2ABXY+ACY^2+A\cdot f(-\dfrac{D}{A},0)=0$$
Pero $AC=B^2.$ Luego la ecuación de la cónica$\diagup XY$ es:
\begin{equation}\label{75}
A^2X^2+2ABXY+B^2Y^2+A\cdot f(-\dfrac{D}{A},0)=0
\end{equation}
\item[$\bullet$] Hallemos ahora las ecuaciones de $\mathscr{L}_1$ y $\mathscr{L}_2\diagup XY.$\\
Ecuaciones de la transformación:
\begin{align*}
\dbinom{X}{Y}=P\dbinom{x'}{y'}&\underset{\uparrow}{=}\left(\begin{array}{ccc}-\dfrac{A}{\sqrt{}}&+&\dfrac{B}{\sqrt{}}\\ -\dfrac{B}{\sqrt{}}&-&\dfrac{A}{\sqrt{}}\end{array}\right)\dbinom{x'}{y'}\\
&\sqrt{}=\sqrt{A^2+B^2}
\end{align*}
Ecuaciones de $\mathscr{L}_1\diagup x'y':\begin{cases}x'=+\sqrt{-\dfrac{M_{22}}{A\omega}}=G\\ y'=g\end{cases}$\\
Ecuaciones de $\mathscr{L}_1\diagup XY: \dbinom{X}{Y}=\left(\begin{array}{ccc}-\dfrac{A}{\sqrt{}}&+&\dfrac{B}{\sqrt{}}\\ -\dfrac{B}{\sqrt{}}&-&\dfrac{A}{\sqrt{}}\end{array}\right)\dbinom{G}{g}$\\
\begin{align*}
X&=-\dfrac{AG}{\sqrt{}}+\dfrac{B}{\sqrt{}}g\\
Y&=-\dfrac{BG}{\sqrt{}}-\dfrac{A}{\sqrt{}}g
\end{align*}
y eliminando el parámetro $g,$ $$AX+BY=-\dfrac{A^2G}{\sqrt{}}-\dfrac{B^2G}{\sqrt{}}=-\dfrac{G(A^2+B^2)}{\sqrt{}}=-G\sqrt{A^2+B^2}.$$
\item[$\bullet$] Ecuación de $\mathscr{L}_2\diagup x'y':x=-\sqrt{\dfrac{M_{22}}{A\omega}}=-G.$\\
Ecuación de $\mathscr{L}_2\diagup XY: \dbinom{X}{Y}=\left(\begin{array}{ccc}-\dfrac{A}{\sqrt{}}&+&\dfrac{B}{\sqrt{}}\\ -\dfrac{B}{\sqrt{}}&-&\dfrac{A}{\sqrt{}}\end{array}\right)\dbinom{-G}{g}$\\
\begin{align*}
X&=+\dfrac{AG}{\sqrt{}}+\dfrac{B}{\sqrt{}}g\\
Y&=+\dfrac{BG}{\sqrt{}}-\dfrac{A}{\sqrt{}}g
\end{align*}
y eliminando a $g,$ $$AX+BY=+G\sqrt{A^2+B^2}$$
Definamos
\begin{align*}
L_1&=AX+BY+G\sqrt{A^2+B^2}\\
L_2&=AX+BY-G\sqrt{A^2+B^2}
\end{align*}
Entonces
\begin{figure}[ht!]
\begin{center}
  \includegraphics[scale=0.5]{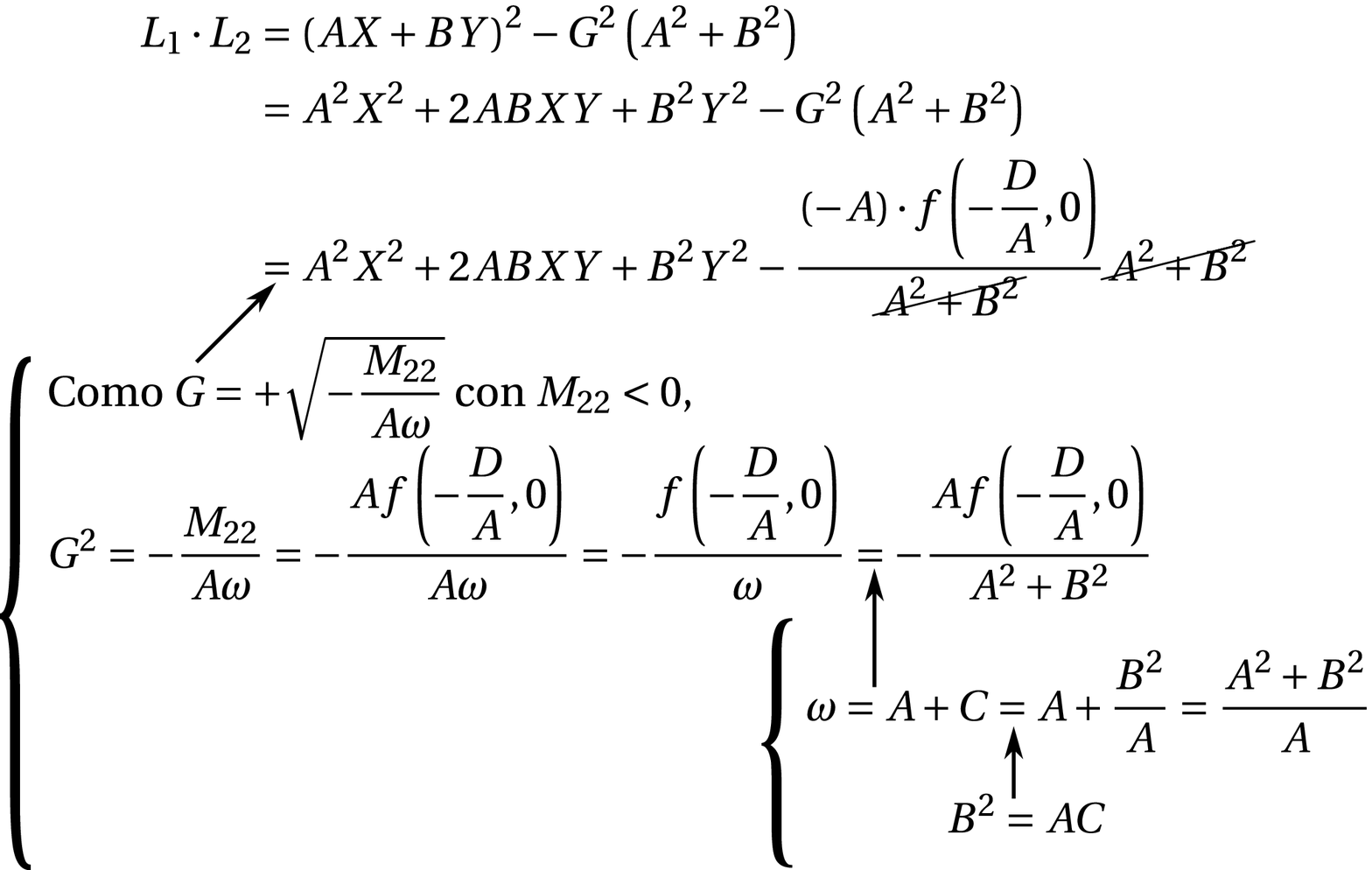}\\
\end{center}
\end{figure}
\newline
$$=A^2X^2+2ABXY+B^2Y^2+Af\left(-\dfrac{D}{A},0\right)$$
O sea que $L_1\cdot L_2\underset{\overset{\uparrow}{[\ref{75}]}}{=}\text{ecuación de la cónica$\diagup XY.$}$\\
Ecuaciones de $\mathscr{L}_1$ y $\mathscr{L}_2\diagup xy.$ Ecuaciones de transformación:
\begin{align*}
x&=X-\dfrac{D}{A}\\
y&=Y\\
\therefore\dbinom{X}{Y}&=\left(\begin{array}{cc}x+\dfrac{D}{A}\\ y\end{array}\right)
\end{align*}
Ecuación de $\mathscr{L}_1\diagup xy.$
\begin{align}
AX+BY&=-G\sqrt{A^2+B^2}\notag\\
A\left(x+\dfrac{D}{A}\right)+By&=-G\sqrt{A^2+B^2}\notag\\
Ax+By+D+G\sqrt{A^2+B^2}=0\label{76}
\end{align}
Ecuación de $\mathscr{L}_2\diagup xy.$
\begin{align}
AX+BY&=+G\sqrt{A^2+B^2}\notag\\
A\left(x+\dfrac{D}{A}\right)+By&=+G\sqrt{A^2+B^2}\notag\\
Ax+By+D-G\sqrt{A^2+B^2}&=0\label{77}
\end{align}
Hagamos
\begin{align*}
\mathscr{L}_1&=Ax+By+D+G\sqrt{A^2+B^2}\\
\mathscr{L}_2&=Ax+By+D-G\sqrt{A^2+B^2};G=\pm\sqrt{-\dfrac{M_{22}}{A\omega}}\\
\mathscr{L}_1\cdot\mathscr{L}_2&=\left(Ax+By+D-G\sqrt{A^2+B^2}\right)\left(Ax+By+D+G\sqrt{A^2+B^2}\right)\\
&={\left(Ax+By+D\right)}^2-G^2\left(A^2+B^2\right)\\
&\underset{\nearrow}{=}A^2x^2+B^2y^2+2ABxy+D^2+2D\left(Ax+By\right)+Af\left(-\dfrac{D}{A},0\right)\\
&\begin{cases}
-G^2\left(A^2+B^2\right)=Af\left(-\dfrac{D}{A},0\right)
\end{cases}\\
&=A\left(Ax^2+2Bxy+\dfrac{B^2}{A}y^2+\dfrac{D^2}{A}+\dfrac{2D}{A}\left(Ax+By\right)+f\left(-\dfrac{D}{A},0\right)\right)\\
&\underset{\uparrow}{=}A\left(Ax^2+2Bxy+Cy^2+\dfrac{D^2}{A}+2Dx+2\dfrac{BD}{A}y+\dfrac{AF-D^2}{A}\right)\\
&\begin{cases}
AC=B^2\therefore\dfrac{B^2}{A}=C\\
f\left(-D/A,0\right)=\dfrac{AF-D^2}{A}
\end{cases}\\
&\underset{\uparrow}{=}A\left(Ax^2+2Bxy+Cy^2+\cancel{\dfrac{D^2}{A}}+2Dx+2Ey+F-\cancel{\dfrac{D^2}{A}}\right)\\
&\begin{cases}
BD=AE\\
\therefore\dfrac{BD}{A}=E
\end{cases}\\
&=A\left(Ax^2+2Bxy+Cy^2Dx+2Ey+F\right)=Af(x,y)
\end{align*}
De nuevo regresamos al caso $(1).$\\
Si $M_{22}=0,$ el lugar es el eje $y'$ o eje de centros: $Ax+By=-D.$
\newpage
\begin{figure}[ht!]
\begin{center}
  \includegraphics[scale=0.5]{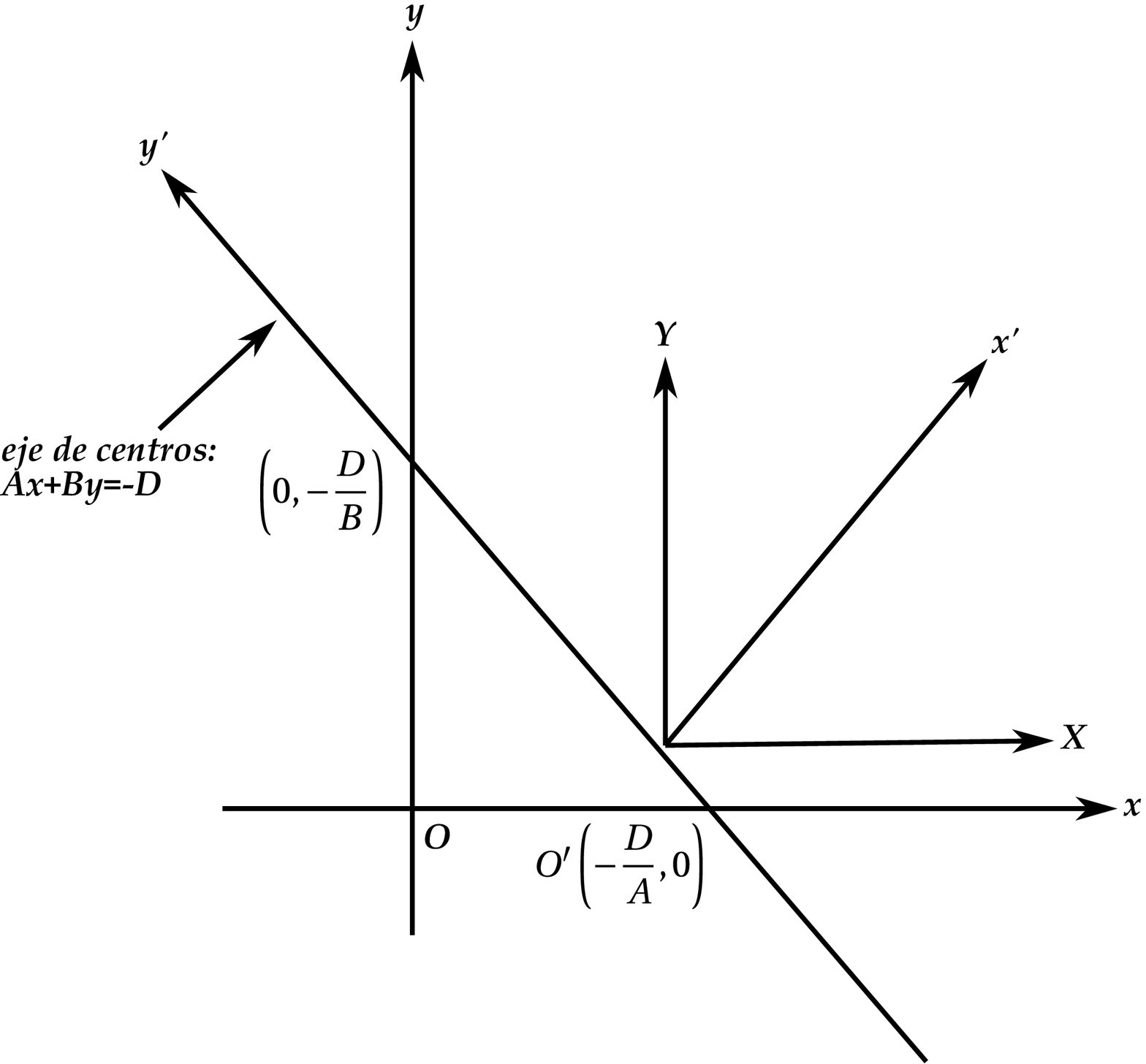}\\
\end{center}
\end{figure}
Vamos a demostrar que $f(x,y)=Ax^2+2Bxy+Cy^2+2Dx+2Ey+F$ factoriza en la forma $Af(x,y)=\mathscr{L}^2$ donde $\mathscr{L}=Ax+By+D.$\\
Ecuación de la cónica$/XY:$
$$A^2X^2+2ABXY+B^2Y^2+Af(-D/A,0)=0$$
\begin{align*}
\mathscr{L}^2&={\left(Ax+By+D\right)}^2\\
&={\left(Ax+By\right)}^2+D^2+2D\left(Ax+By\right)\\
&=A^2x^2+2ABxy+B^2y^2+D^2+2ADx+2BDy\hspace{0.5cm}\star
\end{align*}
Pero $\dfrac{M_{22}}{A}=f(-D/A,0)$ y como $M_{22}=0, f(-D/A,0)=0.$ O sea que $AF-D^2=0\hspace{0.5cm}\therefore\hspace{0.5cm}D^2=AF.$\\
Y al regresar a $\star,$
\begin{align*}
\mathscr{L}^2&=A^2x^2+2ABxy+B^2y^2+AF+2ADx+2BDy\\
&\underset{\uparrow}{=}A^2x^2+2ABxy+ACy^2+2ADx+2AEy+AF\\
&\begin{cases}
BD=AE\\
B^2=AC
\end{cases}\\
&=A\left(Ax^2+2Bxy+Cy^2+2Dx+2Ey+F\right)\\
&=Af(x,y)
\end{align*}
\end{enumerate}
\begin{ejer}
Consideremos la recta
\begin{align*}
\mathscr{L}&=x-y+1=0\\
{\mathscr{L}}^2&=(x-y+1)(x-y+1)=x^2-2xy+y^2+2x-2y+1=0
\end{align*}
Ahora invertimos los paleles.\\
Consideremos la cónica de ecuación
\begin{equation}\label{78}
x^2-2xy+y^2+2x-2y+1=0
\end{equation}
Vamos a demostrar que el lugar representado por es la recta $\mathscr{L}.$
\end{ejer}
\begin{sol}
\begin{center}
\begin{tabular}{ccc}
$A=1$\\
$2B=-2;B=-1$\\
$C=1$\\
$2D=D; D=1$&$\dbinom{-D}{-E}=\dbinom{-1}{-1}=\dbinom{B}{C}$\\
$2E=-2; E=-1$
\end{tabular}
\end{center}
\begin{figure}[ht!]
\begin{center}
  \includegraphics[scale=0.5]{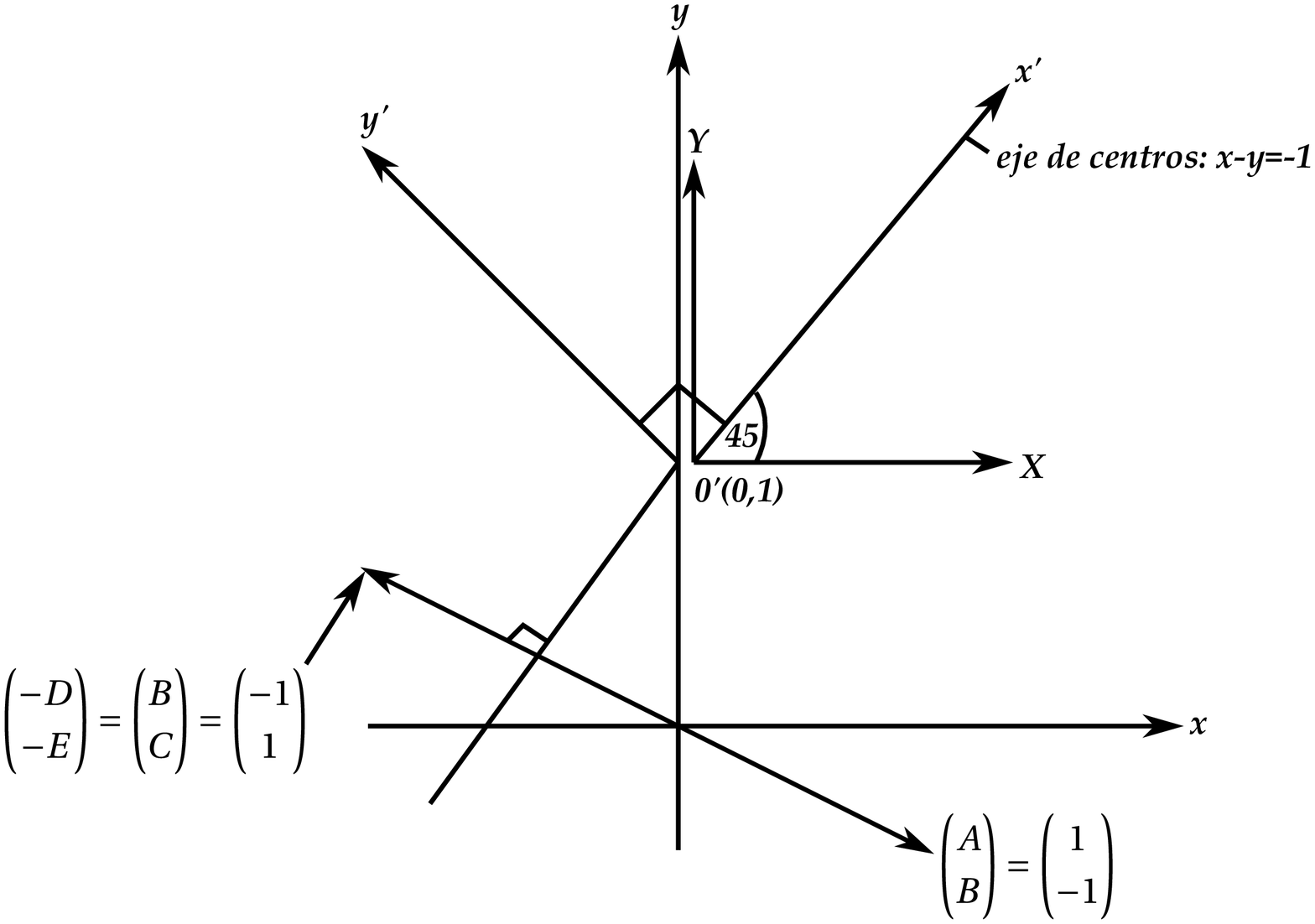}\\
\end{center}
\end{figure}
\underline{Hay $\infty$s centros.}\\
Eje de centros: $Ax+By=-D$ o sea $x-y=-1$ ó $x-y+1=0$ (que es la ecuación de la recta $\mathscr{L}.$)\\
\begin{align*}
M=\left(\begin{array}{cc}1&-1\\-1&1\end{array}\right); \delta&=AC-B^2=1-1=0\\
\omega&=A+C=2
\end{align*}
Como $\Delta=\left|\begin{array}{ccc}1&-1&1\\-1&1&-1\\1&-1&1\end{array}\right|, M_{22}=\left|\begin{array}{cc}1&1\\1&1\end{array}\right|=0.$ \underline{El lugar es una recta.} (El eje de centros.)\\ Tomemos $O'$ en el eje de centros, $O'$ de coord. $(\widetilde{h},\widetilde{k})=(0,1)\diagup xy.$\\
Si hacemos la traslación de los ejes $x-y$ a $O'$, la ecuación de la cónica$\diagup xy.$ es: $$\left(\begin{array}{cc}X&Y\end{array}\right)\left(\begin{array}{ccc}1&-1\\-1&1\end{array}\right)\dbinom{X}{Y}+f(0,1)=0$$
Como
\begin{align}
f(0,1)=0^2-2\cdot0\cdot 1+1^2+2\cdot0-2\cdot1+1&=0,\notag\\
\left(\begin{array}{cc}X&Y\end{array}\right)\left(\begin{array}{ccc}1&-1\\-1&1\end{array}\right)\dbinom{X}{Y}&=0\label{79}
\end{align}
es la ecuación de la cónica$\diagup XY.$
Para seguir transformando a [\ref{79}], consideremos la matriz $M$.\\
$PCM(\lambda)=\lambda^2-\omega+\delta=0.$ Como $\delta=0$ y $\omega=2, \lambda^2-2\lambda=0.$\\
Luego los valores propios de $M$ son 0 y 2.\\
$$EP_0^M=\mathscr{N}(M)=\left\{(u,v)\diagup Au+Bv=0\right\}=\left\{(u,v)\diagup u-v=0\right\}$$
$$(u,v)\in\mathscr{N}(M)\Longleftrightarrow v=u.$$
Todo vector de la forma $u\dbinom{1}{1}=\alpha\dbinom{1/\sqrt{2}}{1/\sqrt{2}}, \alpha\in\mathbb{R}$ está en el $EP_0^M.$
\begin{figure}[ht!]
\begin{center}
  \includegraphics[scale=0.5]{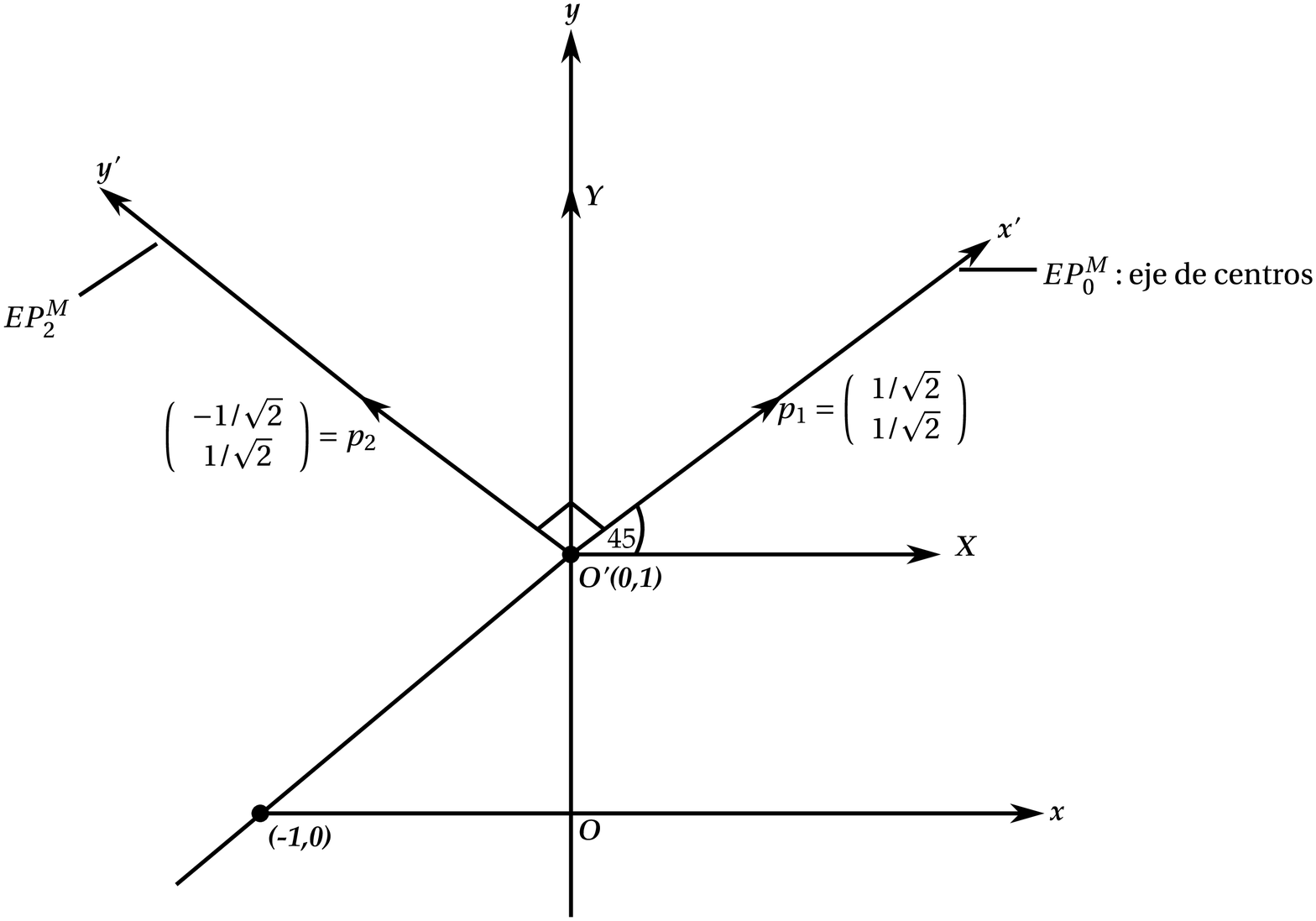}\\
\end{center}
\end{figure}
El espectro de $M$ se organiza así: $\lambda(M)=\left\{0,2\right\}.$ $$P=\left(\begin{array}{cc}1/\sqrt{2}&-1/\sqrt{2}\\1/\sqrt{2}&1/\sqrt{2}\end{array}\right).$$
Después de aplicar el T. Espectral, la ecuación de la cónica$\diagup x'y'$ con origen en $O'$ es: $$\left(\begin{array}{cc}x'&y'\end{array}\right)\left(\begin{array}{cc}0&0\\0&2\end{array}\right)\dbinom{x'}{y'}+f(0,1)=0$$
O sea $2y'^2+0=0,$ i.e., $y'^2=0.$\\
Esto dm. que la cónica es el eje $x'$ de ecuación$\diagup xy: \dfrac{x}{-1}+\dfrac{y}{1}=1,$ o sea $x-y+1=0$ que es $\mathscr{L}.$\\
\framebox[1.2\width]{\parbox[2.6\height]{8.5cm}{Resumiendo, el criterio para identificar las cónicas con $\infty$ s, centros es el sgte:\\
Consideremos la cónica $Ax^2+2Bxy+2Dx+2Ey+F=0$ en la que $A,B,C\neq0.$\\
$\left\{\dbinom{A}{B}\dbinom{B}{C}\right\}$ es L.D., i.e., $\delta=\left|\begin{array}{cc}A&B\\B&C\end{array}\right|=0.$ y que $\dbinom{-D}{-E}\in Sg\left\{\dbinom{A}{B}\right\}=Sg\left\{\dbinom{B}{C}\right\}.$ O sea que $\left\{\dbinom{A}{B},\dbinom{-D}{-E}\right\}$ es L.D., y por lo tanto, $\delta=\left|\begin{array}{cc}A&-D\\B&-E\end{array}\right|=BD-AE=0.$\\
y $\left\{\dbinom{B}{C},\dbinom{-D}{-E}\right\}$ es L.D., y, $\delta=\left|\begin{array}{cc}B&-D\\C&-E\end{array}\right|=CD-BE=0.$\\
La cónica tiene $\infty$s centros. Además, $\Delta=0.$}}\\
\newline
\framebox[1.2\width]{\parbox[2.6\height]{8.5cm}{Si $M_{22}<0,$ el lugar consta de dos rectas $\parallel$s.\\ Si $M_{22}>0,$ el lugar es $\emptyset$.\\ Si $M_{22}=0$ el lugar es una recta.\\
Otro criterio es éste:\\
Si $M_{11}<0$: dos rectas $\parallel$s.\\ Si $M_{11}>0:\emptyset.$\\ Si $M_{11}=0$: una recta.}}
\newline
Los casos en que algunos de los coeficientes $A,B,C$ son cero son:\\
$\begin{tabular}{|c|c|c|}
A&0&0\\ \hline
&&\\
\end{tabular}$
La cónica es $Ax^2+2Dx+2Ey+F=0.$\\
Si $E=0,$ \underline{hay $\infty$s centros.}\\
Si $M_{22}<0:$ dos rectas $\parallel$s.\\
Si $M_{22}=0:$ una recta.\\
Si $M_{22}>0:\emptyset.$\\
Si $E\neq 0,$ la cónica \underline{no tiene centro.}\\
Sea cualquiera $D(D=0\,\,\text{ó}\,\,D\neq0)$ se obtiene una parábola.\\
$\begin{tabular}{|c|c|c|}
0&0&C\\ \hline
&&\\
\end{tabular}$
La ecuación de la cónica es $Cy^2+2Dx+2Ey+F=0.$\\
Si $D=0,$ \underline{hay $\infty$s centros.}\\
Si $M_{11}<0:$ dos rectas $\parallel$s.\\
Si $M_{11}=0:$ una recta.\\
Si $M_{11}>0:\emptyset.$\\
Si $D\neq 0,$ la cónica \underline{no tiene centro.}\\
Sea cualquiera $E(E=0\,\,\text{ó}\,\,E\neq0)$ se obtiene una parábola.\\
\end{sol}
\begin{ejer}
Consideremos las rectas $\parallel$s:
\begin{align*}
\mathscr{L}_1&=3x-2y+2=0\hspace{0.5cm}\text{y}\hspace{0.5cm}\\
\mathscr{L}_2&=3x-2y+1=0\\
\mathscr{L}_1\cdot\mathscr{L}_2&=(3x-2y+2)(3x-2y+1)\\
&=9x^2-12xy+4y^2+3x-2y+6x-4y+2\\
&=9x^2-12xy+4y^2+9x-6y+2.
\end{align*}
Ahora vamos a cambiar los papeles.\\
Consideremos la cónica $9x^2-12xy+4y^2+9x-6y+2=0.$\\
Vamos a reducirla y a dm. que la cónica factoriza así: $$9x^2-12xy+4y^2+9x-6y+2=(3x-3y+2)(3x-2y+1).$$
Consideremos, pués, la cónica
\begin{equation}\label{80}
9x^2-12xy+4y^2+9x-6y+2=0
\end{equation}
\begin{tabular}{ccc}
$A=9$\\
$2B=-12; B=-6$\\
$C=4$\\
$2D=9; D=\dfrac{9}{2}$\\
$2E=-6; E=-3$\\
$F=2$
\end{tabular}\\
Estudiemos los centros de [\ref{80}].\\
\begin{align*}
\dbinom{A}{B}=\dbinom{9}{-6}=3\dbinom{3}{2}\\
\dbinom{B}{C}=\dbinom{-6}{4}=-2\dbinom{3}{-2}\\
\dbinom{-D}{-E}=\dbinom{-9/2}{3}=-\dfrac{3}{2}\dbinom{3}{-2}
\end{align*}
\newpage
\begin{figure}[ht!]
\begin{center}
  \includegraphics[scale=0.4]{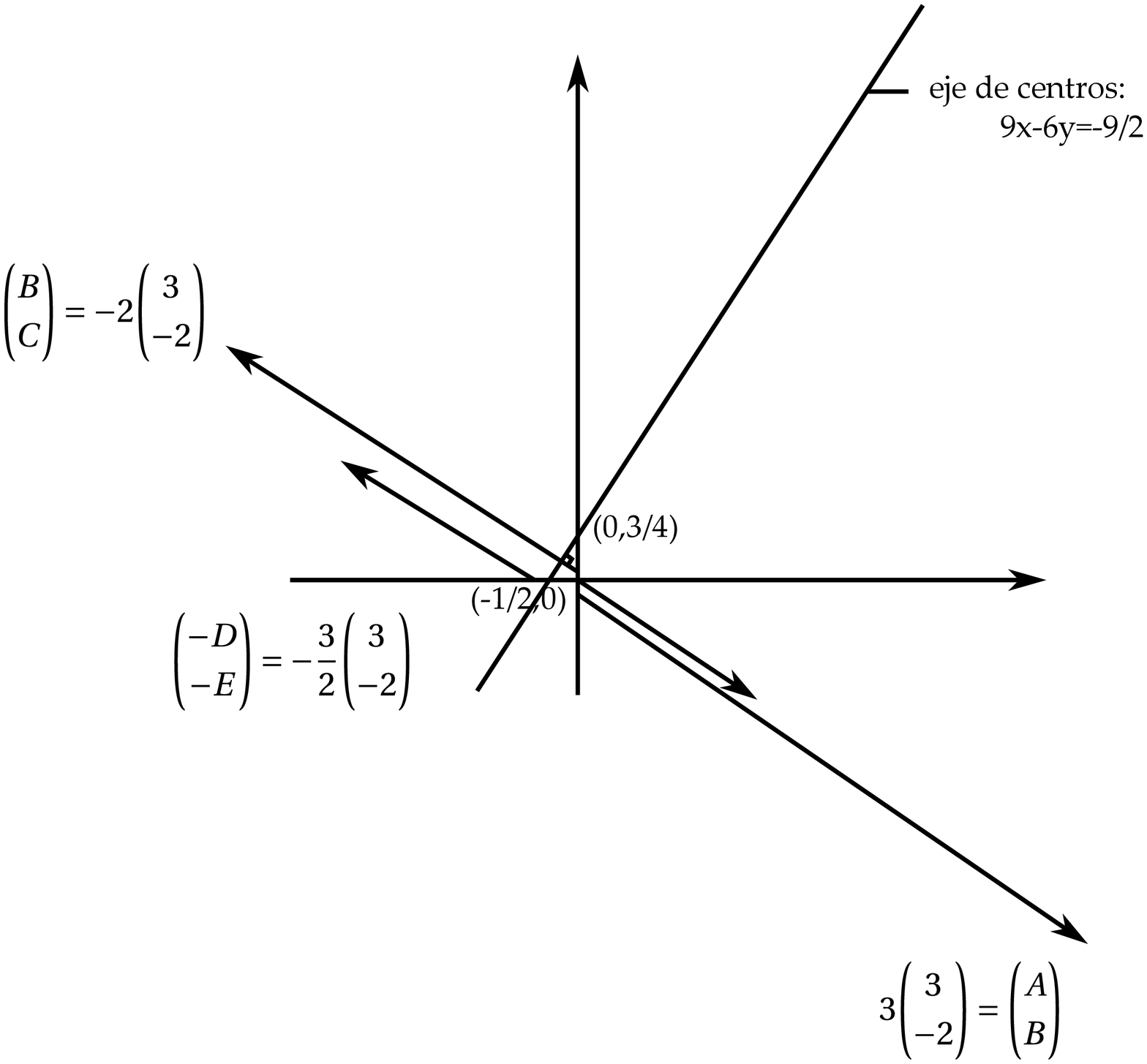}\\
\end{center}
\end{figure}
La cónica tiene $\infty$s centros.\\
El eje de centros es
\begin{align*}
Ax+By&=-D,\,\,\text{o sea,}\\
9x-6y=-9/2
\end{align*}
Si $x=0, y=3/4;\,\,\text{Si}\,\,y=0, x=-1/2.$
\begin{figure}[ht!]
\begin{center}
  \includegraphics[scale=0.5]{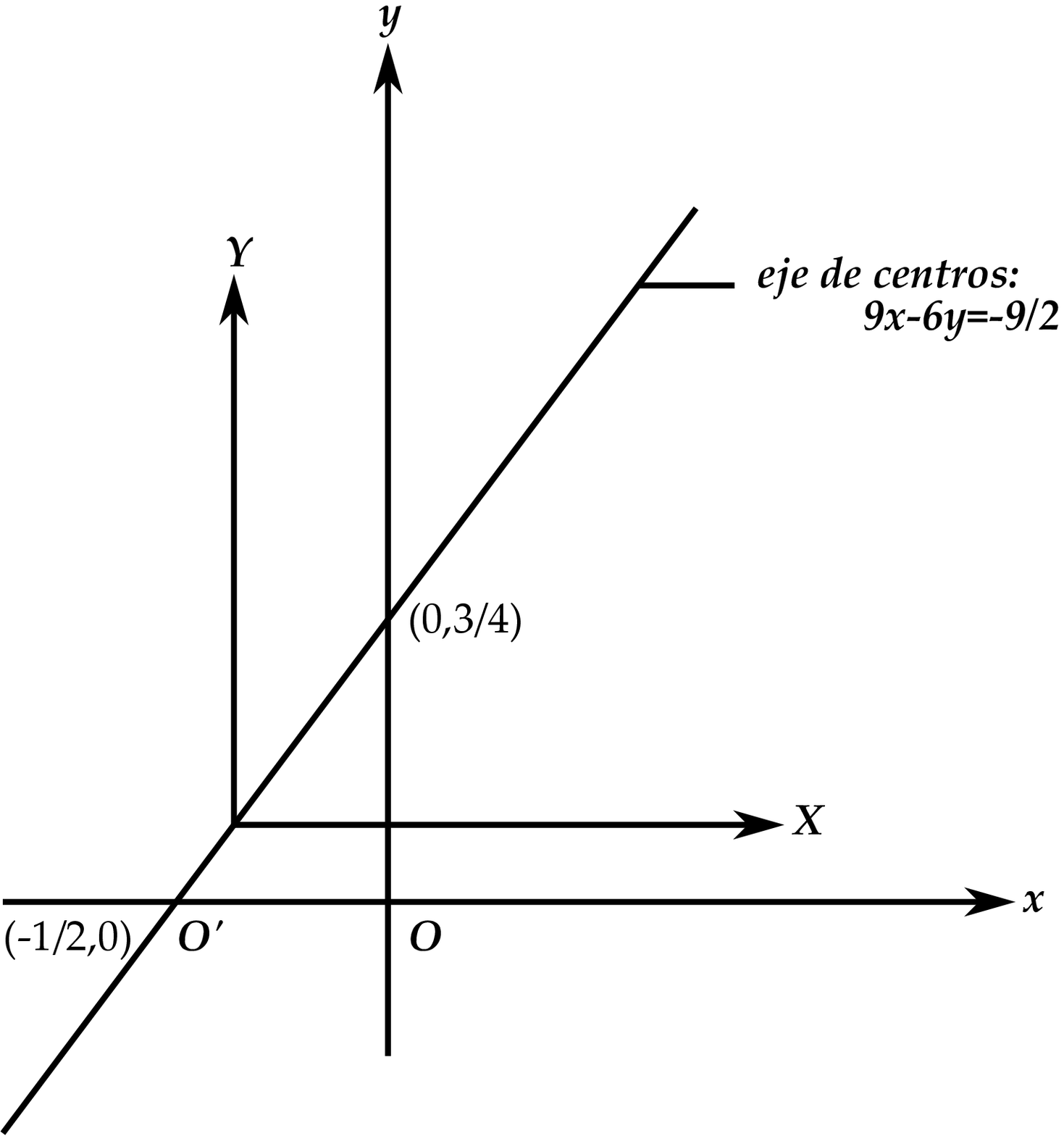}\\
\end{center}
\end{figure}
Al realizar una traslación de ejes al punto $O'(-1/2,0)$ del eje de centros, la ecuación de la cónica$\diagup XY$ es: $$9X^2+12XY+4Y^2+f(-1/2,0)=0$$
\begin{align*}
f(-1/2,0)&=9{\left(-\dfrac{1}{2}\right)}^2+9\left(-\dfrac{1}{2}\right)+2\\
&=\dfrac{9}{4}-\dfrac{9}{2}+\dfrac{2}{1}=\dfrac{9-18+8}{4}=-\dfrac{1}{4}
\end{align*}
Ecu. de la cónica$\diagup XY:$ $$9X^2-12XY+4Y^2-\dfrac{1}{4}=0$$
$M=\left(\begin{array}{cc}9&-6\\-6&4\end{array}\right);\,\,\delta=0;\,\,\omega=9+4=13;\,\,PCM(\lambda)=\lambda^2-\omega\lambda+\delta=\lambda^2-13\lambda=0$\\
Los valores propios de $M$ son 0 y 13.\\
\begin{align*}
EP_0^M&=\left\{\dbinom{u}{v}\diagup\left(\begin{array}{cc}9&-6\\-6&4\end{array}\right)\dbinom{u}{v}=\dbinom{0}{0}\right\}\\
&=\left\{\dbinom{u}{v}\diagup 9u-6v=0\right\}=\left\{\dbinom{u}{v}\diagup 3u-2v=0\right\}\\
&2v=3u\hspace{0.5cm}\therefore\hspace{0.5cm}v=\dfrac{3}{2}u
\end{align*}
Todo vector de la forma $$\left(u,\dfrac{3}{2}u\right)=u\left(1,\dfrac{3}{2}\right)=\alpha(2,3)=\beta\left(\begin{array}{c}2/\sqrt{13}\\ 3/\sqrt{13}\end{array}\right)\,\,\text{con}\,\,\beta\in\mathbb{R}\in EP_0^M.$$
\begin{figure}[ht!]
\begin{center}
  \includegraphics[scale=0.5]{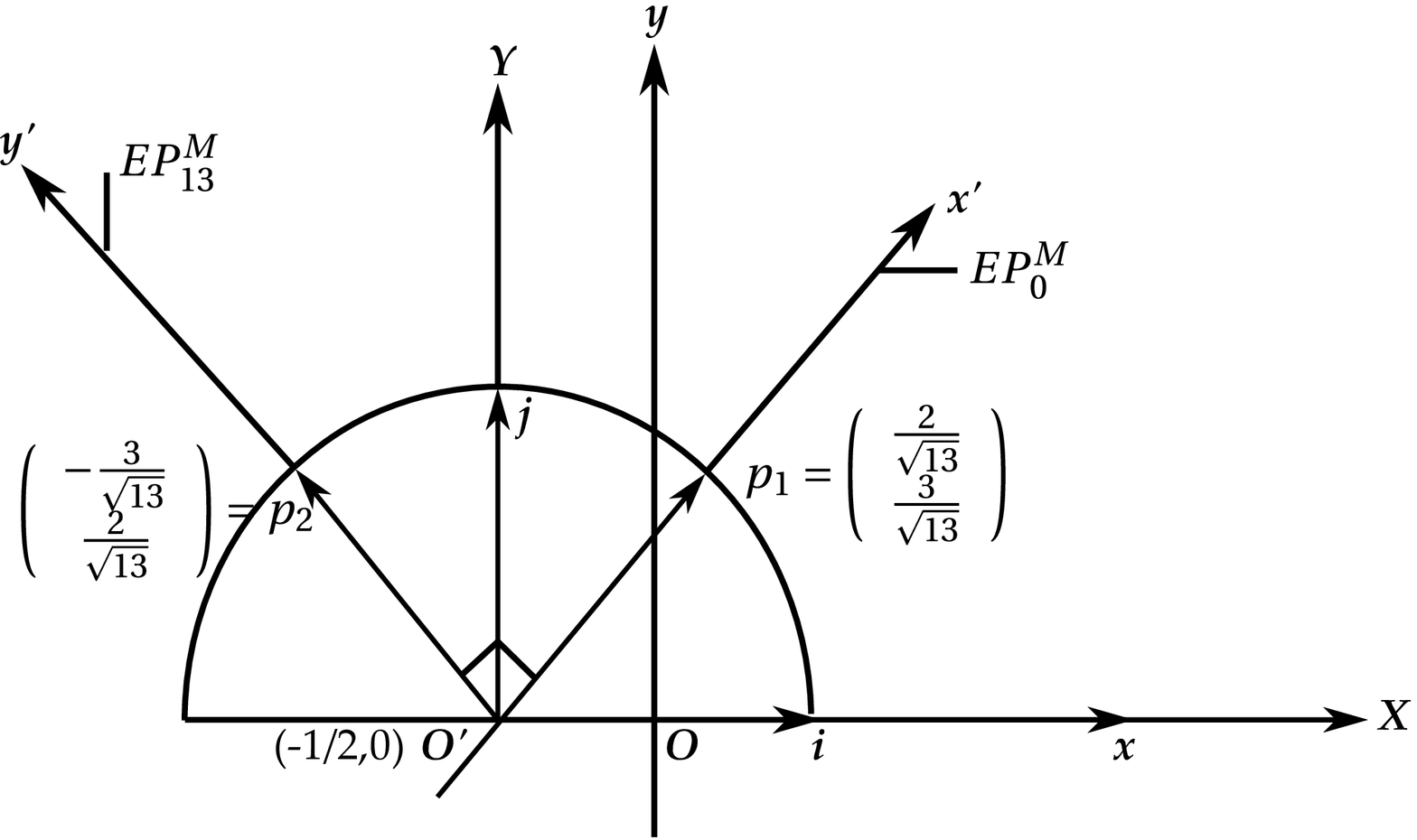}\\
\end{center}
\end{figure}
$\lambda(M)=(0,13);\,\,P=[I]_{ij}^{p_1,p_2}=\left(\begin{array}{cc}\frac{2}{\sqrt{13}}&-\frac{3}{\sqrt{13}}\\ \frac{3}{\sqrt{13}}&\frac{2}{\sqrt{13}}\end{array}\right)$\\
Ecuac. de la cónica$\diagup x'y':$ $$\left(\begin{array}{cc}x'&y'\end{array}\right)\left(\begin{array}{cc}0&0\\0&13\end{array}\right)\dbinom{x'}{y'}-\dfrac{1}{4}=0\hspace{0.5cm}\text{ó}\hspace{0.5cm}13y'^2=\dfrac{1}{4}.$$
El lugar (la cónica) se compone de dos rectas $\mathscr{L}_1$ y $\mathscr{L}_2$ de ecuaciones$\diagup x'y':$
\begin{align*}
\mathscr{L}_1&:y'+\dfrac{\sqrt{13}}{26};\,\,x'=g, g\in\mathbb{R}.\\
\mathscr{L}_2&:y'=-\dfrac{\sqrt{13}}{26};\,\,x'=g, g\in\mathbb{R}.
\end{align*}
Sus ecuaciones$\diagup xy$ se obtienen así: $$\dbinom{X}{Y}=P\dbinom{x'}{y'};\,\,\begin{cases}x=X-1/2\\ y=Y\hspace{0.5cm}\therefore\hspace{0.5cm}\dbinom{X}{Y}=\left(\begin{array}{cc}x+1/2\\ y\end{array}\right)\end{cases}$$
Por lo tanto, las ecuac. de transf. de coor. de $x-y$ a $x'-y'$ son: $$\left(\begin{array}{cc}x+1/2\\ y\end{array}\right)=\left(\begin{array}{cc}\frac{2}{\sqrt{13}}&-\frac{3}{\sqrt{13}}\\ \frac{3}{\sqrt{13}}&\frac{2}{\sqrt{13}}\end{array}\right)\dbinom{x'}{y'}$$
Ec. de $\mathscr{L}_1\diagup xy:$
$$\left(\begin{array}{cc}x+1/2\\ y\end{array}\right)=\left(\begin{array}{cc}\frac{2}{\sqrt{13}}&-\frac{3}{\sqrt{13}}\\ \frac{3}{\sqrt{13}}&\frac{2}{\sqrt{13}}\end{array}\right)\left(\begin{array}{c}g\\ \dfrac{\sqrt{13}}{26}\end{array}\right)$$
\begin{align*}
(3)\hspace{0.5cm}x+\frac{1}{2}&=\frac{2}{\sqrt{13}}g-\frac{3}{26}\\
(2)\hspace{0.5cm}y&=\frac{3}{\sqrt{13}}g+\frac{2}{26}\hspace{0.5cm}\text{Eliminando el parámetro $g$,}
\end{align*}
$$3x+\frac{3}{2}-2y=-\frac{9}{26}-\frac{4}{26}=-\frac{13}{26}=-\frac{1}{2}.$$
Luego $$3x-2y+2=0:\text{ecuac. de $\mathscr{L}_1\diagup xy$}$$
Ec. de $\mathscr{L}_2\diagup xy:$
$$\left(\begin{array}{cc}x+1/2\\ y\end{array}\right)=\left(\begin{array}{cc}\frac{2}{\sqrt{13}}&-\frac{3}{\sqrt{13}}\\ \frac{3}{\sqrt{13}}&\frac{2}{\sqrt{13}}\end{array}\right)\left(\begin{array}{c}g\\ -\dfrac{\sqrt{13}}{26}\end{array}\right)$$
\begin{align*}
(3)\hspace{0.5cm}x+\frac{1}{2}&=\frac{2}{\sqrt{13}}g+\frac{3}{26}\\
(2)\hspace{0.5cm}y&=\frac{3}{\sqrt{13}}g-\frac{2}{26}\hspace{0.5cm}\text{Eliminando el parámetro $g$,}
\end{align*}
$$3x+\frac{3}{2}-2y=\frac{9}{26}+\frac{4}{26}=\frac{13}{26}=\frac{1}{2}.$$
Luego $$3x-2y1=0:\text{ecuac. de $\mathscr{L}_2\diagup xy$}$$
Esto nos dm. que la cónica $9x^2-12xy+4y^2+9x-6y+2=0$ se compone de las rectas $3x-2y+2=0$ y $3x-2y+1=0.$\\
$$9x^2-12xy+4y^2+9x-6y+2=(3x-2y+2)(3x-2y+1)$$
\end{ejer}
\textbf{Problemas. 1.}\\
Empleando los invariantes determine la naturaleza y dibuje las cónicas representadas por las ecuaciones:
\begin{enumerate}
\item $5x^2-4xy+y^2+2x-y=0$\\
\item $3x^2-4xy+y^2+2x-y=0$\\
\item $3x^2-4xy+y^2+15x-6y+7=0$\\
\item $2x^2-7xy+3y^2-9x+7y+4=0$\\
\item $4x^2-12xy+9y^2+4x-5y+3=0$\\
\item $4x^2-12xy+9y^2-8x+12y-7=0$
\end{enumerate}
\textbf{2.} Los siguientes son ejemplos de cónicas degeneradas. Empleando los invariantes diga de que consta y si es posible, factorice la ecuación de la cónica.\\
\begin{enumerate}
\item $6x^2+xy-2y^2+7x-14y-24=0$\\
\item $4x^2+4xy+y^2-2x-y-20=0$\\
\item $x^2+2xy+2y^2-8x-12y+20=0$\\
\item $xy+5x-2x-10=0$\\
\item $6x^2+11xy+3y^2+11x-y-10=0$\\
\item $4x^2+3xy+y^2-10x-2y+8=0$\\
\item $10xy+4x-15y-6=0$\\
\item $4x^2+4xy+y^2-12x-6y+9=0$\\
\item $x^2-4xy+4y^2+2x-4y-3=0$\\
\item $9x^2-6xy+y^2-3x+y-2=0$
\end{enumerate}
\textbf{Ejercicios.}\\
En los siguientes casos se tiene una cónica del tipo $Ax^2+2Bxy+Cy^2+2Dx+2Ey+F=0.$\\
Utilizand los invariantes, pruebe que se trata de una cónica degenerada $\begin{cases}\text{Una recta ó}\\ \text{dos rectas \,\,$\parallel$s\,\, ó que se cortan.}\end{cases}$\\
Una vez definida la naturaleza de la cónica, transforme la ecuación (haga traslaciones ó rotaciones si es necesario) y finalmente dibújela.
\begin{enumerate}
\item[1)] $2x^2+xy-y^2+3y-2=0$\\
\item[2)] $x^2-y^2+x+y=0$\\
\item[3)] $2x^2+xy-2x-y=0$\\
\item[4)] $x^2-2xy+y^2+2x-2y+1=0$\\
\item[5)] $4x^2-4xy+y^2+4x-2y+1=0$
\end{enumerate}
\section[Intersección de una cónica con una recta...]{Intersección de una cónica con una recta. Recta tangente a una cónica por un punto de la curva.}
\begin{enumerate}
\item Vamos a tratar de utilizar la ecuación de incrementos de una cónica para hallar la ecuación de la tangente
a la elipse
\begin{equation}\label{81}
\dfrac{x^2}{a^2}+\dfrac{y^2}{b^2}=1,\hspace{0.5cm}a>b
\end{equation}
en un punto $P(x_0,y_0)$ de la curva.
\begin{figure}[ht!]
\begin{center}
  \includegraphics[scale=0.5]{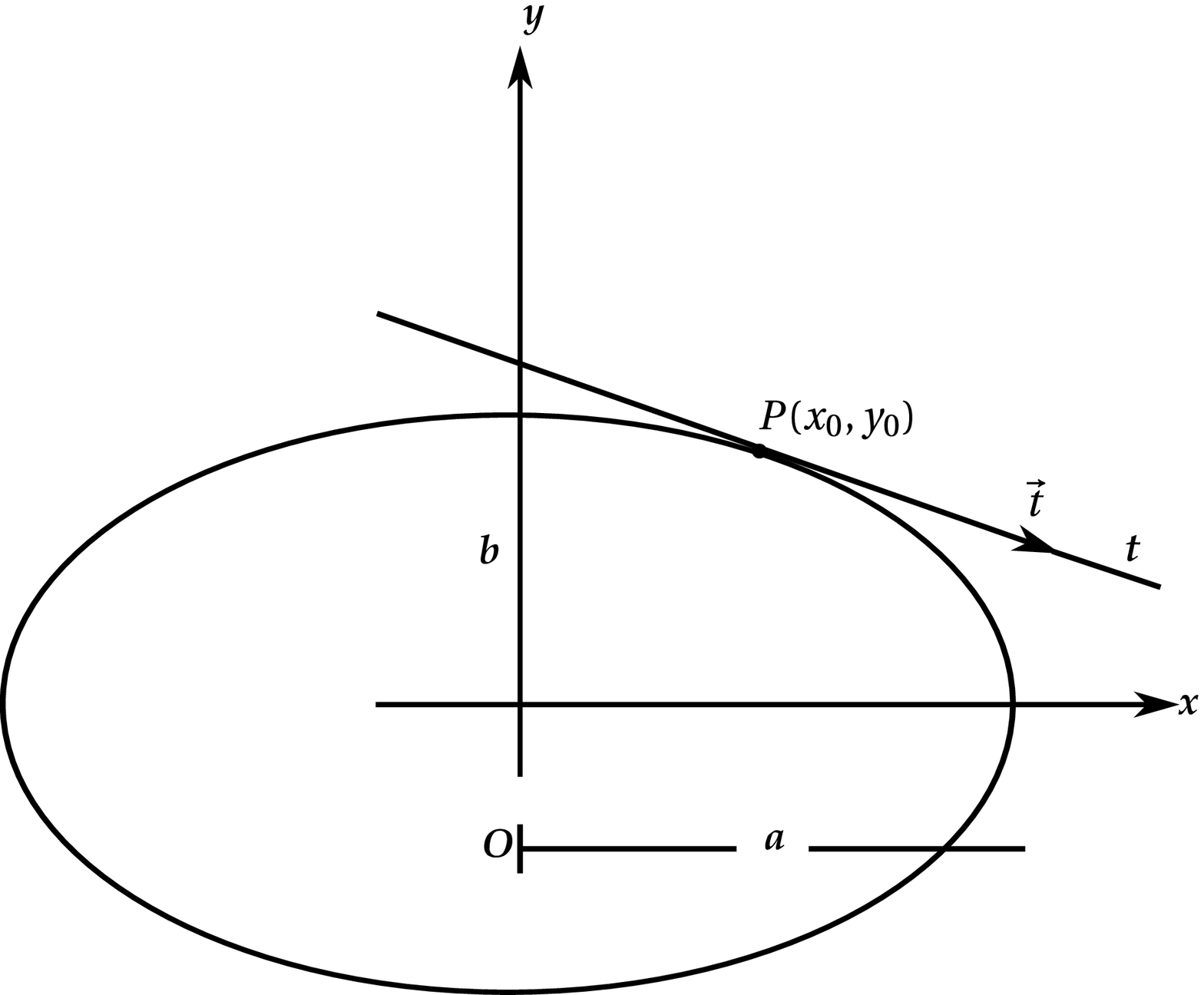}\\
\end{center}
\end{figure}
De [\ref{81}]:
\begin{equation}\label{82}
b^2x^2+a^2y^2-a^2b^2=0
\end{equation}
Consideremos la cónica $b^2x^2+a^2y^2-a^2b^2=0\hspace{0.5cm}\star$\\
En este caso,
\begin{align*}
f&:\mathbb{R}^2\longrightarrow\mathbb{R}\\
&(x,y)\longrightarrow f(x,y)=\underbrace{b^2x^2+a^2y^2}-a^2b^2\\
&\hspace{2.5cm}=q(x,y)-a^2-b^2\\
&\hspace{2.5cm}=\left(\begin{array}{cc}x&y\end{array}\right)\left(\begin{array}{cc}b^2&0\\0&a^2\end{array}\right)\dbinom{x}{y}-a^2b^2\\
q&=\mathbb{R}^2\longrightarrow\mathbb{R}\\
&(x,y)\longrightarrow q(x,y)=b^2x^2+a^2y^2\\
&\hspace{2.5cm}=\left(\begin{array}{cc}x&y\end{array}\right)\left(\begin{array}{cc}b^2&0\\0&a^2\end{array}\right)\dbinom{x}{y}\,\,\text{es la f.c. asociada a la cónica $\star.$}
\end{align*}
Sea $(x_0,y_0)$ un punto de la curva y $t$ la tangente a la curva por $(x_0,y_0)$. Vamos a dm. que la ecuación de $t$ es:
\begin{equation}\label{83}
(x-x_0)\left.\dfrac{\partial f}{\partial x}\right)_{x_0,y_o}+(y-y_0)\left.\dfrac{\partial f}{\partial x}\right)_{x_0,y_o}=0
\end{equation}
Una vez dm. [\ref{83}], como $\dfrac{\partial f}{\partial x}=2b^2x$ y $\dfrac{\partial f}{\partial y}=2a^2x,$\\
$$\left.\dfrac{\partial f}{\partial x}\right)_{x_0,y_0}=2b^2x_0,\,\,\left.\dfrac{\partial f}{\partial y}\right)_{x_0,y_0}=2a^2y_0$$ y regresando a [\ref{83}]:
\begin{align*}
t:(x-x_0)2b^2x_0+(y-y_0)2a^2y_0&=0\\
b^2x_0x-b^2x_0^2x-b^2x_0^2+a^2y_0y-a^2y_0^2&=0\\
b^2x_0x+a^2y_0y=b^2x_0^2+a^2y_0^2&\underset{\uparrow}{=}a^2b^2\\
&\begin{cases}(x_0,y_0)\,\,\text{está en la curva}\end{cases}
\end{align*}
Así que $$t:\dfrac{x_0x}{a^2}+\dfrac{y_0y}{b^2}=1,\hspace{0.5cm}\begin{array}{cc}\text{ecuación que ya habíamos}\\ \text{obtenido de otra forma}\\ \text{en las monografías de las}\\ \text{cónicas.}\end{array}$$
Regresemos a la ecuación de la cónica:
\begin{align*}
b^2x^2+a^2y^2-a^2b^2&=0\\
\left(\begin{array}{cc}x&y\end{array}\right)\underset{\overset{\parallel}{M}}{\underbrace{\left(\begin{array}{cc}b^2&0\\0&a^2\end{array}\right)}}\dbinom{x}{y}-a^2b^2=0
\end{align*}
\begin{align*}
PCM(\lambda)&=\lambda^2-(a^2+b^2)\lambda+a^2b^2\\
&\therefore \lambda=\dfrac{(a^2+b^2)\pm\sqrt{{(a^2+b^2)}^2-4a^2b^2}}{2}=\dfrac{(a^2+b^2)\pm (a^2-b^2)}{2}\begin{cases}\lambda_1=a^2\\ \lambda_2=b^2\end{cases}\\
EP_{b^2}^M&=\left\{(x,y)\diagup\left(\begin{array}{cc}b^2&0\\0&a^2\end{array}\right)\dbinom{x}{y}=b^2\dbinom{x}{y}\right\}\\
&\begin{cases}b^2x=b^2x\hspace{0.5cm}\therefore\hspace{0.5cm}x\in\mathbb{R}\\
a^2y=b^2y\hspace{0.5cm}\therefore\hspace{0.5cm}(a^2-b^2)y=0\hspace{0.5cm}\text{i.e.,}\,\,y=0\end{cases}
\end{align*}
Así que $EP_{b^2}^M$ es el eje $x$.\\
\begin{align*}
EP_{a^2}^M&=\left\{(x,y)\diagup\left(\begin{array}{cc}b^2&0\\0&a^2\end{array}\right)\dbinom{x}{y}=a^2\dbinom{x}{y}\right\}\\
&\begin{cases}b^2x=a^2x\hspace{0.5cm}\therefore\hspace{0.5cm}(a^2-b^2)x=,\hspace{0.5cm}\text{i.e.,}\,\,x=0\\
a^2y=a^2y\hspace{0.5cm}\text{o sea que y $\mathbb{R}.$}\end{cases}\\
\end{align*}
El $EP_{a^2}^M$ es el eje $y.$
\begin{figure}[ht!]
\begin{center}
  \includegraphics[scale=0.5]{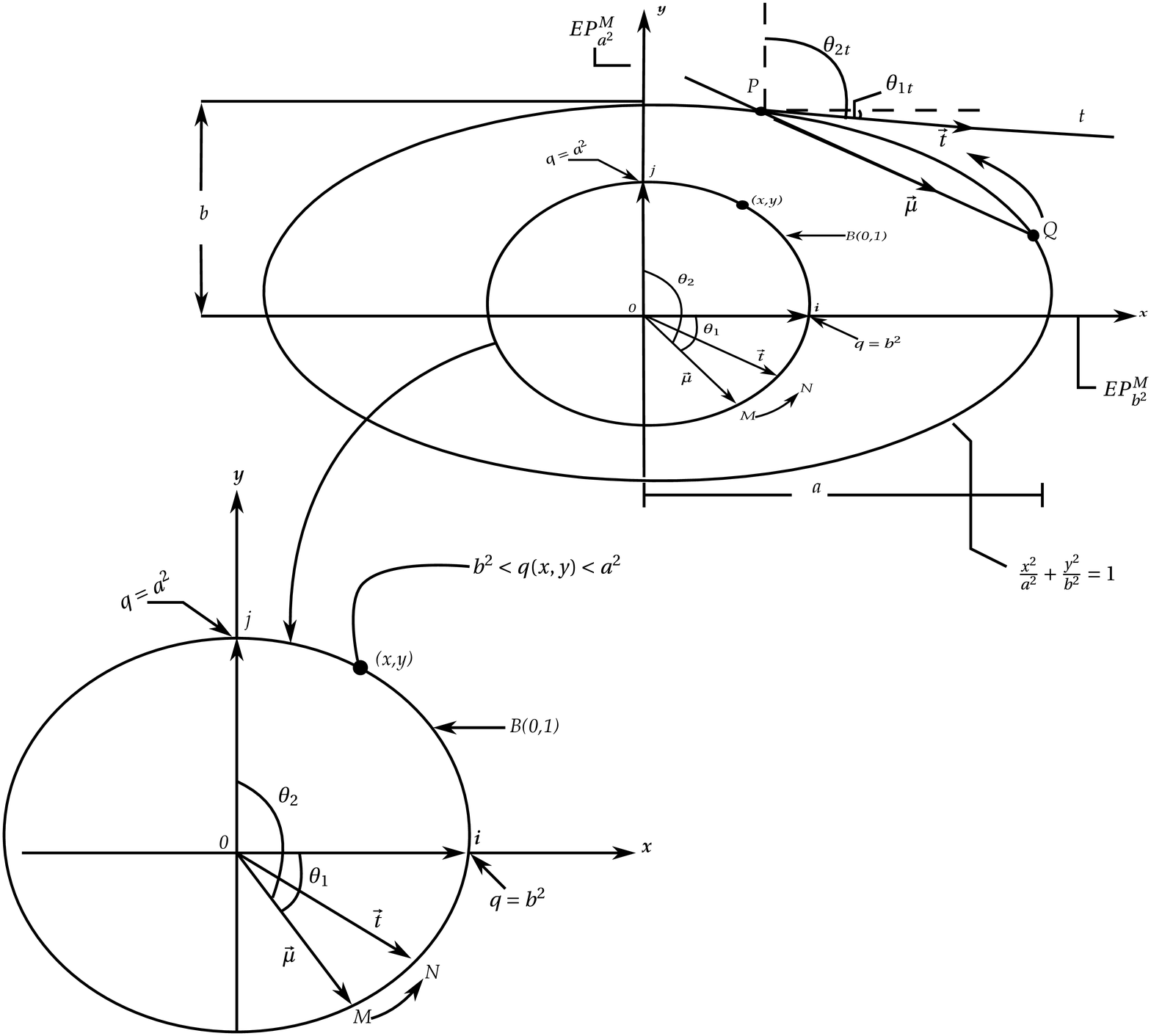}\\
\end{center}
\end{figure}
\newline
Queremos estudiar los valores que toma $q$ sobre los puntos de la $B(0;1)=\left\{(x,y)\diagup x^2+y^2=1\right\}.$\\
Recordemos el \underline{Tma. de Euler:}\\
<<La f.c. <<q>> ($q$ es en este caso la f.c. asociada a la cónica $b^2x^2+a^2y^2-a^2b^2=0$) alcanza sus valores máximo y mínimo sobre la esfera unidad en los puntos donde las direcciones principales de la matriz de la f.c. (M) corta a dicha esfera>>.\\
Como $EP_{b^2}^M$ es el eje $x$ y el $EP_{a^2}^M$ es el eje $y$, las direcciones principales de $M$ son el eje $x$ y el eje $y$. Por el Tma. de Euler se tiene que $\forall (x,y)\in B(0;1),$ i.e., $x^2+y^2=1,$
\begin{align*}
b^2\leqslant q(x,y)=\left(\begin{array}{cc}x&y\end{array}\right)\left(\begin{array}{cc}b^2&0\\0&a^2\end{array}\right)\dbinom{x}{y}=b^2x^2+a^2y^2&\underset{\uparrow}{=}b^2\cos^2\theta_1+a^2\cos^2\theta_2\leqslant a^2\\
&\begin{array}{cc}x=\cos\theta_1\\ y=\cos\theta_2\end{array}
\end{align*}
Tomemos $P$ en la elipse y $Q$ un punto en la curva cercano a $P.$\\
Llamemos $\vec{u}=\xi\vec{i}+\eta\vec{j}=\cos\theta_1\vec{i}+\cos\theta_2\vec{j}$ el vector unitario $\parallel$ a $\overline{PQ}$ y $t=\xi_{t}\vec{i}+\eta_{i}\vec{j}=\cos\theta_{1_t}\vec{i}+\cos\theta_{2_t}\vec{j}$ al vector unitario $\parallel$ a la tangente a la curva por $P$.\\
Cuando $Q\longrightarrow P$ a lo largo de la curva, el punto $M\longrightarrow N$ a lo largo de la $B(0;1)$ y en ningún momento $q(\vec{u})$ se anula.\\
Además,
\begin{align*}
\lim\xi_{\substack{Q\longrightarrow P\\ M\longrightarrow N}}&=\lim_{\substack{Q\longrightarrow P\\ M\longrightarrow N}}\cos^2\theta_1=\cos^2\theta_{1_t}=\xi^2t\\
\lim\eta_{\substack{Q\longrightarrow P\\ M\longrightarrow N}}&=\lim_{\substack{Q\longrightarrow P\\ M\longrightarrow N}}\cos^2\theta_2=\cos^2\theta_{2_t}=\eta^2t\hspace{0.5cm}\text{teniéndose que}\\
\lim_{Q\longrightarrow P}q(\vec{u})&=\lim_{Q\longrightarrow P}\left(b^2\cos^2\theta_1+a^2\cos^2\theta_2\right)=b^2\cos^2\theta_{1_t}+a^2\cos^2\theta_{2_t}=q(\vec{t})
\end{align*}
Ahora si volvamos al problema de encontrar la ecuación de la tangente a la elipse por un punto de la curva.\\
Consideremos la elipse $\dfrac{x^2}{a^2}+\dfrac{y^2}{b^2}=1, \quad a>b,$ un punto $P(x_0, y_0)$ del plano de la curva, $P$ no necesariamente en la curva, y la secante $\mathscr{L} \parallel$ al vector unitario $\vec{\mu}=\xi\vec{i}+\eta\vec{j},$
\begin{figure}[ht!]
\begin{center}
  \includegraphics[scale=0.5]{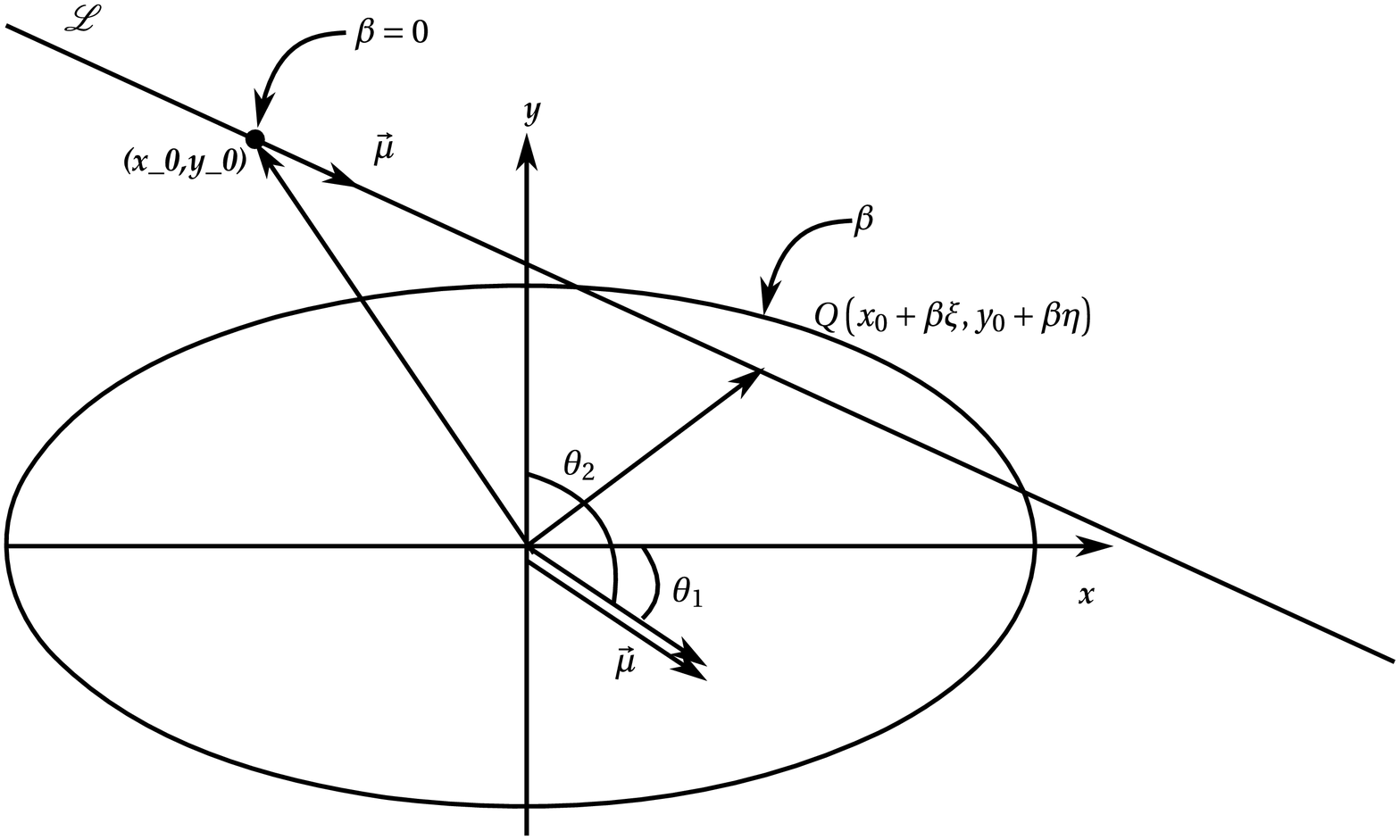}\\
\end{center}
\end{figure}
\newpage
$\xi^2+\eta^2=1,\quad \xi=\cos\theta_1,\quad \eta=\cos\theta_2.$\\
Si $Q((\vec{r})$ es un punto de la secante, $\vec{r}=\vec{r}_0+\beta\vec{\mu},$ y como $\parallel\vec{\mu}\parallel=1,$ el parámetro $\beta=PQ=d(P,Q):\text{distancia entre P y Q.}$\\
Si llamamos $Q(x,y), (x,y)=(x_0,y_0)+\beta(\xi,\eta)$\\
\begin{equation*}
\left.
\begin{split}
\therefore\hspace{0.5cm}x&=x_0+\beta\xi\\ y&=y_0+\beta\eta
\end{split}
\right\}\text{Ecuaciones paramétricas de $\mathscr{L}$.}
\end{equation*}
Recordemos la ecuación de incrementos de la cónica:\\
$\forall(x,y),\quad(h,k)\in\mathbb{R}^2:$
\begin{align*}
f\left(X+h,Y+k\right)&=q(X,Y)+\left(\left.\dfrac{\partial f}{\partial x}\right)_{h,k}\quad\left.\dfrac{\partial f}{\partial y}\right)_{h,k}\right)\dbinom{X}{Y}+f(h,k).\\
\intertext{En virtud de ésta ecuación,}\\
f\left(\beta\xi+x_0,\beta\eta+y_0\right)&=q(\beta\xi,\beta\eta)+\left(\left.\dfrac{\partial f}{\partial x}\right)_{h,k}\quad\left.\dfrac{\partial f}{\partial y}\right)_{h,k}\right)\dbinom{\beta\xi}{\beta\eta}+f(x_0,y_0)\\
&=\left(\begin{array}{cc}\beta\xi&\beta\eta\end{array}\right)\left(\begin{array}{cc}b^2&0\\0&a^2\end{array}\right)\dbinom{\beta\xi}{\beta\eta}+\beta\xi\left.\dfrac{\partial f}{\partial x}\right)_{x_0,y_0}+\beta\eta\left.\dfrac{\partial f}{\partial y}\right)_{x_0,y_0}+f(x_0,y_0)\\
&=\left(\begin{array}{cc}\beta\xi&\beta\eta\end{array}\right)\left(\begin{array}{cc}b^2\xi\beta\\a^2\eta\beta\end{array}\right)+\left(\left.\xi\dfrac{\partial f}{\partial x}\right)_{x_0,y_0}+\eta\left.\dfrac{\partial f}{\partial y}\right)_{x_0,y_0}\right)+f(x_0,y_0)\\
&=\left(\begin{array}{cc}b^2\xi^2+a^2\eta^2\end{array}\right)\beta^2+\left(\left.\xi\dfrac{\partial f}{\partial x}\right)_{h,k}+\left.\dfrac{\partial f}{\partial y}\right)_{h,k}\right)\beta+f(x_0,y_0)\\
&=q(\xi,\eta)\beta^2+\left(\left.\xi\dfrac{\partial f}{\partial x}\right)_{h,k}+\eta\left.\dfrac{\partial f}{\partial y}\right)_{h,k}\right)\beta+f(x_0,y_0)\hspace{0.5cm}\star
\end{align*}
Los punto de la cónica y la la recta $\mathscr{L}$ se consiguen al hacer simultaneas
\begin{equation}\label{84}
f(x,y)=0
\end{equation}
y
\begin{equation}\label{85}
\left.
\begin{split}
x&=x_0+\beta\xi\\
y&=y_0+\beta\eta
\end{split}
\right\}
\end{equation}
Al llevar [\ref{85}] a [\ref{84}], $f(x_0+\beta\xi, y_0+\beta\eta)=0.$\\
Luego, en virtud de $\star,$ $$q(\xi,\eta)\beta^2+\left(\xi\left.\dfrac{\partial f}{\partial x}\right)_{x_0,y_0}+\eta\left.\dfrac{\partial f}{\partial y}\right)_{x_0,y_0}\right)\beta+f(x_0,y_0)=0\hspace{0.5cm}\star\star$$
Las raíces de ésta ecuación (en general hay dos raíces simples y cuando $\mathscr{L}$ sea una secante) son los valores del parámetro $\beta$ que llevados a [\ref{85}] nos permiten encontrar las coordenadas de los puntos donde la secante corta a la cónica.\\
Supongamos ahora que $P(x_0,y_0)$ está en la curva.\\
Entonces $f(x_0,y_0)=0$ y al regresar a $\star\star,$
\begin{align*}
q(\xi,\eta)\beta^2+\left(\left.\xi\dfrac{\partial f}{\partial x}\right)_{x_0,y_0}+\left.\eta\dfrac{\partial f}{\partial y}\right)_{x_0,y_0}\right)\beta&=0\\
\beta\left(q(\xi,\eta)\beta+\left(\left.\xi\dfrac{\partial f}{\partial x}\right)_{x_0,y_0}+\left.\eta\dfrac{\partial f}{\partial y}\right)_{x_0,y_0}\right)\right)=0\\
\therefore\quad\beta_1=0\hspace{0.5cm}\text{o}\hspace{0.5cm}\beta_2&=-\dfrac{\left.\xi\dfrac{\partial f}{\partial x}\right)_{x_0,y_0}+\left.\eta\dfrac{\partial f}{\partial y}\right)_{x_0,y_0}}{q(\xi,\eta)}
\end{align*}
lo cual nos indica que la recta $\mathscr{L}$ corta a la cónica en dos puntos: $P(x_0,y_0)$ que corresponde al valor $\beta_1=0$ del parámetro y $Q(x_0+\beta_2\xi, y_0+\beta_2\eta)$ que corresponde al valor $$\beta_2=-\dfrac{\left.\xi\dfrac{\partial f}{\partial x}\right)_{x_0,y_0}+\left.\eta\dfrac{\partial f}{\partial y}\right)_{x_0,y_0}}{q(\xi,\eta)}\quad\text{del parámetro}:$$
\begin{figure}[ht!]
\begin{center}
  \includegraphics[scale=0.5]{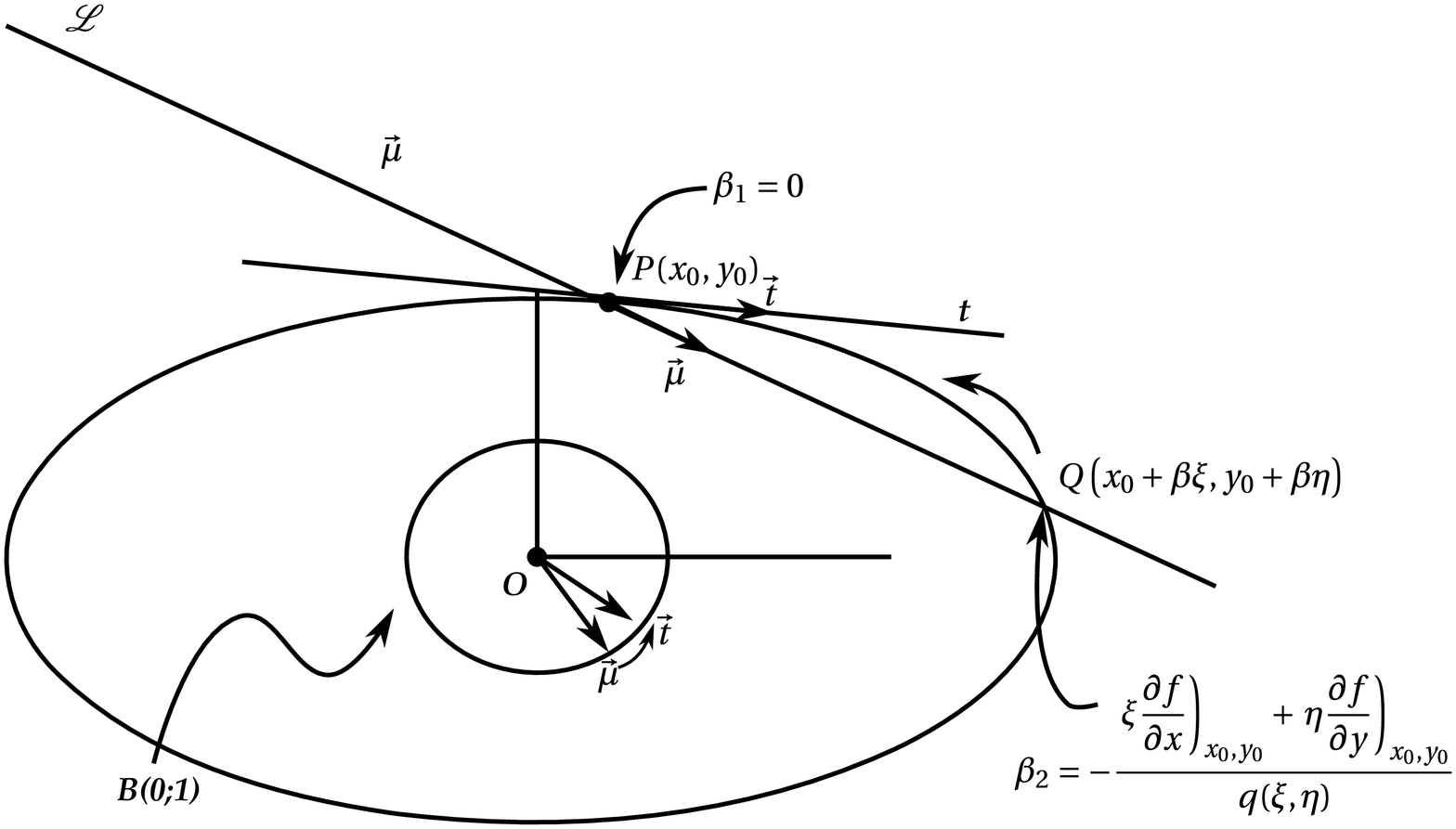}\\
\end{center}
\end{figure}
Ahora acerquemos $Q$ a $P$ a través de la curva.\\
Entonces $\vec{\mu}\longrightarrow\vec{t},\quad\lim\limits_{Q\rightarrow P}\beta_2=\beta_1=0.$\\
Pero
\begin{figure}[ht!]
\begin{center}
  \includegraphics[scale=0.5]{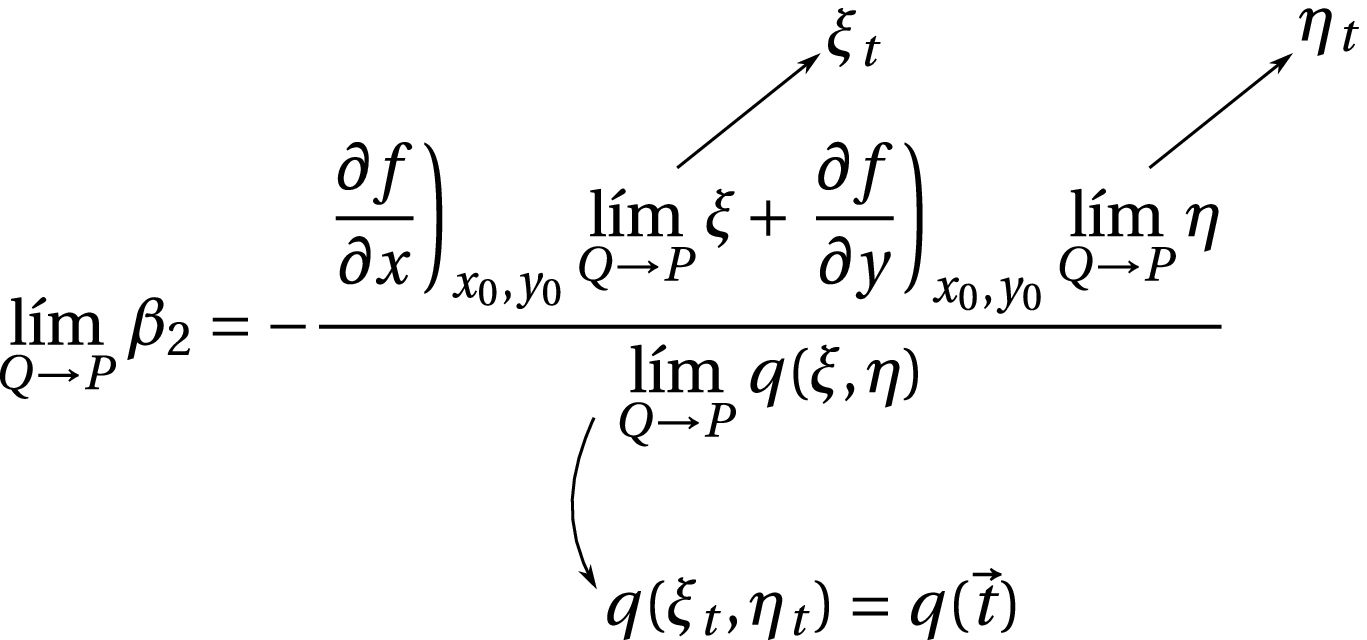}\\
\end{center}
\end{figure}
\newline
Así que $$\left.\dfrac{\partial f}{\partial x}\right)_{x_0,y_0}\xi_t+\left.\dfrac{\partial f}{\partial y}\right)_{x_0,y_0}\eta_t=0\quad\star$$
donde $\vec{t}=\xi_t\vec{i}+\eta_t\vec{j}$ es el vector unitario $\parallel$ a la tangente de la cónica por $P(x_0,y_0).$\\
De $\star:$
\begin{align*}
\eta_t&=-\dfrac{\left.\dfrac{\partial f}{\partial x}\right)_{x_0,y_0}\xi_t}{\left.\dfrac{\partial f}{\partial y}\right)_{x_0,y_0}}\quad\text{y}\\
\underset{\nwarrow}{\underbrace{m_t}}&=\dfrac{\eta_t}{\xi_t}=-\dfrac{\left.\dfrac{\partial f}{\partial x}\right)_{x_0,y_0}}{\left.\dfrac{\partial f}{\partial y}\right)_{x_0,y_0}}\\
&\begin{array}{cc}\text{pendiente}\\ \text{de la tangente $t$ por $P(x_0,y_0)$}\\ \text{a la elipse}\end{array}
\end{align*}
De éste modo, la ecuación de $t$ es: $$y-y_0=-\dfrac{\left.\dfrac{\partial f}{\partial x}\right)_{x_0,y_0}}{\left.\dfrac{\partial f}{\partial y}\right)_{x_0,y_0}}(x-x_0)\quad\text{y finalmente,}$$
$$(x-x_0)\left.\dfrac{\partial f}{\partial x}\right)_{x_0,y_0}+(y-y_0)\left.\dfrac{\partial f}{\partial y}\right)_{x_0,y_0}=0.$$
De paso hallemos la ecuación de la normal a la curva en $P(x_0,y_0)$
\begin{figure}[ht!]
\begin{center}
  \includegraphics[scale=0.5]{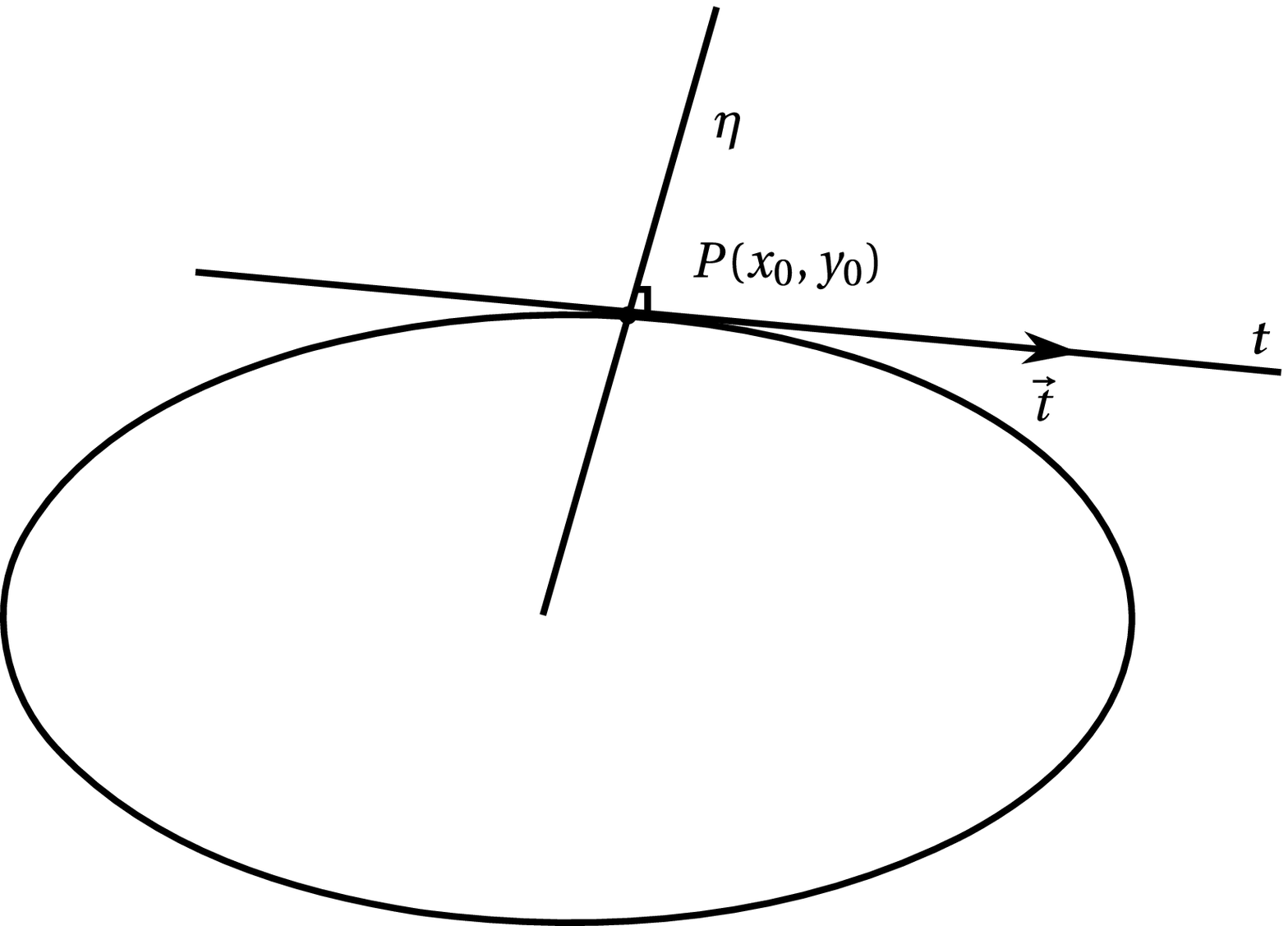}\\
\end{center}
\end{figure}
\begin{align*}
\underset{\nwarrow}{\underbrace{m_\eta}}&=-\dfrac{1}{m_t}=\dfrac{\left.\dfrac{\partial f}{\partial y}\right)_{x_0,y_0}}{\left.\dfrac{\partial f}{\partial x}\right)_{x_0,y_0}}\\
&\begin{array}{c}\text{pendiente de la normal}\end{array}
\end{align*}
Ecuación de la normal: $$y-y_0=\dfrac{\left.\dfrac{\partial f}{\partial y}\right)_{x_0,y_0}}{\left.\dfrac{\partial f}{\partial x}\right)_{x_0,y_0}}$$
que podemos escribir así: $$(x-x_0)\left.\dfrac{\partial f}{\partial y}\right)_{x_0,y_0}-(y-y_0)\left.\dfrac{\partial f}{\partial x}\right)_{x_0,y_0}=0.$$
\item[(2)] Supongamos ahora que la elipse es $\dfrac{x^2}{a^2}+\dfrac{y^2}{b^2}=1$ con $a<b.$
\begin{figure}[ht!]
\begin{center}
  \includegraphics[scale=0.5]{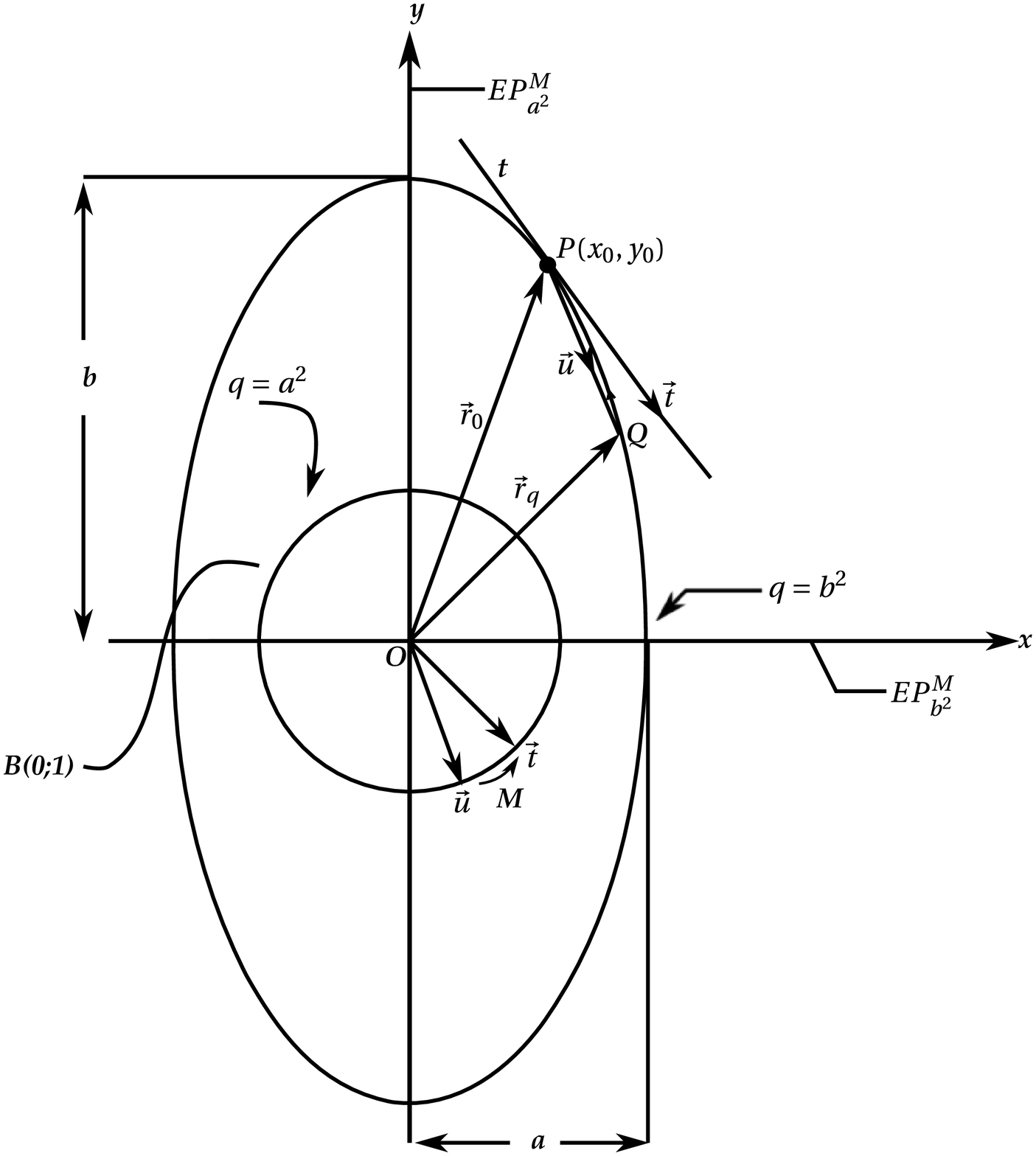}\\
\end{center}
\end{figure}
La cónica es $b^2x^2+a^2y^2-a^2b^2=0.$\\
La componente cuadrática es
\begin{align*}
q(x,y)&=b^2x^2+a^2y^2\\
&=\left(\begin{array}{cc}x&y\end{array}\right)\underset{\overset{\parallel}{M}}{\left(\begin{array}{cc}b^2&0\\ 0&a^2\end{array}\right)}\dbinom{x}{y}
\end{align*}
El $EP_{b^2}^M$ es el eje $x$.\\
El $EP_{a^2}$ es el eje $y$ y $\forall(x,y)\in B(0;1),$ por el T. de Euler, $a^2\leqslant q(x,y)\leqslant b^2.$\\
Si tomamos $(x_0,y_0)$ en la curva y $Q$ un punto de la curva cercano a $P,$ la ecuación de la secante $\overset{\longleftrightarrow}{PQ}$ es $\vec{r}=\vec{r}_0+\beta\vec{\mu}$ donde $\vec{\mu}=\xi\vec{i}+\eta\vec{j}$ es un vector $\parallel$ a $\overset{\longleftrightarrow}{PQ}.$ La secante $\overset{\longleftrightarrow}{PQ}$ corta a la curva en los parámetros $(x_0,y_0)$ y $Q$ correspondientes a los valores del parámetro $\beta_1=0$ y $\beta_2=-\dfrac{\left.\xi\dfrac{\partial f}{\partial x}\right)_{x_0,y_0}+\left.\eta\dfrac{\partial f}{\partial y}\right)_{x_0,y_0}}{q(\xi,\eta)}.$\\
Cuando $Q\longrightarrow P$ a través de la curva, $M\longrightarrow N$ a través de la $$B(0;1),\quad\lim\limits_{Q\longrightarrow P}\xi=\xi_t,\quad\lim\limits_{Q\longrightarrow P}\eta=\eta_t$$ con $\vec{t}=\xi_t\vec{i}+\eta_t\vec{j}.$\\
$\lim\limits_{Q\longrightarrow P}q(\xi,\eta)=q(\vec{t}).$\\
\begin{figure}[ht!]
\begin{center}
  \includegraphics[scale=0.5]{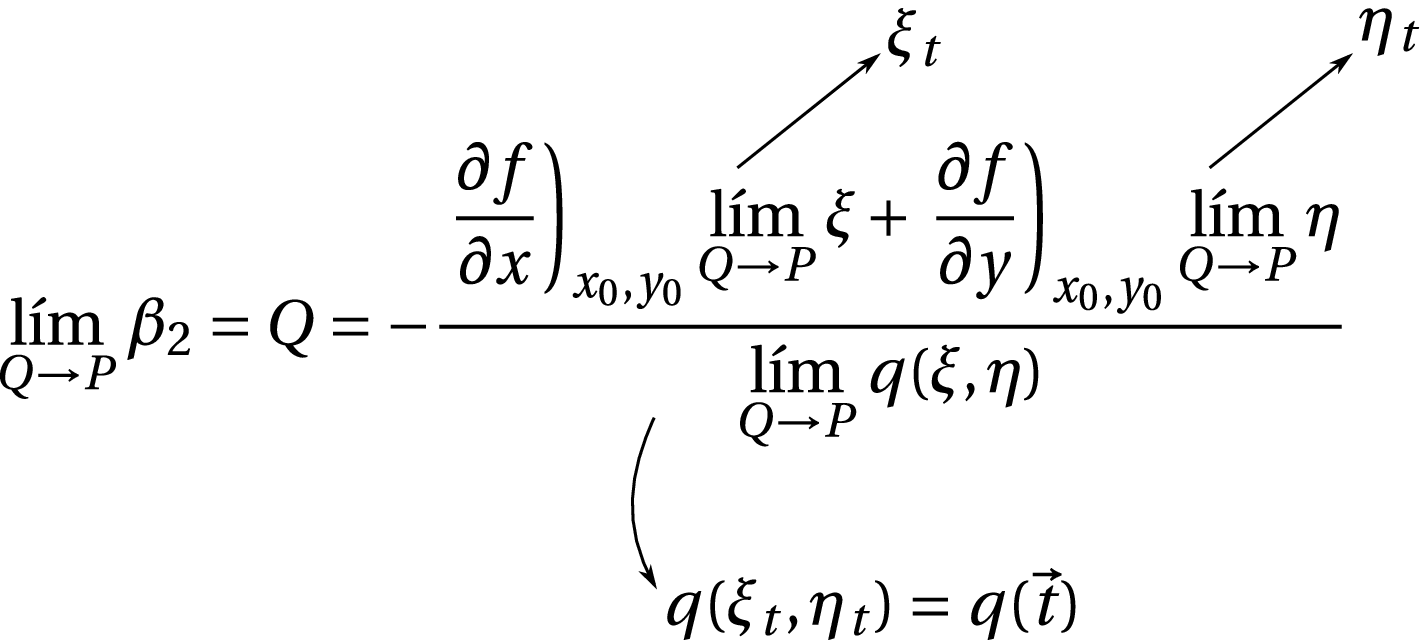}\\
\end{center}
\end{figure}
\newline
De suerte que $$m_t=\dfrac{\eta_t}{\xi_t}=-\dfrac{\left.\dfrac{\partial f}{\partial x}\right)_{x_0,y_0}}{\left.\dfrac{\partial f}{\partial y}\right)_{x_0,y_0}}$$ y la ecuación de la recta $t$ tangente a la curva en $(x_0,y_0)$ es:
\begin{align*}
y-y_0&=-\dfrac{\left.\dfrac{\partial f}{\partial x}\right)_{x_0,y_0}}{\left.\dfrac{\partial f}{\partial y}\right)_{x_0,y_0}}(x-x_0),\quad\text{o sea}:\\
t&:(x-x_0)\left.\dfrac{\partial f}{\partial x}\right)_{x_0,y_0}+(y-y_0)\left.\dfrac{\partial f}{\partial y}\right)_{x_0,y_0}=0
\end{align*}
En este caso, $\left.\dfrac{\partial f}{\partial x}\right)_{x_0,y_0}=2b^2x_0;\quad\left.\dfrac{\partial f}{\partial y}\right)_{x_0,y_0}=2a^2y_0.$\\
y la ecuación de la tangente es\\
\begin{align*}
t:(x-x_0)2b^2x_0+(y-y_0)2a^2y_0&=0\\
b^2x_0x-b^2x_0^2+a^2y_0y-a^2y_0^2&=0\\
b^2x_0x+a^2y_0y&=b^2x_0^2+a^2y_0^2\\
&\underset{\uparrow}{=}a^2b^2\\
&\begin{cases}
\text{Como}\quad (x_0,y_0)\quad\text{está en la curva,}\\
b^2x_0^2+a^2y_0^2=a^2b^2
\end{cases}
\end{align*}
$$t:\dfrac{x_0x}{a^2}+\dfrac{y_0y}{b^2}=1$$
\item[3)] Consideremos ahora la hipérbola $\dfrac{x^2}{a^2}-\dfrac{y^2}{b^2}=1$
\begin{figure}[ht!]
\begin{center}
  \includegraphics[scale=0.5]{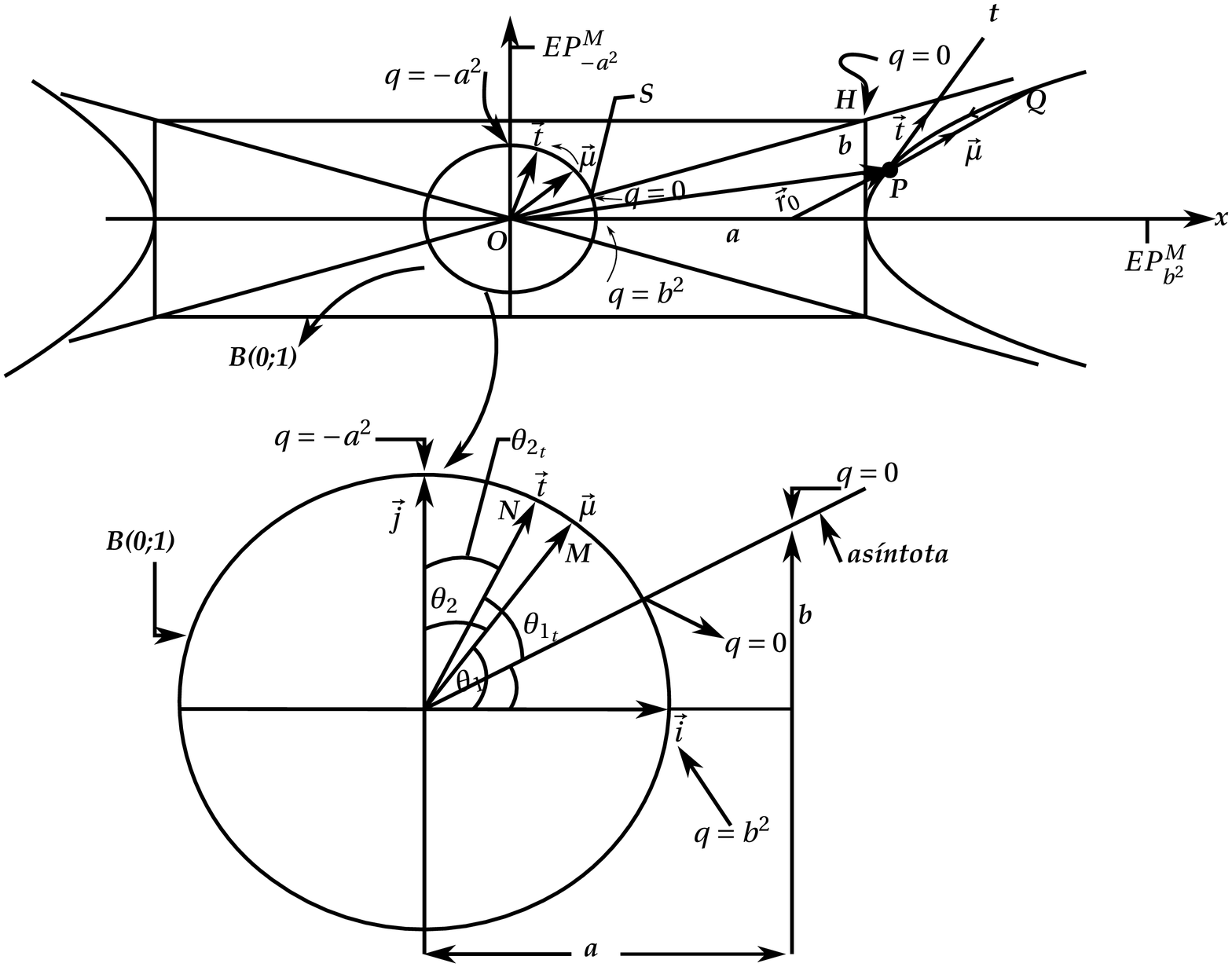}\\
\end{center}
\end{figure}
La cónica es:
\begin{align*}
b^2x^2-a^2y^2-a^2b^2&=0\quad\text{y su f.c. es}\\
q(x,y)=b^2x^2-a^2y^2&=\left(\begin{array}{cc}x&y\end{array}\right)\underset{\overset{\parallel}{M}}{\left(\begin{array}{cc}b^2&0\\0&-a^2\end{array}\right)}\dbinom{x}{y}
\end{align*}
\begin{align*}
PCM(\lambda)&=\lambda^2-(b^2-a^2)\lambda-a^2b^2=0\\
\lambda&=\dfrac{(b^2-a^2)\pm\sqrt{{(b^2-a^2)}^2+4a^2b^2}}{2}=\dfrac{(b^2-a^2)\pm(b^2+a^2)}{2}\begin{cases}\lambda_1=b^2\\ \lambda_2=-a^2\end{cases}
\end{align*}
Queremos estudiar los valores que toma $q$ en la $B(0;1).$\\
\begin{align*}
EP_{b^2}^M=\left\{(x,y)\diagup\left(\begin{array}{cc}b^2&0\\0&-a^2\end{array}\right)\dbinom{x}{y}=b^2\dbinom{x}{y}\right\}\\
\begin{cases}b^2x=b^2x\quad\text{:}\quad x\in\mathbb{R}\\
-a^2y=b^2y\quad\therefore\quad(b^2+a^2)y=0\quad\text{O sea $y=0$.}\end{cases}
\end{align*}
El $EP_{b^2}^M$ es el eje $x.$\\
\begin{align*}
EP_{-a^2}^M=\left\{(x,y)\diagup\left(\begin{array}{cc}b^2&0\\0&-a^2\end{array}\right)\dbinom{x}{y}=-a^2\dbinom{x}{y}\right\}\\
\begin{cases}b^2x=-a^2x\quad\therefore\quad(b^2+a^2)x=0\quad\text{;}\quad x=0\\
-a^2y=-a^2y\quad\text{;}\quad y\in\mathbb{R}\end{cases}
\end{align*}
El $EP_{-a^2}^M$ es el eje $y$.\\
Luego, por el T. de Euler, $\forall(\xi,\eta)\in B(0;1):-a^2\leq q(\xi,\eta)\leq b^2$\\
\begin{align*}
q(H)&=q(a,b)=\left(\begin{array}{cc}a&b\end{array}\right)\left(\begin{array}{cc}b^2&0\\0&-a^2\end{array}\right)\dbinom{a}{b}=\left(\begin{array}{cc}a&b\end{array}\right)\left(\begin{array}{cc}b^2a\\-a^2b\end{array}\right)=0\\
q(S)&=\left(\begin{array}{cc}\dfrac{a}{\sqrt{a^2+b^2}}&\dfrac{b}{\sqrt{a^2+b^2}}\end{array}\right)\left(\begin{array}{cc}b^2&0\\0&-a^2\end{array}\right)\left(\begin{array}{cc}\dfrac{a}{\sqrt{a^2+b^2}}\\ \dfrac{b}{\sqrt{a^2+b^2}}\end{array}\right)\\
&=\left(\begin{array}{cc}\dfrac{a}{\sqrt{a^2+b^2}}&\dfrac{b}{\sqrt{a^2+b^2}}\end{array}\right)\left(\begin{array}{cc}\dfrac{ab^2}{\sqrt{a^2+b^2}}&-\dfrac{a^2b}{\sqrt{a^2+b^2}}\end{array}\right)=0
\end{align*}
Consideremos ahora la secante $\overline{PQ}$ de vector unitario $\vec{\mu}\parallel\overline{PQ},$ $$\vec{\mu}=\xi\vec{i}+\eta\vec{j}=\cos\theta_1\vec{i}+\cos\theta_2\vec{j}.$$ Como en los casos anteriores, llamemos $t=\xi_t\vec{j}+\eta_t\vec{j}=\cos\theta_{1_t}\vec{i}+\cos\theta_{2_t}\vec{j}$ el vector unitario $\parallel$ a la tangente $t$ por $P$ a la curva.\\
Cuando el punto $Q\longrightarrow P$ a través de la curva, $M\longrightarrow N$ a través de la $B(0;1).$\\
\begin{align*}
\lim\xi_{\substack{Q\longrightarrow P\\M\longrightarrow N}}&=\lim\limits_{Q\longrightarrow P}\cos\theta_1=\cos\theta_{1_t}=\xi_t\\
\lim\limits_{Q\longrightarrow P}\eta&=\lim\limits_{Q\longrightarrow P}\cos\theta_2=\cos\theta_{2_t}=\eta_t\\
\lim\limits_{Q\longrightarrow P}q(\vec{u})&=\lim\limits_{Q\longrightarrow P}(b^2\cos\theta_1-a^2\cos\theta_2)=b^2\cos\theta_{1_t}-a^2\cos\theta_{2_t}=q(\vec{t})
\end{align*}
La ecuación de la secante $\overset{\longleftrightarrow}{PQ}$ es $\vec{r}=\vec{r}_0+\beta\vec{\mu}.$\\
Dicha secante corta a la curva en los puntos $P$ y $Q$ asociados a los valores $\beta_1=0$ y\\ $\beta_2=-\dfrac{\left.\xi\dfrac{\partial f}{\partial x}\right)_{x_0,y_0}+\left.\eta\dfrac{\partial f}{\partial y}\right)_{x_0,y_0}}{q(\xi,\eta)}$ del parámetro.\\
Cuando $Q\longrightarrow P$ a través de la curva, $M\longrightarrow N$ se tiene que
\begin{figure}[ht!]
\begin{center}
  \includegraphics[scale=0.5]{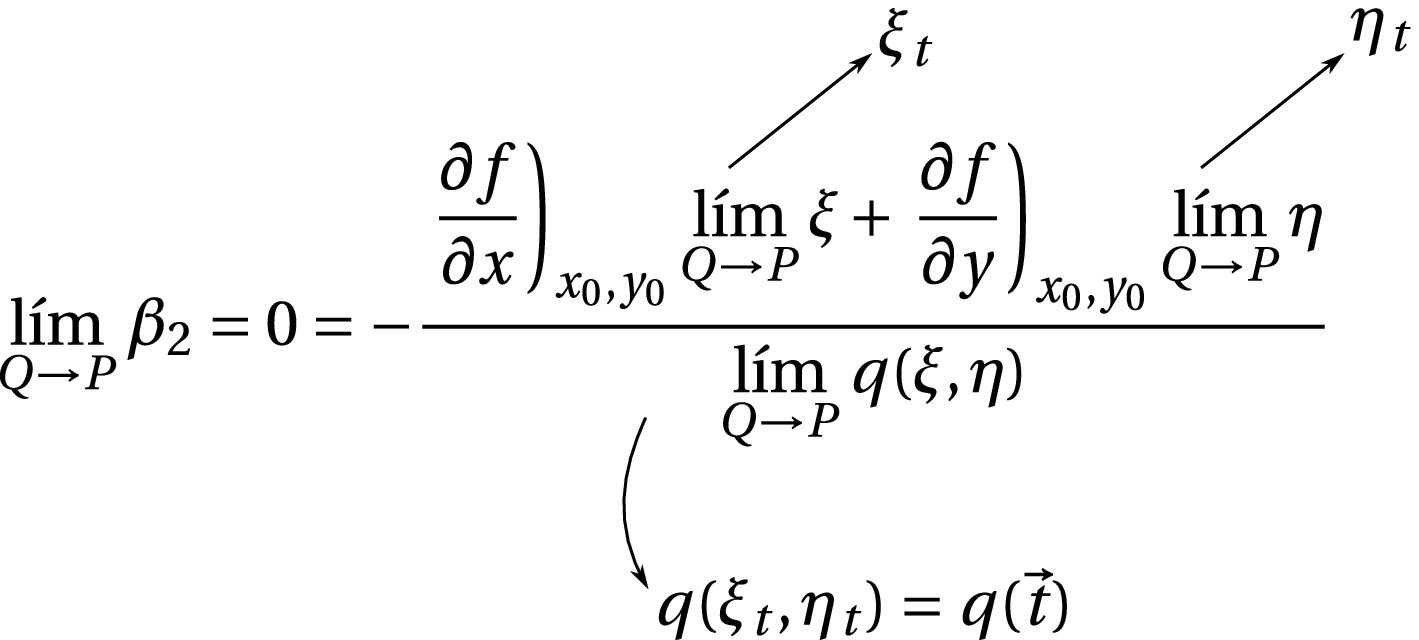}\\
\end{center}
\end{figure}
\newline
y se tiene de nuevo que $$m_t=\dfrac{\eta_t}{\xi_t}=-\dfrac{\left.\dfrac{\partial f}{\partial x}\right)_{x_0,y_0}}{\left.\dfrac{\partial f}{\partial y}\right)_{x_0,y_0}}$$
La ecuación de la tangente es $$t:(x-x_0)\left.\dfrac{\partial f}{\partial x}\right)_{x_0,y_0}+(y-y_0)\left.\dfrac{\partial f}{\partial y}\right)_{x_0,y_0}=0$$
En este caso, $f(x,y)=b^2x^2-a^2y^2-a^2b^2.$\\
\begin{align*}
\dfrac{\partial f}{\partial x}&=2b^2x;\quad\dfrac{\partial f}{\partial y}=-2a^2y,\\
\dfrac{\partial f}{\partial x}&=2b^2x_0;\quad\dfrac{\partial f}{\partial y}=-2a^2y_0
\end{align*}
Así que
\begin{align*}
t:(x-x_0)2b^2x_0-(y-y_0)2a^2y_0&=0\\
b^2x_0x-b^2x_0^2-a^2y_0y+a^2y_0^2&=0\\
b^2x_0x-a^2y_0y&=b^2x_0^2-a^2y_0^2\\
&=a^2b^2
\end{align*}
y finalmente, $$t:\dfrac{x_0x}{a^2}-\dfrac{y_0y}{b^2}=1$$
\item[4)]
Tomemos ahora la hipérbola $\dfrac{y^2}{b^2}-\dfrac{x^2}{a^2}=1.$\\
La cónica es $a^2y^2-b^2x^2-a^2b^2=0$ y su f.c. asociada es $$q(x,y)=-b^2x^2+a^2y^2=\left(\begin{array}{cc}x&y\end{array}\right)\underset{\overset{\parallel}{M}}{\left(\begin{array}{cc}-b^2&0\\0&a^2\end{array}\right)}\dbinom{x}{y}$$
Vamos a dm. que la ecuación de la tangente $t$ a la curva por el punto $P(x_0,y_0)$ de la curva se construye así:
$$\text{ecua. de la curva:}\dfrac{yy}{b^2}-\dfrac{xx}{a^2}=1;\quad t:\dfrac{y_0y}{b^2}-\dfrac{x_0x}{a^2}=1$$
Podemos dm. que los valores propios de $M$ son $-b^2,a^2$ y que $EP_{-b^2}^M$ es el eje $x$, $EP_{a^2}^M$ es el eje $y.$
\begin{figure}[ht!]
\begin{center}
  \includegraphics[scale=0.3]{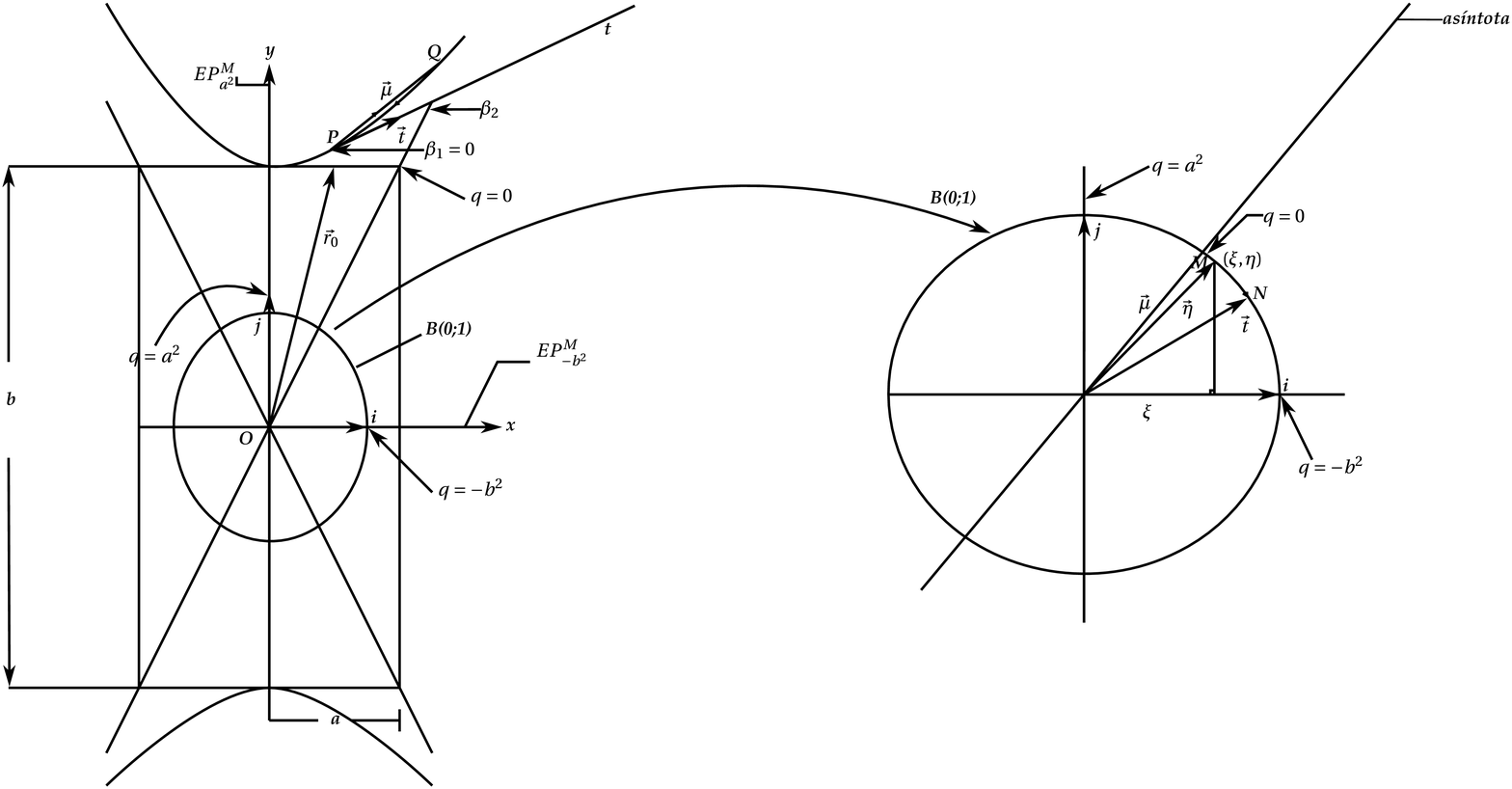}\\
\end{center}
\end{figure}
\newline
Por el Tma. de Euler, $\forall(\xi,\eta)\in B(0;1):-b^2\leq q(\xi,\eta)\leq a^2.$\\
La ecuación de la secante $\overset{\longleftrightarrow}{PQ}$ donde $P(x_0,y_0)$ es un punto de la curva es: $\vec{r}=\vec{r}_0+\beta\vec{\mu},\quad \vec{\mu}=\xi\vec{i}+\eta\vec{j}:$ vector unitario $\parallel$ a $\overline{PQ}.$\\
El valor del parámetro $\beta$ es: $\beta_1=0$ en $P$ y $\beta_2=-\dfrac{\xi\left.\dfrac{\partial f}{\partial x}\right)_{x_0,y_0}+\eta\left.\dfrac{\partial f}{\partial y}\right)_{x_0,y_0}}{q(\xi,\eta)}$ en $Q.$\\
Cuando $Q\longrightarrow P$ a través de la curva, $M\longrightarrow N$ a través de la $B(0;1)$ y
\begin{figure}[ht!]
\begin{center}
  \includegraphics[scale=0.5]{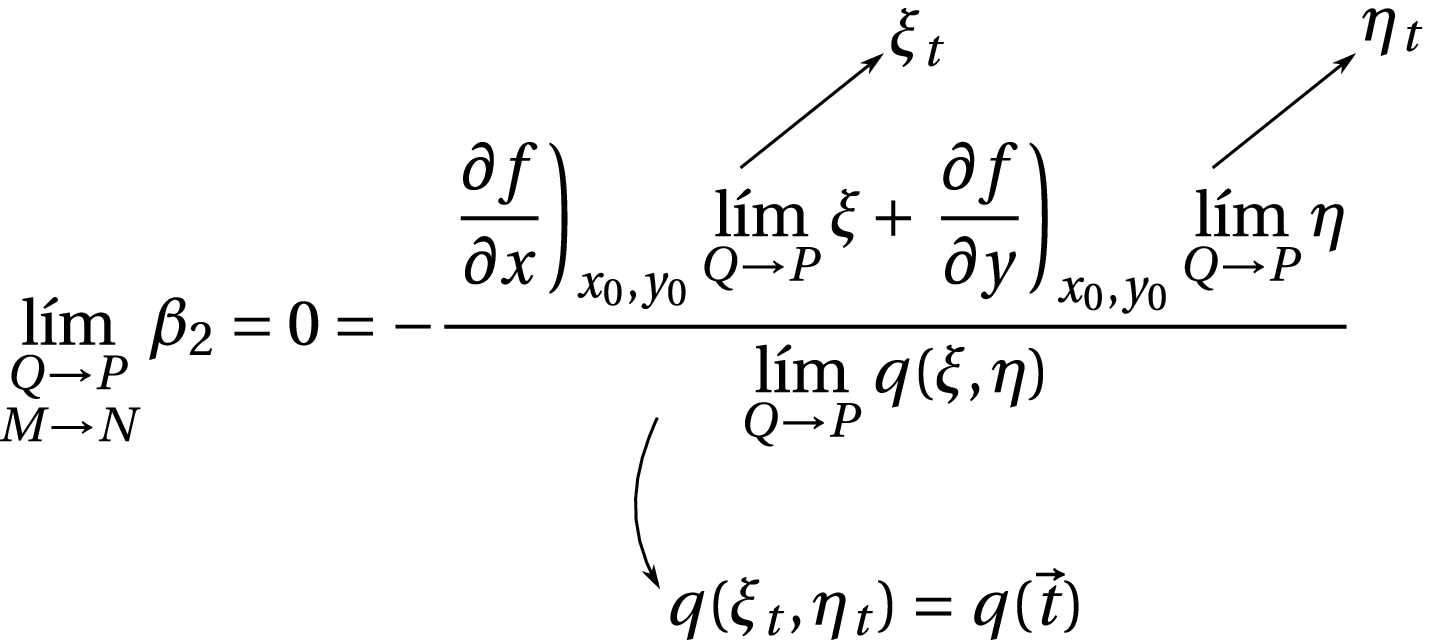}\\
\end{center}
\end{figure}
donde $\vec{t}=\xi_t\vec{i}+\eta_t\vec{j}$ es el vector unitario $\parallel$ a la tangente $t$ por el punto $P(x_0,y_0)$ de la curva.\\
Se tiene de nuevo que $$\underset{\text{pendiente de t}}{\underbrace{m_t}}=\dfrac{\eta_t}{\xi_t}=-\dfrac{\left.\dfrac{\partial f}{\partial x}\right)_{x_,y_0}}{\left.\dfrac{\partial f}{\partial y}\right)_{x_0,y_0}}$$
$$t:(x-x_0)\left.\dfrac{\partial f}{\partial x}\right)_{x_,y_0}+(y-y_0)\left.\dfrac{\partial f}{\partial y}\right)_{x_,y_0}=0$$
En este caso,
\begin{align*}
f(x,y)&=-b^2x^2+a^2y^2-a^2b^2\\
\dfrac{\partial f}{\partial x}&=-2b^2x;\quad\left.\dfrac{\partial f}{\partial y}\right)_{x_0,y_0}=-2b^2x_0\\
\dfrac{\partial f}{\partial x}&=2a^2y;\quad\left.\dfrac{\partial f}{\partial y}\right)_{x_0,y_0}=2a^2y_0
\end{align*}
Así que la ecuación de $t$ es
\begin{align*}
t:-(x-x_0)\cancel{2}b^2x_0+(y-y_0)\cancel{2}a^2y_0&=0\\
-b^2x_0x+b^2x_0^2+a^2y_0y-a^2y_0^2&=0\\
a^2y_0y-b^2x_0x=a^2y_0^2-b^2x_0^2&\underset{\uparrow}{=}a^2b^2\\
&(x_0,y_0)\quad\text{está en la curva}
\end{align*}
y finalmente,\\
$$t:\dfrac{y_0y}{b^2}-\dfrac{x_0x}{a^2}=1$$
\item[5)] Consideremos la parábola $y^2=2px$
\newpage
\begin{figure}[ht!]
\begin{center}
  \includegraphics[scale=0.5]{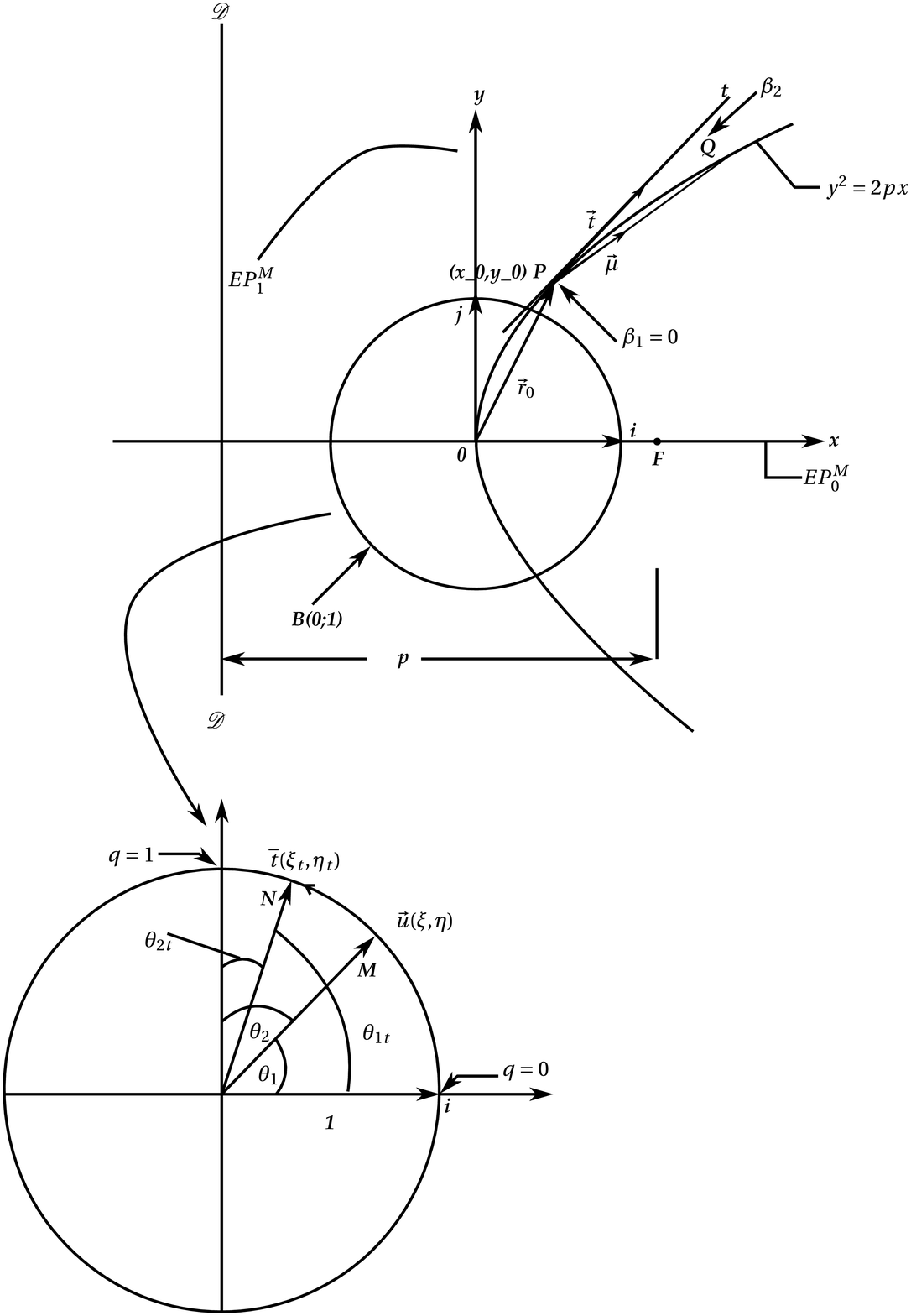}\\
\end{center}
\end{figure}
Vamos a dm. que la ecuación de la tangente $t$ a la curva por un punto $P(x_0,y_0)$ de ella se construye así:\\
ecuación de la curva: $y^2=2px$ que podemos escribir así:
\begin{align*}
y\cdot y&=p(x+x)\\
t:y_0y&=p(x+x_0)
\end{align*}
La cónica es
\begin{align*}
y^2-2pxy&=0\\
q(x,y)=y^2&=\left(\begin{array}{cc}x&y\end{array}\right)\underset{\overset{\parallel}{M}}{\left(\begin{array}{cc}0&0\\0&1\end{array}\right)}\dbinom{x}{y}=y^2\\
PCM(\lambda)&=\lambda^2-\omega\lambda+\delta\\
&\underset{\nearrow}{=}\lambda^2-\lambda=\lambda(\lambda-1)=0\\
&\begin{cases}
\omega&=1\quad\text{los valores propios de M son 1 y 0.}\\
\delta&=1
\end{cases}\\
EP_0^M&=\left\{(x,y)\diagup\left(\begin{array}{cc}0&0\\0&1\end{array}\right)\dbinom{x}{y}=\dbinom{0}{0}\right\}\\
&=\left\{(x,y)\diagup 0\cdot x+1\cdot y=0\hspace{0.5cm}x\in\mathbb{R}, y=0\right\}\\
\intertext{El $EP_0^M$ es el eje $x$}\\
EP_1^M&=\left\{(x,y)\diagup\left(\begin{array}{cc}0&0\\0&1\end{array}\right)\dbinom{x}{y}=\dbinom{x}{y}\right\}\\
&\begin{cases}
0\cdot+0\cdot y&=x;\quad x=0\\
0\cdot x+1\cdot y&=y;\quad y\in\mathbb{R}
\end{cases}\\
\intertext{El $EP_1^M$ es el eje $y$}
\end{align*}
Así que las direcciones principales de la matriz $M$ son el eje $x$ y el eje $y$ y por el Tma. de Euler, $\forall(\xi,\eta)\in B(0;1):0\leqslant q(\xi,\eta)\leqslant 1.$\\
Se traza la secante $\overline{PQ}$ donde $P(x_0,y_0)$ es un punto de la curva.\\
Su ecuación es: $\vec{r}=\vec{r}_0+\beta\vec{\mu}$ donde $\mu=\xi\vec{i}+\eta\vec{j}$ es el vector unitario $\parallel$ a $\overline{PQ}.$\\
En este caso, $$\beta_2=-\dfrac{\left.\dfrac{\partial f}{\partial x}\right)_{x_0,y_0}+\left.\eta\dfrac{\partial f}{\partial y}\right)_{x_0,y_0}}{q(\xi,\eta)}$$
Si ahora hacemos que $Q\longrightarrow P$ y llamamos $\vec{t}=\xi_t\vec{t}+\eta_t\vec{j}$ al vector unitario $\parallel$ a la tangente $t$ llevada a la curva por $P$, se tiene de nuevo que
\begin{figure}[ht!]
\begin{center}
  \includegraphics[scale=0.5]{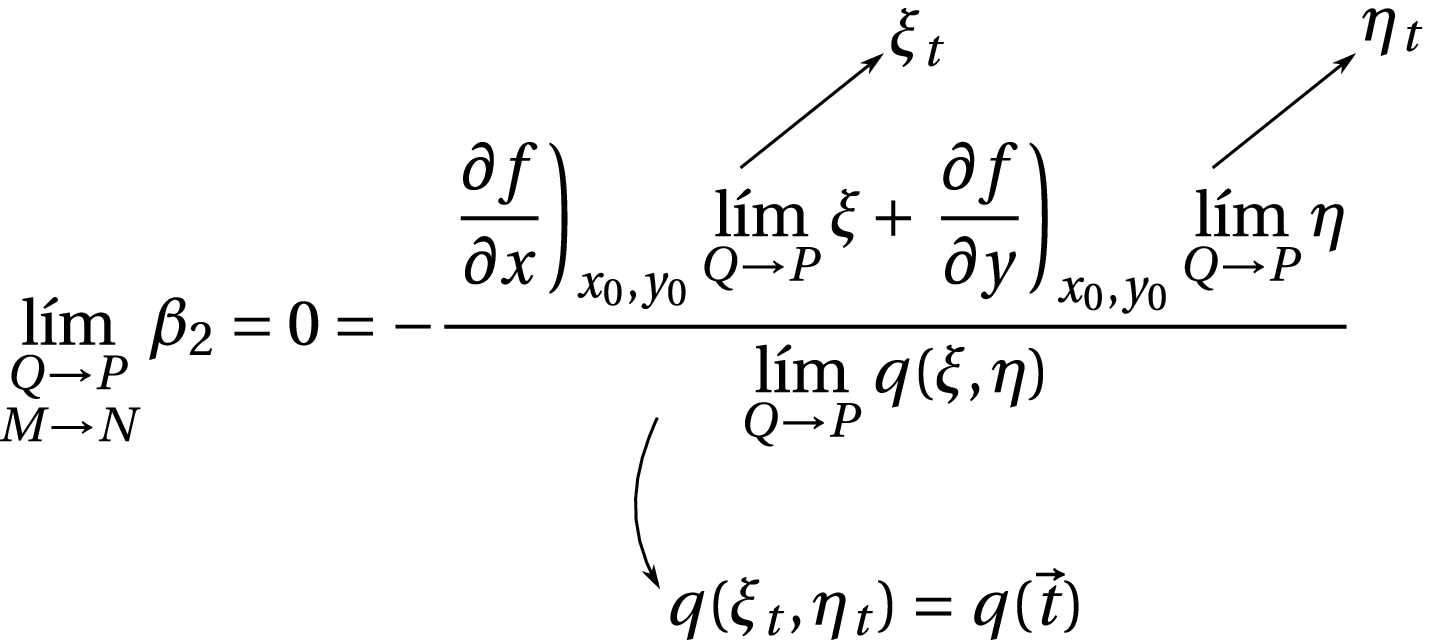}\\
\end{center}
\end{figure}
\newline
(Como $0\leqslant q(\xi,\eta)\leqslant 1$, en el movimiento de $Q\longrightarrow P$ a través de la curva, el denominador nunca se anula.)\\
Así que $$t:(x-x_0)\left.\dfrac{\partial f}{\partial x}\right)_{x_0,y_0}+(y-y_0)\left.\dfrac{\partial f}{\partial y}\right)_{x_0,y_0}=0$$
En este caso, $f(x,y)=y^2-2px$\\
$$\dfrac{\partial f}{\partial x}=-2p;\quad \dfrac{\partial f}{\partial y}=2y;\quad \left.\dfrac{\partial f}{\partial x}\right)_{x_0,y_0}=-2p;\quad\left.\dfrac{\partial f}{\partial y}\right)_{x_0,y_0}=-2y_0$$
y
\begin{align*}
t:-(x-x_0)\cancel{2}p+2(y-y_0y_0)&=0\\
\therefore\quad y_0y&=y_0^2-x_0p+px\\
&\underset{\nearrow}{=}2px_0-px_0+px\\
&\begin{cases}
y_0^2=2px_0\\
\text{ya que $P(x_0,y_0)$ está en la curva}
\end{cases}
\end{align*}
y finalmente: $$t:y_0y=p(x+x_0)$$
\item[6)] Finalmente consideremos la parábola $x^2=2py.$
\begin{figure}[ht!]
\begin{center}
  \includegraphics[scale=0.5]{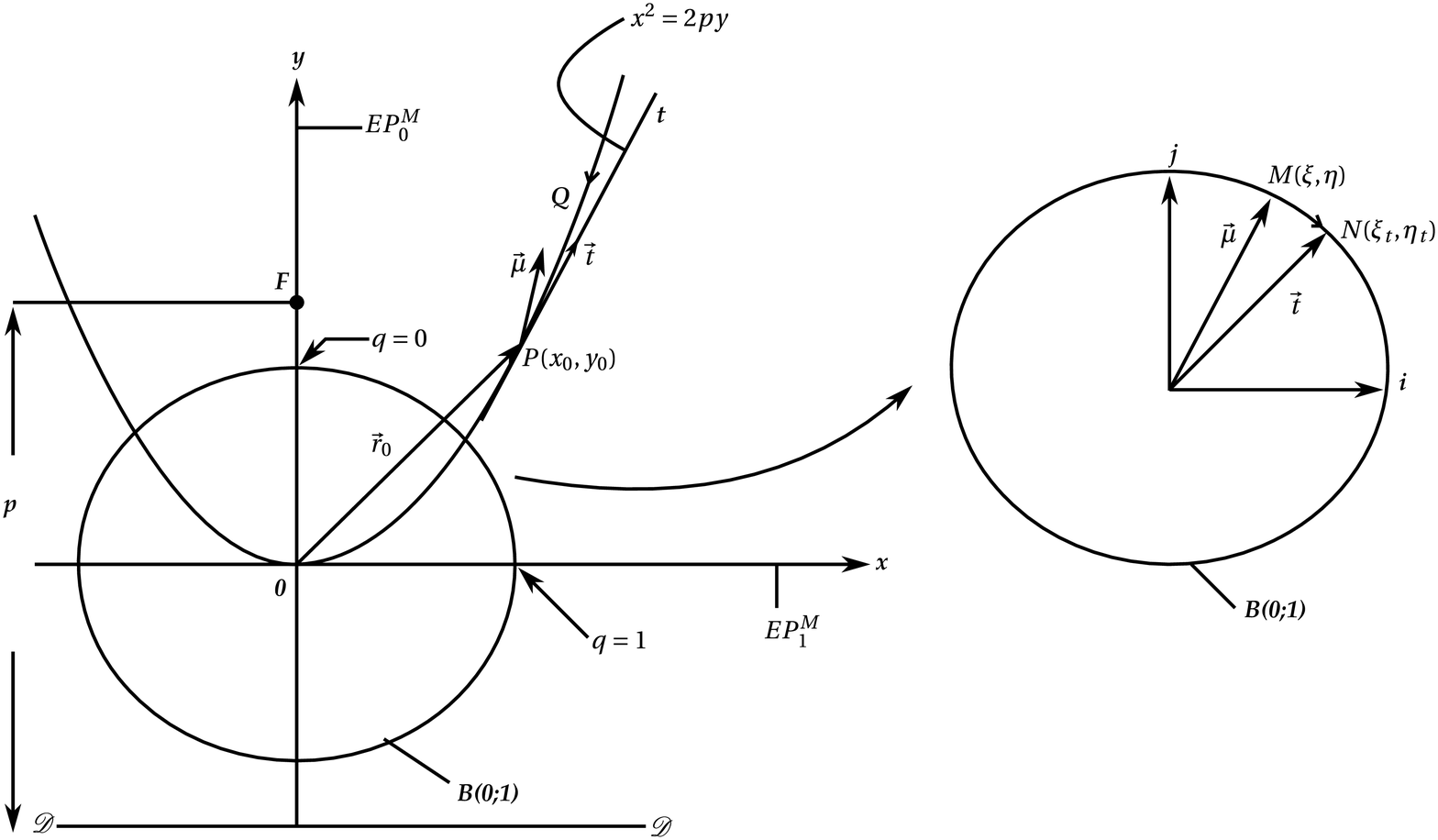}\\
\end{center}
\end{figure}
\newline
Vamos a dm. que la ecuación de la tangente $t$ a la curva por el punto $P(x_0,y_0)$ de la curva se construye así:
\begin{align*}
x^2&=2py\\
x\cdot x&=p(y+y)\\
t:x_0x&=p(y_0+y)
\end{align*}
La cónica es $x^2-2py=0$\\
$$q(x,y)=x^2=\left(\begin{array}{cc}x&y\end{array}\right)\underset{\overset{\parallel}{M}}{\underbrace{\left(\begin{array}{cc}1&0\\0&0\end{array}\right)}}\dbinom{x}{y}$$
\begin{align*}
PCM(\lambda)=\lambda^2+\omega\lambda+\delta&\underset{\uparrow}{=}\lambda^2-\lambda=\lambda(\lambda-1)\\
&\begin{array}{cc}\omega=1\\ \delta=0\end{array}
\end{align*}
Los valores propios de $M$ son 0 y 1.\\
El $EP_1^M$ es el eje, el $EP_0^M$ es el eje $y.$\\
Luego por el Tma. de Euler, $\forall(\xi,\eta)\in B(0;1):0\leqslant q(\xi,\eta)\leqslant 1.$\\
Consideremos la secante $\overline{PQ}$ a la curva por $P(x_0,y_0)$. Llamemos $\vec{\mu}=\xi\vec{i}+\eta\vec{j}$
al vector unitario $\parallel$ a la tangente $t$ a la curva por $P(x_0,y_0).$\\
$\overline{PQ}=\vec{r}=\vec{r}_0+\beta\vec{\mu}.$\\
El parámetro $\beta$ vale $\beta_1=0$ en $P$ y en $Q,$ $$\beta_2=-\dfrac{\left.\dfrac{\partial f}{\partial x}\right)_{x_0,y_0}+\left.\eta\dfrac{\partial f}{\partial y}\right)_{x_0,y_0}}{q(\xi,\eta)}$$
Cuando $Q\longrightarrow P$ a través de la curva, $M\longrightarrow N$ a través de la $B(0;1)$ y $q(\xi,\eta)=\xi^2$ nunca se anula.
\begin{figure}[ht!]
\begin{center}
  \includegraphics[scale=0.5]{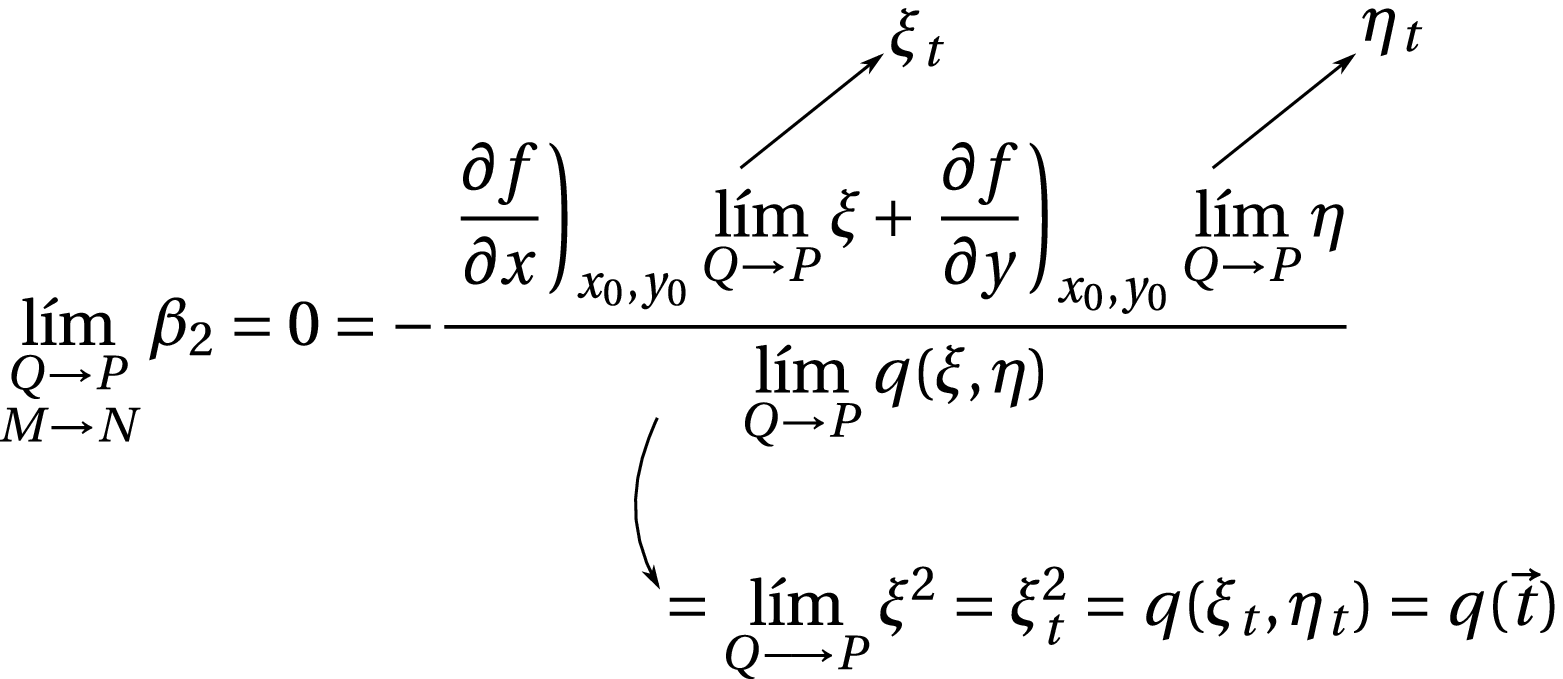}\\
\end{center}
\end{figure}
\newline
De nuevo: $$\underset{\text{pendiente de t}}{\underbrace{m_t}}=-\dfrac{\left.\dfrac{\partial f}{\partial x}\right)_{x_0,y_0}}{\left.\dfrac{\partial f}{\partial y}\right)_{x_0,y_0}}$$ y la ecuación de $t$ es:
$$(x-x_0)\left.\dfrac{\partial f}{\partial x}\right)_{x_0,y_0}+(y-y_0)\left.\dfrac{\partial f}{\partial y}\right)_{x_0,y_0}=0$$
En este caso, $f(x,y)=x^2-2py;\quad\dfrac{\partial f}{\partial x}=2x\quad\dfrac{\partial f}{\partial y}=-2p$\\
Luego la ecuación de $t$ es:
\begin{align*}
(x-x_0)2x_0+(y-y_0)(-2p)&=0\\
x_0x-x_0^2-py+py_0&=0\\
x_0x=py+x_0^2-py_0&\underset{\uparrow}{=}py+2py_0-py_0\\
&\begin{cases}
x_0^2=2py_0\\
\text{ya que $P(x_0,y_0)$}\\
\text{está en la curva}
\end{cases}
\end{align*}
y finalmente, $x_0x=p(y_0+y).$\\
\item[7)] Consideremos la cónica
\begin{align*}
Ax^2+2Bxy+Cy^2+2Dx+2Ey+F=0\\
q(x,y)=Ax^2+2Bxy+Cy^2=\left(\begin{array}{cc}x&y\end{array}\right)\underset{\overset{\parallel}{M}}{\underbrace{\left(\begin{array}{cc}A&B\\B&C\end{array}\right)}}\dbinom{x}{y}
\end{align*}
con $\delta=AC-B^2\neq 0.$ O sea que la cónica tiene centro único.\\
Supongamos que luego de trasladar y rotar los ejes se obtiene una elipse de ecuación $\lambda_1x'^2+\lambda_2y'^2=-\dfrac{\Delta}{\delta}$ donde $\lambda_1$ y $\lambda_2$ son las raíces del $$PCM(\lambda)=\lambda^2-\omega\lambda+\delta=\lambda^2-(A+C)\lambda+(AC-B^2).$$ Por el Tma. de Euler, $\forall(\xi,\eta)\in B(0;1)\quad\min\{\lambda_1,\lambda_2\}\leqslant q(\xi,\eta)\leqslant\max{\lambda_1,\lambda_2}.$\\
(Recuérdese que en el caso de una elipse, $\lambda_1$ y $\lambda_2$ tiene el mismo signo que el signo de $-\dfrac{\Delta}{\delta}$)
\begin{figure}[ht!]
\begin{center}
  \includegraphics[scale=0.5]{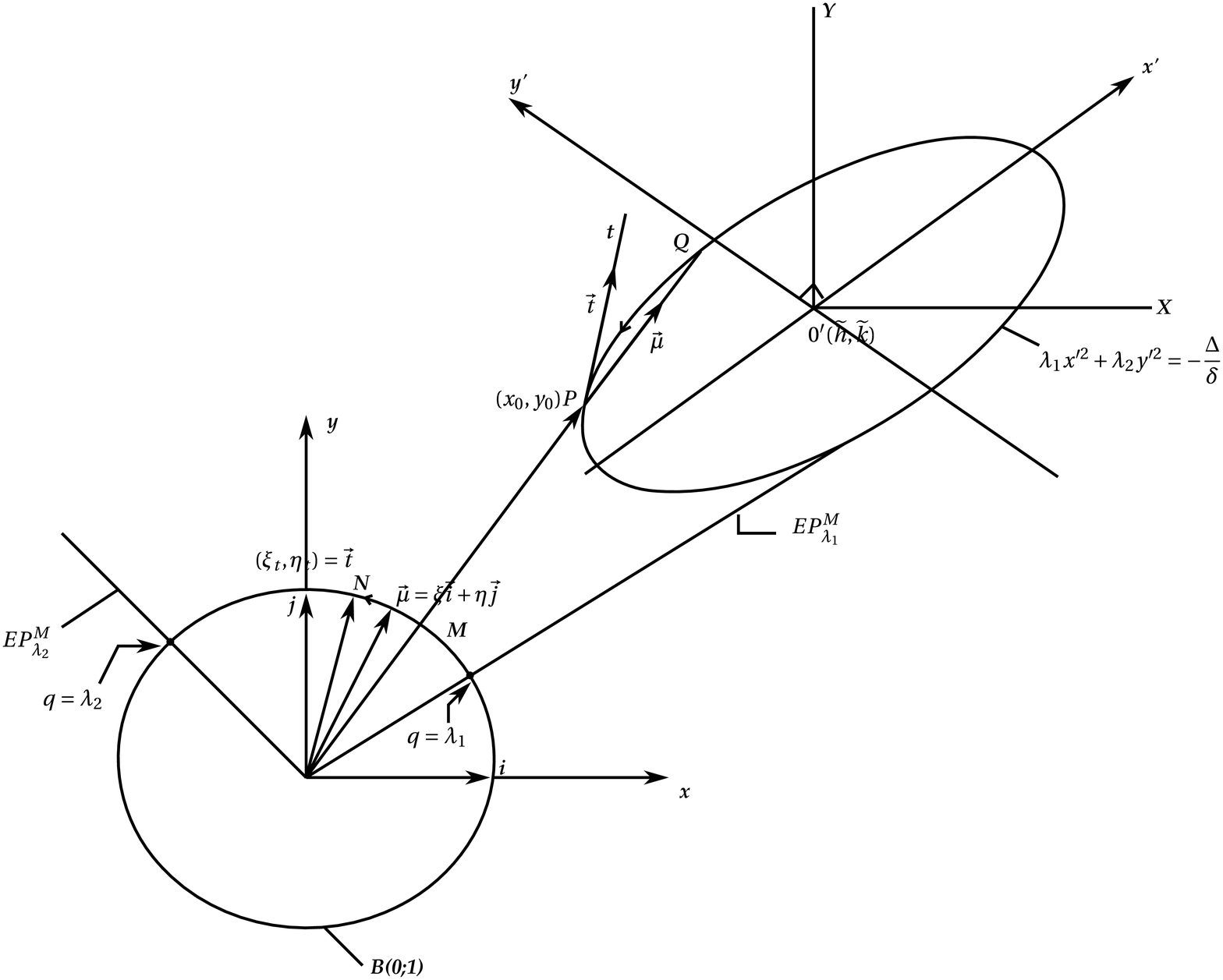}\\
\end{center}
\end{figure}
\newline
Tomemos un punto $P(x_0,y_0)$ de la curva y $Q$ un punto cercano a $P.$\\
Llamemos $\vec{\mu}=\xi\vec{i}+\eta\vec{j}$ al vector unitario $\parallel$a $\overline{PQ}$ y $\vec{t}=\xi_t\vec{i}+\eta_t\vec{j}$ el vector unitario $\parallel$ a la tangente $t$ a la curva por $P.$\\
La ecuación de $\overset{\longleftrightarrow}{PQ}$ es: $\vec{t}=\vec{r}_0+\beta\vec{\mu}.$ El valor del parámetro $\beta$ es $\beta_1=0$ en $P$ y $\beta_2=-\dfrac{\left.\dfrac{\partial f}{\partial x}\right)_{x_0,y_0}+\left.\eta\dfrac{\partial f}{\partial y}\right)_{x_0,y_0}}{q(\xi,\eta)}$ en el punto $Q.$\\
Cuando $Q\longrightarrow P$ a través de la curva, $M\longrightarrow N$ a través de la $B(0;1)$ y el denominador en la expresión para $\beta_2$ no se anula en ningún momento. Luego
\begin{figure}[ht!]
\begin{center}
  \includegraphics[scale=0.5]{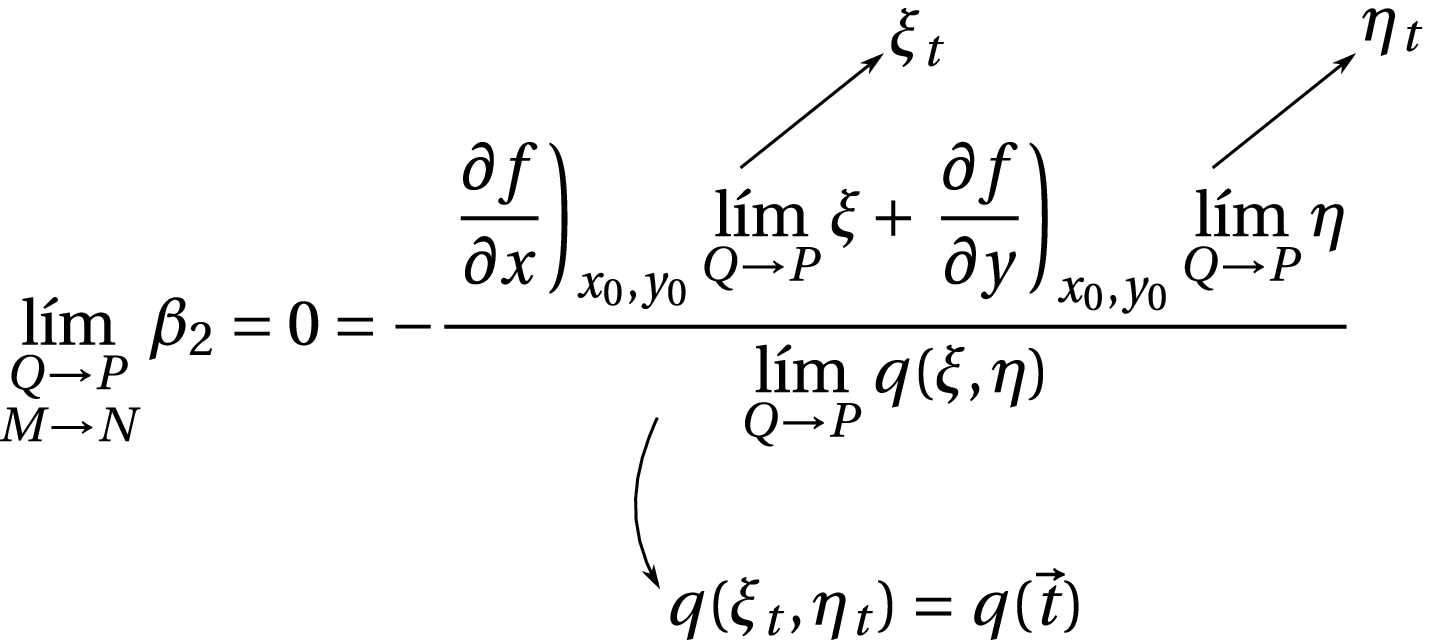}\\
\end{center}
\end{figure}
\newline
Se tiene que:
$$\underset{\text{pendiente de t}}{\underbrace{m_t}}=-\dfrac{\left.\dfrac{\partial f}{\partial x}\right)_{x_0,y_0}}{\left.\dfrac{\partial f}{\partial y}\right)_{x_0,y_0}}$$
y la ecuación de la tangente es:
$$t:(x-x_0)\left.\dfrac{\partial f}{\partial x}\right)_{x_0,y_0}+(y-y_0)\left.\dfrac{\partial f}{\partial y}\right)_{x_0,y_0}=0$$
En este caso,
\begin{align*}
\left.\dfrac{\partial f}{\partial x}\right)_{x_0,y_0}&=2Ax_0+2By_0+2D\\
\left.\dfrac{\partial f}{\partial y}\right)_{x_0,y_0}&=2Cy+2Bx+2E
\end{align*}
Luego
\begin{align*}
&t:\left(2Ax_0+2By_0+2D\right)(x-x_0)+\left(2Cy_0+Bx_0+2E\right)(y-y_0)=0\\
&Ax_0x-Ax_0^2+By_0x-2B_0y_0+Dx-Dx_0+Bx_0y+Cy_oy-Cy_0^2+Ey-Ey_0=0
\end{align*}
O sea:
\begin{figure}[ht!]
\begin{center}
  \includegraphics[scale=0.5]{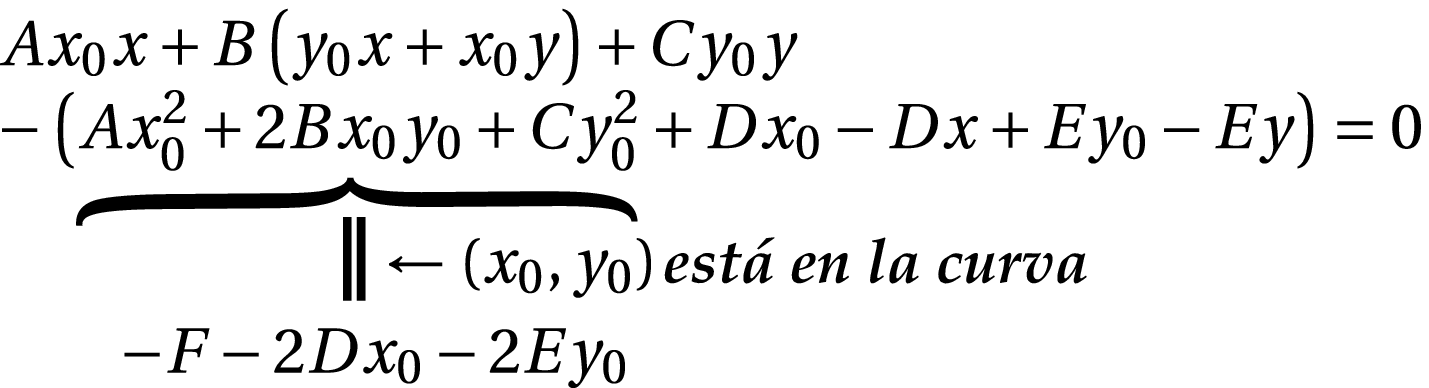}\\
\end{center}
\end{figure}
\newline
\begin{align*}
Ax_0x+B\left(y_0x+x_0y\right)+Cy_oy-\left(-F-2Dx_0-2Ey_0+Dx_0-Dx+Ey_0-Ey\right)&=0\\
Ax_0x+B\left(y_0y+x_0y\right)+Cy_0y+F+Dx_0+Dx+Ey_0+Ey+F&=0
\end{align*}
y finalmente:
$$t:Ax_0+B\left(y_0x+x_0y\right)+Cy^2+D\left(x_0+x\right)+E\left(y_0+y\right)+F=0$$
que se puede construir a partir de la ecuación de la cónica así: $$Ax^2+2Bxy+Cy^2+2Dx+2Ey+F=0$$
\begin{figure}[ht!]
\begin{center}
  \includegraphics[scale=0.5]{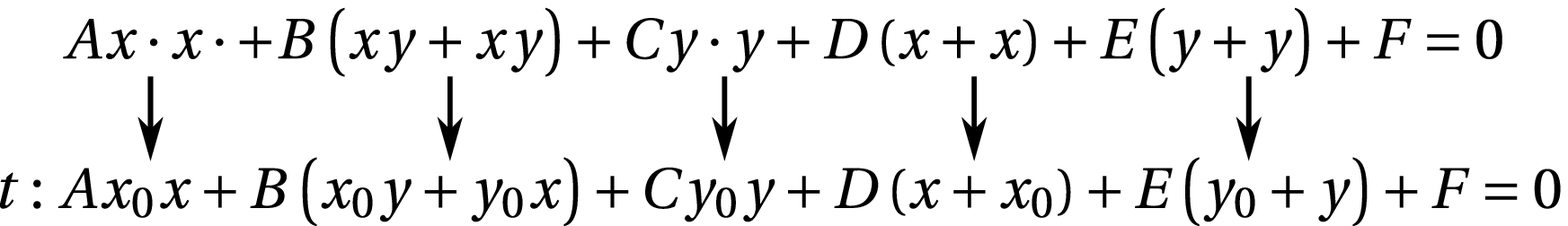}\\
\end{center}
\end{figure}
\end{enumerate}
\begin{ejem}
Considere la cónica $$3x^2+4xy+2y^2+3x+y-11=0.$$
Utilizando los invariantes diga que tipo de cónica es, dibújela y encuentre la ecuación de la tangente en el punto $(\xi,1)$ de la curva.\\
$A=3; 2B=4, B=2; C=2; 2D=3, D=\dfrac{3}{2}; 2E=1, E=\dfrac{1}{2}; F=-11.$ $P(\underset{\overset{\parallel}{x_0}}{2},\underset{\overset{\parallel}{y_0}}{-1}).$ La ecuación de la tangente es:
\begin{align*}
3x_0x+2\left(x_0y+y_oy\right)+2y_0y+\dfrac{3}{2}\left(x+x_0\right)+\dfrac{1}{2}\left(y_0+y\right)-11&=0\\
3\cdot 2\cdot x+2\left(2y-x\right)+2(-1)y+\dfrac{3}{2}\left(x+2\right)+\dfrac{1}{2}\left(-1+y\right)-11&=0
\end{align*}
que después de simplificar se convierte en $$11x+5y-17=0.$$
\end{ejem}
\section{Las secciones cónicas ó intersecciones de un cono con un plano.}
Las secciones cónicas se las define como las curvas al cortar un cono circular recto con un plano. Fue desde ese punto de vista como fueron estudiadas por los primeros geómetras que se obtienen.\\
\begin{enumerate}
\item[i)] Si el plano es $\perp$ al eje del cono, la sección es una circunferencia.\\
\item[ii)] Si el plano es $\parallel$ al eje del cono la sección que se obtiene es una hipérbola.\\
\item[iii)] Si el plano contiene al eje se obtienen dos rectas (dos generatrices del cono.)
\end{enumerate}
Estas tres definiciones son fáciles de demostrar.\\
Vamos a dm. la $ii).$\\
Consideremos el cono circular recto de la Fig. con vértice en $V$ y que tiene por directriz la circunferencia $\mathscr{C}$ de centro en el eje $Z$, radio $R$ y a una distancia $VA=\delta$ de $V; Z$ es el eje del cono.\\
Tomemos como plano $xy$ el plano que por $V$ es $\perp$ al eje $Z.$
\newpage
\begin{figure}[ht!]
\begin{center}
  \includegraphics[scale=0.5]{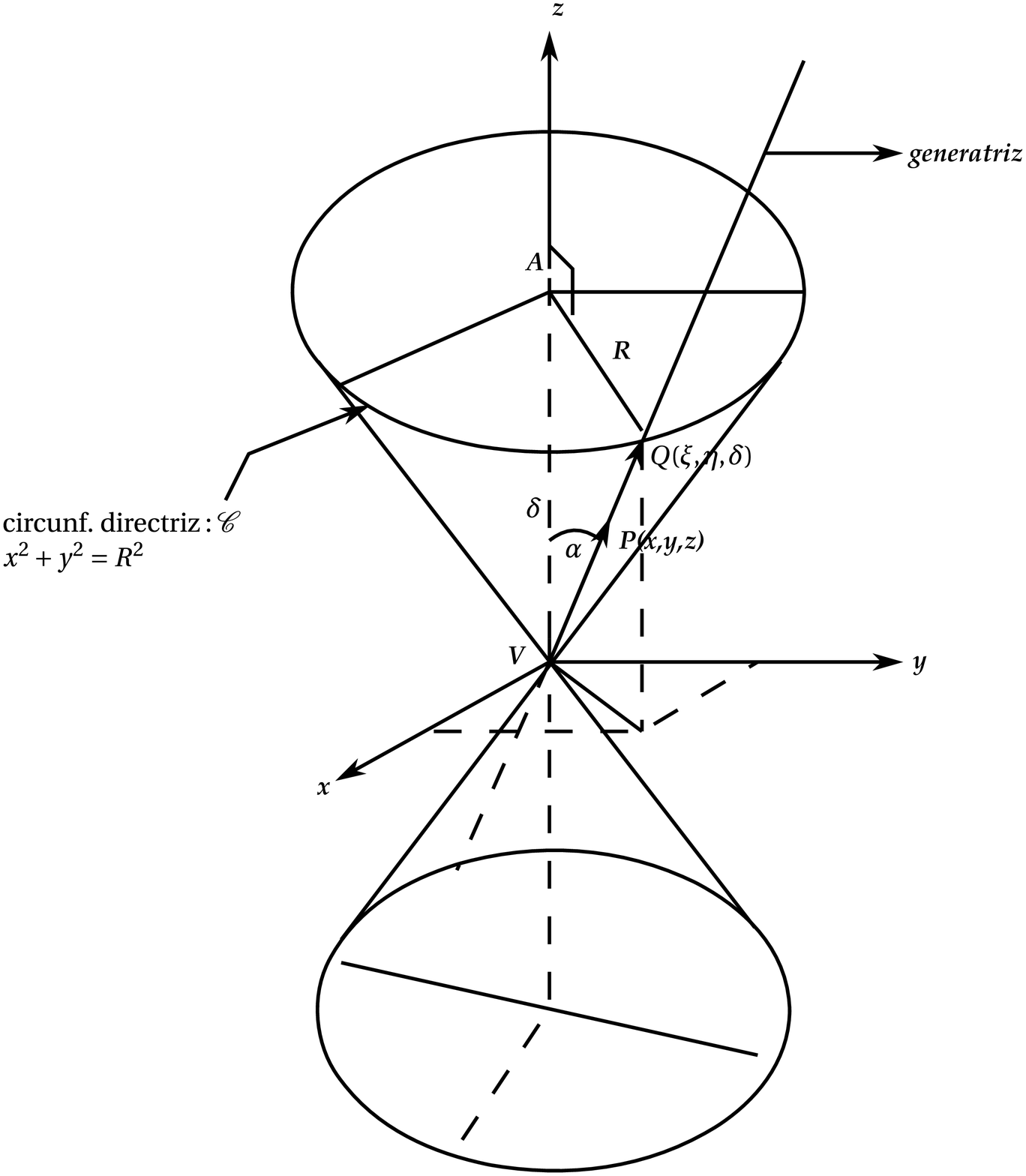}\\
  \caption{}\label{}
\end{center}
\end{figure}
Vamos a obtener en primer lugar la ecuación del cono$\diagup xyz.$\\
Sea $P(x,y,z)$ un punto de la superficie.\\
Llamemos $Q(\xi,\eta,\delta)$ al punto de la generatriz $\overset{\longleftrightarrow}{VP}$ que está en la directriz $\mathscr{C}.$\\
Como $Q\in\mathscr{C}, \xi^2+\eta^2=R^2.$\\
Ahora, $$\overset{\longleftrightarrow}{VP}=(x,y,z)=\lambda\overset{\longleftrightarrow}{OQ}=\lambda(\xi,\eta,\delta)$$
\begin{equation*}
\left.
\begin{split}
\therefore\quad x&=\lambda\xi\\
y&=\lambda\eta\\
z&=\lambda\delta
\end{split}
\right\}\hspace{0.5cm}x^2+y^2=\lambda^2(\xi^2+\eta^2)=\lambda^2R^2\underset{\overset{\uparrow}{\lambda=\dfrac{z}{\delta}}}{=}\dfrac{R^2}{\delta^2}z^2
\end{equation*}
Llamemos $\alpha$ el ángulo en el vértice del cono.\\
Entonces $\tan\alpha=\dfrac{R}{\delta}$ y por lo tanto, $\dfrac{R^2}{\delta^2}=\tan^2\alpha.$\\
Así que $x^2+y^2=\tan^2\alpha\cdot z^2$ es la ecuación del cono$\diagup xyz.$\\
Ahora se trata de probar que la sección del cono con un plano $\parallel$ al eje del cono y que no pase por $V$ es una hipérbola.
\begin{figure}[ht!]
\begin{center}
  \includegraphics[scale=0.5]{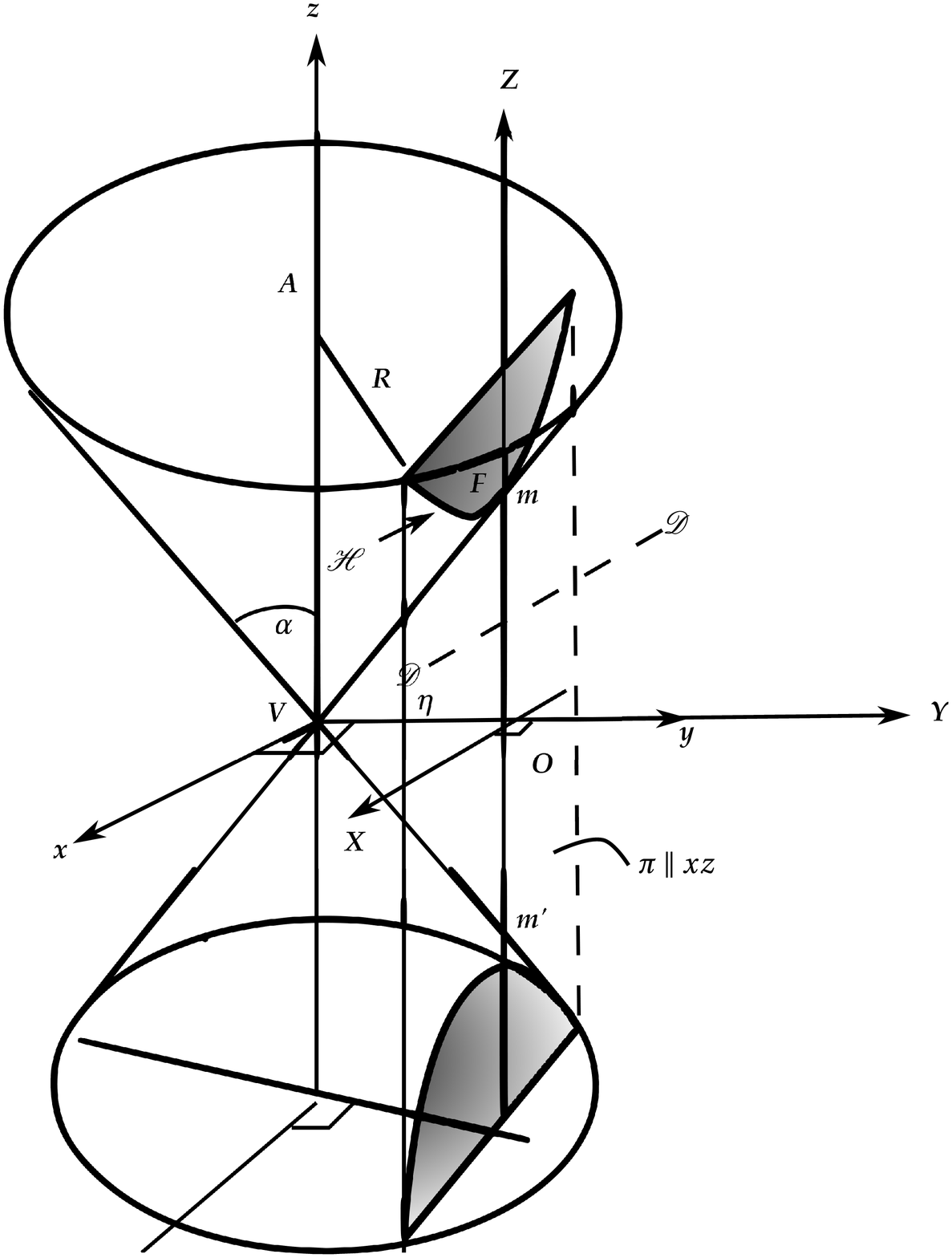}\\
  \caption{}\label{}
\end{center}
\end{figure}
\newline
La situación está representada en la Fig:1.7 en la que el plano $\pi$ es $\parallel$ al plano $xz$ y queda definido dando la distancia $VO=\eta$ que suponemos conocida. La $\bigcap$ del cono con el plano $\pi$ es la curva $\mathscr{H}$. Vamos a dm. que $\mathscr{H}$ es una hipérbola y a definir todos sus elementos: focos, directriz, excentricidad, etc.\\
Definimos ejes $XYZ$ paralelos a $xyz$ y con origen en $O(0,\eta,0),$ Fig:1.7. Las ecuaciones de transf. de coordenadas son:
\begin{equation}\label{86}
\left.
\begin{split}
x=X\\
y=Y+\eta\\
z=Z
\end{split}
\right\}
\end{equation}
Ahora, la ecuación del cono$\diagup xyz$ es: $x^2+y^2=\tan^2\alpha\cdot z^2$ y al tener en cuenta [\ref{86}]
\begin{equation}\label{87}
X^2+{\left(Y+\eta\right)}^2=\tan^2\alpha\cdot Z^2:\text{ecuación del cono$\diagup XYZ$ con origen en $0.$}
\end{equation}
$\pi=\left\{(X,Y,Z)\diagup Y=0\right\}.$ O sea que la ecuación del plano $\pi\diagup XYZ$ es $Y=0$ que llevada a [\ref{87}] nos da: $$X^2+\eta^2=\tan^2\alpha\cdot Z^2:\text{ecuación de $\mathscr{H}\diagup XYZ$}$$
Así que la curva $\mathscr{H}$ está en el plano $XZ$ y tiene ecuación:
\begin{align*}
X^2+\eta^2=\tan^2\alpha\cdot Z^2\\
\tan^2\alpha\cdot Z^2-X^2=\eta^2\\
\intertext{O sea que}\\
\dfrac{Z^2}{{\left(\dfrac{\eta}{\tan\alpha}\right)}^2}-\dfrac{X^2}{\eta^2}=1
\end{align*}
lo que nos dm. que $\mathscr{H}$ es una hipérbola.\\
Si miramos la curva y sus ejes desde el punto $V$ aparece así:
\begin{figure}[ht!]
\begin{center}
  \includegraphics[scale=0.5]{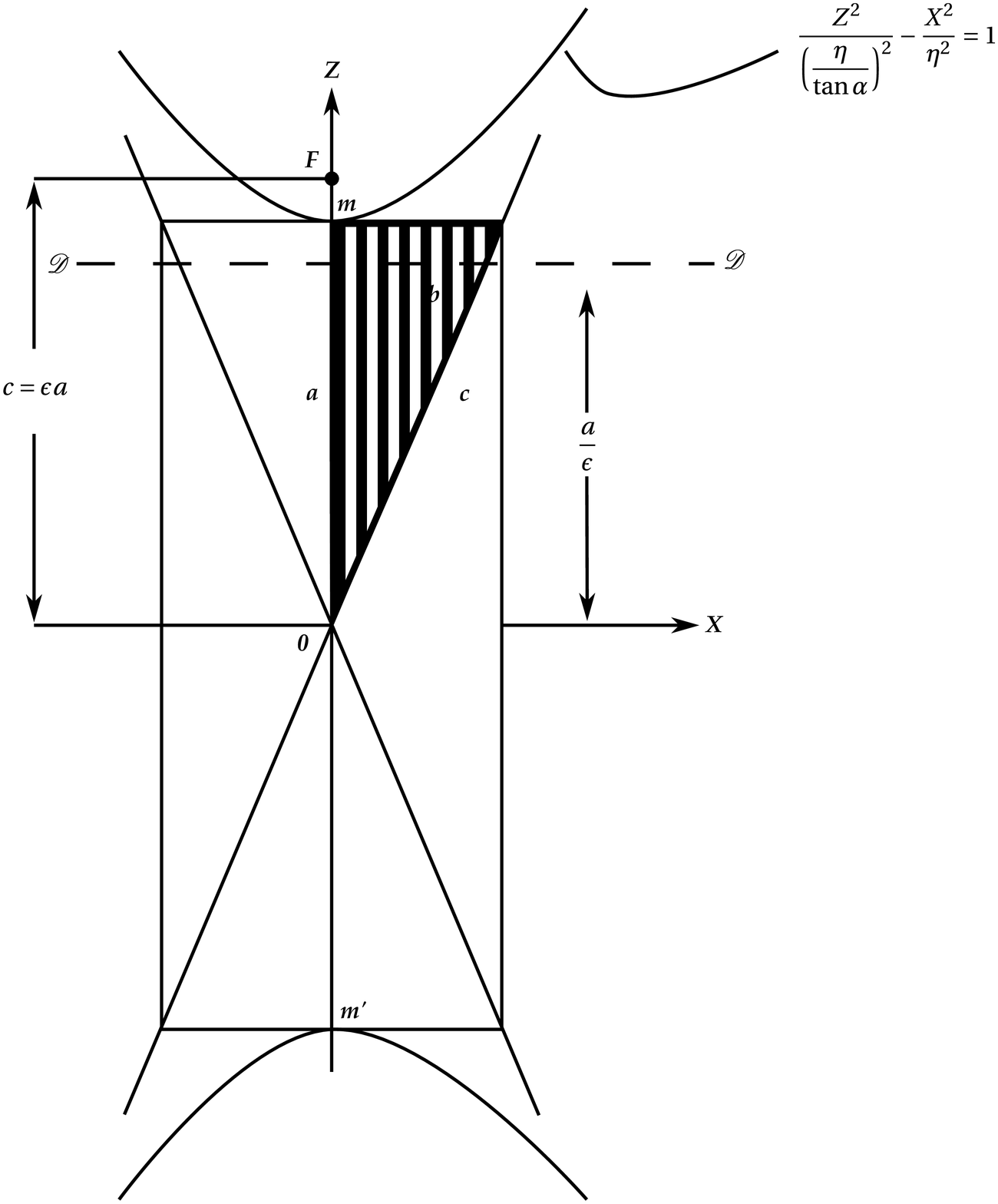}\\
  \caption{}\label{}
\end{center}
\end{figure}
\newline
$$a=\dfrac{\eta}{\tan\alpha};\quad b=\eta;\quad c=\sqrt{a^2+b^2}=\sqrt{\dfrac{\eta^2}{\tan^2\alpha}+\eta^2}=\dfrac{\eta}{\tan\alpha}\sqrt{1+\tan^2\alpha}=\dfrac{\eta}{\sin\alpha}$$
Nótese que $$\tan\alpha=\dfrac{\sin\alpha}{\cos\alpha}\quad\therefore\quad\sin\alpha=\cos\alpha\cdot\tan\alpha\underset{\overset{\uparrow}{\cos\alpha<1}}{<}\tan\alpha$$
Luego $$\dfrac{\eta}{\tan\alpha}<\dfrac{\eta}{\sin\alpha},\quad\text{o sea que $a<c.$}$$
$$\epsilon=\dfrac{c}{a}=\dfrac{\dfrac{\eta}{\sin\alpha}}{\dfrac{\eta}{\tan\alpha}}=\sec\alpha.$$
Esto es un hecho sorprendente!.\\
\textit{Todos las secciones obtenidas al cortar el cono con planos $\parallel$s al plano xy son hipérbolas con la misma excentricidad: $\epsilon=\sec\alpha$.}\\
La directriz $\mathscr{DD}$ se localiza así (Fig. 1.8.): $$\dfrac{a}{\epsilon}=\dfrac{\dfrac{\eta}{\tan\alpha}}{\sec\alpha}=\dfrac{\eta\cos^2\alpha}{\sin\alpha}.$$
Las coordenadas del foco $F\diagup XYZ$ son $\left(0,0,c=\dfrac{\eta}{\sin\alpha}\right)$ y respecto a $xyz$ serían, regresando a [\ref{86}],
\begin{align*}
x&=0\\
y&=\eta\quad\therefore\quad z=\dfrac{1}{\sin\alpha}y\\
z&=\dfrac{\eta}{\sin\alpha}
\end{align*}
Esto nos señala que cuando el plano $\pi$ se desplaza paralelamente a $xy,$ el foco $F$ de la hipérbola se mueve por la recta $$\begin{cases}x&=0\\ t&=\dfrac{1}{\sin\alpha}y\end{cases}.$$\\
Ahora cortemos el cono con un plano $\pi$ no $\parallel$ al eje del cono.\\
Llamemos $\beta$ al ángulo que hace el plano $\pi$ con el eje del cono y $0$ al punto donde el plano $\pi$ corta al eje del cono, $0\neq V.$
\begin{figure}[ht!]
\begin{center}
  \includegraphics[scale=0.4]{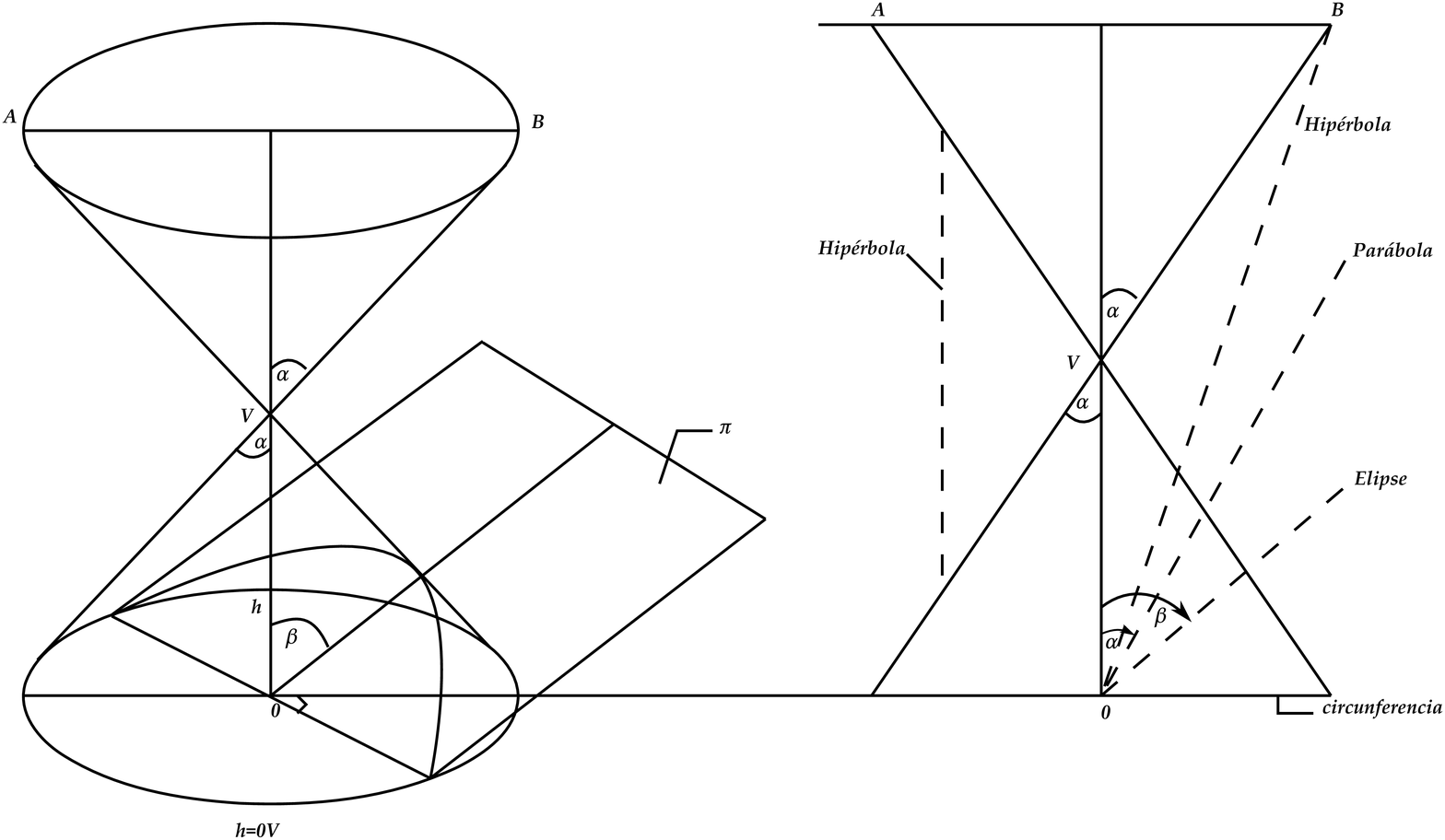}\\
  \caption{}\label{}
\end{center}
\end{figure}
Vamos a dm. que
\begin{itemize}
\item Si $0<\beta<\alpha,$ la sección es una hipérbola.\\
\item Si $\beta=\alpha,$ la sección es una parábola.\\
\item $\beta_i\quad\alpha<\beta<\pi/2,$ la sección cónica es una elipse. (Fig. 1.9)
\end{itemize}
Recordemos que $x^2+y^2=\tan^2\alpha\cdot z^2$ es la ecuación del cono$\diagup xyz$ con origen (Fig. 1.10.)\\
Llamemos $h=0V.$\\
Primero vamos a trasladar los ejes $XYZ$ con origen en $V$ al punto $O$ donde $O(0,0,h)\diagup xyz.$
\begin{figure}[ht!]
\begin{center}
  \includegraphics[scale=0.5]{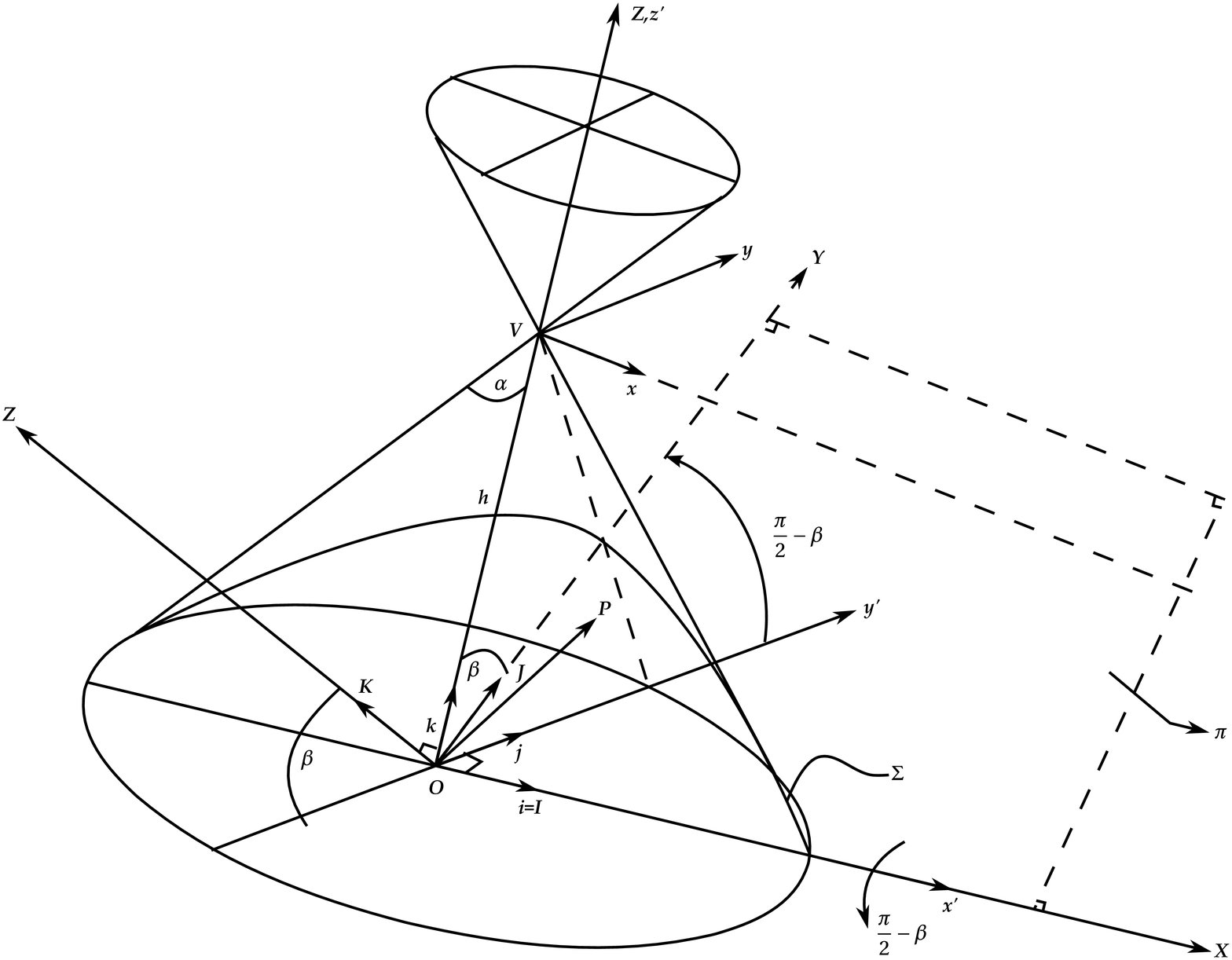}\\
  \caption{}\label{}
\end{center}
\end{figure}
\newline
Las ecuaciones de la transformación son:
\begin{align*}
x&=x'\\
y&=y'\\
z&=z'-h
\end{align*}
que llevamos a la ecuación del cono obteniendo: $\tan^2\alpha{(z-h)}^2=x'^2+y'^2:$ ecuación del cono$\diagup x'y'z'$ con origen en $O.$ Ahora consideremos el sistema ortogonal $XYZ$ con origen en $O$ y conseguido al rotar los ejes $x'y'z'$ un ángulo $\pi/2-\beta\circlearrowleft$ al rededor del eje $x'.$\\
Si $i,j,k, I,J,K$ son los vectores unitarios en las direcciones de los ejes,
\begin{align*}
I&=i\\
J&=\sin\beta j+\cos\beta k\\
K&=-\cos\beta j+\sin\beta k
\end{align*}
Nótese que la curva $\Sigma,$ intersección de $\pi$ con el cono está en el plano $XY$ y que el eje $Z$ es $\perp$ a $\pi.$\\
Sea $P(x',y',z')$ (ó $P(X,Y,Z)$) un punto cualquiera del cono.\\
Entonces $[\vec{OP}]_{ijk}=[I]_{ijk}^{IJK}[\vec{OP}]_{IJK}.$ O sea que
\begin{align*}
\left(\begin{array}{cc}x'\\y'\\z'\end{array}\right)&=\left(\begin{array}{ccc}1&0&0\\0&\sin\beta&-\cos\beta\\0&\cos\beta&\sin\beta\end{array}\right)\left(\begin{array}{cc}X\\Y\\Z\end{array}\right)\\
\therefore\quad x'&=X\\
y'&=\sin\beta Y-\cos\beta Z\\
z'&=\cos\beta Y+\sin\beta Z
\end{align*}
que llevamos a la ecuación del cono$\diagup x'y'z'.$ $$\tan^2\alpha{\left(\cos\beta Y+\sin\beta Z-h\right)}^2=X^2+{\left(\sin\beta Y-\cos\beta Z\right)}^2:\text{ecuación del cono$\diagup XYZ$}$$
Si en la ecuación anterior hacemos $Z=0$ obtenemos la ecuación de la curva $\Sigma$. O sea que $$\tan^2\alpha{\left(\cos\beta Y-h\right)}^2=X^2+\sin^2\beta Y^2$$ es la ecuación de $\Sigma\diagup XY$ y que podemos simplificar así:
\begin{align*}
\tan^2\alpha\left(\cos^2\beta Y^2-2h\cos\beta Y+h^2\right)=X^2+\sin^2\beta Y^2\\
\intertext{ó}\\
X^2+\left(\sin^2\beta-\tan^2\alpha\cos^2\beta\right)Y^2+2h\tan^2\alpha\cos\beta Y=h^2\tan^2\alpha\quad\star
\end{align*}
Vamos a analizar tres casos en la ecuación $\star.$
\begin{enumerate}
\item Supongamos que $\beta=\alpha.$ Esto significa que $\pi$ es $\parallel$ a una generatriz del cono. El
coeficiente de $Y^2$ en la ecuación $\star$ se anula y la ecuación de $\Sigma\diagup XY$ es:
\begin{align*}
X^2+2h\tan^2\alpha\cos\alpha Y=\tan^2\alpha h^2\\
\therefore\quad Y=-\dfrac{1}{2h\tan^2\alpha\cos\alpha}X^2+\dfrac{h}{2\cos\alpha}
\end{align*}
La sección $\Sigma$ es una \underline{PARÁBOLA}.\\
Los elementos de la curva, foco, directriz, etc... se calculan fácilmente$\diagup XY.$
\begin{figure}[ht!]
\begin{center}
  \includegraphics[scale=0.5]{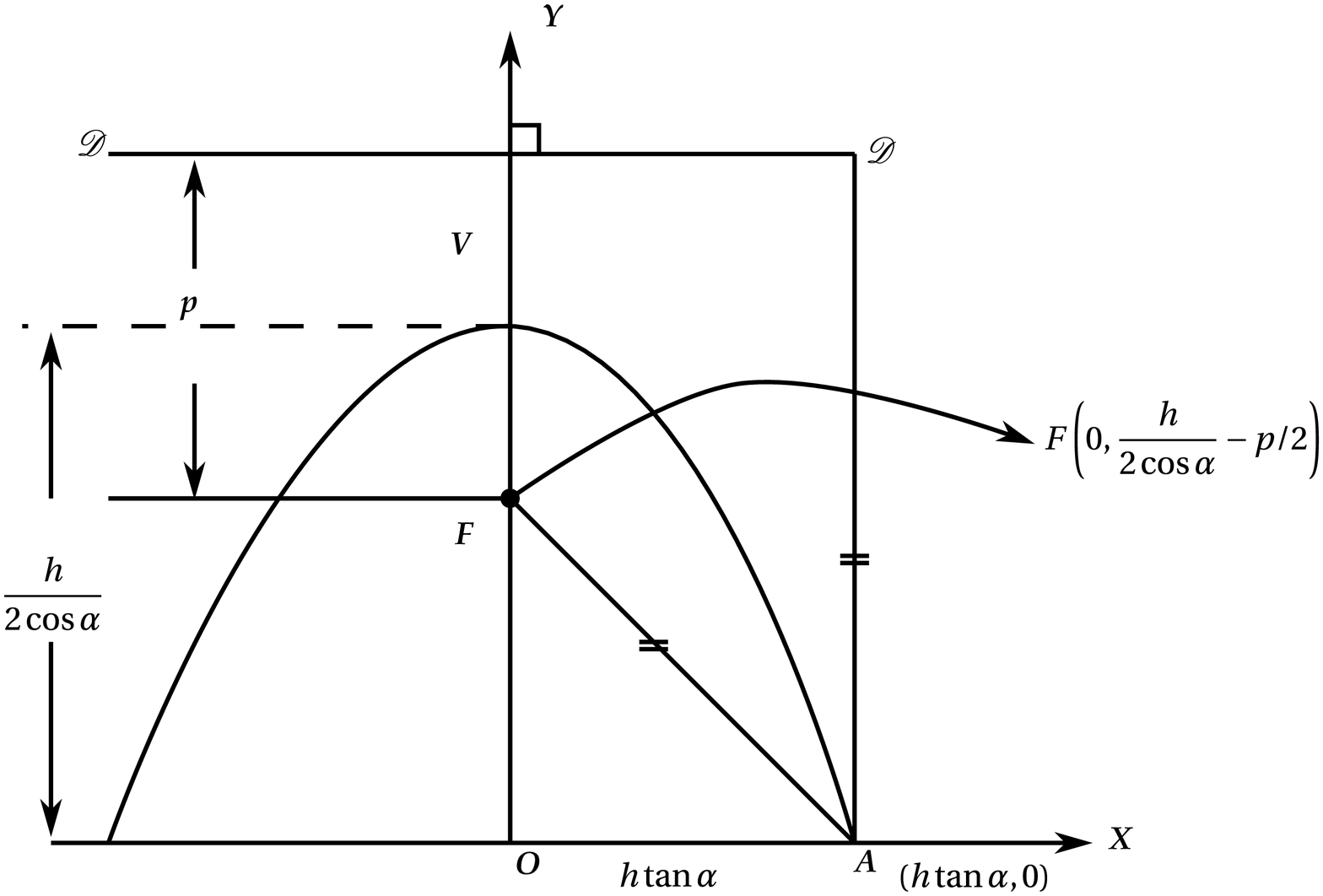}\\
  \caption{}\label{}
\end{center}
\end{figure}
\newline
Si $X=0,\quad Y=OV=\dfrac{h}{2\cos\alpha}$\\
Si $Y=0,\quad O=-\dfrac{1}{2h\tan^2\alpha\cos\alpha}OA^2+\dfrac{h}{2\cos\alpha}$
\begin{align*}
OA^2&=\dfrac{h}{\cancel{2}\cos\alpha}\cancel{2}h\tan^2\alpha\cos\alpha=\dfrac{h^2}{\cancel{\cos\alpha}}\tan\alpha\cancel{\cos\alpha}\\
OA^2&=h\tan\alpha
\end{align*}
Como $A$ está en la parábola, $AD=AF.$ O sea que
\begin{align*}
\dfrac{h}{2\cos\alpha}+\dfrac{p}{2}&=\sqrt{h^2\tan^2\alpha+{\left(\dfrac{h}{2\cos\alpha}-\dfrac{p}{2}\right)}^2}\\
{\left(\dfrac{h}{2\cos\alpha}+\dfrac{p}{2}\right)}^2&=h^2\tan^2\alpha+{\left(\dfrac{h}{2\cos\alpha}+\dfrac{p}{2}\right)}^2\\
&\therefore\quad p=h\tan^2\alpha\cos\alpha=d(F;\mathscr{DD}):\text{distancia foco-directriz}
\end{align*}
El foco $F$ se localiza así: $$OF=\dfrac{h}{2\cos\alpha}-\dfrac{h\tan^2\alpha\cos\alpha}{2}=\dfrac{1}{2\cos\alpha}(1-\sin^2\alpha)$$
\item Supongamos ahora que $\alpha<\beta<\pi/2$ y consideremos la cónica $\Sigma$
$$X^2+\left(\sin^2\beta-\tan^2\alpha\cos^2\beta\right)Y^2+2h\tan^2\alpha\cos\beta Y-h^2\tan^2\alpha=0$$
Vamos a emplear lo estudiado en las secciones anteriores para reducirla. Si consideramos la ecuación general de las cónicas $AX^2+2BY+CY^2+2DX+2EY+F$ se tiene que\\
$\begin{cases}
A=1\\
B=0\\
C=\sin^2\beta-\tan^2\alpha\cos^2\beta\\
D=0\\
2E=2h\tan^2\alpha\cos\beta\quad\therefore\quad E=h\tan^2\alpha\cos\beta\\
F=-h^2\tan^2\alpha
\end{cases}$\\
Llamemos $e=\dfrac{\cos\beta}{\cos\alpha}.$ Como $\alpha<\beta,\quad\cos\beta<\cos\alpha$ y por lo tanto, $e=\dfrac{\cos\beta}{\cos\alpha}<1.$ Vamos a dm. que en éste $\Sigma$ es una elipse de excentricidad $e=\dfrac{\cos\beta}{\cos\alpha}.$\\
Como $e=\dfrac{\cos\beta}{\cos\alpha},\quad\cos\beta=e\cos\alpha.$\\
Luego
\begin{align*}
C&=\sin^2\beta-\tan^2\alpha^2\beta=1-\cos^2\beta-\tan^2\alpha\cos^2\beta\\
&=1-\cos^2\beta(1+\tan^2\alpha)=1-\cos^2\beta\sec^2\alpha\underset{\overset{\uparrow}{\cos\beta}=e\cos\alpha}{=}1-e^2\cos^2\alpha\sec^2\alpha=1-e^2
\end{align*}
y la ecuación de $\Sigma$ puede ponerse así:
\begin{align*}
X^2+(1-e^2)Y^2+2h\tan^2\alpha\cos\beta Y-h^2\tan^2\alpha=0\quad\text{con}1-e^2>0\\
M=\left(\begin{array}{cc}A&B\\B&C\end{array}\right)=\left(\begin{array}{cc}1&0\\0&1-e^2\end{array}\right);\quad\delta=|M|=1-e^2&=C
\end{align*}
Para hallar el centro debe resolverse el sistema:$\begin{cases}Ax+By=-D\\ Bx+Cy=-E\end{cases}$\\
O sea\\
$\begin{cases}
1\cdot x+0\cdot y=0\\
0\cdot x+(1-e^2)y=-h\tan^2\alpha\cos\beta
\end{cases}$\\
Hay solución única:
\begin{align*}
\widetilde{h}=0,\quad\widetilde{k}=-\dfrac{h\tan^2\alpha\cos\beta}{1-e^2}&\underset{\uparrow}{<}0\\
&1-e^2>0\\
&\begin{array}{cc}\text{y por lo tanto se trata de}\\ \text{una cónica con centro único}\end{array}
\end{align*}
\begin{figure}[ht!]
\begin{center}
  \includegraphics[scale=0.4]{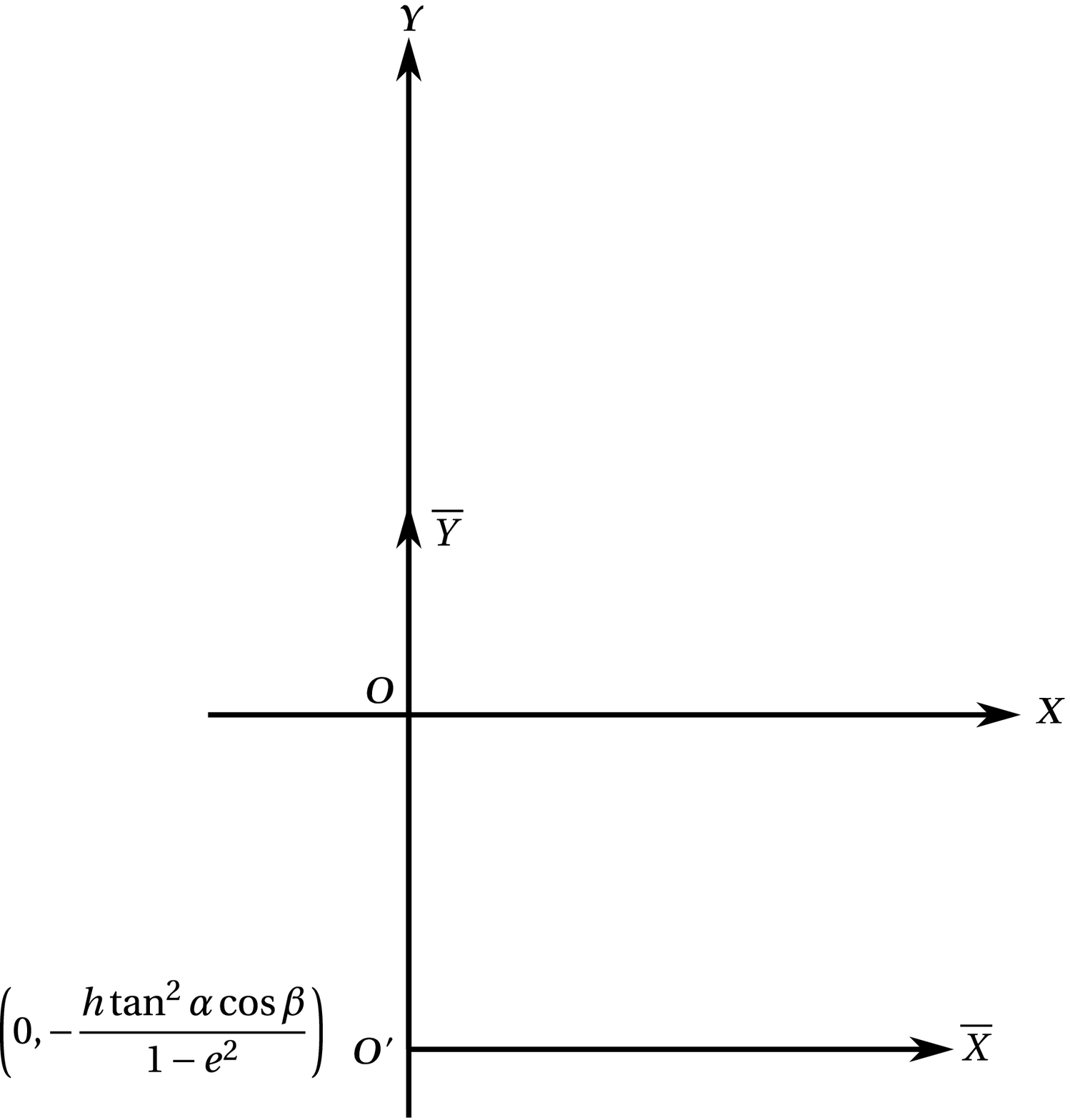}\\
  \caption{}\label{}
\end{center}
\end{figure}
\newline
Al trasladar los ejes $XY$ al punto $O'$ de coor. $\left(0,-\dfrac{h\tan^2\alpha\cos\beta}{1-e^2}\right)\diagup X$ se eliminan los términos lineales en $\star.$\\
La ecuación de la cónica $\Sigma\diagup\overline{X}\overline{Y}$ con origen en $O'$ es: $$\overline{X}^2+(1-e^2)\overline{Y}^2=-\dfrac{\Delta}{\delta}$$
\begin{align*}
\dfrac{\Delta}{\delta}&=\dfrac{1}{1-e^2}\left|\begin{array}{ccc}A&B&D\\B&C&E\\D&E&F\end{array}\right|=\dfrac{1}{1-e^2}\left|\begin{array}{ccc}1&0&0\\0&1-e^2&h\tan^2\alpha\cos\beta\\0&h\tan^2\alpha\cos\beta&-h^2\tan^2\alpha\end{array}\right|\\
&=\dfrac{1}{1-e^2}\left[-(1-e^2)h^2\tan^2\alpha-h^2\tan^4\cos^2\beta\right]\\
&\underset{\uparrow}{=}\left[-(1-e^2)h^2\tan^2\alpha-h^2e^2\tan^2\alpha(1-\cos^2\alpha)\right]\\
&\begin{cases}
\tan^4\alpha\cos^2\beta\underset{\overset{\uparrow}{\cos\beta}=e\cos\alpha}{=}\tan^4\alpha{(e\cos\alpha)}^2&=e^2\tan^2\alpha\dfrac{\sin^2\alpha}{\cancel{\cos^2\alpha}}\cancel{\cos^2\alpha}\\
&=e^2\tan^2\alpha\sin^2\alpha\\
&=e^2\tan^2\alpha(1-\cos^2\alpha)
\end{cases}\\
&=\dfrac{1}{1-e^2}\left[h^2e^2\tan^2\alpha\cos^2\alpha\cdot h^2\tan^2\alpha\right]\\
&=\dfrac{h^2\tan^2\alpha}{1-e^2}(e^2\cos^2\alpha-1)=\dfrac{h^2\tan^2\alpha}{1-e^2}(\cos^2\beta-1)
\end{align*}
Luego $$-\dfrac{\Delta}{\delta}=\dfrac{h^2\tan^2\alpha}{1-e^2}(1-\cos^2\beta)=\dfrac{h^2\tan^2\alpha\cdot\sin^2\beta}{1-e^2}$$
La ecuación de la cónica $\Sigma\diagup\overline{X}\overline{Y}$ es finalmente $$\overline{X}^2+(1-e^2)\overline{Y}^2=\dfrac{h^2\tan^2\alpha\sin^2\beta}{1-e^2}$$ que podemos finalmente escribir así: $$\dfrac{\overline{X}^2}{{\left(\dfrac{h\tan\alpha\sin\beta}{\sqrt{1-e^2}}\right)}^2}+\dfrac{\overline{Y}^2}{{\left(\dfrac{h\tan\alpha\sin\beta}{1-e^2}\right)}^2}=1$$
lo que nos dm. que la curva $\Sigma$ es una \underline{Elipse} de centro en $O'$.
\begin{figure}[ht!]
\begin{center}
  \includegraphics[scale=0.5]{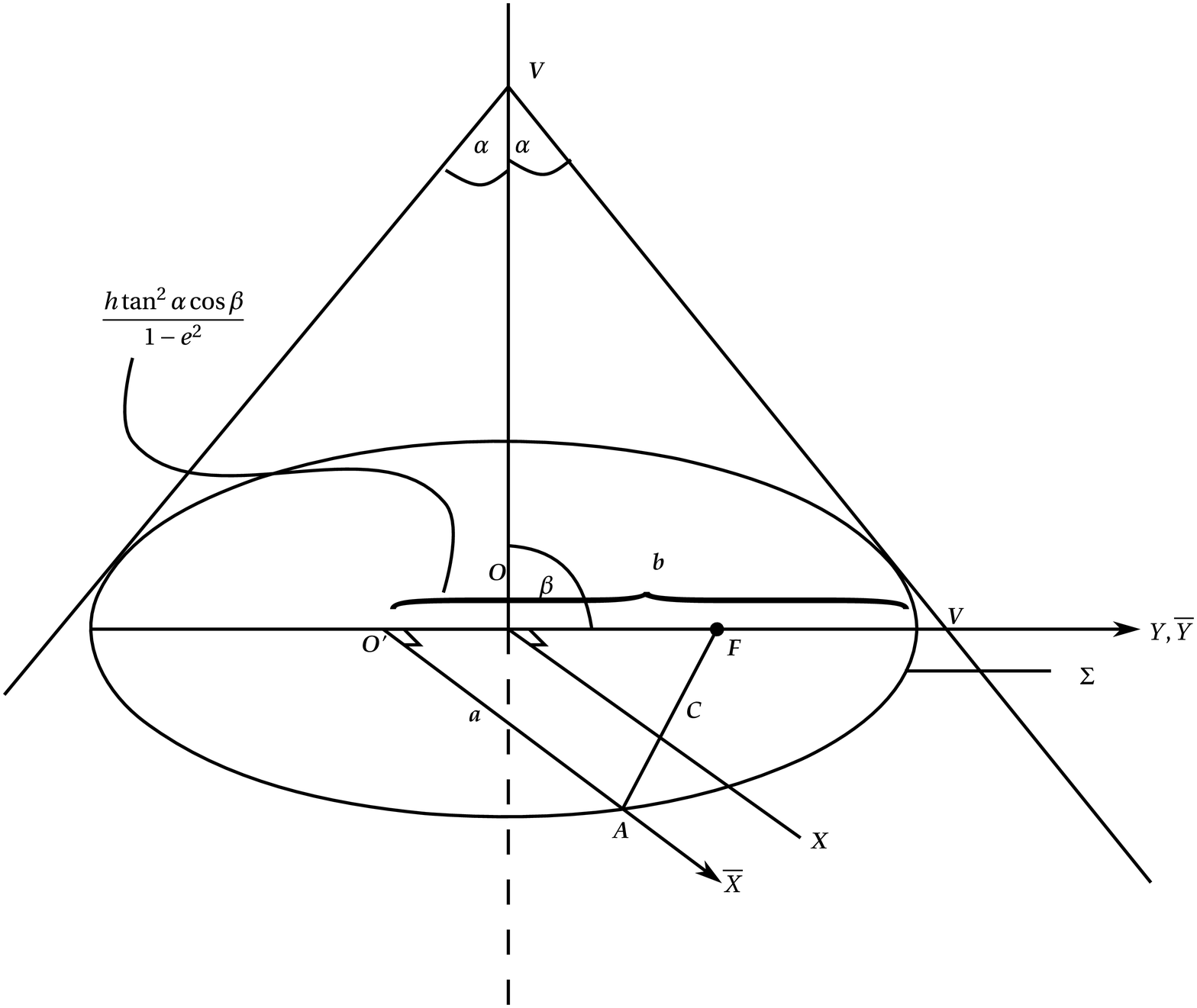}\\
  \caption{}\label{}
\end{center}
\end{figure}
\newline
en la que los semiejes son
\begin{align*}
O'A&=a=\dfrac{h\tan\alpha\sin\beta}{\sqrt{1-e^2}}\\
\intertext{y}\\
O'V&=b=\dfrac{h\tan\alpha\sin\beta}{1-e^2};\quad e=\dfrac{\cos\beta}{\cos\alpha}<1
\end{align*}
Vamos a dm. que $a<b$ con lo cual queda establecido que los focos de la elipse están en el eje $\overline{Y}$ y localizados así: $$0'F=c=\sqrt{b^2-a^2}$$ Sabemos que si $x<1,\quad x^2<x.$\\
Como $\sqrt{1-e^2}<1,\quad 1-e^2<\sqrt{1-e^2}.$ Luego $\dfrac{1}{\sqrt{1-e^2}}<\dfrac{1}{1-e^2}$ y por tanto, $\dfrac{h\tan\alpha\sin\beta}{\sqrt{1-e^2}}<\dfrac{h\tan\alpha\sin\beta}{1-e^2},$ o sea que $a<b.$\\
$$O'F=c=\sqrt{b^2-a^2}=h\tan\alpha\sin\beta\dfrac{1}{1-e^2}.$$
Sea $\epsilon$ la excentricidad de $\Sigma.$\\
$$c=0F'=\epsilon b\quad\therefore\quad\epsilon=\dfrac{c}{b}=\dfrac{h\tan\alpha\sin\beta\frac{e}{1-e^2}}{\dfrac{h\tan\alpha\sin\beta}{1-e^2}}=e=\dfrac{\cos\beta}{\cos\alpha}$$
Esto dm. que la excentricidad $\epsilon$ de la elipse es la misma para todas las elipses obtenidas al cortar el cono con los planos $\parallel$s a $\pi.$\\
\item[3)] Supongamos que $0<\beta<\alpha$ y consideremos la cónica $$\Sigma:X^2+\left(\sin^2\beta-\tan^2\alpha\cos^2\beta\right)Y^2+2h\tan^2\alpha\cos\beta Y-h^2\tan^2\alpha=0$$
De nuevo:
\begin{align*}
A&=1\\
B&=0\\
C&=\sin^2\beta-\tan^2\alpha\cos^2\beta\\
D&=0\\
E&=h\tan^2\alpha\cos\beta\\
F&=-h\tan^2\alpha
\end{align*}
Llamemos $e=\dfrac{\cos\beta}{\cos\alpha}.$ Como $\beta<\alpha,\quad\cos\beta>\cos\alpha$ y se tiene ahora que $e=\dfrac{\cos\beta}{\cos\alpha}>1\quad\therefore e^2-1\quad\text{y}1-e^2<0.$\\
Vamos a dm. que $\Sigma$ es una hipérbola de excentricidad $e$.\\
En este caso,
\begin{align*}
C&=\sin^2\beta-\tan^2\alpha\cos^2\beta=1-\cos^2\beta-\tan^2\alpha\cos^2\beta\\
&=1-\cos^2\beta(1-\tan^2\alpha)=1-\cos^2\beta\sec^2\alpha\\
&\underset{\nearrow}{=}1-e^2\cos^2\alpha\sec^2\alpha=1-e^2<0\\
&\cos\beta=e\cos\alpha
\end{align*}
y la ecuación de la cónica se puede escribir: $$X^2+(1-e^2)Y^2+2h\tan^2\alpha\cos\beta Y-h^2\tan^2\alpha=0$$
con $1-e^2<0.$\\
En análisis de centros es el mismos que hizo en el caso [2)].\\
La cónica tiene centro único en $O'\left(0,-\dfrac{h\tan^2\alpha\cos\beta}{1-e^2}\right),$ solo que ahora $1-e^2<0.$ O sea que el centro $O'$ está situado así:
\begin{figure}[ht!]
\begin{center}
  \includegraphics[scale=0.4]{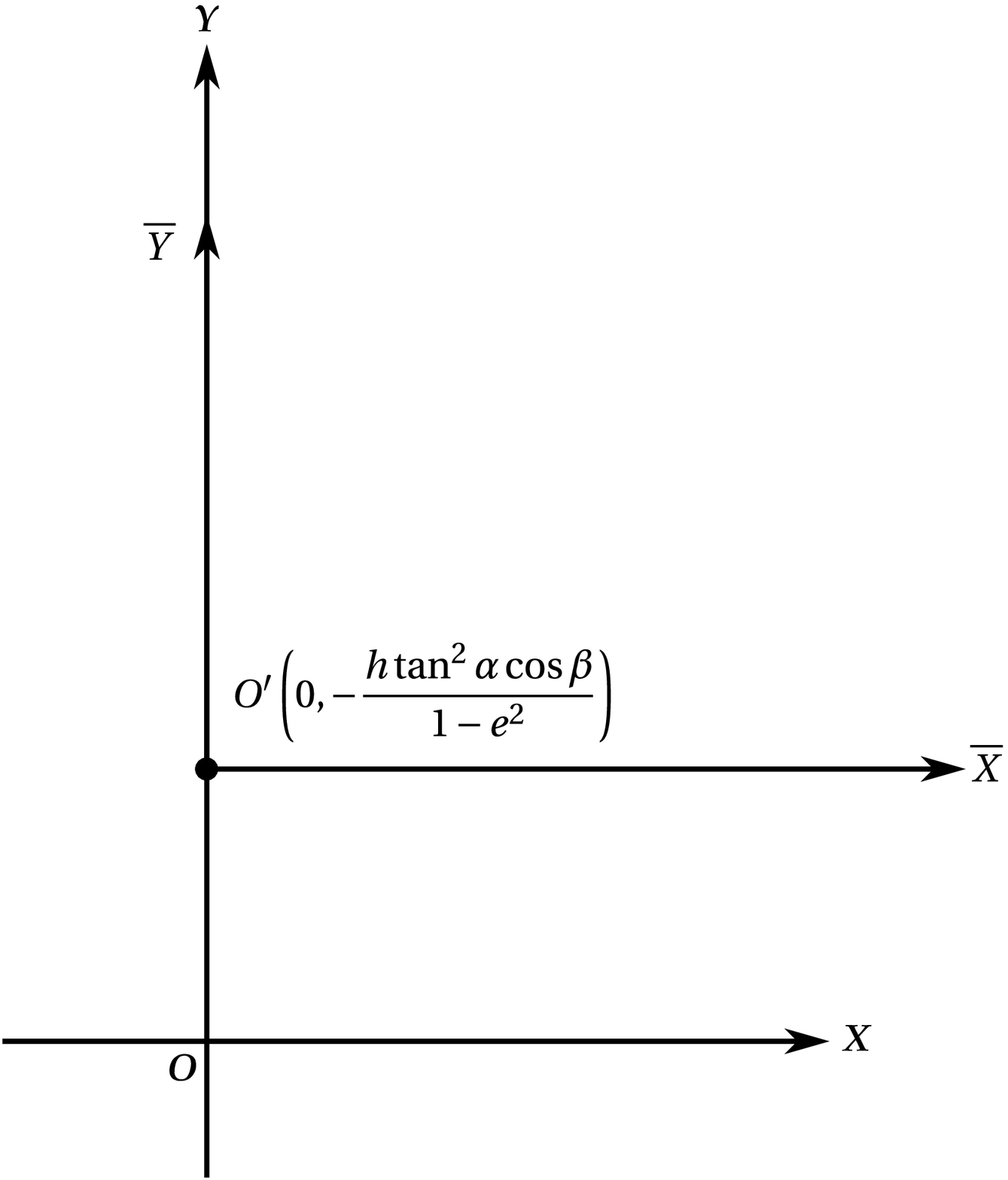}\\
  \caption{}\label{}
\end{center}
\end{figure}
\newline
Al trasladar los ejes al centro $O'$ de la cónica, desaparecen los términos lineales y la ecuación de $\Sigma\diagup\overline{X},\overline{Y}$ es: $$\overline{X}^2+(1-e^2)\overline{Y}^2=-\dfrac{\Delta}{\delta}\quad\text{con}\quad 1-e^2<0.$$
El cálculo de $-\dfrac{\Delta}{\delta}$ es el mismo que se llevó a cabo en el caso [2)]: $$-\dfrac{\Delta}{\delta}=\dfrac{h^2\tan^2\alpha\sin^2\beta}{1-e^2}\quad\text{con}\quad 1-e^2<0.$$
La ecuación de la cónica$\diagup\overline{X}\overline{Y}$ es
\begin{align*}
\overline{X}^2+(1-e^2)\overline{Y}^2&=\dfrac{h^2\tan^2\alpha\sin^2\beta}{1-e^2}\quad\text{y con}\quad 1-e^2<0,\\
\overline{X}^2-(1-e^2)\overline{Y}^2&=-\dfrac{h^2\tan^2\alpha\sin^2\beta}{e^2-1}\quad\text{con}\quad e^2-1<0
\end{align*}
O sea que la ecuación es: $$\dfrac{\overline{Y}^2}{{\left(\dfrac{h\tan\alpha\sin\beta}{e^2-1}\right)}^2}-\dfrac{\overline{X}^2}{{\left(\dfrac{h\tan\alpha\sin\beta}{\sqrt{e^2-1}}\right)}^2}=1$$
lo que nos indica que $\Sigma$ es una \underline{Hipérbola} de centro en $O'.$\\
y con semiejes
\begin{align*}
O'V'&=b=\dfrac{h\tan\alpha\sin\beta}{e^2-1}\\
V'N&=a=\dfrac{h\tan\alpha\sin\beta}{\sqrt{e^2-1}}\quad[Fig.]\\
O'N&=c=\sqrt{a^2+b^2}=\sqrt{\dfrac{h^2\tan^2\alpha\sin^2\beta}{e^2-1}+\dfrac{h^2\tan^2\alpha\sin^2\beta}{{(e^2-1)}^2}}=\\
&=h\tan\alpha\sin\beta\dfrac{e}{e^2-1}
\end{align*}
\begin{figure}[ht!]
\begin{center}
  \includegraphics[scale=0.5]{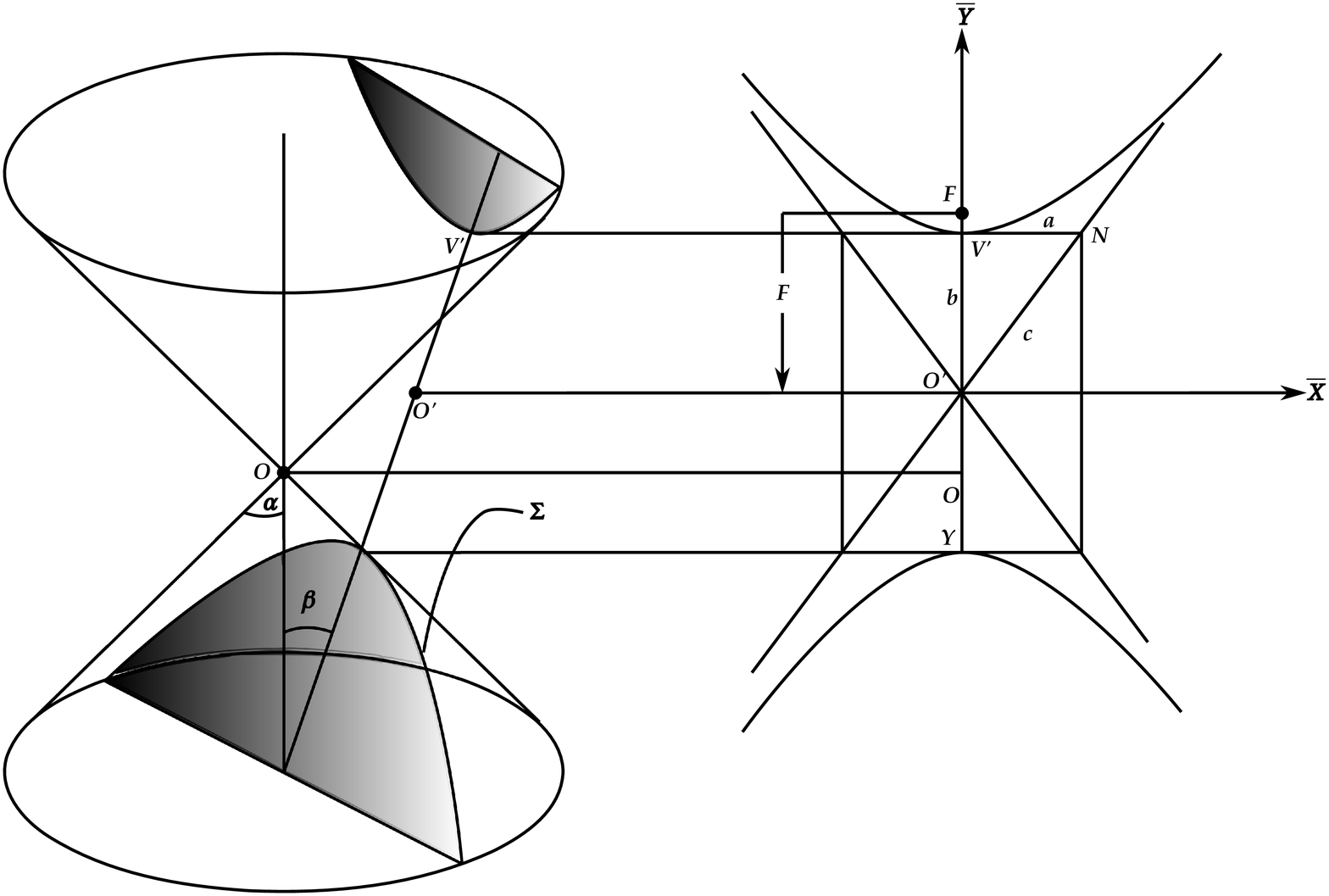}\\
  \caption{}\label{}
\end{center}
\end{figure}
\newline
Llamemos $\epsilon$ a la excentricidad de $\Sigma$ y $F$ al foco.\\
$$c=0'F=\epsilon p\quad\therefore\quad\epsilon=\dfrac{c}{b}=\dfrac{h\tan\alpha\sin\beta\dfrac{e}{e^2-1}}{\dfrac{h\tan\alpha\sin\beta}{e^2-1}}=e=\dfrac{\cos\beta}{\cos\alpha}$$
lo que nos dm. que todas las hipérbolas obtenidas al cortar el cono con los planos $\parallel$s a $\pi$ tienen la misma excentricidad: $$\epsilon=\dfrac{\cos\beta}{\cos\alpha}>1$$
\end{enumerate}
\newpage
\section{Intersección de un cilindro circular recto con un plano que corta al eje del cilindro}
Consideremos el cilindro circular recto de radio $R$ y cuyo eje es el eje $z$.  Su ecuación es\\
$\begin{cases}
x^2+y^2=R^2\\
z\in\mathbb{R}
\end{cases}$\\
Si lo cortamos con un plano $\pi$ que pase por $O$ y haga con el eje $z$ un ángulo $\beta$, la curva que se obtiene es una \underline{elipse} de excentricidad $\epsilon=\cos\beta.$ Vamos a demostrar esto.
\begin{figure}[ht!]
\begin{center}
  \includegraphics[scale=0.4]{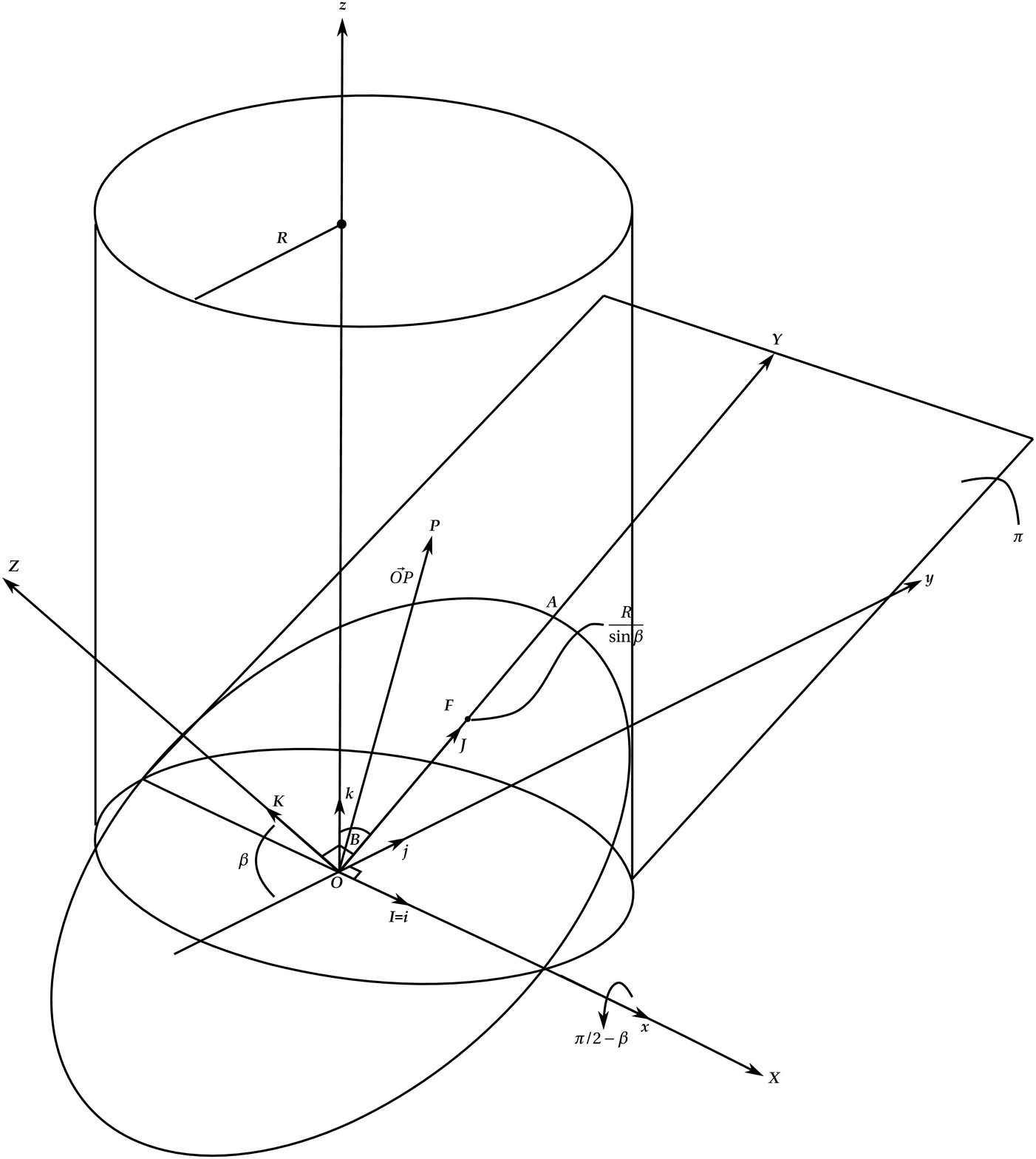}\\
  \caption{}\label{}
\end{center}
\end{figure}
\newpage
Si giramos los ejes $xyz$ un ángulo $\pi/2-\beta\circlearrowleft$ al rededor del eje $x$ obtenemos un sistema $XYZ$ en el que eje $Z$ es $\perp$ al plano $\pi$ y los ejes $X$ y $Y$ están en $\pi.$\\
Sea $P$ un punto del cilindro de coord. $(x,y,z)$ y $(X,Y,Z)$ respecto a los ejes $xyz$ y $XYZ.$\\
\begin{align*}
[\vec{OP}]_{ijk}=\left(\begin{array}{cc}x\\y\\z\end{array}\right)&=[I]_{ijk}^{IJK}[\vec{OP}]_{IJK}\\
\intertext{O sea que}\\
\left(\begin{array}{cc}x\\y\\z\end{array}\right)&=\left(\begin{array}{ccc}1&0&\\0&\sin\beta&-\cos\beta\\0&\cos\beta&\sin\beta\end{array}\right)\left(\begin{array}{cc}X\\Y\\Z\end{array}\right)\\
\therefore\quad x&=X\\
y&=\sin\beta Y-\cos\beta Z\\
z&=\cos\beta Y+\sin\beta Z
\end{align*}
Así que la ecuación del cilindro$\diagup XYZ$ es:
\begin{align*}
X^2+{\left(\sin\beta Y-\cos\beta Z\right)}^2=R^2\\
X^2+\sin^2\beta Y^2-2\cos\beta\sin\beta YZ+\cos^2\beta Z^2=R^2
\end{align*}
Si $Z=0,$ obtenemos la ecuación de la sección del cilindro con el plano:
\begin{align*}
X^2+\sin^2 Y^2&=R^2\\
\dfrac{X^2}{R^2}+\dfrac{Y^2}{{\left(\dfrac{R}{\sin\beta}\right)}^2}&=1\quad\text{Elipse}\\
a&=R\\
b&=\dfrac{R}{\sin\beta}>R=a,
\end{align*}
lo cual explica que los focos están en el eje $Y.$\\
$$c=OF=\sqrt{{\left(\dfrac{R}{\sin\beta}\right)}^2-R^2}=\dfrac{R\cos\beta}{\sin\beta}=\dfrac{R}{\tan\beta}$$
Si $\epsilon$ es la excentricidad de la curva, $$\epsilon=\dfrac{c}{b}=\dfrac{R\dfrac{\cos\beta}{\sin\beta}}{\dfrac{R}{\sin\beta}}=\cos\beta$$

\end{document}